\documentclass[11pt]{amsart}
\usepackage{pdfrender,listings}

\usepackage[dvipsnames,svgnames,x11names]{xcolor}
\usepackage{textgreek,bbm,url,graphicx,verbatim,amssymb,enumerate,stmaryrd,
booktabs,lmodern,mathtools,microtype}
\SetSymbolFont{stmry}{bold}{U}{stmry}{m}{n}
\usepackage[
  pagebackref,
  colorlinks,
  citecolor=Mahogany,
  linkcolor=Mahogany,
  urlcolor=Mahogany,
  filecolor=Mahogany
]{hyperref}
\usepackage[capitalize]{cleveref}
\usepackage[mathscr]{euscript}
\usepackage{appendix}
\usepackage{tikz,tikz-cd}
\usetikzlibrary{matrix,calc,positioning,arrows,decorations.pathreplacing,
patterns,arrows}
\usepackage[a4paper,
            bindingoffset=0.2in,
            left=0.8in,
            right=1in,
            top=1in,
            bottom=1in,
            footskip=.25in]{geometry}
\usepackage{enumitem}
\usepackage{bm}

\definecolor{leanblue}{RGB}{0,0,160}
\definecolor{leancomment}{RGB}{96,96,96}
\newcommand{\leanletter}[1]{\ensuremath{#1}}
\newcommand{\leansymbol}[1]{\textpdfrender{TextRenderingMode=2,LineWidth=0.12pt}{\ensuremath{#1}}}
\lstdefinelanguage{LeanShowcase}{
  sensitive=true,
  alsoletter={_.'},
  morekeywords={theorem,lemma,def,abbrev,example,structure,class,instance,
    inductive,where,open,import,namespace,section,end,variable,universe,
    noncomputable,axiom,by,have,show,from,fun,let,in,match,with,do,
    if,then,else,calc,sorry,Prop,Type},
  morecomment=[l]{--},
  morecomment=[s]{/-}{-/},
  morestring=[b]",
}

\lstdefinestyle{lean}{
  language=LeanShowcase,basicstyle=\small\ttfamily,
  keywordstyle=\color{leanblue}\bfseries,
  commentstyle=\color{leancomment}\itshape,stringstyle=\color{leancomment},
  columns=fullflexible,keepspaces=true,showstringspaces=false,
  breaklines=true,breakatwhitespace=true,breakindent=1.5em,
  aboveskip=7pt,belowskip=7pt,tabsize=2,
  literate=
    {ℝ}{{\leanletter{\mathbb R}}}1 {ℕ}{{\leanletter{\mathbb N}}}1
    {ℕ∞}{{\leanletter{\mathbb N_\infty}}}2 {⊤}{{\leansymbol{\top}}}1
    {∀}{{\leansymbol{\forall}}}1 {∃}{{\leansymbol{\exists}}}1
    {→}{{\leansymbol{\to}}}1 {↔}{{\leansymbol{\leftrightarrow}}}1
    {∧}{{\leansymbol{\wedge}}}1 {∈}{{\leansymbol{\in}}}1
    {≠}{{\leansymbol{\ne}}}1 {≤}{{\leansymbol{\le}}}1 {≥}{{\leansymbol{\ge}}}1
    {⨆}{{\leansymbol{\textstyle\bigsqcup}}}1
    {ε}{{\leanletter{\varepsilon}}}1 {δ}{{\leanletter{\delta}}}1
    {₀}{{\leanletter{{}_0}}}1
    {π}{{\leanletter{\pi}}}1 {∫}{{\leansymbol{\textstyle\int}}}1 {×}{{\leansymbol{\times}}}1
    {×ˢ}{{\leansymbol{\times}\leanletter{{}^{\mathrm s}}}}2
}
\lstnewenvironment{leancode}[1][]{\lstset{style=lean,#1}}{}
\newcommand{\lean}{\lstinline[style=lean]}

\DeclareFontFamily{U}{mathx}{}
\DeclareFontShape{U}{mathx}{m}{n}{<-> mathx10}{}
\DeclareSymbolFont{widecheckfont}{U}{mathx}{m}{n}
\SetSymbolFont{widecheckfont}{bold}{U}{mathx}{m}{n}
\DeclareMathAccent{\widecheck}{\mathord}{widecheckfont}{"71}

\newtheorem{theorem}{Theorem}[section]
\newtheorem{lemma}[theorem]{Lemma}
\newtheorem{proposition}[theorem]{Proposition}

\newtheorem*{corollary*}{Corollary}

\theoremstyle{definition}
\newtheorem{definition}[theorem]{Definition}

\theoremstyle{remark}

\newtheorem{remark}[theorem]{Remark}

\AddToHook{env/theorem/begin}{\crefalias{theorem}{theorem}}
\AddToHook{env/lemma/begin}{\crefalias{theorem}{lemma}}
\AddToHook{env/proposition/begin}{\crefalias{theorem}{proposition}}
\AddToHook{env/corollary/begin}{\crefalias{theorem}{corollary}}
\AddToHook{env/conjecture/begin}{\crefalias{theorem}{conjecture}}
\AddToHook{env/atheorem/begin}{\crefalias{atheorem}{atheorem}}
\AddToHook{env/acorollary/begin}{\crefalias{atheorem}{acorollary}}
\AddToHook{env/definition/begin}{\crefalias{theorem}{definition}}
\AddToHook{env/axiom/begin}{\crefalias{theorem}{axiom}}
\AddToHook{env/notation/begin}{\crefalias{theorem}{notation}}
\AddToHook{env/convention/begin}{\crefalias{theorem}{convention}}
\AddToHook{env/example/begin}{\crefalias{theorem}{example}}
\AddToHook{env/nonexample/begin}{\crefalias{theorem}{nonexample}}
\AddToHook{env/remark/begin}{\crefalias{theorem}{remark}}
\AddToHook{env/warning/begin}{\crefalias{theorem}{warning}}
\AddToHook{env/question/begin}{\crefalias{theorem}{question}}
\AddToHook{env/discussion/begin}{\crefalias{theorem}{discussion}}

\newcommand{\subref}[2]{\hyperref[#2]{\ref{#1}.\ref{#2}}}

\numberwithin{equation}{section}
\crefname{equation}{}{}

\newcommand{\sq}[1]{\widetilde{#1}}
\newcommand{\R}{\mathbb{R}}

\newcommand{\N}{\mathbb{N}}

\newcommand{\C}{\mathbb{C}}

\newcommand{\mc}[1]{\mathcal{#1}}

\newcommand{\ov}[1]{\overline{#1}}
\newcommand{\ang}[1]{\langle #1 \rangle}

\date{\today}
\title{Self-Similar Solutions for the gSQG patch equations}
\author{Javier G\'omez-Serrano}
\address{Department of Mathematics, Brown University, 02912 Providence, Rhode Island, USA}
\email{javier\_gomez\_serrano@brown.edu}
\author{Marcus Pasquariello}
\address{Department of Mathematics, Brown University, 02912 Providence, Rhode Island, USA}
\email{marcus\_pasquariello@brown.edu}

\begin{document}

\begin{abstract}
    The inviscid generalized surface quasi-geostrophic (gSQG) equations are a family of active scalar equations parametrized by a scalar $\alpha$ that includes the two-dimensional Euler equations ($\alpha = 0)$ and the SQG equation $(\alpha = 1)$. 
    We prove that, for every $\alpha \in (\frac{7}{4}, 2)$ and every $T > 0$, there exists a self-similar patch solution 
    whose boundary is a smooth, unbounded graph for $0 \leq t < T$ and develops a corner singularity at time $T$.      The proof has been formally verified in Lean~4.
\end{abstract}

\maketitle

\tableofcontents

\section{Introduction}\label{sec:intro}

\subsection{The gSQG patch equations}

In this paper we are concerned with the inviscid generalized surface quasi-geostrophic (gSQG) equations
\begin{equation}\label{eqn:gsqg}
\begin{cases}
    \partial_t \theta + u \cdot \nabla \theta = 0, & (x, t) \in \R^2 \times \R_+, \\
    u = -\nabla^\perp(-\Delta)^{-1+\alpha / 2}\theta, \\
    \theta(\cdot, 0) = \theta_0,
\end{cases}
\end{equation}
where $\alpha \in (0,2)$. The family \eqref{eqn:gsqg} was proposed by C\'ordoba--Fontelos--Mancho--Rodrigo \cite{Cordoba-Fontelos-Mancho-Rodrigo:evidence-singularities-contour-dynamics} as an interpolation between the 2D incompressible Euler equations, which correspond to $\alpha = 0$, and the surface quasi-geostrophic (SQG) equation, which corresponds to $\alpha = 1$ and models the evolution of the potential temperature in rapidly rotating, stratified flows \cite{Constantin-Majda-Tabak:formation-fronts-qg,Held-Pierrehumbert-Garner-Swanson:sqg-dynamics}. In the range $\alpha \in (1,2)$, the velocity $u$ is more singular than $\theta$, and the equations were first studied by Chae--Constantin--C\'ordoba--Gancedo--Wu \cite{Chae-Constantin-Cordoba-Gancedo-Wu:gsqg-singular-velocities}.

We will work in the patch setting, where $\theta(\cdot, t) = \mathbbm{1}_{\Omega(t)}$ is the characteristic function of a domain $\Omega(t)$ transported by the flow. In this case, the evolution is completely determined by the boundary $\partial\Omega(t)$, and \eqref{eqn:gsqg} reduces to a nonlocal one-dimensional equation for the interface. We will consider unbounded patches whose boundary is a graph,
\[\Omega(t) = \{(x_1, x_2) \in \R^2 : x_2 < F(x_1, t)\}.\]
Then, the interface $F$ satisfies the front equation
\begin{equation}\label{eqn:front_equation}
    \partial_t F(x,t) = c_\alpha\int_{\R}
    \frac{\partial_x F(x,t)-\partial_x F(x',t)}
    {\big((x-x')^2+(F(x,t)-F(x',t))^2\big)^{\alpha/2}}\,dx',
\end{equation}
where $c_\alpha > 0$ is a constant that depends on the normalization of the kernel of $(-\Delta)^{-1 + \alpha/2}$.  For $\alpha \in (1,2)$, the integral in \eqref{eqn:front_equation} converges absolutely as long as $\partial_x F$ is bounded and Lipschitz, and no decay of $F$ at infinity is required.

Two features of \eqref{eqn:front_equation} will play a central role in this paper. First, the equation is invariant under translations in space and time, $F(x,t) \mapsto F(x - x_0, t - t_0) + h_0$, under the scaling
\[F(x,t) \mapsto \frac{1}{c}F(cx, c^{\alpha}t), \qquad c > 0,\]
and under the reflection $F(x,t) \mapsto F(-x,-t)$. In particular, interfaces that are homogeneous of degree one, such as the corner $F(x) = s\, x_-$ with $x_- = \max\{-x, 0\}$, are scale-invariant. Moreover, since \eqref{eqn:gsqg} is invariant under rotations, so is the patch problem: rotating the domain $\Omega(t)$ of a patch solution yields another patch solution. In the graph formulation \eqref{eqn:front_equation}, this symmetry is only visible for rotations that keep the interface a graph, as is the case for small rotations of interfaces with small slope. Together with the translation invariance, this shows that the normalization of the corner $s\, x_-$, with one flat side and the vertex at the origin, entails no loss of generality: it represents any corner with interior angle $\pi + \arctan s$, located at any point and with any orientation. Second, the linearization of \eqref{eqn:front_equation} around the flat interface $F \equiv 0$ is, after normalizing $c_\alpha$, the dispersive equation
\[\partial_t F = |D|^{\alpha-1}\partial_x F,\]
with dispersion relation $\xi|\xi|^{\alpha-1}$. Dispersion is the mechanism behind the global existence of small, smooth and decaying perturbations of the flat interface \cite{Cordoba-GomezSerrano-Ionescu:global-generalized-sqg-patch}. As we will see, it is also the mechanism behind the singularities constructed in this paper.

\subsection{Historical background}

The local regularity theory for \eqref{eqn:gsqg} and its patch version is well understood. For smooth solutions, local well-posedness in Sobolev spaces was established by Chae--Constantin--C\'ordoba--Gancedo--Wu \cite{Chae-Constantin-Cordoba-Gancedo-Wu:gsqg-singular-velocities} in the range $\alpha \in (1,2)$ and extended down to the critical Sobolev exponent by Hu--Kukavica--Ziane \cite{Hu-Kukavica-Ziane:local-existence-active-scalar} (see also \cite{Chae-Constantin-Wu:inviscid-models-2d-euler-sqg,Inci:well-posedness-sqg,Yu-Zhang:local-wellposedness-sqg}). Below that threshold, the equations are ill-posed: strong ill-posedness, nonexistence and instantaneous loss of regularity were proved by C\'ordoba--Mart\'inez-Zoroa \cite{Cordoba-MartinezZoroa:ill-posedness-sqg,Cordoba-MartinezZoroa:ill-posedness-gsqg}, Jeong--Kim \cite{Jeong-Kim:illposedness-sqg}, C\'ordoba--Mart\'inez-Zoroa--O\.za\'nski \cite{Cordoba-MartinezZoroa-Ozanski:instantaneous-loss-regularity-sqg}, C\'ordoba--Lucas-Manch\'on--Mart\'inez-Zoroa \cite{Cordoba-LucasManchon-MartinezZoroa:ill-posedness-gsqg-sobolev} and, in the range $\alpha \in (1,2)$ and through a degenerate dispersion mechanism, Chae--Jeong--Oh \cite{Chae-Jeong-Oh:illposedness-degenerate-dispersion-gsqg}. For patches, we refer to Rodrigo \cite{Rodrigo:evolution-sharp-fronts-qg}, Gancedo \cite{Gancedo:existence-alpha-patch-sobolev}, Chae--Constantin--C\'ordoba--Gancedo--Wu \cite{Chae-Constantin-Cordoba-Gancedo-Wu:gsqg-singular-velocities} and Gancedo--Patel \cite{Gancedo-Patel:local-existence-blowup-gsqg} for local existence results, and to C\'ordoba--C\'ordoba--Gancedo \cite{Cordoba-Cordoba-Gancedo:uniqueness-sqg-patch} for the uniqueness of patch solutions. For fronts, that is, patches whose boundary is a graph as above, analytic solutions of the SQG equation were constructed by Fefferman--Rodrigo \cite{Fefferman-Rodrigo:analytic-sharp-fronts-sqg} (see also \cite{Fefferman-Rodrigo:almost-sharp-fronts-sqg}). Hunter--Shu--Zhang \cite{Hunter-Shu:regularized-approximate-model-sqg,Hunter-Shu-Zhang:local-wellposedness-model-sqg,Hunter-Shu-Zhang:global-model-front-sqg,Hunter-Shu-Zhang:global-gsqg} derived regularized and approximate equations for gSQG fronts and proved local and global well-posedness results for them, and Ai--Avadanei \cite{Ai-Avadanei:well-posedness-sqg-front} proved well-posedness for the SQG front equation. In the range $\alpha \in (1,2)$ considered here, C\'ordoba--G\'omez-Serrano--Ionescu \cite{Cordoba-GomezSerrano-Ionescu:global-generalized-sqg-patch} showed that small, smooth and sufficiently decaying perturbations of the flat interface lead to global solutions of \eqref{eqn:front_equation}, which disperse and undergo modified scattering as $t \to \infty$.

The question of whether the gSQG evolution can develop singularities in finite time from smooth data remains open, both in the smooth and in the patch setting (see \cite{Kiselev:open-problems}). In the patch setting, there are several numerical scenarios that suggest the formation of singularities. C\'ordoba--Fontelos--Mancho--Rodrigo \cite{Cordoba-Fontelos-Mancho-Rodrigo:evidence-singularities-contour-dynamics} considered two patches and observed an asymptotically self-similar collapse, in which the distance between the patches goes to zero while the curvature of their boundaries blows up. Scott--Dritschel \cite{Scott-Dritschel:self-similar-sqg,Scott-Dritschel:self-similar-sqg-patch} found numerical evidence of a self-similar cascade of filament instabilities and of a scale-invariant singularity of the SQG patch, with blow-up of the curvature of the boundary (see also Scott \cite{Scott:scenario-singularity-quasigeostrophic}). In the presence of a boundary, Kiselev--Ryzhik--Yao--Zlato\v{s} \cite{Kiselev-Ryzhik-Yao-Zlatos:singularity-alpha-patch-boundary} proved finite-time singularity formation for gSQG patches with small $\alpha$ on the half-plane (see also \cite{Kiselev-Yao-Zlatos:local-regularity-sqg-patch-boundary,Gancedo-Patel:local-existence-blowup-gsqg,Miao-Tan-Xue-Xue:singularity-generalized-sqg-patches-halfplane}), and Zlato\v{s} \cite{Zlatos:local-regularity-singularity-gsqg-halfplane} obtained finite-time blow-up from smooth data for the gSQG equation on the half-plane. On the whole plane, Jeon--Zlato\v{s} \cite{Jeon-Zlatos:touching-gsqg-patches} proved singularity formation for touching gSQG patches with small $\alpha$. In the smooth setting, Kiselev--Nazarov \cite{Kiselev-Nazarov:simple-energy-pump-sqg} constructed solutions of the SQG equation that start arbitrarily small and become arbitrarily large at a later time, and He--Kiselev \cite{He-Kiselev:small-scale-creation-sqg} constructed solutions that either blow up in finite time or exhibit exponential growth of H\"older norms. Singular solutions with infinite energy were constructed by Castro--C\'ordoba \cite{Castro-Cordoba:infinite-energy-sqg} (see also \cite{Hou-Qin-Sire-Wu:self-similar-blowup-1d-gsqg}), and Bronzi--Guimar\~aes--Mondaini \cite{Bronzi-Guimaraes-Mondaini:locally-self-similar-blowup-gsqg} ruled out broad classes of locally self-similar blow-up profiles. When an external force is allowed, C\'ordoba--Dom\'inguez--Lucas-Manch\'on--Mart\'inez-Zoroa \cite{Cordoba-Dominguez-LucasManchon-MartinezZoroa:blowup-forced-gsqg} constructed classical, finite-energy solutions of the forced gSQG equations in the range $\alpha \in (1,2)$ that blow up in finite time, with both the initial datum and the force in the Sobolev well-posedness regime. Several singular scenarios have also been ruled out: splash singularities, in which the interface touches itself while remaining regular, cannot occur, as shown by Gancedo--Strain \cite{Gancedo-Strain:absence-splash-muskat-SQG}, Kiselev--Luo \cite{Kiselev-Luo:nonexistence-splash-sqg} and Jeon--Zlato\v{s} \cite{Jeon-Zlatos:no-splash-gsqg}. This suggests that singularities of a single patch, if they exist, are local and of self-similar type.

Self-similar solutions and patches with singular boundaries have attracted considerable attention in recent years. Elgindi--Jeong \cite{Elgindi-Jeong:singular-vortex-patches-i,Elgindi-Jeong:singular-vortex-patches-ii} developed a well-posedness theory for vortex patches with corners for the 2D Euler equations and studied their long-time dynamics. Garc\'ia--G\'omez-Serrano \cite{Garcia-GomezSerrano:self-similar-spirals-gsqg} constructed self-similar spiral solutions of \eqref{eqn:gsqg} for $\alpha \in (0,1)$, and Abe--G\'omez-Serrano--Jeong \cite{Abe-GomezSerrano-Jeong:homogeneous-steady-states-gsqg} studied homogeneous steady states of \eqref{eqn:gsqg}.

\subsection{Main result}

The main goal of this paper is to construct solutions of the gSQG patch equation \eqref{eqn:front_equation} whose interface is $C^\infty$ for $t < T$ and develops a corner at time $T$. The solutions are self-similar, and they are small perturbations of the flat interface in the sense that the slope $\sqrt{\epsilon}$ of the corner is small. Our main result is the following theorem:

\begin{theorem}\label{thm:main}
Let $\alpha \in (\frac{7}{4}, 2)$. Then, there exists $\epsilon_0 > 0$ such that for every $0 < \epsilon < \epsilon_0$ and every $T > 0$ there exists a classical solution $F : \R \times (-\infty, T) \to \R$ of \eqref{eqn:front_equation} of the form
\begin{equation}\label{eqn:intro_self_similar_ansatz}
    F(x,t) = \sqrt{\epsilon}\,L(t)\, f_\epsilon\left(\frac{x}{L(t)}\right), \qquad L(t) = (\alpha(T-t))^{1/\alpha},
\end{equation}
where the profile $f_\epsilon \in C^\infty(\R)$ satisfies
\[|\partial^j(f_\epsilon(y) + y)| \lesssim \ang{y}^{-\alpha + 1 -j}, \quad 0 \leq j \leq 2, \quad y \leq 0,\]
and
\[|f_\epsilon(y)| \lesssim \ang{y}^{-\alpha + 1}, \quad |\partial^{j+1} f_\epsilon(y)| \lesssim \ang{y}^{-\frac{\alpha}{2(\alpha-1)} + \frac{j}{\alpha - 1}}, \quad 0 \leq j \leq 1, \quad y \geq 0.\]
Moreover, the interface converges to a corner at time $T$:
\begin{equation}\label{eqn:intro_corner_convergence}
    \|F(\cdot, t) - \sqrt{\epsilon}\, x_-\|_{L^\infty(\R)} \lesssim \sqrt{\epsilon} (T - t)^{1/\alpha},
\end{equation}
and its curvature blows up at the origin: for every $\delta > 0$ there exist $c = c(\epsilon) > 0$ and $t_\delta < T$ such that
\begin{equation}\label{eqn:intro_curvature_blowup}
    \|\partial_x^2 F(\cdot, t)\|_{L^\infty(-\delta, \delta)} \geq c (T-t)^{-1/\alpha}, \qquad t_\delta < t < T.
\end{equation}
\end{theorem}

\begin{remark}
Since \eqref{eqn:front_equation} is invariant under $(x,t) \mapsto (-x,-t)$, the function $\sq{F}(x,t) = F(-x,-t)$ with $T = 0$ is a solution of \eqref{eqn:front_equation} for $t > 0$ with initial datum the corner $\sq{F}(x,0) = \sqrt{\epsilon}\, x_+$. Therefore, Theorem~\ref{thm:main} can also be read as a desingularization result: a corner with small slope is instantaneously smoothed into a $C^\infty$ interface.
\end{remark}


\begin{remark}
The restriction $\alpha > \frac{7}{4}$ comes from the choice of the weighted spaces $X_\gamma$ and $Y_\gamma$ in Section~\ref{sec:fixed_point}, which requires a decay exponent $\gamma$ with $\frac{\alpha+2}{2(\alpha-1)} < \gamma \leq 2\alpha - 1$. This interval is nonempty if and only if $\alpha(4\alpha - 7) > 0$, that is, $\alpha > \frac{7}{4}$. We make no claim that $\alpha = \frac{7}{4}$ is optimal.
\end{remark}

\begin{remark}
Although we work with the single equation \eqref{eqn:front_equation}, none of the ingredients of the proof is specific to the gSQG kernel, and we view the argument as an instance of a general strategy for constructing corner-type self-similar solutions of quasilinear equations whose linearization around a trivial state is a one-dimensional dispersive equation with a homogeneous dispersion relation. The self-similar linearized equation $\mc{L}_0 f = 0$ admits an explicit solution with corner asymptotics at $-\infty$ and an Airy-type oscillatory tail at $+\infty$ (\cref{lem:approx_linear_seed}). This is the same object that underlies the self-similar solutions of the modified Korteweg--de Vries equation \cite{Perelman-Vega:self-similar-mkdv,Correia-Cote-Vega:asymptotics-mkdv} and the corner-forming self-similar solutions of the binormal flow \cite{Gutierrez-Rivas-Vega:self-similar-binormal,Banica-Vega:stability-selfsimilar-binormal}, where a corner at time $T$ is likewise smoothed instantaneously. We expect to carry out a more global program in the future.
\end{remark}

\subsection*{Acknowledgements}

JGS and MP were partially supported by the National Science
Foundation through Grants DMS-2247537, DMS-2434314, and DMS-2554957. We thank Gonzalo Cao-Labora and Eduardo Garc\'ia-Ju\'arez for many helpful comments and discussions over the years.

\subsection*{AI Declaration}

The first author started working on this project in 2015. The second author joined the project in early March 2026, and over the following few months, the main ideas of this project were developed by the authors, with little AI assistance. As the capabilities of AI improved, we began to use AI tools more in our workflow. \par
We used GPT-5.5, GPT-5.6, and GPT-6 during several stages of the subsequent research process. During the exploratory phase of this project, we used GPT-5.5 and GPT-5.6 to perform computations and numerical experiments and conduct literature searches that helped us develop the current proof strategy. During the writing phase, we occasionally used GPT-6 to fill in computational details, especially in the section containing the construction of the approximate solution. However, this paper was largely human written.  More specifically, almost all of the content in \cref{sec:intro}, \cref{sec:main_ideas}, and \cref{sec:fixed_point} was generated solely by the authors. GPT-6 was used occasionally to write formulas here, but all of the exposition was written by the authors. More AI assistance was used in the remaining sections, where we used GPT-6 to assist with individual statements and proof steps. In \cref{sec:analysis_equation}, GPT-6 assisted in writing some of the steps in the proofs. In \cref{sec:approximate_solution}, GPT-6 was used to decompose the residual, write a few of the lemmas for the residual terms, and fill in some portions of the proofs, especially for the later terms $\mc{E}_5$-$\mc{E}_9$. In \cref{sec:linear_inverse} and \cref{sec:nonlinear_estimates}, GPT-6 was used to fill in some small details in various proof steps. Finally, in the appendix, GPT-6 generated the proof of \cref{lem:appendix_triangle} and \cref{lem:appendix_oscillatory_multiplier}. The proofs were then checked and modified by the authors.   \par 
The Lean code was autoformalized by a mixture of GPT-6 for the blueprint and Opus 5.5 for the execution, needing 4 passes to further refine the blueprint whenever the models got stuck, or when formalizing an auxiliary result from the literature such as the interior Schauder estimates from \cite{RosOton-Serra:dirichlet-fractional-laplacian}.
\par Every time AI was used to generate content, we carefully checked the content and made the necessary modifications to ensure that the content was correct and well presented. All of the main ideas of the paper were generated by the authors, and all mathematical statements and proofs have been checked by the authors, who take full responsibility for their correctness.

\section{Main Ideas of the Proof}\label{sec:main_ideas}

In this section, we discuss the main ideas used to construct self-similar solutions to the gSQG patch equations. Our construction follows a perturbative strategy consisting of two main steps:
\begin{enumerate}
    \item Construct an approximate solution (\cref{thm:fixed_point_approximation}).
    \item Close a fixed point argument for the correction by linearizing the equation around this approximate solution and studying the corresponding linear inverse and nonlinear terms (\cref{thm:fixed_point_correction}).
\end{enumerate}
There are several different methods for attacking either of these steps, and depending on the equation, the interplay between these two steps can be quite challenging to manage. See \cite{GarciaJuarez-GomezSerrano-Nguyen-Pausader:self-similar-muskat,Garcia-GomezSerrano:self-similar-spirals-gsqg,Merle-Raphael-Rodnianski-Szeftel:implosion-i,Merle-Raphael-Rodnianski-Szeftel:implosion-ii,Buckmaster-CaoLabora-GomezSerrano:imploding-compressible-fluids,Buckmaster-Shkoller-Vicol:point-shocks-3d-euler} for related constructions of self-similar solutions and singularity formation.

\subsection{General Setup}

Let us denote the nonlinearity by 
\[\mc{J}(F) = c_\alpha \int_{\R} \frac{\partial_x F(x) - \partial_x F(x')}{\big((x-x')^2+(F(x)-F(x'))^2\big)^{\alpha/2}}\,dx', \qquad c_\alpha = \frac{\Gamma(\alpha)\sin(\pi(\alpha-1)/2)}{\pi}\]
so that the front evolution can be written as the quasilinear equation
\[\partial_t F = \mc{J}(F).\]
Here $c_\alpha$ is chosen so that the linearization of $\mc{J}$ around the flat interface $F \equiv 0$ is given by $|D|^{\alpha - 1}\partial_x$. We will search for singularities occurring at a fixed time $T > 0$. If we introduce the self-similar variables
\[\tau = \log \frac{1}{(\alpha(T - t))^{1/\alpha}}, \quad y = \frac{x}{(\alpha(T - t))^{1/\alpha}}, \quad F(x, t) = (\alpha(T - t))^{1/\alpha}f(y, \tau),\]
then the self-similar evolution is given by 
\begin{equation}\label{eqn:self_similar_front_equation}
    \partial_\tau f = f - y\partial_yf + \mc{J}(f).
\end{equation}
We search for solutions to \cref{eqn:self_similar_front_equation} satisfying the asymptotic behavior 
\[\lim_{y \to - \infty} \partial_y f(y, \tau) = -\sqrt{\epsilon}, \quad \lim_{y \to + \infty} f(y, \tau) = 0,\]
where $\epsilon$ is some small parameter. By replacing $f$ with $\sqrt{\epsilon}f$, we can equivalently study the equation
\begin{equation}\label{eqn:self_similar_front_equation_scaled}
    \partial_\tau f = f - y\partial_yf + \mc{J}_\epsilon(f), \quad \lim_{y \to - \infty} f'(y) = -1, \quad \lim_{y \to + \infty} f(y) = 0
\end{equation}
where 
\[\mc{J}_\epsilon(f) = c_\alpha \int_{\R} \frac{\partial_y f(y) - \partial_y f(z)}{\big((y-z)^2+ \epsilon(f(y)-f(z))^2\big)^{\alpha/2}}\,dz.\]
To separate the linear and nonlinear parts of the equation, set $A_\epsilon(u) = (1 + \epsilon u^2)^{-\alpha / 2}$. Then we can write
\[\mc{J}_\epsilon = |D|^{\alpha - 1} \partial_y + \mc{Q}_\epsilon, \quad \mc{Q}_{\epsilon}(f)(y) = c_\alpha \int \frac{\partial_y f(y) - \partial_y f(z)}{|y - z|^\alpha}\left(A_\epsilon\left(\frac{f(y) - f(z)}{y - z}\right) - 1\right) \, dz.\]
From here, we can see that the evolution can be expressed as 
\[\partial_\tau f = \underbrace{|D|^{\alpha - 1} \partial_y f + f - y\partial_y f}_{\textup{linear}} \quad  + \underbrace{\mc{Q}_{\epsilon}(f)}_{\textup{cubic and higher}}\]
Both operators are of order $\alpha$ in terms of the number of derivatives, so carrying out any kind of fixed point argument involving this formulation would need to be quite precise. Fortunately, for the \emph{time-independent} equation, there is a semilinear formulation that will require less delicate analysis. To see this, we extract the leading-order behavior of the nonlinear term:
\begin{align*}
    \mc{J}_\epsilon(f) &= c_\alpha \int_{\R} \frac{\partial_y f(y) - \partial_y f(z)}{\big((y-z)^2+\epsilon(f(y)-f(z))^2\big)^{\alpha/2}}\,dz \\
    & = A_\epsilon(\partial_y f) |D|^{\alpha - 1} \partial_yf + \sq{\mc{J}}_\epsilon(f)
\end{align*}
where $\sq{\mc{J}}_\epsilon$ is the lower-order nonlinear term 
\[\sq{\mc{J}}_\epsilon(f) = c_\alpha \int \frac{\partial_y f(y) - \partial_y f(z)}{|y - z|^\alpha}\left(A_\epsilon\left(\frac{f(y) - f(z)}{y - z}\right) - A_\epsilon(\partial_y f(y))\right) \, dz.\]
This term is lower order because, if $f \in W^{2,\infty}_{\textup{loc}}$, then the integrand will be $\mc{O}(|y - z|^{-\alpha + 2})$ as opposed to $\mc{O}(|y - z|^{-\alpha + 1})$.  The self-similar evolution can now be written as
\begin{equation}\label{eqn:main_ideas_self_similar_front_equation_separated}
    \partial_\tau f = f - y\partial_yf + A_\epsilon(\partial_y f)|D|^{\alpha - 1} \partial_yf + \sq{\mc{J}}_\epsilon(f)
\end{equation}
where $\sq{\mc{J}}_\epsilon(f)$ is of order $1$. We now make the observation that this leading-order term factors exactly as the product of the first-order operator $f \mapsto A_\epsilon(\partial_y f)$ and the \emph{linear} leading-order operator $|D|^{\alpha - 1}\partial_y$. After imposing $\partial_\tau f = 0$ and dividing the equation by $A_\epsilon(\partial_y f)$, we are left with a semilinear equation of the form
\begin{equation}\label{eqn:main_ideas_self_similar_front_equation_semilinear}
    |D|^{\alpha - 1}\partial_y f + \frac{1}{A_\epsilon(\partial_y f)}\left(f - y\partial_yf + \sq{\mc{J}}_\epsilon(f)\right) = 0.
\end{equation}
With this formulation, the nonlinear terms have more favourable mapping properties, and we will be able to close a fixed point argument (see \cref{prop:analysis_semilinear_kernel,thm:fixed_point_nonlinear_estimates}). 

\subsection{Fixed Point Argument}

We perform all of our analysis in physical space. We will construct an approximate solution $\ov{f}_\epsilon$ and set up a fixed point argument to solve the following equation for the correction $r$ so that $f = \ov{f}_\epsilon + r$ solves the full equation:
\[\mc{L}_\epsilon(r) = -\mc{E}_\epsilon(\ov{f}_\epsilon) - \mc{N}_{\epsilon, \ov{f}_\epsilon}(r)\]
where $\mc{E}_\epsilon(\ov{f}_\epsilon)$ is the residual, $\mc{L}_\epsilon$ is a linear operator, and $\mc{N}_{\epsilon, \ov{f}_\epsilon}$ is a nonlinear operator. The space for the correction will be a weighted $W^{2,\infty}$ space defined by the norm 
\[\begin{split}
 \|r\|_{X_\gamma} :=\sup_{y\ge0}\ang{y}^\gamma(|r|+\ang{y}^{-1/(\alpha - 1)}|r'|+\ang{y}^{-2/(\alpha - 1)}|r''|) +\sup_{y\le0}\ang{y}^\gamma(|r|+\ang{y}\,|r'|+\ang{y}^2|r''|),
 \end{split}\]
and the norm for the forcing/residual will be defined by 
\[
\begin{aligned}
 \|G\|_{Y_\gamma} := \sup_{y \geq 0}\ang{y}^\gamma\bigl(|G|+\ang{y}^{-1/(\alpha - 1)}|G'|\bigr) +\sup_{y\le0}\ang{y}^\gamma\bigl(|G|+\ang{y}\,|G'|\bigr)
 \end{aligned}
\]
for an appropriate $\gamma > 0$. These spaces are constructed to capture the Airy-type behavior coming from solutions to the linearized equation. On the negative end, the solutions are slowly varying, and on the positive end, the solutions are rapidly oscillating.
The exact choice of $\gamma$ comes from the mapping properties of the linear inverse, the nonlinear terms, and the number of corrections we use in the approximate solution. See \cref{sec:choice_of_space} for a discussion.
\par
Unfortunately, for our choice of $\gamma$, the linear inverse won't be bounded from $Y_\gamma$ to $X_\gamma$. Indeed, our inverse problem reduces to the study of the free operator through a conjugation $\mc{L}_\epsilon = \mc{V}_\epsilon \mc{L}_0 \mc{U}_\epsilon$ (\cref{lem:linear_conjugation}), where
\[\mc{L}_0(f) = |D|^{\alpha - 1}f' + f - yf'\]
whose formal inverse (see \eqref{eqn:linear_free_inverse}) is given by 
\[\widehat{\mc{L}_0^{-1} G}(\xi) = \frac{e^{-i\Phi_0(\xi)}}{\xi^2}\int_{0}^{\xi}\eta e^{i\Phi_0(\eta)} \widehat{G}(\eta) \, d\eta, \quad \Phi_0(\xi) = \frac{\xi|\xi|^{\alpha - 1}}{\alpha}\]
for $\xi > 0$ and extended to $\xi < 0$ by symmetry. For a general Schwartz function $G$, we then have 
\[\widehat{\mc{L}_0^{-1} G}(\xi) = \left(\int_{0}^\infty\ \eta e^{i\Phi_0(\eta)} \widehat{G}(\eta) \,d\eta\right)\frac{e^{-i\Phi_0(\xi)}}{\xi^2} + \text{lower-order terms}, \quad \xi \to \infty\]
and one can therefore conclude that $\mc{L}_0^{-1}(G)$ can only decay as fast as $y^{- \frac{\alpha + 2}{2(\alpha - 1)}}$ which is too slow for the space $X_\gamma$. To get around this issue, we must project the forcing onto the subspace $\sq{Y}_\gamma \subset Y_\gamma$ consisting of $G$ satisfying $\ell(G) = 0$ where $\ell: Y_\gamma \to \C$ is given by
\[\ell(G) = \int_{0}^{\infty}\eta e^{i\Phi_0(\eta)}\widehat{G}(\eta) \, d\eta.\]
We will then be able to bound $\mc{L}_0^{-1}$ from $\sq{Y}_\gamma$ to $X_\gamma$ (\cref{thm:fixed_point_linear_inverse}). To actually implement this, we will need to perform a Lyapunov-Schmidt reduction and rewrite the equation for $w = \mc{U}_\epsilon r$
\[\mc{L}_0(w) = -\mc{V}_\epsilon^{-1}\mc{E}_\epsilon(\ov{f}_\epsilon) - (\mc{V}_\epsilon^{-1} \mc{N}_{\epsilon,\ov{f}_\epsilon}\mc{U}_\epsilon^{-1})w\]
as
\[\begin{cases}
    \mc{L}_0(w) & = \Pi\left[-\mc{V}_\epsilon^{-1}\mc{E}_\epsilon(\ov{f}_{\epsilon}) - \mc{V}_\epsilon^{-1}\mc{N}_{\epsilon, \ov{f}_{\epsilon}}(\mc{U}_\epsilon^{-1}w)\right] \\
    0 & = \Pi^{\perp}\left[-\mc{V}_\epsilon^{-1}\mc{E}_\epsilon(\ov{f}_{\epsilon}) - \mc{V}_\epsilon^{-1}\mc{N}_{\epsilon, \ov{f}_{\epsilon}}(\mc{U}_\epsilon^{-1}w)\right]
\end{cases}\]
where $\Pi$ denotes the projection onto $\sq{Y}_\gamma$ and $\Pi^\perp:=I-\Pi$ projects onto the chosen two-dimensional complement $\sq{Y}_\gamma^{\perp}:=\operatorname{span}_{\R}\{G_+,G_-\}$, with $G_\pm$ defined in \cref{sec:fixed_point}. The first equation can be solved by a fixed point argument because the forcing is now in the right space, but to solve the second (two-dimensional) equation, we will introduce a complex parameter $\mu \in \C$ into the approximate solution and consider the (now consistent) system 
\[\begin{cases}
    \mc{L}_0(w) & = \Pi\left[-\mc{V}_\epsilon^{-1}\mc{E}_\epsilon(\ov{f}_{\epsilon, \mu}) - \mc{V}_\epsilon^{-1}\mc{N}_{\epsilon, \ov{f}_{\epsilon, \mu}}(\mc{U}_\epsilon^{-1}w)\right] \\
    0 & = \Pi^{\perp}\left[-\mc{V}_\epsilon^{-1}\mc{E}_\epsilon(\ov{f}_{\epsilon, \mu}) - \mc{V}_\epsilon^{-1}\mc{N}_{\epsilon, \ov{f}_{\epsilon, \mu}}(\mc{U}_\epsilon^{-1}w)\right].
\end{cases}\]
After doing this, we will be able to solve for $r = r_{\epsilon, \mu}$ by solving the first equation and $\mu = \mu(\epsilon)$ by solving the second equation. The resulting correction $r = r_{\epsilon, \mu(\epsilon)}$ will then solve the full equation (see the proof of \cref{thm:fixed_point_correction}).

The procedure described here is a Lyapunov--Schmidt reduction. We refer the reader to \cite{Kielhofer:bifurcation-theory} for the general method and \cite{Correia-Cote-Vega:asymptotics-mkdv} for a related construction of self-similar solutions to the modified Korteweg--de Vries equation.

\subsection{Construction of the Approximate Solution}

The first step in executing the fixed point argument is to construct an approximate solution $\ov{f}$ so that the residual
\[\sq{\mc{E}}_\epsilon(\ov{f}) = \mc{L}_{0}(\ov{f}) + \mc{Q}_\epsilon(\ov{f})\]
is small in the forcing space $Y_\gamma$. To do this, we consider the solution to the linearized equation
\[\mc{L}_0(f) = 0, \quad \mc{L}_0(f) = |D|^{\alpha - 1}f' + f - yf'\]
with the right asymptotic behavior (\cref{lem:approx_linear_seed}), which is given by
\[f_{0}(y) = -\frac{y}{2} + \frac{1}{\pi}\int_0^\infty
\frac{1-\cos\bigl(y\xi-\Phi_0(\xi)\bigr)}{\xi^2}\,d\xi,
\qquad \Phi_0(\xi)=\frac{\xi|\xi|^{\alpha - 1}}{\alpha}.\]
We then use the function 
\[f_{0, \mu} = f_{0} + 2\text{Re}\left[\left(\frac{\mu}{\mu_*} - 1\right)f_{0,+}\right]\]
as a seed for the approximate solution, where $f_{0,+}$ is the positive high-frequency part of $f_0$ and $\mu_*$ is a complex number specified in \cref{sec:approx_linear_solution}. The approximate solution will then be constructed as a perturbation of this seed, with the perturbation being small in $\epsilon$ and decaying at infinity. The choice of the term
\[2\text{Re}\left[\left(\frac{\mu}{\mu_*} - 1\right)f_{0,+}\right]\]
comes from the fact that the linear response 
\[\mc{L}_0\left(2\text{Re}\left[\left(\frac{\mu}{\mu_*} - 1\right)f_{0,+}\right]\right)\]
is a Schwartz function in the two-dimensional forcing space $\sq{Y}_\gamma^{\perp}$. Choosing the complex parameter $\mu$ correctly (using the fixed point argument for $\mu$) will allow us to cancel the forcing in $\sq{Y}_\gamma^{\perp}$ and solve the full equation. \par
To construct the approximate solution, we express the (derivative of the) high-frequency part of the linear solution $f_{0,+}$ as
\[f_{0,+}'(y) = -\frac{i}{2\pi}\int_0^\infty\frac{\chi_+(\xi)}{\xi}e^{i(y\xi-\Phi_0(\xi))}\,d\xi = a_0(y)e^{i(\Psi_0(y)-\pi/4)},\] 
where the phase and slowly varying coefficient satisfy
\[\Psi_0(y) = \frac{\alpha - 1}{\alpha}y^{\frac{\alpha}{\alpha - 1}}, \qquad |\partial_y^j a_0(y)|+|\partial_y^j a(y)|\lesssim_j y^{-\frac{\alpha}{2(\alpha - 1)}-j} \quad (y\gg1,\ j\geq0).\]
The approximate solution will then consist of three parts (see \cref{sec:approximate_family})
\[\ov{f}_{\epsilon, \mu} = f_{0, \mu} + f_{\textup{slow}} + f_{\textup{osc}}.\]
These are the linear seed, a slowly varying correction $f_{\textup{slow}}$, and a rapidly oscillating correction $f_{\textup{osc}}$. Here, $f_{\textup{slow}}$ is defined by 
\[f_{\textup{slow}} = \begin{cases}
c_+|y|^{-\alpha + 1}, \quad y \gg 1 \\
c_-|y|^{-\alpha + 1}, \quad y \ll -1
\end{cases} \]
for some constants $c_+,c_- = \mc{O}(\epsilon)$ and its purpose is to cancel the nonlinear response to the growing tail of the linear seed (see \cref{prop:approx_E1,prop:approx_residual_negative}):
\[\mc{Q}_\epsilon(W) = -\alpha \begin{cases}
c_+|y|^{-\alpha + 1}, \quad y \gg 1 \\
c_-|y|^{-\alpha + 1}, \quad y \ll -1
\end{cases}, \quad W = \max\{-y, 0\}, \quad \]
The oscillatory correction $f_{\textup{osc}}$ is defined by
\[f_{\textup{osc}}'(y) = 2\operatorname{Re}\left[a(y)e^{i(\Psi_{\epsilon,\mu}(y)-\pi/4)} + c_{\textup{cubic}}a(y)^3e^{3i(\Psi_{\epsilon,\mu}(y)-\pi/4)} - \frac{\mu}{\mu_*}f_{0,+}'(y)\right],\quad y \gg 1\]
with
\[\lim_{y\to\infty}f_{\textup{osc}}(y) = 0, \quad a(y) = \frac{\mu}{\mu_*}\left(\frac{\Psi_0'(y)}{\Psi_{\epsilon,\mu}'(y)}\right)^{\alpha/2}a_0(y)\]
and where $\Psi_{\epsilon, \mu}$ is the modified phase function 
\[\Psi_{\epsilon,\mu}(y)=\Psi_0(y)+c_{\textup{power}}\left(y^{\frac{\alpha(2 - \alpha)}{\alpha - 1}}-R^{\frac{\alpha(2 - \alpha)}{\alpha - 1}}\right)+c_{\textup{log}}(\mu)\log\frac{y}{R},\qquad y\geq R.\]
It follows that we can express the oscillatory part of $\ov{f}'_{\epsilon, \mu}$ as
\[2\operatorname{Re}\left[a(y)e^{i(\Psi_{\epsilon,\mu}(y)-\pi/4)} + c_{\textup{cubic}}a(y)^3e^{3i(\Psi_{\epsilon,\mu}(y)-\pi/4)}\right]\]
which contains a cubic correction to the linear solution, as well as a modified phase function that accounts for certain responses to the nonlinear operator. Indeed, there are two problematic terms that arise from the nonlinear operator $\mc{Q}_\epsilon$ acting on the linear seed $f_{0, \mu}$, namely, the resonant cubic term (see \cref{prop:approx_E5_cubic})
\[\ov{a_0e^{i(\Psi_{0} - \pi/4)}}\left[a_0 e^{i(\Psi_{0} - \pi/4)}\right]^2 = |a_0|^2a_0e^{i(\Psi_{0} - \pi/4)}\]
and the interaction of the positive oscillatory tail with the slowly varying, growing, negative tail (see \eqref{eqn:approx_wedge_first_variation} and \cref{prop:approx_E6}):
\[D\mc{Q}_\epsilon(W)[f](y) = c_\epsilon y^{-\alpha + 1}f'(y) + \text{lower-order terms}, \quad y > 0\]
so that
\[D\mc{Q}_\epsilon(W)[f_{0,\mu} - W](y) = c_\epsilon y^{-\alpha + 1} 2\text{Re}\left[\frac{\mu}{\mu_*}a_0 e^{i(\Psi_{0} - \pi/4)}\right] + \text{lower-order terms}, \quad y > 0\]
Each of these nonlinear responses needs to be cancelled by a linear response, and the only way to do this is to modify the phase function $\Psi_{\epsilon, \mu}$ with the logarithmic and power corrections. The logarithmic correction cancels the resonant cubic response and the power correction cancels the tail response. After incorporating these corrections, we will be able to bound the residual $\mc{E}_\epsilon(\ov{f}_{\epsilon, \mu})$ in the forcing space $Y_\gamma$. See \cref{thm:fixed_point_approximation} and its proof in \cref{sec:approximate_solution}.

\subsection{Linear and Nonlinear Estimates}

Once we have constructed the family of approximate solutions $\ov{f}_{\epsilon, \mu}$ with small residuals in the desired space, we can set up the fixed point argument for the correction $r$ where $f = \ov{f}_{\epsilon, \mu} + r$. The semilinear formulation of the equation will be expressed as
\[\mc{E}_\epsilon(f) := \mc{L}_0(f) + \mc{N}_\epsilon(f) = 0,\]
where $\mc{N}_\epsilon$ is some nonlinear operator that contains fewer derivatives than $\mc{L}_0$. Here the residual of the semilinear formulation is related to the residual of the quasilinear formulation by
\[\sq{\mc{E}}_\epsilon(f) = A_\epsilon(f')\mc{E}_\epsilon(f),\]
and this coefficient does not affect the mapping properties of the residual. It follows that we can express the equation for the correction as
\[\mc{E}_\epsilon(\ov{f}) + \mc{L}_0(r) + \mc{N}_{\epsilon}(\ov{f} + r) - \mc{N}_\epsilon(\ov{f}) = 0.\]
After performing a careful analysis of the nonlinear operator, we will see that every interaction is perturbative except for the nonlocal interaction between the growing tail in $\ov{f}$ for $y \ll 1$ and the correction for $y \gg 1$. In other words, we can express our correction equation as
\[\mc{E}_\epsilon(\ov{f}) + \mc{L}_\epsilon(r) + \mc{N}_{\epsilon, \ov{f}}(r) = 0\]
where $\mc{L}_\epsilon$ is a small perturbation of the linear operator $\mc{L}_0$ obtained by adding a slow term $c_\epsilon y^{-\alpha + 1}\partial_y$ for $y \gg 1$. This operator is too slow to be bounded from $X_\gamma$ to $Y_\gamma$ even though $c_\epsilon = \mc{O}(\epsilon)$, so it will need to be included in the linear inverse problem. Fortunately, we can perform an isomorphism of $X_\gamma$ and $Y_\gamma$ to reduce the problem to inverting the linear operator $\mc{L}_0$ (see \cref{lem:linear_conjugation}). \par
To bound the inverse $\mc{T}_0 = \mc{L}_0^{-1}$ from $\sq{Y}_\gamma$ to $X_\gamma$, we will express its kernel as a Fourier integral: 
\[\mc{T}_0(G)(y) = \int_{\R}K(y,z)G(z)\,dz, \quad K(y,z) = \frac{1}{2\pi}\int_{\R}\left[\int_0^\xi\frac{\eta}{\xi^2}e^{i(y\xi-z\eta+\Phi_0(\eta)-\Phi_0(\xi))}\,d\eta\right]d\xi.\]
To perform the desired weighted estimates, we will decompose the inverse into three pieces
\[\mc{T}_0 = \mc{T}_{\textup{l}} + \mc{T}_{\textup{h}, \textup{s}} + \mc{T}_{\textup{h}, \textup{ns}}\]
corresponding to the low-frequency part, the high-frequency stationary part, and the high-frequency nonstationary part (see \eqref{eqn:linear_frequency_split} and \cref{sec:linear_low_frequency,sec:linear_high_frequency}). Here stationary and nonstationary refer to the phase $\xi \mapsto y\xi - \Phi_0(\xi)$. The reason we extract the low-frequency part and handle it separately is that the symbol of $\mc{L}_0$ is not smooth at the origin. To leading order, the low-frequency part of the inverse is given by 
\[\widehat{\mc{T}_{\textup{l}}G}(\xi) \sim \chi(\xi)\Phi_0(\xi)\int_{0}^{1}t(1 - t^{\alpha})\widehat{G}(t\xi) \, dt\]
where $\chi$ is a low-frequency cutoff (see the decomposition \eqref{eqn:linear_low_frequency_series} in the proof of \cref{prop:linear_low_frequency}). We are able to bound this term by proving weighted $L^\infty$ estimates for the Fourier multipliers with the non-smooth symbol $\chi(\xi)\Phi_0(\xi)$ (see \cref{prop:appendix_nonsmooth_fourier_inverse}). Next, to leading order, the high-frequency stationary part contains the oscillatory term
\begin{equation}\label{eqn:main_ideas_high_frequency_stationary}
G \mapsto
y^{-\frac{\alpha+2}{2(\alpha-1)}}
\int_{Cy}^\infty
z^{\frac{4-\alpha}{2(\alpha-1)}}
\cos \bigl(\Psi_0(y)-\Psi_0(z)\bigr)\,G(z)\,dz,
\qquad y\gg1,
\end{equation}
for some $C \gg 1$. This term comes from the stationary contribution of the kernel in both variables $\xi = y^{1/(\alpha-1)}$ and $\eta = z^{1/(\alpha - 1)}$ in the projected representation \eqref{eqn:linear_projected_tail}. In our weighted spaces, the bounds of terms like this are straightforward (see \cref{prop:linear_stationary} and \eqref{eqn:linear_positive_stationary_contribution}). To leading order, the high-frequency nonstationary part contains the term (see \cref{lem:linear_nonstationary_symbol_expansion})
\[G \mapsto 
\frac1{2\pi}\int_{\mathbb R}e^{iy\xi}
\frac{(1-\chi(\xi))\chi_{\textup{ns}}(y,|\xi|)}
{i\xi\bigl(|\xi|^{\alpha-1}-y\bigr)}
\widehat G(\xi)\,d\xi
\]
where the cutoffs localize away from $0$ and away from the stationary point $\xi = y^{1/(\alpha-1)}$. Therefore this term is a pseudodifferential operator with a smooth symbol. By analyzing the symbol, we can prove weighted $L^\infty$ estimates for this term as well (see \cref{lem:linear_nonstationary_leading_kernels,prop:linear_high_frequency}). \par
Finally, the bounds for the nonlinear operator $\mc{N}_{\epsilon, \ov{f}}$ are relatively straightforward in our function space setup. The local interactions are bounded immediately from the definition of our spaces (\cref{lem:N1,lem:N2}), and to bound the nonlocal interactions, we simply perform a near-far splitting of the integrals and the desired bounds follow from the decay of the kernel and cancellation of the singularity (\cref{lem:N3}).

\subsection{Choice of Space}\label{sec:choice_of_space}

The spaces $X_\gamma$ and $Y_\gamma$ have three defining properties
\begin{enumerate}
    \item Derivative count ($2$ for $X_\gamma$ and $1$ for $Y_\gamma$)
    \item Derivative decay/growth ($\ang{y}^{-1}$ for $y \leq 0$, and $\ang{y}^{1/(\alpha - 1)}$ for $y \geq 0$)
    \item Base decay rate $\gamma$
\end{enumerate}

The choice of two derivatives in $X_\gamma$ and one derivative in $Y_\gamma$ is motivated by the fact that the nonlinear operator $\mc{N}_{\epsilon, \ov{f}}$ is of order $1$, but is only well-defined for functions in $W^{2,\infty}_{\textup{loc}}$. Indeed, the nonlocal free part of $\mc{N}_{\epsilon, \ov{f}}$ contains a leading-order term of the form 
\[r \mapsto \int \frac{r'(y) - r'(z)}{|y - z|^\alpha}\left(\frac{r(y) - r(z)}{y - z} - r'(y)\right)\left(\frac{r(y) - r(z)}{y - z} + r'(y)\right) \, dz.\]
And to utilize the cancellation in the integrand, we need $r \in W^{2,\infty}_{\textup{loc}}$ (see \cref{rmk:analysis_vanishing_diagonal}). \par
The choice of the weights on the derivatives comes from the fact that the linearized solution $f_0$ has leading-order behavior (see \cref{lem:approx_linear_seed})
\[f_0(y) + y \sim c|y|^{-\alpha + 1}, \quad y \to -\infty\]
and 
\[f_0'(y) \sim \text{Re}\left[ C y^{- \frac{\alpha}{2(\alpha - 1)}}e^{i \left(\frac{\alpha - 1}{\alpha}\right) y^{\left(1 + \frac{1}{\alpha - 1}\right)}} \right], \quad y \to \infty.\]
Therefore we expect the correction to be slowly varying on the negative end, and oscillating at a frequency of $\ang{y}^{1 + \frac{1}{\alpha - 1}}$ on the positive end. The weights on the derivatives are chosen to capture this behavior. \par
Finally, the base decay rate is determined by the mapping properties of the fixed point map. To see this, suppose that we have the bound
\[\mc{N}_{\epsilon, \ov{f}} : X_{\gamma_1} \to Y_{\gamma_2}.\]
The nonlinear operator $\mc{N}_{\epsilon, \ov{f}}$ contains local terms of the form (see \eqref{eq:nonlinear_decomposition} and \cref{lem:N1,lem:N2}) 
\[r \mapsto y(\overline f')^2r' , \quad r \mapsto y\overline f'(r')^2, \quad r \mapsto y(r')^3\]
and in order for these operators to be bounded from $X_{\gamma_1}$ to $Y_{\gamma_2}$, for $y \gg 1$, we require the respective restrictions
\[\gamma_2 \leq \gamma_1, \quad \gamma_2 \leq 2\gamma_1 - \left(\frac{\alpha + 2}{2(\alpha - 1)}\right), \quad \gamma_2 \leq 3\gamma_1 - 2\left(\frac{\alpha + 2}{2(\alpha - 1)}\right).\]
On the other hand, suppose that we have either a bound on the full space or the projected space
\[\mc{L}_0^{-1} : Y_{\gamma_2} \to X_{\gamma_1} \quad \text{or} \quad \mc{L}_0^{-1} : \sq{Y}_{\gamma_2} \to X_{\gamma_1}.\]
Since, for example,
\[\mc{L}_0(\chi_+ y^{-\gamma_2})(y) = (1 + \gamma_2)y^{-\gamma_2} + \text{lower-order terms}, \quad y \gg 1,\]
we must have $\gamma_2 \geq \gamma_1$ in order for the linear inverse to be bounded. Therefore, we must have $\gamma_1 = \gamma_2 =: \gamma$. Based on the restrictions coming from the nonlinear estimates, we then must have
\[\gamma \geq \frac{\alpha + 2}{2(\alpha - 1)}.\]
If we have a bound on the full space $Y_{\gamma_2}$, then by the discussions above, 
\[\widehat{\mc{L}_0^{-1}(G)}(\xi) = \left(\int_{0}^\infty\ \eta e^{i\Phi_0(\eta)} \widehat{G}(\eta) \,d\eta\right)\frac{e^{-i\Phi_0(\xi)}}{\xi^2} + \text{lower-order terms}, \quad \xi \to \infty\]
which implies that $\mc{L}_0^{-1}(G)$ can only decay as fast as $y^{- \frac{\alpha + 2}{2(\alpha - 1)}}$ and this therefore restricts $\gamma$ to equal
\[\gamma = \frac{\alpha + 2}{2(\alpha - 1)}.\]
However, a bound $\mc{L}_0^{-1}: Y_\gamma \to X_\gamma$ fails with this choice of $\gamma$ because for example if
\[\widehat{G}(\xi) = \frac{\chi_+(\xi)}{\xi^2}e^{-i\Phi_0(\xi)}, \quad \xi > 0\]
with $\widehat{G}(\xi) = \overline{\widehat{G}(-\xi)}$ for $\xi < 0$, so that $G$ is real-valued, then $G \in Y_\gamma$, but \eqref{eqn:linear_free_inverse} gives
\[\widehat{\mc{L}_0^{-1}(G)}(\xi) = \xi^{-2}e^{-i\Phi_0(\xi)} (\log \xi + C), \quad \xi \geq 2,\]
and a stationary phase argument shows that $\mc{L}_0^{-1}(G)$ fails to lie in $X_\gamma$ by a factor of $\log y$. Therefore, we will need to project the forcing onto the subspace $\sq{Y}_\gamma$ in order to bound the linear inverse. By the leading order term in \eqref{eqn:main_ideas_high_frequency_stationary} (see also the tail integral in \eqref{eqn:linear_positive_stationary_contribution}), we require
\[\gamma > \frac{\alpha + 2}{2(\alpha - 1)}.\]
Finally, the linear inverse $\mc{L}_0^{-1}$ contains the term 
\[G \mapsto \left(\int_{-\infty}^{\infty}G(y) \, dy\right)\frac{\text{sgn}(y)}{|y|^{\alpha + 1}}\]
and this gives the restriction $\gamma \leq \alpha + 1$ (see the low-frequency kernel bound \eqref{eqn:linear_Hk_kernel}). An additional restriction $\gamma \leq 2\alpha - 1$ comes from the bound \eqref{eqn:approx_E1_bound} on the residual $\mc{E}_\epsilon(\ov{f})$. This yields the final restriction 
\[\frac{\alpha + 2}{2(\alpha - 1)} < \gamma \leq 2\alpha - 1\]
which is nonempty for $\frac{7}{4} < \alpha < 2$. By adding additional corrections to the approximate solution, we can likely extend the upper bound on $\gamma$ to $\alpha + 1$ which would only require
\[\alpha > \frac{1 + \sqrt{33}}{4} \approx 1.686,\]
but we will not pursue this in this paper. To go further below the threshold $(1 + \sqrt{33})/4$, one would need to assume additional moment conditions on the forcing $G$ (like $\int G = 0$) to bound the low-frequency part of the linear inverse into a space with faster decay, as well as add additional corrections to the approximate solution to improve the decay of the residual. We believe that this is possible with our approach, but we will not pursue this matter in this paper.

\subsection{Organization of the Paper} In \cref{sec:fixed_point}, we set up the strategy and state the main theorems \cref{thm:fixed_point_approximation}, \cref{thm:fixed_point_linear_inverse}, and \cref{thm:fixed_point_nonlinear_estimates}, regarding the approximate solution, the linear inverse, and the nonlinear estimates, and then we show how these theorems imply \cref{thm:main} using a fixed point argument. In \cref{sec:analysis_equation}, we prove some miscellaneous facts about the equation. Namely, we rewrite the nonlinear operator in a more convenient form in \cref{prop:analysis_semilinear_kernel}, and we prove the bootstrapping result in \cref{thm:fixed_point_smoothness}. In \cref{sec:approximate_solution}, we construct the approximate solution and bound the residual, which proves \cref{thm:fixed_point_approximation}. In \cref{sec:linear_inverse}, we bound the linear inverse, which proves \cref{thm:fixed_point_linear_inverse}. In \cref{sec:nonlinear_estimates}, we bound the nonlinear operator, which proves \cref{thm:fixed_point_nonlinear_estimates}. Finally, in Appendix~\ref{app:lean} we describe the Lean formalization of Theorem~\ref{thm:main}.

\subsection{Notation}
We use $x$ for the original spatial variables and $t$ for time, and $y$ and $z$ for the self-similar
spatial variables and $\tau$ for self-similar time. We will use $\xi$ and $\eta$ for Fourier
variables. We will use $\Phi, \Psi$ for phase functions and we will use the lowercase letters
$\phi,\psi$ to denote changes of spatial coordinates.

Calligraphic letters, such as $\mc{E},\mc{L},\mc{N}$ and $\mc{T}$, denote operators and residual terms. Uppercase
letters, such as $A_\epsilon,B_\epsilon$ and $\Gamma_\epsilon$, denote scalar coefficient functions or
kernels. We will use lowercase letters like $f, r, g, h$ to denote functions related to the
correction, and we will use $G$, with subscripts as needed, to denote functions related to the forcing.

We will use $M, N, R$ to denote large numbers. We will use the letter $K$ to denote kernels. We
write $p=\alpha-1$, use $\beta,\gamma$ for decay and weight exponents, and $\epsilon$ for the small
perturbation parameter.

The Japanese bracket is $\ang{y}=(1+y^2)^{1/2}$. For nonnegative quantities $a,b$, we write
$a\lesssim b$ if $a\le Cb$ for a constant $C>0$, and $a\approx b$ if both $a\lesssim b$ and
$b\lesssim a$. Implicit constants may change from line to line and may depend on the fixed exponents
and cutoffs, but are independent of the spatial variables and of $\epsilon$ in the stated range.
Additional dependencies are indicated by subscripts, as in $\lesssim_{M,N}$ or $\lesssim_K$, or
specified locally.
We will use the Fourier transform convention
\[\hat{f}(\xi) = \int_{-\infty}^{\infty} f(x) e^{-i x \xi} dx.\]
For a function $f$, we will denote the secant by
\[S_f(y, z) = \frac{f(y) - f(z)}{y - z}.\]

\section{Fixed Point Argument}\label{sec:fixed_point}

In this section, we set up the fixed point argument and state the main theorems regarding the approximate solution, the linear inverse, and the nonlinear estimates. We then show how these theorems imply \cref{thm:main} using a fixed point argument. \par Write $p = \alpha - 1$. Define the correction norm 
\[
 \begin{split}
 \|r\|_{X_\gamma}:=\sup_{y\ge0}\ang{y}^\gamma(|r|+\ang{y}^{-1/p}|r'|+\ang{y}^{-2/p}|r''|) +\sup_{y\le0}\ang{y}^\gamma(|r|+\ang{y}\,|r'|+\ang{y}^2|r''|),
 \end{split}
\]
and the forcing norm
\begin{equation}\label{eq:Y}
\begin{aligned}
 \|G\|_{Y_\gamma}:= \sup_{y \geq 0}\ang{y}^\gamma\bigl(|G(y)|+\ang{y}^{-1/p}|G'(y)|\bigr) +\sup_{y\le0}\ang{y}^\gamma\bigl(|G(y)|+\ang{y}\,|G'(y)|\bigr).
 \end{aligned}
\end{equation}
We will be working with the correction space
\[X_\gamma = \{r \in W^{2,\infty}_{\textup{loc}}(\R;\R) : \|r\|_{X_\gamma} < \infty\}\]
and the forcing space
\[Y_\gamma = \{G \in  W^{1,\infty}_{\textup{loc}}(\R;\R) : \|G\|_{Y_\gamma} < \infty\}.\]
We now fix the parameters
\[\frac{3}{4} < p < 1, \quad \beta = \frac{p + 1}{2p}, \quad \beta + \frac{1}{p} < \gamma_1 < \gamma_2 < 2p + 1.\]
Our first theorem concerns the existence of a family of approximate solutions to \cref{eqn:main_ideas_self_similar_front_equation_semilinear}. For a function $f$, let the residual be given by 
\[\mc{E}_\epsilon(f) := |D|^{\alpha - 1}\partial_y f + \frac{1}{A_\epsilon(\partial_y f)}\left(f - y\partial_yf + \sq{\mc{J}}_\epsilon(f)\right).\]

\begin{theorem}\label{thm:fixed_point_approximation}(Existence of Approximate solutions)
    There exist small constants $\epsilon_0, \delta > 0$ and a nonzero complex number $\mu_* \in \mathbb{C} \setminus \{0\}$ so that the following holds. Let 
    \[\ov{B} = \ov{B}(\mu_*, \delta) := \{\mu \in \mathbb{C} : |\mu - \mu_*| \leq \delta\}.\]
    There exists a family of real-valued $C^\infty$ functions $\ov{f}_{\epsilon, \mu}$ where $0 \leq \epsilon \leq \epsilon_0$ and $\mu \in \ov{B}$ so that the following properties hold for any $\beta + \frac{1}{p} < \gamma < 2p + 1$:
    \begin{enumerate}
        \item The approximate solutions have the asymptotic behavior
        \[\lim_{y \to -\infty} [\ov{f}_{\epsilon, \mu}(y) + y] = 0, \quad \lim_{y \to \infty} \ov{f}_{\epsilon, \mu}(y) = 0\]
        \item We have the following estimates for $y \geq 0$:
        \[|\ov{f}_{\epsilon, \mu}(y)| \lesssim \ang{y}^{-p}, \quad |\partial^{j}\ov{f}'_{\epsilon, \mu}(y)| \lesssim_j \ang{y}^{-\beta + \frac{j}{p}}, \quad j \geq 0\]
        and for $y \leq 0$,
        \[|\partial^{j}(\ov{f}_{\epsilon, \mu}(y) + y)| \lesssim_j \ang{y}^{-p - j}, \quad j \geq 0.\]
        \item The residuals lie in the forcing space: $\mc{E}_\epsilon(\ov{f}_{\epsilon, \mu}) \in Y_\gamma$. Moreover, for each $\mu \in \ov{B}$, there exists a Schwartz function $G_\mu$ independent of $\epsilon$ so that
        \begin{equation}\label{eqn:fixed_point_residual_bound}
\sup_{\mu \in \ov{B}}\|\mc{E}_\epsilon(\ov{f}_{\epsilon, \mu}) - G_\mu\|_{Y_\gamma} \lesssim \epsilon
\end{equation} 
        and $G_{\mu_*} = 0$. In fact, $G_{\mu}$ has the explicit form 
        \[\widehat{G}_{\mu}(\xi) = \left(1 - \frac{\mu}{\mu_*}\right) \frac{\chi_+'(\xi)}{\xi}e^{-i \Phi_0(\xi)}, \quad \Phi_0(\xi) = \frac{1}{\alpha}\xi|\xi|^{p}\]
        for $\xi > 0$ where $\chi_+$ is a smooth cutoff equal to $1$ for $\xi \geq 2$ and equal to $0$ for $\xi \leq 1$. For $\xi < 0$, we define $\widehat{G}_{\mu}(\xi) = \overline{\widehat{G}_{\mu}(-\xi)}$ so that $G_{\mu}$ is real-valued.
        \item For each $\epsilon$, the maps $\mu \mapsto \ov{f}_{\epsilon, \mu}$ and $\mu \mapsto \mc{E}_\epsilon(\ov{f}_{\epsilon, \mu})$ are continuous into $C^2_{\textup{loc}}$ and $Y_\gamma$, respectively.
    \end{enumerate}
    All of the constants in the above estimates are uniform in $\epsilon$ and $\mu$.
\end{theorem}

Now we state the main theorem about the existence of the correction. Decompose $f = \overline{f} + r$ and decompose the equation into its residual, a linear part and a nonlinear part
\[\mc{E}_{\epsilon}(f) = \mc{E}_{\epsilon}(\ov{f}) + \mc{L}_{\epsilon}(r) + \mc{N}_{\epsilon, \ov{f}}(r)\]
where 
\[\mc{L}_{\epsilon} = |D|^{p}\partial_y + I - y\partial_y + c_\epsilon \chi_+y^{-p} \partial_y + \mc{R}_\epsilon\]
where $\mc{R}_\epsilon$ is the perturbative remainder term defined in the proof of \cref{lem:linear_conjugation} and 
\[\mc{N}_{\epsilon, \ov{f}}(r) = \mc{E}_{\epsilon}(\ov{f} + r) - \mc{E}_{\epsilon}(\ov{f}) - \mc{L}_{\epsilon}(r).\]
Here $c_\epsilon \chi_+ y^{-p} \partial_y$ is a nonperturbative linear term that arises from the nonlinear and nonlocal interaction of the tail $\ov{f} \approx -y$ as $y \to - \infty$ and the correction $r(y)$ for $y > 0$. Fortunately, every other interaction is perturbative and can be absorbed into the nonlinear operator.

The following theorem concerns the existence of a linear inverse of the operator $\mc{L}_\epsilon$. We are fortunately able to reduce matters to the study of the free operator $\mc{L}_0$. The operator $\mc{L}_0$ doesn't have a bounded inverse from $Y_\gamma$ to $X_\gamma$, so to obtain an inverse, we will project onto a codimension two subspace. Consider the functional 
\[\ell(G) := \int_{0}^{\infty}\xi e^{i \Phi_0(\xi)}\widehat{G}(\xi) \, d\xi, \quad \Phi_0(\xi) = \frac{\xi|\xi|^p}{\alpha}.\]
Let $\sq{Y}_\gamma$ denote the subspace of $Y_\gamma$ consisting of functions $G$ so that $\ell(G) = 0$
\[\sq{Y}_\gamma = \{G \in Y_\gamma : \ell(G) = 0\}\]
and let $\Pi:\ Y_\gamma \to \sq{Y}_\gamma$ be the projection onto this subspace. I.e. we define real-valued Schwartz functions $G_+, G_-$ by 
\[\widehat{G_+}(\xi) = \frac{\chi_+'(\xi)}{\xi}e^{-i\Phi_0(\xi)}, \quad \widehat{G_-}(\xi) = i\frac{\chi_+'(\xi)}{\xi}e^{-i\Phi_0(\xi)}, \quad \xi \geq 0.\]
It's immediate from the definition of $\ell$ that $\ell(G_+) = 1$ and $\ell(G_-) = i$. Therefore, we can define the projection $\Pi$ by
\[\Pi (G) = G - \text{Re}(\ell(G))G_+ - \text{Im}(\ell(G))G_-.\]
It's not immediately obvious that $\ell$ is a well-defined continuous functional on $Y_\gamma$. This will be proved in \cref{sec:linear_inverse}.

\begin{theorem}\label{thm:fixed_point_linear_inverse}(Linear Inverse) Suppose that $\beta + \frac{1}{p} < \gamma < 2p + 1$. Then the functional $\ell: Y_{\gamma} \to \C$ is continuous. Moreover, there exists a bounded linear operator $\mc{T}_0: \sq{Y}_\gamma \to X_\gamma$ and isomorphisms $\mc{U}_\epsilon: X_\gamma \to X_\gamma$, $\mc{V}_\epsilon: Y_\gamma \to Y_\gamma$
so that
\begin{enumerate}
    \item $\mc{T}_0$ is a right inverse of $\mc{L}_0$:
    \begin{equation}\label{eqn:fixed_point_free_inverse}
\mc{L}_0 \mc{T}_0 G = G, \quad G \in \sq{Y}_\gamma.
\end{equation}
    \item The operators conjugate $\mc{L}_\epsilon$ to $\mc{L}_0$: $\mc{L}_\epsilon = \mc{V}_\epsilon \mc{L}_0 \mc{U}_\epsilon$ and we have the uniform estimates 
    \[\sup_{0 \leq \epsilon \leq \epsilon_0}\|\mc{U}_{\epsilon}\|_{X_\gamma \to X_\gamma} < \infty, \quad \sup_{0 \leq \epsilon \leq \epsilon_0}\|\mc{V}_{\epsilon}\|_{Y_\gamma \to Y_\gamma} < \infty\]
    and
    \[\sup_{0 \leq \epsilon \leq \epsilon_0}\|\mc{U}_{\epsilon}^{-1}\|_{X_\gamma \to X_\gamma} < \infty, \quad \sup_{0 \leq \epsilon \leq \epsilon_0}\|\mc{V}_{\epsilon}^{-1}\|_{Y_\gamma \to Y_\gamma} < \infty\]
    and 
    \begin{equation}\label{eqn:fixed_point_residual_transport}
\|(\mc{V}_\epsilon^{-1} - I)G_\mu\|_{Y_\gamma} \lesssim \epsilon.
\end{equation}
\end{enumerate}
None of the operators defined above depends on $\gamma$.
\end{theorem}

\begin{theorem}\label{thm:fixed_point_nonlinear_estimates}(Nonlinear Estimates) Let $\beta + \frac{1}{p} < \gamma < 2p + 1$ and write $\ov{f} = \ov{f}_{\epsilon, \mu}$. The nonlinear operator $\mc{N}_{\epsilon, \ov{f}}$ satisfies the following estimates for any $r_1, r_2, r \in X_\gamma$ with \[\|r\|_{X_\gamma}, \|r_1\|_{X_\gamma}, \|r_2\|_{X_\gamma} \leq N.\] 
\begin{enumerate}
    \item We have that $\mc{N}_{\epsilon, \ov{f}}: X_\gamma \to Y_\gamma$ with the estimate
    \begin{equation}\label{eqn:fixed_point_nonlinear_bound}
\|\mc{N}_{\epsilon, \ov{f}}(r)\|_{Y_\gamma} \lesssim_N \epsilon\|r\|_{X_\gamma}.
\end{equation}
    \item We have the Lipschitz estimate
    \begin{equation}\label{eqn:fixed_point_nonlinear_lipschitz}
\|\mc{N}_{\epsilon, \ov{f}}(r_1) - \mc{N}_{\epsilon, \ov{f}}(r_2)\|_{Y_\gamma} \lesssim_{N} \epsilon \|r_1 - r_2\|_{X_\gamma}.
\end{equation}
    \item For each fixed $r \in X_{\gamma_2}$ and $\epsilon$, the map
    \[\mu \mapsto \mc{N}_{\epsilon, \ov{f}_{\epsilon, \mu}}(r)\]
    is continuous as a map into $Y_{\gamma_1}$.
\end{enumerate}
All of the constants in the above estimates are uniform in $\epsilon$ and $\mu$.
\end{theorem}

By applying \cref{thm:fixed_point_approximation,thm:fixed_point_linear_inverse,thm:fixed_point_nonlinear_estimates}, we can now prove the existence of the correction $r$.

\begin{theorem}\label{thm:fixed_point_correction}
    For any $\epsilon$ small enough, there exist a function $r = r_\epsilon \in X_{\gamma_2}$ and a parameter $\mu = \mu_\epsilon$ satisfying the correction equation 
    \begin{equation}\label{eqn:fixed_point_correction}
\mc{L}_\epsilon(r) = -\mc{E}_\epsilon(\ov{f}_{\epsilon, \mu}) - \mc{N}_{\epsilon, \ov{f}_{\epsilon, \mu}}(r).
\end{equation}
    Moreover, $r_\epsilon$ and $\mu_\epsilon$ satisfy the uniform bounds
    \begin{equation}\label{eqn:fixed_point_correction_bound}
|\mu_* - \mu_\epsilon| + \|r_\epsilon\|_{X_{\gamma_2}} \lesssim \epsilon.
\end{equation}
\end{theorem}

\begin{proof}[Proof of \cref{thm:fixed_point_correction}]
We prove the theorem by setting up a fixed point argument. By using that $\mc{L}_\epsilon = \mc{V}_\epsilon \mc{L}_0 \mc{U}_\epsilon$ from \cref{thm:fixed_point_linear_inverse}, we will rewrite the equation in terms of $w = \mc{U}_\epsilon r$. To solve \eqref{eqn:fixed_point_correction}, we equivalently want to solve
\begin{equation}\label{eqn:fixed_point_conjugated_equation}
\mc{L}_0(w) = - \mc{V}_\epsilon^{-1} \mc{E}_\epsilon(\ov{f}_{\epsilon, \mu}) - \mc{V}_\epsilon^{-1}\mc{N}_{\epsilon, \ov{f}_{\epsilon, \mu}}(\mc{U}_\epsilon^{-1}(w)).
\end{equation}
Write 
\begin{equation}\label{eqn:fixed_point_forcing}
\mc{G}_{\epsilon, \mu}(w) = - [\mc{V}_\epsilon^{-1} \mc{E}_\epsilon(\ov{f}_{\epsilon, \mu}) - G_\mu] - \mc{V}_\epsilon^{-1}\mc{N}_{\epsilon, \ov{f}_{\epsilon, \mu}}(\mc{U}_\epsilon^{-1}(w)).
\end{equation}
The result follows once we have a solution to the system
\begin{equation}\label{eqn:fixed_point_system}
\begin{cases}
    w = \mc{T}_0\Pi \mc{G}_{\epsilon, \mu}(w) \\
    \ell(\mc{G}_{\epsilon, \mu}(w)) = 1 - \frac{\mu}{\mu_*}
\end{cases}.
\end{equation}
Indeed, if we have a solution $(w, \mu)$ to \eqref{eqn:fixed_point_system}, then by \eqref{eqn:fixed_point_free_inverse} and \eqref{eqn:fixed_point_forcing}, we have 
\begin{align*}
    \mc{L}_0(w) & = \Pi \mc{G}_{\epsilon, \mu}(w) \\ 
    & = \mc{G}_{\epsilon, \mu}(w) - \text{Re}[\ell(\mc{G}_{\epsilon, \mu}(w))]G_+ - \text{Im}[\ell(\mc{G}_{\epsilon, \mu}(w))]G_- \\
    & = \mc{G}_{\epsilon, \mu}(w) - \text{Re}\left(1 - \frac{\mu}{\mu_*}\right)G_+ - \text{Im}\left(1 - \frac{\mu}{\mu_*}\right)G_- \\
    & = \mc{G}_{\epsilon, \mu}(w) - G_\mu \\
    & = - \mc{V}_\epsilon^{-1} \mc{E}_\epsilon(\ov{f}_{\epsilon, \mu}) - \mc{V}_\epsilon^{-1}\mc{N}_{\epsilon, \ov{f}_{\epsilon, \mu}}(\mc{U}_\epsilon^{-1}(w))
\end{align*}
This, by \eqref{eqn:fixed_point_conjugated_equation}, implies that $r = \mc{U}_\epsilon^{-1}(w)$ is a solution to \eqref{eqn:fixed_point_correction}. It remains to solve \eqref{eqn:fixed_point_system} for $(w, \mu)$. We start by solving the first equation. Consider the fixed point map
\begin{equation}\label{eqn:fixed_point_map}
\mc{H} = \mc{H}_{\epsilon, \mu} = \mc{T}_0\Pi\mc{G}_{\epsilon, \mu}.
\end{equation}
We claim that $\mc{H}$ is a contraction on the ball of radius $C\epsilon$ in $X_{\gamma_2}$ for some large enough $C$ and small enough $\epsilon$. Indeed, let $w, w_1, w_2 \in X_{\gamma_2}$ with $\|w\|_{X_{\gamma_2}}, \|w_1\|_{X_{\gamma_2}}, \|w_2\|_{X_{\gamma_2}} \leq C\epsilon$. By \eqref{eqn:fixed_point_residual_bound}, \eqref{eqn:fixed_point_residual_transport}, and \cref{lem:linear_conjugation}, 
\begin{equation}\label{eqn:fixed_point_conjugated_residual_bound}
\|\mc{V}_\epsilon^{-1} \mc{E}_\epsilon(\ov{f}_{\epsilon, \mu}) - G_\mu\|_{Y_{\gamma_2}} \lesssim \epsilon.
\end{equation}
Therefore by applying the uniform bounds for $\mc{E}_\epsilon(\ov{f}_{\epsilon, \mu}) - G_\mu$ in \eqref{eqn:fixed_point_residual_bound}, $\mc{U}_\epsilon^{-1}$, $\mc{V}_\epsilon^{-1}$, and $\mc{N}_{\epsilon, \ov{f}_{\epsilon, \mu}}$ in \eqref{eqn:fixed_point_nonlinear_bound}--\eqref{eqn:fixed_point_nonlinear_lipschitz}, we have
\begin{equation}\label{eqn:fixed_point_self_map_bound}
\|\mc{H}(w)\|_{X_{\gamma_2}} \leq C_{1}\epsilon + C_{2}\epsilon \|w\|_{X_{\gamma_2}} \leq C_1\epsilon + C_2 C\epsilon^2
\end{equation}
and 
\begin{equation}\label{eqn:fixed_point_contraction_bound}
\|\mc{H}(w_1) - \mc{H}(w_2)\|_{X_{\gamma_2}} \leq C_3\epsilon \|w_1 - w_2\|_{X_{\gamma_2}}
\end{equation}
where $C_1$, $C_2$, and $C_3$ are uniform constants coming from the bounds in \cref{thm:fixed_point_linear_inverse,thm:fixed_point_nonlinear_estimates}. By choosing $C$ large enough and then $\epsilon$ small enough, we can make the constant in \eqref{eqn:fixed_point_self_map_bound} less than $C\epsilon$ and the constant in \eqref{eqn:fixed_point_contraction_bound} less than $1/2$. This shows that $\mc{H}$ is a contraction on the ball of radius $C\epsilon$ in $X_{\gamma_2}$. Therefore, by Banach's fixed point theorem, there exists a unique solution $w_{\epsilon, \mu} \in X_{\gamma_2}$ to the equation $w = \mc{H}_{\epsilon, \mu}(w)$ that satisfies
\begin{equation}\label{eqn:fixed_point_solution_bound}
\|w_{\epsilon, \mu}\|_{X_{\gamma_2}} \lesssim \epsilon.
\end{equation}
For the scalar equation in \eqref{eqn:fixed_point_system}, define the fixed point map
\begin{equation}\label{eqn:fixed_point_obstruction_map}
\mc{K}_\epsilon(\mu) = \mu_*[1 - \ell(\mc{G}_{\epsilon, \mu}(w_{\epsilon, \mu}))].
\end{equation}  
We will show that the map in \eqref{eqn:fixed_point_obstruction_map} is continuous. First we show that the map $\mu \mapsto w_{\epsilon, \mu}$ is continuous into $X_{\gamma_1}$. Suppose $\mu_j \to \mu$. Then 
\begin{align*}
    \|w_{\epsilon, \mu_j} - w_{\epsilon, \mu}\|_{X_{\gamma_1}} & = \|\mc{H}_{\epsilon, \mu_j}(w_{\epsilon, \mu_j}) - \mc{H}_{\epsilon, \mu}(w_{\epsilon, \mu})\|_{X_{\gamma_1}} \\
    & \leq \|\mc{H}_{\epsilon, \mu_j}(w_{\epsilon, \mu_j}) - \mc{H}_{\epsilon, \mu_j}(w_{\epsilon, \mu})\|_{X_{\gamma_1}} + \|\mc{H}_{\epsilon, \mu_j}(w_{\epsilon, \mu}) - \mc{H}_{\epsilon, \mu}(w_{\epsilon, \mu})\|_{X_{\gamma_1}} \\
    & \leq \frac{1}{2}\|w_{\epsilon, \mu_j} - w_{\epsilon, \mu}\|_{X_{\gamma_1}} + \|\mc{H}_{\epsilon, \mu_j}(w_{\epsilon, \mu}) - \mc{H}_{\epsilon, \mu}(w_{\epsilon, \mu})\|_{X_{\gamma_1}},
\end{align*}
by the bounds for $\mc{H}$ from \cref{thm:fixed_point_linear_inverse,thm:fixed_point_nonlinear_estimates}. Next, for a fixed $w \in X_{\gamma_2}$, we have that 
\[\mc{H}_{\epsilon, \mu_j}(w) \to \mc{H}_{\epsilon, \mu}(w)\]
in $X_{\gamma_1}$ by the continuity of $\mc{E}_\epsilon(\ov{f}_{\epsilon, \mu})$ and $\mc{N}_{\epsilon, \ov{f}_{\epsilon, \mu}}$ in $\mu$ from \cref{thm:fixed_point_approximation,thm:fixed_point_nonlinear_estimates}. Therefore, $w_{\epsilon, \mu_j} \to w_{\epsilon, \mu}$ in $X_{\gamma_1}$. It follows by a similar argument that $\mc{G}_{\epsilon, \mu}(w_{\epsilon, \mu})$ is continuous in $\mu$ into $Y_{\gamma_1}$. Finally, since $\ell$ is a bounded linear functional on $Y_{\gamma_1}$, we have that $\mc{K}_\epsilon(\mu)$ is continuous in $\mu$. Using the fact that this map is continuous, combined with the bound
\begin{equation}\label{eqn:fixed_point_obstruction_bound}
|\mc{K}_\epsilon(\mu) - \mu_*| = |\mu_*||\ell(\mc{G}_{\epsilon, \mu}(w_{\epsilon, \mu}))| \lesssim \epsilon,
\end{equation}
we can choose $\epsilon$ small so that $\mc{K}_\epsilon$ maps the closed ball $\ov{B}$ into itself. By the Brouwer fixed point theorem, there exists a fixed point $\mu_\epsilon \in \ov{B}$ so that $\mc{K}_\epsilon(\mu_\epsilon) = \mu_\epsilon$ and $|\mu_\epsilon - \mu_*| \lesssim \epsilon$. This gives a solution $(w_{\epsilon, \mu_\epsilon}, \mu_\epsilon)$ to \eqref{eqn:fixed_point_system} and therefore a solution $r_{\epsilon, \mu_\epsilon} = \mc{U}_\epsilon^{-1}(w_{\epsilon, \mu_\epsilon})$ to \eqref{eqn:fixed_point_correction}. The $\epsilon$ bounds in \eqref{eqn:fixed_point_correction_bound} for $r_{\epsilon, \mu_\epsilon}$ follow from the bounds for $w_{\epsilon, \mu_\epsilon}$ in \eqref{eqn:fixed_point_solution_bound} and the uniform bounds for $\mc{U}_\epsilon^{-1}$.
\end{proof}

Finally, we are able to prove that the profiles constructed in \cref{thm:fixed_point_approximation,thm:fixed_point_correction} are smooth by a bootstrapping argument. The proof of this theorem is given in \cref{sec:analysis_equation}.

\begin{theorem}[Smoothness of profiles]\label{thm:fixed_point_smoothness}
    Let $0<p<1$, $\alpha=p+1$, and $\epsilon\geq0$. Suppose that
    \[f\in W^{2,\infty}_{\textup{loc}}(\R) \quad \text{and} \quad f'\in L^\infty(\R)\]
    satisfies $\mc{E}_\epsilon(f)=0$ in the sense of distributions. Then
    \[f \in C^\infty(\R).\] 
    In particular, the profile
    $f=\ov{f}_{\epsilon,\mu_\epsilon}+r_\epsilon$ and the correction
    $r_\epsilon$ constructed in \cref{thm:fixed_point_correction} are smooth.
\end{theorem}

Using the previous theorems, we can now prove the main theorem of the paper.

\begin{proof}[Proof of \cref{thm:main}]
    By \cref{thm:fixed_point_correction}, for every $\epsilon$ small enough, there exists a solution $f_\epsilon = \ov{f}_{\epsilon, \mu_\epsilon} + r_\epsilon$ to the front equation \eqref{eqn:main_ideas_self_similar_front_equation_semilinear} satisfying the estimates in the theorem. Moreover, by \cref{thm:fixed_point_smoothness}, $f_\epsilon$ is smooth. It follows that 
    \[F_\epsilon(x, t) = \sqrt{\epsilon}\,L(t)f_\epsilon\left(\frac{x}{L(t)}\right), \qquad L(t) = (\alpha(T-t))^{1/\alpha}\]
    solves the front equation \eqref{eqn:front_equation}. The convergence claims of the theorem follow from the self-similar ansatz.
\end{proof}

\section{Analysis of the Equation}\label{sec:analysis_equation}

In this section, we will write the nonlinear operators in a more convenient form for the purpose of doing the nonlinear estimates. We will also prove the smoothness of the profiles constructed in \cref{thm:fixed_point_approximation,thm:fixed_point_correction}. Recall that the semilinear equation can be written as
\[\mc{E}_\epsilon(f) := |D|^{\alpha - 1}\partial_y f + \frac{1}{A_\epsilon(\partial_y f)}\left(f - y\partial_yf + \sq{\mc{J}}_\epsilon(f)\right) = 0.\]
In the following proposition, we will write the integral operator in a more convenient form for the purpose of doing the nonlinear estimates.

\begin{proposition}\label{prop:analysis_semilinear_kernel}
    Let $f\in W^{2,\infty}_{\textup{loc}}(\R)$ with $f'\in L^\infty(\R)$. We can write
    \[\mc{E}_{\epsilon}(f) = |D|^p f' + B_\epsilon(f')\left(f-yf'\right) + \mc{I}_\epsilon(f)\]
    where $B_\epsilon(u)=A_\epsilon(u)^{-1}=(1+\epsilon u^2)^{\alpha/2}$ and $\mc{I}_\epsilon$ is the nonlinear operator defined by
    \[\mc{I}_\epsilon(f) = \int \frac{1}{|y - z|^\alpha} \Gamma_\epsilon\left(f'(y), \frac{f(y) - f(z)}{y - z} \right) \, dz.\]
    Here, $\Gamma_\epsilon: \R^2 \to \R$ is the smooth function
    \begin{equation}\label{eq:G_epsilon}
    \begin{aligned}
    \Gamma_\epsilon(u, v) & = c_\alpha\Bigg[p \int_u^v
    \left(1-\frac{A_\epsilon(s)}{A_\epsilon(u)}\right)\,ds + (v-u)\left(1-\frac{A_\epsilon(v)}{A_\epsilon(u)}\right)\Bigg] 
    \end{aligned}
    \end{equation}
\end{proposition}

\begin{remark}\label{rmk:analysis_vanishing_diagonal}
    Note that we can equivalently write
    \begin{equation}\label{eqn:analysis_kernel_quadratic_factor}
    \begin{aligned}
        \Gamma_\epsilon(u,v)&=(u-v)^2\sq{\Gamma}_\epsilon(u,v),\\
        \sq{\Gamma}_\epsilon(u,v)&=-\frac{c_\alpha}{A_\epsilon(u)}
        \int_0^1\bigl[1+p(1-t)\bigr]A_\epsilon'\bigl(u+t(v-u)\bigr)\,dt
    \end{aligned}
    \end{equation}
    where $\sq{\Gamma}_\epsilon$ is smooth and satisfies, for $|u|, |v| \leq R$,
    \[|\sq{\Gamma}_\epsilon(u, v)| \lesssim_{R} \epsilon (|u| + |v|).\]
    Therefore for $f \in W^{2, \infty}_{\textup{loc}}$, the integrand in $\mc{I}_\epsilon(f)$ vanishes to order $|y - z|^{2 - \alpha}$ as $z \to y$. Moreover, for each $0 \leq \epsilon \leq 1$, $j \geq 0$, for $|u|, |v| \leq R$, we have that 
    \begin{equation}\label{eqn:analysis_kernel_derivative_bounds}|\nabla \Gamma_\epsilon(u, v)| \lesssim_{j, R} \epsilon |u - v|(|u| + |v|), \quad |\nabla^2 \Gamma_\epsilon| \lesssim_{j, R} \epsilon(|u| + |v|), \quad |\nabla^j \Gamma_\epsilon| \lesssim_{j, R} \epsilon\end{equation}
    and 
    \begin{equation}\label{eqn:analysis_kernel_diagonal_derivative_bounds}|(\partial_1 + \partial_2)^j \Gamma_\epsilon(u, v)| \lesssim_{j, R} \epsilon|u - v|^2, \quad |(\partial_1 + \partial_2) \nabla \Gamma_\epsilon(u, v)| \lesssim_{j, R} \epsilon|u - v|.\end{equation}
\end{remark}

\begin{proof}
    Dividing the remainder $\sq{\mc{J}}_\epsilon(f)$ by $A_\epsilon(f'(y))$, we define
    \begin{equation}\label{eqn:analysis_divided_remainder}
        \mc{I}_\epsilon(f)(y)
        :=\frac{\sq{\mc{J}}_\epsilon(f)(y)}{A_\epsilon(f'(y))}
        =c_\alpha\int_{\R}\frac{f'(y)-f'(z)}{|y-z|^\alpha}
        \left(\frac{A_\epsilon(S_f(y,z))}{A_\epsilon(f'(y))}-1\right)\,dz.
    \end{equation}
    Define the antiderivative
    \[
        P_\epsilon(v,q):=\int_v^q\left(1-\frac{A_\epsilon(u)}{A_\epsilon(v)}\right)\,du.
    \]
    Since 
    \[\partial_z S_f(y, z) = \frac{S_f(y, z) - f'(z)}{y - z},\]
    the chain rule gives
    \begin{equation}\label{eqn:analysis_secant_integration_by_parts}
    \begin{aligned}
        (y-z)\partial_zP_\epsilon(f'(y),S_f(y,z)) & = \left(S_f(y, z) - f'(z)\right)\left(1-\frac{A_\epsilon(S_f(y,z))}{A_\epsilon(f'(y))}\right) \\
        & = \left(S_f(y, z) - f'(y)\right)\left(1-\frac{A_\epsilon(S_f(y,z))}{A_\epsilon(f'(y))}\right) \\
        & + \left(f'(y) - f'(z)\right)\left(1-\frac{A_\epsilon(S_f(y,z))}{A_\epsilon(f'(y))}\right)
    \end{aligned}
    \end{equation}
    We substitute this equation into the integral defining $\mc{I}_\epsilon$ and integrate by parts in $z$. Indeed, 
    \begin{align*}
        \mc{I}_\epsilon(f)(y)
        &=c_\alpha\int_{\R}\frac{1}{|y-z|^\alpha}
        \left[(S_f(y, z) - f'(y)) \left(1 - \frac{A_\epsilon(S_f(y,z))}{A_\epsilon(f'(y))}\right)\right]\,dz \\
        & -c_\alpha \int  \frac{y - z}{|y - z|^\alpha} \partial_z P_\epsilon(f'(y),S_f(y,z))\,dz.
    \end{align*}
    We use the identity
    \[
        \partial_z\left(\frac{y-z}{|y-z|^\alpha}\right)
        =\frac{p}{|y-z|^\alpha},\qquad z\ne y.
    \]
    Note that the boundary terms vanish because
    \[f'(y) - \frac{f(y) - f(z)}{y - z} \to 0 \quad \text{as } z \to y\]
    since $f \in W^{2,\infty}_{\textup{loc}}$ and thus
    \[
        P_\epsilon(f'(y),S_f(y,z)) = O(|y-z|^{2}).
    \]
    At infinity, boundedness of $f'$ implies boundedness of $P_\epsilon(f'(y),S_f(y,z))$, so the same boundary term is $O(|y-z|^{-p})\to0$. We therefore obtain
    \begin{align*}
        \mc{I}_\epsilon(f)(y)
        &=\int_{\R}\frac{\Gamma_\epsilon(f'(y),S_f(y,z))}{|y-z|^\alpha}\,dz,
    \end{align*}
    where $\Gamma_\epsilon$ is the function in \eqref{eq:G_epsilon}.
\end{proof}

Using the proposition above, we can prove \cref{thm:fixed_point_smoothness}.

\begin{proof}[Proof of \cref{thm:fixed_point_smoothness}]
    In this proof, all implicit constants can depend on the norms of $f$ given in the assumptions. We will prove the implication
    \[f' \in C^s_{\textup{loc}} \Rightarrow \mc{I}_\epsilon(f) \in C^s_{\textup{loc}}\]
    for every noninteger $s$. Indeed, suppose $f' \in C_{\textup{loc}}^s$. Then we can write
    \[\mc{I}_\epsilon(f) = \int \frac{\Gamma_\epsilon(f'(y), S_f(y, y - t))}{|t|^{\alpha}}\,dt.\]
    Let $\chi$ be a smooth cutoff function equal to $1$ for $|t| \leq 1$ and equal to $0$ for $|t| \geq 2$. We can split the integral into a near part and a far part:
    \[\mc{I}_\epsilon(f) = \int \frac{\chi(t)\Gamma_\epsilon(f'(y), S_f(y, y - t))}{|t|^{\alpha}}\,dt + \int \frac{(1 - \chi(t))\Gamma_\epsilon(f'(y), S_f(y, y - t))}{|t|^{\alpha}}\,dt.\]
We first control the singular part of the integral when $|t| \leq 1$. Fix some number $R > 1$. We have that
    \[|f'(y) - S_{f}(y, y - t)| \lesssim_ R |t|\]
    for all $|y| \leq R$ and $|t| \leq 2$. We also have that 
    \[f' \in C_{\textup{loc}}^s, \quad (y \mapsto S_f(y, y - t)) \in C_{\textup{loc}}^s\]
    uniformly for $|t| \leq 1$, which can be seen by writing
    \[S_f(y, y - t) = \int_{0}^{1} f'(y - t't) \, dt'.\]
    Using the representation $\Gamma_\epsilon(u, v) = (u - v)^2\sq{\Gamma}_\epsilon(u, v)$ and the product estimate
    \[\|ab\|_{C^s} \lesssim \|a\|_{L^\infty} \|b\|_{C^s} + \|b\|_{L^\infty} \|a\|_{C^s},\]
    we can bound
    \begin{align*}
        \|\Gamma_\epsilon(f', S_f(\cdot, \cdot - t))\|_{C^s([-R,R])} & \lesssim_R |t|^2 \|\sq{\Gamma}_\epsilon(f', S_f(\cdot, \cdot - t))\|_{C^s([-R,R])} \\
        & + |t|\|\sq{\Gamma}_\epsilon(f', S_f(\cdot, \cdot - t))\|_{L^\infty} \\
        & \lesssim_R |t|
    \end{align*}
    where we used that $\sq{\Gamma}_\epsilon$ is smooth. It follows that we can bound the $C^s([-R,R])$ norm of the near part of the integral by 
    \[\int_{|t| \leq 1}|t|^{-p} \, dt < \infty.\]
    For the far part, observe that the function 
    \[y \mapsto S_f(y, z)(1 - \chi(y - z)) = \frac{f(y) - f(z)}{y - z}(1 - \chi(y - z))\]
    has a uniform $C^s([-R,R])$ bound for all $|y - z| \geq 1$. Therefore we can bound the $C^s([-R,R])$ norm of 
    \[y \mapsto \frac{S_f(y, z)(1 - \chi(y - z))}{|y - z|^\alpha}\]
    by a constant times $\ang{z}^{-\alpha}$. Here we used that the $C^s([-R,R])$ norm of $y \mapsto |y - z|^{-\alpha}(1 - \chi(y - z))$ is bounded by a constant times $\ang{z}^{-\alpha}$ for all $|y - z| \geq 1$. Therefore using that $\Gamma_\epsilon$ is smooth, we can bound the $C^s([-R,R])$ norm of the far part by a constant. Therefore we have shown that $\mc{I}_\epsilon(f) \in C^s_{\textup{loc}}$ whenever $f' \in C^s_{\textup{loc}}$. \par
    To perform the bootstrap, observe that if $f' \in W_{\textup{loc}}^{1,\infty}$ then we know that $f' \in C_{\textup{loc}}^{s_0}$ for any $0<s_0<1$. By writing 
    \[|D|^pf' = -B_\epsilon(f')(f - yf') - \mc{I}_\epsilon(f),\]
    we can use the assumptions above and the regularity of the local term to see that $|D|^pf'\in C_{\textup{loc}}^{s_0}$.
    By the interior Schauder estimate (see \cite[Corollary~2.4]{RosOton-Serra:dirichlet-fractional-laplacian}), we obtain $f'\in C_{\textup{loc}}^{s_0+p}$ whenever $s_0,s_0+p\notin\N$. \footnote{The theorem we are citing assumes that the function is smooth, but we can apply the theorem to $f'*\rho_\delta$ and then pass to the limit as $\delta\to0$. The other assumptions of the theorem are a consequence of the fact that $\|f'\|_{L^\infty} < \infty$.}
\end{proof}

\begin{lemma}\label{lem:analysis_coefficient}
    Consider the function $b_\epsilon(y)$ defined by 
    \[b_\epsilon(y) = \chi_+(y)\int \frac{\partial_1\Gamma_\epsilon(W'(y), S_W(y,z))}{|y-z|^{\alpha}}\,dz, \quad W(y) = \max\{-y,0\},\]
    where $\chi_+$ is a smooth cutoff function equal to $1$ for $y \geq 2$ and equal to $0$ for $y \leq 1$. Then we can equivalently write 
    \[b_\epsilon(y) = c_\epsilon\chi_+(y)y^{-p}, \quad |c_\epsilon| \lesssim \epsilon.\]
\end{lemma}

\begin{remark}\label{rmk:analysis_coefficient}
    The coefficient $b_\epsilon$ arises from the nonlocal interaction between the negative tail of the approximate solution $\ov{f}$ and the correction $r$ for $y > 0$. Indeed, the linearization $D\mc{I}_\epsilon(\ov{f})[r]$ contains the term
    \[\left(\int \frac{\partial_1\Gamma_\epsilon(\ov{f}'(y), S_{\ov{f}}(y,z))}{|y-z|^\alpha}\,dz\right)r'(y).\]
    This coefficient is nonperturbative because for $y > 0$ each derivative loses a power $\ang{y}^{1/p}$ while the gain $\ang{y}^{-p}$ is not enough to compensate for it. Therefore it must be included in the linear inverse problem.
\end{remark}

\begin{proof}
    It suffices to consider $y>1$, where $W'(y)=0$. For $z\geq0$, we have $S_W(y,z)=0$ and $\partial_1\Gamma_\epsilon(0,0)=0$, so this region contributes nothing. For $z<0$, we have
    \[S_W(y, z) = \frac{z}{y - z}.\]
    Therefore the integral equals
    \[\int_{z < 0} \frac{\partial_1\Gamma_\epsilon(0, z / (y - z))}{|y - z|^\alpha} \, dz.\] 
    Differentiating \eqref{eq:G_epsilon} gives
    \begin{equation}\label{eqn:analysis_kernel_first_partial}
    \begin{aligned}
        \partial_1\Gamma_\epsilon(u,v)
        &=c_\alpha\left(\frac{A_\epsilon(v)}{A_\epsilon(u)}-1\right)\\
        &\quad+c_\alpha\frac{A_\epsilon'(u)}{A_\epsilon(u)^2}
        \left[p\int_u^v A_\epsilon(s)\,ds+(v-u)A_\epsilon(v)\right].
    \end{aligned}
    \end{equation}
    In particular, since $A_\epsilon(0)=1$ and $A_\epsilon'(0)=0$,
    \[
        \partial_1\Gamma_\epsilon(0,v)
        =c_\alpha\bigl(A_\epsilon(v)-1\bigr)
        =c_\alpha\left[(1+\epsilon v^2)^{-\alpha/2}-1\right].
    \]
    We make the change of variables
    \[
        t=\frac{-z}{y-z},\qquad z=-\frac{yt}{1-t},\qquad
        y-z=\frac{y}{1-t},\qquad dz=-\frac{y}{(1-t)^2}\,dt.
    \]
    Therefore the integral equals
    \[c_\alpha y^{-p}\int_0^1\left[(1+\epsilon t^2)^{-\alpha/2}-1\right](1-t)^{p-1}\,dt = c_\epsilon y^{-p}\]
    where 
    \[c_\epsilon = c_\alpha\int_0^1\left[(1+\epsilon t^2)^{-\alpha/2}-1\right](1-t)^{p-1}\,dt, \quad |c_\epsilon| \lesssim \epsilon.\]
\end{proof}

\section{Construction of the Approximate Solution}\label{sec:approximate_solution}

In this section, we will prove \cref{thm:fixed_point_approximation}. We will find a family of approximate solutions $\ov{f}_{\epsilon, \mu}$ so that the residual 
\[\sq{\mc{E}}_{\epsilon}(f) := A_\epsilon(\partial_y f) \mc{E}_\epsilon(f) = \mc{L}_0(f) + \mc{Q}_\epsilon(f)\]
is small. Recall that the nonlinear operator $\mc{Q}_\epsilon$ is given by
\[\mc{Q}_\epsilon(f) = c_\alpha \int \frac{f'(y) - f'(z)}{|y - z|^\alpha}\left[A_\epsilon\left(\frac{f(y) - f(z)}{y - z}\right) - 1\right] \, dz\]
which only contains terms of cubic order and higher. There is no reason to work with the divided/semilinear equation in this section because the undivided equation has a simpler form, and we don't need to worry about quasilinear structure when constructing the approximate solution.

\subsection{Study of the Linear Solution}\label{sec:approx_linear_solution}

In this section, we will study the solution to the linearized equation $\mc{L}_0(f) = 0$. Let $f_{0}$ denote the solution to the linearized equation $\mc{L}_0(f_{0}) = 0$ satisfying the boundary conditions
\[\lim_{y \to -\infty} f_{0}'(y) = -1, \quad \lim_{y \to \infty} f_{0}'(y) = 0.\]
We will see in the next lemma that we can write $f_{0}$ explicitly as
\begin{equation}\label{eqn:fixed_point_linear_solution}
f_{0}(y) = -\frac{y}{2} + \frac{1}{\pi}\int_0^\infty
\frac{1-\cos\bigl(y\xi-\Phi_0(\xi)\bigr)}{\xi^2}\,d\xi,
\qquad \Phi_0(\xi)=\frac{\xi|\xi|^p}{p+1}.
\end{equation}
Let $\chi_+$ be a smooth cutoff function equal to $1$ for $y \geq 2$ and equal to $0$ for $y \leq 1$. Similarly $\chi_-(\xi) = \chi_+(-\xi)$. Define the positive high-frequency part of $f_{0}$ by
\[f_{0, +}(y) = -\frac{1}{2\pi}\int_{0}^{\infty}\frac{\chi_+(\xi)}{\xi^2}e^{i(y\xi - \Phi_0(\xi))} \, d\xi\]
and let 
\[f_{0, \textup{high}}(y) = 2 \text{Re} f_{0, +}(y) , \quad f_{0, \textup{low}}(y) = f_{0}(y) - f_{0, \textup{high}}(y)\]
denote the high- and low-frequency portions of $f_{0}$, respectively. We also define the physical space phase and slowly varying amplitude 
\[\lambda_0(y)=y^{1/p},\quad
 \Psi_0(y)=\frac{p}{p+1}y^{1+1/p},\quad
 a_*(y)=\frac{e^{-i(\Psi_0(y)-\pi/4)}}
    {\mu_*\lambda_0(y)^{-(p+1)/2}}f_{0,+}'(y), \quad \mu_* = - \frac{i}{\sqrt{2\pi p}}.\]
Therefore we can write the positive high-frequency part of $f_{0}'$ as
\[f_{0, +}'(y) = \mu_* \lambda_0^{-(p+1)/2} a_* e^{i(\Psi_0(y) - \pi/4)}, \quad y \gg 1.\]
In the following lemma, we will study the asymptotic behavior of $f_{0}$. 
\begin{lemma}[The linear solution]\label{lem:approx_linear_seed}
The function $f_0$ is smooth and satisfies
$\mc L_0(f_0)=0$. For every $j \geq 0$, we have the following estimates:
\begin{enumerate}
    \item The coefficient $a_*$ is slowly varying:
    \[|\partial_y^j(a_*-1)(y)|\lesssim_j y^{-2\beta-j}, \quad y \gg 1.\]
    \item The low-frequency portion of  $f_0$ satisfies
    \[|\partial_y^j(f_{0,\textup{low}}-W)(y)|\lesssim_j |y|^{-p-j},\qquad |y|\gg1,\]
    and $f_0$ satisfies
    \[|\partial_y^j(f_0(y)+y)|\lesssim_j |y|^{-p-j},\qquad y\ll-1.\]
    \item The function
        \begin{equation}\label{eqn:approx_amplitude_seed}
        f_{0,\mu}=f_0
        +2\operatorname{Re}\left[\left(\frac\mu{\mu_*}-1\right)f_{0, +}\right]
        \end{equation}
        satisfies $\mc L_0(f_{0,\mu}) = G_\mu$.
\end{enumerate}
\end{lemma}

\begin{proof}
    To show that $f_0$ solves the equation $\mc{L}_0(f_0) = 0$, observe that $\partial \mc{L}_0 = (|D|^p - y)\partial^2$. Write
    \[f_0''(y) = \frac{1}{\pi}\int_{0}^{\infty}\cos(y \xi - \Phi_0(\xi)) \, d\xi = \frac{1}{2\pi} \int e^{i(y \xi - \Phi_0(\xi))} \, d\xi\]
    so that $\widehat{f_0''} = e^{-i\Phi_0}$. Therefore
    \[\widehat{\partial_y \mc{L}_0 f_0}(\xi) = |\xi|^pe^{-i\Phi_0} - i \partial_\xi e^{-i\Phi_0} = (|\xi|^p - \Phi_0'(\xi))e^{-i\Phi_0(\xi)} = 0.\]
    It follows that $\mc{L}_0(f_0)$ is constant. To evaluate this constant, note that
    \[f_0(0) = \frac{1}{\pi}\int_{0}^{\infty}\frac{1 - \cos(\Phi_0(\xi))}{\xi^2} \, d\xi = \frac{1}{\pi}\int_{0}^{\infty}\xi^{p-1}\sin(\Phi_0(\xi)) \, d\xi = -(|D|^pf_0')(0)\]
    and thus $\mc{L}_0(f_0)(0) = 0$, which implies that $\mc{L}_0(f_0) = 0$. \par
    For the assertion about the slowly varying amplitude $a_*$, we apply \cref{lem:appendix_homogeneous_stationary_phase} to the integral defining $f_{0, +}'$ with $a(\xi) = \chi_+(\xi) / \xi$ to obtain 
    \begin{align*}
        f_{0, +}'(y) & = - \frac{i}{2\pi} \int_{0}^{\infty} \frac{\chi_+(\xi)}{\xi} e^{i(y\xi - \Phi_0(\xi))} \, d\xi \\
        & = -\frac{i}{2\pi} e^{i(\Psi_0(y) - \pi/4)}\left(\frac{2\pi}{p \lambda_0(y)^{p - 1}}\right)^{1/2}(a(\lambda_0(y)) + r(y)) \\
        & = \mu_* e^{i(\Psi_0(y) - \pi/4)}\lambda_0^{-(p-1)/2}(a(\lambda_0(y)) + r(y)). 
    \end{align*}
    Therefore we can write $a_*$ as
    \[a_*(y) = \lambda_0(y)\bigl(a(\lambda_0(y)) + r(y)\bigr).\]
    Since $a(\lambda_0(y)) = \lambda_0(y)^{-1}$ for $y \geq 2^p$, we have that
    \[a_*(y) - 1 = \lambda_0(y)r(y).\]
    The bounds for $a_*$ follow immediately. \par
    For the second assertion, using the integral representation
    \[W(y)=-\frac{y}{2}+\frac1\pi\int_0^\infty\frac{1-\cos(y\xi)}{\xi^2}\,d\xi,\]
    we write
    \begin{align*}
        f_{0,\textup{low}}(y)-W(y)
        &=\frac1\pi\operatorname{Re}\int_0^\infty e^{iy\xi}b(\xi)\,d\xi
        +\frac1\pi\operatorname{Re}\int_0^\infty e^{iy\xi}\frac{\chi_+(\xi)}{\xi^2}\,d\xi\\
        &=\mc{I}_1(y)+\mc{I}_2(y),
    \end{align*}
    where
    \[b(\xi)=(1-\chi_+(\xi))\frac{1-e^{-i\Phi_0(\xi)}}{\xi^2}.\]
    Since
    \[1-e^{-i\Phi_0(\xi)}=i\Phi_0(\xi)\int_0^1e^{-it\Phi_0(\xi)}\,dt,\]
    we have
    \[|\partial_\xi^k b(\xi)|\lesssim_k\xi^{p-1-k},\qquad 0<\xi\leq2,\quad k\geq0,\]
    and $b$ vanishes for $\xi\geq2$.
    Fix $j\geq0$ and set $\rho=|y|^{-1}$. Differentiating the compactly supported integral first and then splitting at $\rho$, we obtain
    \[
        |\partial_y^j\mc{I}_1(y)|
        \lesssim\int_0^\rho\xi^{p+j-1}\,d\xi
        +\left|\int_\rho^\infty e^{iy\xi}\xi^jb(\xi)\,d\xi\right|.
    \]
    The first term is $O(|y|^{-p-j})$. Integrating the second term by parts $j+1$ times gives
    \begin{align*}
        \left|\int_\rho^\infty e^{iy\xi}\xi^jb(\xi)\,d\xi\right|
        &\lesssim_j\sum_{k=0}^{j}|y|^{-k-1}\rho^{p+j-1-k}
        +|y|^{-j-1}\int_\rho^2\xi^{p-2}\,d\xi\\
        &\lesssim_j |y|^{-p-j},
    \end{align*}
    where we used $0<p<1$.
    For $\mc{I}_2$, integrate by parts sufficiently many times before differentiating in $y$. Since
    \[
        \left|\partial_\xi^N\left(\frac{\chi_+(\xi)}{\xi^2}\right)\right|
        \lesssim_N\xi^{-2-N},\qquad \xi\geq1,
    \]
    the resulting differentiated integrals converge absolutely, and
    \[|\partial_y^j\mc{I}_2(y)|\lesssim_{j,M}|y|^{-M},\qquad |y|\gg1,\quad M\geq0.\]
    This proves the estimate for $f_{0,\textup{low}}-W$.
    Finally, for $y<0$,
    \[
        f_{0,+}(y)=-\frac1{2\pi}\int_0^\infty
        \frac{\chi_+(\xi)}{\xi^2}e^{-i(|y|\xi+\Phi_0(\xi))}\,d\xi.
    \]
    The phase derivative has magnitude $|y|+\xi^p$. Repeated integration by parts, again before differentiating in $y$, yields
    \[|\partial_y^j f_{0,+}(y)|\lesssim_{j,M}|y|^{-M},\qquad y\ll-1,\quad M\geq0.\]
    Since $W(y)=-y$ there and
    \[f_0+y=(f_{0,\textup{low}}-W)+2\operatorname{Re}f_{0,+},\]
    the negative-end estimate follows. \par
    For the last assertion, we can compute 
    \[\widehat{\mc{L}_0g}(\xi) = i\xi|\xi|^p\widehat{g}(\xi)+\widehat{g}(\xi)+\partial_\xi\bigl(\xi\widehat{g}(\xi)\bigr) = \xi\partial_\xi\widehat{g}(\xi)+(2+i\xi|\xi|^p)\widehat{g}(\xi).\]
    Since $\Phi_0'(\xi)=|\xi|^p$, we have
    \[\widehat{f_{0,+}}(\xi)=-\frac{\chi_+(\xi)}{\xi^2}e^{-i\Phi_0(\xi)}, \qquad \widehat{\mc{L}_0f_{0,+}}(\xi)=-\frac{\chi_+'(\xi)}{\xi}e^{-i\Phi_0(\xi)}.\]
    Using $\mc{L}_0f_0=0$ and the positive frequency support of $\widehat{f_{0,+}}$, it follows that
    \[\widehat{\mc{L}_0(f_{0,\mu})}(\xi)=\left(1-\frac{\mu}{\mu_*}\right)\frac{\chi_+'(\xi)}{\xi}e^{-i\Phi_0(\xi)}=\widehat{G_\mu}(\xi), \qquad \xi>0.\]
    The equality for $\xi<0$ follows by conjugate symmetry. Thus $\mc{L}_0f_{0,\mu}=G_\mu$.
\end{proof}

\subsection{Definition of the Approximate Solution Family}\label{sec:approximate_family}

In this subsection, we will define the family of approximate solutions $\ov{f}_{\epsilon, \mu}$. We will use the function $f_{0, \mu}$ as a seed to construct the approximate solution $\ov{f}_{\epsilon, \mu}$. We now introduce several terms that will be used to define the approximate solution family. Let $R$ be a large parameter and let $\chi_R$ be a smooth cutoff function supported in $(R, \infty)$ and equal to 1 for $y > 2R$, with $|\partial_y^j\chi_R|\lesssim_j R^{-j}$ for every $j\geq0$.
Consider the cubic part of the nonlinearity
\[\mc{Q}_{\epsilon,3}(v)(y)
    =-\frac{(p+1)\epsilon c_\alpha}{2}\int_{\R}
    \frac{v'(y)-v'(z)}{|y-z|^\alpha}
    \left(\frac{v(y)-v(z)}{y-z}\right)^2\,dz.\]

\begin{lemma}
    There are constants $c_1, c_2$ satisfying $|c_1|,|c_2| \lesssim \epsilon$ so that 
    \[\mc{Q}_{\epsilon, 3}(\sin(\cdot))(y) = c_1 \cos(y) + c_2 \cos(3y).\]
\end{lemma}

\begin{proof}
    Let $\sigma \in \{\pm 1\}$. When we compute $\mc{Q}_{\epsilon, 3}(e^{i\sigma \cdot})$, the numerator will involve $(e^{i\sigma y} - e^{i\sigma z})^3$. It follows that the response will be a linear combination of the four terms $e^{\pm i y}$ and $e^{\pm 3i y}$. On the other hand, since $\mc{Q}_{\epsilon,3}(f)(-z) = -\mc{Q}_{\epsilon,3}(f(-\cdot))(z)$, we know that $\mc{Q}_{\epsilon, 3}(\sin(\cdot))$ is even. Therefore we can write
    \[\mc{Q}_{\epsilon, 3}(\sin(\cdot))(y) = c_1 \cos(y) + c_2 \cos(3y).\]
    The bounds for $c_1$ and $c_2$ follow from the fact that
    \[\|\mc{Q}_{\epsilon,3}(\sin(\cdot))\|_{L^\infty} \lesssim \epsilon \left\|\int \frac{\min\{1, |\cdot - z|^3\}}{|\cdot - z|^{\alpha + 2}} \, dz \right\|_{L^\infty} \lesssim \epsilon.\]
\end{proof}
It follows from this lemma that we can define $n_\epsilon$ and $c_{\textup{cubic}}$ as the constants defined so that 
\begin{equation}\label{eqn:approx_cubic_normalization}
\mc{Q}_{\epsilon,3}(2 \sin(\cdot))(y) = 2n_\epsilon \cos(y) - 2(3^p - 1)c_{\textup{cubic}}\cos(3y).
\end{equation}
Also define the constants
\begin{equation}\label{eqn:approx_correction_constants}
\begin{gathered}
c_{\textup{log}}(\mu)=-\frac{n_\epsilon|\mu|^2}{p}, \quad c_{\textup{power}} =-\frac{c_\epsilon}{1-p^2}\\
c_+=-\frac{c_\epsilon}{p+1}, \quad c_- =\frac{c_\alpha}{p+1}\int_0^1[A_\epsilon(t)-1]t^{p-1}\,dt
\end{gathered}
\end{equation}
and for $y \geq R$, let
\begin{equation}\label{eqn:approx_modified_frequency}
\lambda_{\epsilon, \mu}(y) := \Psi_{\epsilon,\mu}'(y)
   =y^{1/p} + c_{\textup{power}}\left(\frac{1 - p^2}{p}\right) y^{1/p-p-1} +\frac{c_{\textup{log}}(\mu)}{y}, \quad \Psi_{\epsilon,\mu}(R) = \Psi_{0}(R)
\end{equation}
Define the modified leading oscillatory term 
\begin{equation}\label{eqn:approx_Z_definition}
Z_{\epsilon, \mu}(y) = \mu a_*(y) \lambda_{\epsilon, \mu}(y)^{-(p + 1)/2}e^{i(\Psi_{\epsilon,\mu}(y) - \pi/4)}.
\end{equation}
We will abbreviate the notation in the following way:
\[f=\ov{f}_{\epsilon,\mu},\qquad Z=Z_{\epsilon,\mu},\qquad \lambda=\lambda_{\epsilon,\mu},\qquad \Psi=\Psi_{\epsilon,\mu},\]
and, for these quantities, we will use the subscript $0$ to denote the corresponding quantities when $\epsilon=0$ and $\mu=\mu_*$. We remark that $Z_{0, \mu_*} = f_{0,+}'$.

\begin{definition}(Approximate solution family)\label{def:approximate_solution_family}
For $0 \leq \epsilon \leq \epsilon_0$ and $\mu \in \ov{B}$,
define the slow correction term 
\begin{equation}\label{eqn:approx_slow_correction}
f_{\textup{slow}}(y) = c_{+}y^{-p}\chi_R(y) + c_{-}|y|^{-p}\chi_R(-y)
\end{equation}
and the oscillatory term 
\begin{equation}\label{eqn:approx_oscillatory_correction}
f_{\textup{osc}}(y) = -2 \chi_R(y)\operatorname{Re} \int_{y}^{\infty}\left(Z_{\epsilon, \mu}(t) + c_{\textup{cubic}}Z_{\epsilon, \mu}(t)^3 - \frac{\mu}{\mu_*}Z_0(t)\right) \, dt.
\end{equation}
The approximate solution family is then defined as
\begin{equation}\label{eqn:approx_family_definition}
    \ov{f}_{\epsilon, \mu} = f_{0, \mu} + f_{\textup{slow}} + f_{\textup{osc}}.
\end{equation}
In other words, when $y \geq 2R$, we can write
\begin{equation}\label{eqn:approx_family_positive}
\begin{aligned}
    \ov{f}_{\epsilon, \mu}(y) & = f_{0} - 2 \operatorname{Re}f_{0,+} + c_+y^{-p} - 2\operatorname{Re} \int_{y}^{\infty}\left(Z_{\epsilon, \mu}(t) + c_{\textup{cubic}}Z_{\epsilon, \mu}(t)^3\right) \, dt \\
    & = f_{0,\textup{low}} + c_+y^{-p} - 2\operatorname{Re} \int_{y}^{\infty}\left(Z_{\epsilon, \mu}(t) + c_{\textup{cubic}}Z_{\epsilon, \mu}(t)^3\right) \, dt
\end{aligned}
\end{equation}
and when $y \leq -2R$, we can write
\begin{equation}\label{eqn:approx_family_negative}
\ov{f}_{\epsilon, \mu}(y) = f_{0}(y) + 2\operatorname{Re}\left[\left(\frac{\mu}{\mu_*} - 1\right)f_{0,+}(y)\right] + c_-|y|^{-p}.
\end{equation}
\end{definition}

Given \cref{def:approximate_solution_family}, we can now prove the properties of the approximate solution family in \cref{thm:fixed_point_approximation}. 

\begin{lemma}\label{lem:approx_phase_differences}
    For every integer $j\geq0$ and $y\geq R$, we have the estimates
    \begin{align}
        \left|\partial_y^j\bigl(\Psi_{\epsilon,\mu}(y)-\Psi_0(y)\bigr)\right|
        &\lesssim_j\epsilon y^{\frac1p-p-j},
        \label{eqn:approx_phase_difference}\\
        \left|\partial_y^j\bigl(\lambda_{\epsilon,\mu}(y)-\lambda_0(y)\bigr)\right|
        &\lesssim_j\epsilon y^{\frac1p-p-1-j},
        \label{eqn:approx_lambda_difference}\\
        \left|\partial_y^j\left(Z_{\epsilon,\mu}(y)-\frac{\mu}{\mu_*}Z_0(y)\right)\right|
        &\lesssim_j\epsilon y^{\frac1p-p-\beta+\frac{j}{p}}.
        \label{eqn:approx_Z_difference}
    \end{align}
    Moreover, for $y \geq R$,
    \begin{equation}\label{eqn:approx_fosc_bounds}
        |\partial_y^j f_{\textup{osc}}(y)|
        \lesssim_j \min\{\epsilon\ang{y}^{-p-\beta+\frac{j}{p}}, \ang{y}^{-\beta + \frac{j-1}{p}} \}, \qquad j \geq 0
    \end{equation}
    In particular,
    \begin{equation}\label{eqn:approx_fosc_L1_bound}
        \|f_{\textup{osc}}\|_{L^1(\R)}
        \lesssim\epsilon R^{1-p-\beta}\lesssim\epsilon.
    \end{equation}
    The implicit constants are uniform in $0\leq\epsilon\leq\epsilon_0$, $\mu\in\ov{B}$, and $R$ sufficiently large.
\end{lemma}

\begin{proof}
    The estimate for $\lambda$ follows immediately from \eqref{eqn:approx_modified_frequency}. The estimate for $\Psi$ follows from integrating the estimate for $\lambda$:
    \[\Psi_{\epsilon,\mu}(y) - \Psi_0(y) = \int_R^y (\lambda_{\epsilon,\mu}(t) - \lambda_0(t)) \, dt.\]
    The estimate for $Z_{\epsilon,\mu}$ follows from \eqref{eqn:approx_Z_definition} by writing 
    \[
    \begin{aligned}
        Z_{\epsilon,\mu}(y)-\frac{\mu}{\mu_*}Z_0(y)
        &=\mu a_*(y)\Bigl[
        \bigl(\lambda_{\epsilon,\mu}(y)^{-(p+1)/2}-\lambda_0(y)^{-(p+1)/2}\bigr)
        e^{i(\Psi_{\epsilon,\mu}(y)-\pi/4)}\\
        &\qquad+\lambda_0(y)^{-(p+1)/2}
        \bigl(e^{i(\Psi_{\epsilon,\mu}(y)-\pi/4)}-e^{i(\Psi_0(y)-\pi/4)}\bigr)
        \Bigr]
    \end{aligned}
    \]
    and using the estimates for $\lambda$ and $\Psi$. To prove \eqref{eqn:approx_fosc_bounds}, using \eqref{eqn:approx_phase_difference}--\eqref{eqn:approx_lambda_difference} and the bounds for $a_*$ in \cref{lem:approx_linear_seed}, we obtain
    \[
    \begin{aligned}
        \left|\int_y^\infty\left(Z_{\epsilon,\mu}(t)-\frac{\mu}{\mu_*}Z_0(t)\right)\,dt\right|
        &\lesssim
        \left|\frac{Z_{\epsilon,\mu}(y)}{\lambda_{\epsilon,\mu}(y)}
        -\frac{\mu}{\mu_*}\frac{Z_0(y)}{\lambda_0(y)}\right|
        +\epsilon\int_y^\infty t^{-p-\beta-1}\,dt\\
        &\lesssim\epsilon y^{-p-\beta},\qquad y\geq R.
    \end{aligned}
    \]
    Here the integral bounds the difference of the remainder integrands, in which the derivative falls on the slow amplitude $a_*\lambda^{-(p+3)/2}$. The cubic contribution satisfies
    \[
        \left|c_{\textup{cubic}}\int_y^\infty Z_{\epsilon,\mu}(t)^3\,dt\right|
        \lesssim\epsilon\int_y^\infty t^{-3\beta}\,dt
        \lesssim\epsilon y^{1-3\beta}
        \lesssim\epsilon y^{-p-\beta},
    \]
    since $2\beta=1+1/p$ and $p<1$. For the other estimate in \eqref{eqn:approx_fosc_bounds}, we bound $f_{\textup{osc}}$ as
    \[|f_{\textup{osc}}| \lesssim \left|\int_{y}^{\infty}Z(t) \, dt\right| + \left|\int_{y}^{\infty}Z_0(t) \, dt\right| \lesssim y^{-\beta - 1/p}\]
    by integrating by parts and using the bounds for $a_*$ and $\lambda$. This proves \eqref{eqn:approx_fosc_bounds} for $j=0$. For $j\geq1$, we differentiate \eqref{eqn:approx_oscillatory_correction} and use \eqref{eqn:approx_Z_difference}, $|\partial_y^{j-1}(Z_{\epsilon,\mu}^3)|\lesssim_j y^{-3\beta+(j-1)/p}$, and the bounds for $\chi_R$. Finally, integrating the $j=0$ estimate over $(R,\infty)$ proves \eqref{eqn:approx_fosc_L1_bound}, since $p+\beta>1$.
\end{proof}

\begin{proposition}\label{prop:approx_profile_bounds}
    Let $\ov{f}_{\epsilon, \mu}$ be the family of approximate profiles from \cref{def:approximate_solution_family}. We have the following estimates for $y \geq 0$:
        \[|\ov{f}_{\epsilon, \mu}(y)| \lesssim \ang{y}^{-p}, \quad |\partial^{j+1}\ov{f}_{\epsilon, \mu}(y)| \lesssim \ang{y}^{-\beta + \frac{j}{p}}, \quad j \geq 0\]
        and for $y \leq 0$,
        \[|\partial^{j}(\ov{f}_{\epsilon, \mu}(y) + y)| \lesssim \ang{y}^{-p - j}.\]
    Moreover, for each $\epsilon$, the map $\mu \mapsto \ov{f}_{\epsilon, \mu}$ is continuous as a map into $C_{\textup{loc}}^2$.
\end{proposition}

\begin{proof}
    For $y \geq 2R$, we use the bounds
    \[|\partial^{j}(Z_{\epsilon, \mu})(y)| \lesssim \ang{y}^{-\beta + \frac{j}{p}},\]
    to obtain 
    \begin{align*}
        |\ov{f}_{\epsilon,\mu}(y)| \lesssim |f_{0, \textup{low}}| + |y|^{-p} + \left|\int_{y}^{\infty}a_*(t) \lambda(t)^{-(p+1)/2}e^{i\Psi_{\epsilon,\mu}(t)}\, dt\right|.
    \end{align*}
    Here the cubic term is bounded by $|y|^{1-3\beta} \lesssim |y|^{-p}$, so it was absorbed into the $|y|^{-p}$ term. Integrating by parts in the last term bounds it as 
    \begin{align*}
        &\lesssim \int_{y}^{\infty}|\partial_t(a_* \lambda^{-(p+1)/2 - 1})| \, dt + |a_*(y)\lambda(y)^{-(p+1)/2 - 1}|\\
        &\lesssim \int_y^\infty t^{-\beta-\frac1p-1}\,dt+y^{-\beta-\frac1p}\\
        &\lesssim y^{-\beta-\frac1p} \\
        & \lesssim y^{-p}.
    \end{align*}
    Using the estimates on $f_{0,\textup{low}}$ from \cref{lem:approx_linear_seed}, we obtain the desired bound for $\ov{f}_{\epsilon, \mu}$. The bounds for the derivatives of $\ov{f}_{\epsilon, \mu}$ follow from the bounds for $Z_{\epsilon,\mu}$. On the other hand, for $y \leq -2R$, \eqref{eqn:approx_family_negative} gives 
    \[\ov{f}_{\epsilon,\mu}(y) = f_0(y) + 2\text{Re}\left[\left(\frac{\mu}{\mu_*} - 1\right)f_{0,+}(y)\right] + c_- |y|^{-p}.\]
    The desired bounds for $y \leq -2R$ then follow from the bounds for $f_0$ in \cref{lem:approx_linear_seed}. The assertion of continuity in the parameters follows immediately from \eqref{eqn:approx_family_definition} and the continuity in $\mu$ of $Z_{\epsilon, \mu}$ defined in \eqref{eqn:approx_Z_definition}.
\end{proof}

\subsection{Residual Bounds for \texorpdfstring{$y \geq 8R$}{y >= 8R}} 
In this section, we will prove estimates on the residual $\sq{\mc{E}}$ in the positive region $(8R, \infty)$. 

\begin{proposition}\label{prop:approx_residual_positive}
    For $y \geq 8R$, we have the estimate
    \[|\partial^j(\sq{\mc{E}}_\epsilon(f)-G_\mu)(y)| \lesssim \epsilon \ang{y}^{-\gamma + \frac{j}{p}}, \quad j = 0,1.\]
\end{proposition}

To prove this proposition, we will decompose the residual $\sq{\mc{E}}_\epsilon(f)-G_\mu$ into nine terms, each controlled by a different mechanism. Fix $Y\geq8R$ and suppose $Y\leq y\leq2Y$. Choose a smooth cutoff $\chi_Y$ supported in $(Y/4,4Y)$ and equal to $1$ on $[Y/2,3Y]$, with $|\partial^j\chi_Y|\lesssim_j Y^{-j}$.
We set
\begin{equation}\label{eqn:approx_local_oscillatory_front}
    v_Y(z)=-2\chi_Y(z)\operatorname{Re}\int_z^\infty
    \left(Z(t)+c_{\textup{cubic}}Z(t)^3\right)\,dt.
\end{equation}
Denote the cubic part of $\mc{Q}_\epsilon$ by
\begin{equation}\label{eqn:approx_cubic_operator}
    \mc{Q}_{\epsilon,3}(v)(y)
    =-\frac{(p+1)\epsilon c_\alpha}{2}\int_{\R}
    \frac{v'(y)-v'(z)}{|y-z|^\alpha}
    \left(\frac{v(y)-v(z)}{y-z}\right)^2\,dz.
\end{equation}
Since $\mc{L}_0(f_{0,\mu})=G_\mu$, we have the exact decomposition
\begin{equation}\label{eqn:approx_residual_decomposition}
    \sq{\mc{E}}_\epsilon(f)-G_\mu
    =\mc{L}_0(f_{\textup{slow}})+\mc{L}_0(f_{\textup{osc}})+\mc{Q}_\epsilon(f)
    =\sum_{k=1}^9\mc{E}_k,
\end{equation}
where
\begingroup
\allowdisplaybreaks[4]
\begin{align}
    \mc{E}_1={}&\mc{L}_0(f_{\textup{slow}})+\mc{Q}_\epsilon(W),
        \label{eqn:approx_E1}\\
    \mc{E}_2={}&2\left(\lambda^p-y-p\lambda_0^{p-1}\frac{c_{\textup{log}}(\mu)}{y}
        +c_\epsilon y^{-p}\right)\operatorname{Re}Z,
        \label{eqn:approx_E2}\\
    \mc{E}_3={}&2\left(p\lambda_0^{p-1}\frac{c_{\textup{log}}(\mu)}{y}
        +n_\epsilon\lambda^p|Z|^2\right)\operatorname{Re}Z,
        \label{eqn:approx_E3}\\
    \mc{E}_4={}&2c_{\textup{cubic}}\left(\lambda^p-y+c_\epsilon y^{-p}\right)
        \operatorname{Re}(Z^3),
        \label{eqn:approx_E4}\\
    \mc{E}_5={}&\mc{L}_0(f_{\textup{osc}})+\mc{Q}_{\epsilon,3}(v_Y)\notag\\*
        &-2\left(\lambda^p-y+n_\epsilon\lambda^p|Z|^2\right)\operatorname{Re}Z
        -2c_{\textup{cubic}}(\lambda^p-y)\operatorname{Re}(Z^3),
        \label{eqn:approx_E5}\\
    \mc{E}_6={}&D\mc{Q}_\epsilon(W)[f-W]
        -2c_\epsilon y^{-p}\operatorname{Re}\left(Z+c_{\textup{cubic}}Z^3\right),
        \label{eqn:approx_E6}\\
    \mc{E}_7={}&\mc{Q}_{\epsilon,3}(\chi_Y f)-\mc{Q}_{\epsilon,3}(v_Y),
        \label{eqn:approx_E7}\\
    \mc{E}_8={}&\mc{Q}_\epsilon(\chi_Y f)-\mc{Q}_{\epsilon,3}(\chi_Y f),
        \label{eqn:approx_E8}\\
    \mc{E}_9={}&\mc{Q}_\epsilon(f)-\mc{Q}_\epsilon(W)
        -D\mc{Q}_\epsilon(W)[f-W]-\mc{Q}_\epsilon(\chi_Y f).
        \label{eqn:approx_E9}
\end{align}
\endgroup
The identity \eqref{eqn:approx_residual_decomposition} follows by adding these definitions. 

\subsection{Control of \texorpdfstring{$\mc{E}_1$}{E1}}
In this subsection, we will control the term 
\[\mc{E}_1 = \mc{L}_0(f_{\textup{slow}}) + \mc{Q}_\epsilon(W).\] Observe that we can compute, for $y > 0$, 
\[\mc{Q}_\epsilon(W)(y) =  c_\alpha\int_{z < 0}\frac{A_\epsilon(z/(y-z)) - 1}{|y - z|^\alpha} \, dz = c_\epsilon y^{-p}\]
exactly as in \cref{lem:analysis_coefficient}. On the other hand, we can compute
\[\mc{L}_0(f_{\textup{slow}})(y) -|D|^p f_{\textup{slow}}'(y) = f_{\textup{slow}}(y) - y f_{\textup{slow}}'(y) = c_+(1 + p)y^{-p}.\]
Since $c_+ = -c_\epsilon / (1 + p)$ by \eqref{eqn:approx_correction_constants}, we have that
\[\mc{E}_1(y) = |D|^p f_{\textup{slow}}'(y).\]
\begin{proposition}\label{prop:approx_E1}
    The residual $\mc{E}_1$ satisfies the estimate
    \[|\partial^j \mc{E}_1(y)| \lesssim \epsilon \ang{y}^{-\gamma + \frac{j}{p}}, \quad y \geq 2R, \quad j = 0,1.\]
\end{proposition}

\begin{proof}
    In fact, we will show the stronger estimate
    \begin{equation}\label{eqn:approx_E1_bound}
        |\partial^j \mc{E}_1(y)| \lesssim \epsilon \ang{y}^{-2p-1-j}, \qquad j=0,1.
    \end{equation}
    Recall from above that $\mc{E}_1(y) = |D|^p f_{\textup{slow}}'(y)$. Write $h = f_{\textup{slow}}$ and let $\chi$ be a smooth cutoff function supported in $[-1/2,1/2]$ and equal to one on $[-1/4,1/4]$. Then $h$ satisfies the estimates
    \[|\partial^j h(y)| \lesssim \epsilon \ang{y}^{-p  - j}.\]
    We write 
    \begin{align*}
        |D|^p h'(y) & = c_\alpha \int \frac{h'(y) - h'(z)}{|y - z|^{1 + p}} \, dz \\
        & = c_\alpha \int \chi\left(\frac{y - z}{y}\right) \frac{h'(y) - h'(z)}{|y - z|^{1 + p}} \, dz  \\
        & + c_\alpha \int \left(1 - \chi\left(\frac{y - z}{y}\right)\right) \frac{h'(y) - h'(z)}{|y - z|^{1 + p}} \, dz \\
        & = I_{\textup{near}} + I_{\textup{far}}.
    \end{align*}
    In the near region, we use 
    \[|h'(y) - h'(z)| \lesssim \epsilon \ang{y}^{-p-2}|y - z|\]
    and thus
    \[|I_{\textup{near}}| \lesssim \epsilon \ang{y}^{-p - 2} \int_{|y - z| \lesssim |y|} |y - z|^{- p} \, dz \lesssim \epsilon \ang{y}^{-2p - 1}.\]
    In the far region, we use the triangle inequality and bound
    \[\int_{|y - z| \gtrsim |y|} \frac{|h'(y)|}{|y - z|^{1 + p}} \, dz \lesssim \epsilon \ang{y}^{-2p - 1}.\]
    For the last term involving $h'(z)$, we integrate by parts to obtain 
    \begin{align*}
        \left|\int \left(1 - \chi\left(\frac{y - z}{y}\right)\right) \frac{h'(z)}{|y - z|^{1 + p}} \, dz\right| & \lesssim \ang{y}^{-1}\int_{|y - z| \gtrsim |y|} \frac{|h(z)|}{|y - z|^{p+1}} \, dz \\
        & + \int_{|y - z| \gtrsim |y|} \frac{|h(z)|}{|y - z|^{p + 2}} \, dz \\
        & \lesssim \epsilon\ang{y}^{-1}\int_{|y - z| \gtrsim |y|}\frac{\ang{z}^{-p}}{|y - z|^{p+1}} \, dz + \epsilon\int_{|y - z| \gtrsim |y|}\frac{\ang{z}^{-p}}{|y - z|^{p + 2}} \, dz \\
        & \lesssim \epsilon\ang{y}^{-2p - 1}.
    \end{align*}
    We can repeat the same argument with $h'$ in place of $h$ for the derivative of $\mc{E}_1$ to obtain the desired estimate. The only change is that we integrate by parts twice in the far integral.
\end{proof}

\subsection{Control of \texorpdfstring{$\mc{E}_2$}{E2}}

In this section, we will control the term
\[\mc{E}_2 = 2\left(\lambda^p-y-p\lambda_0^{p-1}\frac{c_{\textup{log}}(\mu)}{y}+c_\epsilon y^{-p}\right)\operatorname{Re}Z.\]

\begin{proposition}\label{prop:approx_E2}
    The residual $\mc{E}_2$ satisfies the estimate
    \[|\partial^j \mc{E}_2(y)| \lesssim \epsilon^2 \ang{y}^{-2p - 1 - \beta + \frac{j}{p}}, \quad y \geq 2R, \quad j = 0,1.\]
\end{proposition}

\begin{proof}
    Using \eqref{eqn:approx_modified_frequency} and \eqref{eqn:approx_correction_constants}, write 
    \[\lambda = \lambda_0(1 + \sq{\lambda}), \quad \sq{\lambda} = -\frac{c_\epsilon}{p}y^{-p - 1} + c_{\textup{log}}(\mu)y^{-1 - \frac{1}{p}}.\]
    Then we can express the residual as
    \[\mc{E}_2 = 2y[(1 + \sq{\lambda})^p - 1 - p\sq{\lambda}]\operatorname{Re}Z.\]
    By Taylor's formula, we have that
    \[(1 + \sq{\lambda})^p - 1 - p\sq{\lambda} = p(p - 1)(\sq{\lambda})^2\int_{0}^{1}(1 - t)(1 + t \sq{\lambda})^{p - 2} \, dt.\]
    It follows that 
    \[y|(1 + \sq{\lambda})^p - 1 - p\sq{\lambda}| \lesssim \epsilon^2 y(y^{-p-1})^{2} = \epsilon^2 y^{-2p - 1}\]
    and
    \[|\partial(y((1 + \sq{\lambda})^p - 1 - p\sq{\lambda}))| \lesssim y|\sq{\lambda}||\sq{\lambda}'| +  |\sq{\lambda}|^2 \lesssim \epsilon^2 y^{-2p-2}.\]
    Using the asymptotics for $a_*$ provided in \cref{lem:approx_linear_seed} and the fact that $\lambda \approx y^{1/p}$, we also obtain 
    \[|\partial^{j}Z| \lesssim \ang{y}^{-\beta + \frac{j}{p}}.\]
    Combining these estimates with the product rule gives the conclusion of the proposition.
\end{proof}

\subsection{Control of \texorpdfstring{$\mc{E}_3$}{E3}}

In this section, we will control the term
\[\mc{E}_3 = 2\left(p\lambda_0^{p-1}\frac{c_{\textup{log}}(\mu)}{y}
        +n_\epsilon\lambda^p|Z|^2\right)\operatorname{Re}Z.\]
\begin{proposition}\label{prop:approx_E3}
    The residual $\mc{E}_3$ satisfies the estimate
    \begin{equation}\label{eqn:approx_E3_bound}
        |\partial^j \mc{E}_3(y)|
        \lesssim \epsilon^2\ang{y}^{-p-3\beta+\frac{j}{p}}
        +\epsilon\ang{y}^{1-5\beta+\frac{j}{p}},
        \qquad y\geq2R,\quad j=0,1.
    \end{equation}
\end{proposition}

\begin{proof}
    Since $c_{\textup{log}}(\mu)=-n_\epsilon|\mu|^2/p$ by \eqref{eqn:approx_correction_constants} and $\lambda_0^p=y$, \eqref{eqn:approx_Z_definition} gives
    \[p\lambda_0^{p-1}\frac{c_{\textup{log}}(\mu)}{y}
      =-\frac{n_\epsilon|\mu|^2}{\lambda_0}.
      \qquad
      \lambda^p|Z|^2=\frac{|\mu|^2|a_*|^2}{\lambda}.\]
    Thus the leading cubic term cancels, and
    \begin{equation}\label{eqn:approx_E3_cancellation}
        \mc{E}_3=2n_\epsilon|\mu|^2
        \left(\frac{|a_*|^2}{\lambda}-\frac{1}{\lambda_0}\right)
        \operatorname{Re}Z.
    \end{equation}
    As in the estimate for $\mc{E}_2$, write $\lambda=\lambda_0(1+\sq{\lambda})$. Then
    \[\frac{|a_*|^2}{\lambda}-\frac{1}{\lambda_0}
      =\frac{1}{\lambda_0}\frac{|a_*|^2-1-\sq{\lambda}}{1+\sq{\lambda}}.\]
    By \cref{lem:approx_linear_seed} and the explicit formula for $\sq{\lambda}$,
    \[|\partial_y^k(|a_*|^2-1)|\lesssim y^{-2\beta-k},\qquad
      |\partial_y^k\sq{\lambda}|\lesssim\epsilon y^{-p-1-k},\qquad k=0,1,\]
    where we used $2\beta=1+1/p>p+1$. For sufficiently small $\epsilon$, we have $|\sq{\lambda}|\leq1/2$. Therefore,
    \begin{equation}\label{eqn:approx_E3_coefficient_bounds}
        \left|\partial_y^k\left(\frac{|a_*|^2}{\lambda}-\frac{1}{\lambda_0}\right)\right|
        \lesssim y^{-1/p-k}\left(y^{-2\beta}+\epsilon y^{-p-1}\right),
        \qquad k=0,1.
    \end{equation}
    Combining \eqref{eqn:approx_E3_cancellation} and \eqref{eqn:approx_E3_coefficient_bounds} with $|n_\epsilon|\lesssim\epsilon$, and $|\partial_y^jZ|\lesssim y^{-\beta+j/p}$, the product rule gives
    \[|\partial_y^j\mc{E}_3(y)|
      \lesssim\epsilon y^{-3\beta-1/p+j/p}
      +\epsilon^2y^{-\beta-1/p-p-1+j/p},\qquad j=0,1.\]
    The identities $3\beta+1/p=5\beta-1$ and $\beta+1/p+p+1=p+3\beta$ prove \eqref{eqn:approx_E3_bound}.
\end{proof}

\subsection{Control of \texorpdfstring{$\mc{E}_4$}{E4}}

In this section, we will control the term
\[\mc{E}_4=2c_{\textup{cubic}}\left(\lambda^p-y+c_\epsilon y^{-p}\right)\operatorname{Re}(Z^3).\]
\begin{proposition}\label{prop:approx_E4}
    The residual $\mc{E}_4$ satisfies the estimate
    \begin{equation}\label{eqn:approx_E4_bound}
        |\partial^j\mc{E}_4(y)|
        \lesssim\epsilon^2\ang{y}^{1-5\beta+\frac{j}{p}},
        \qquad y\geq2R,\quad j=0,1.
    \end{equation}
\end{proposition}

\begin{proof}
    The calculation for $\mc{E}_2$ gives
    \[\lambda^p-y+c_\epsilon y^{-p}
      =p\,c_{\textup{log}}(\mu)y^{-1/p}+O(\epsilon^2y^{-2p-1}),\]
    with the derivative bound on the remainder. Since $c_{\textup{log}}(\mu)=O(\epsilon)$ and $2p+1>1/p$, it follows that
    \[\left|\partial_y^k\left(\lambda^p-y+c_\epsilon y^{-p}\right)\right|
      \lesssim\epsilon y^{-1/p-k},\qquad k=0,1.\]
    Combining this with $|c_{\textup{cubic}}|\lesssim\epsilon$ and $|\partial_y^j(Z^3)|\lesssim y^{-3\beta+j/p}$, the product rule proves \eqref{eqn:approx_E4_bound}, since $3\beta+1/p=5\beta-1$.
\end{proof}

\subsection{Control of \texorpdfstring{$\mc{E}_5$}{E5}}

In this section, we will control the term
\[\mc{E}_5=\mc{L}_0(f_{\textup{osc}})+\mc{Q}_{\epsilon,3}(v_Y)-2\left(\lambda^p-y+n_\epsilon\lambda^p|Z|^2\right)\operatorname{Re}Z
        -2c_{\textup{cubic}}(\lambda^p-y)\operatorname{Re}(Z^3).\]

The idea is that $\mc{E}_5$ is the error in computing the linear and cubic response to the oscillatory term $f_{\textup{osc}}$. We will write
\begin{equation}\label{eqn:approx_E5_linear_response}
\mc{L}_0(f_{\textup{osc}}) = 2(\lambda^p - y)\operatorname{Re}(Z) + 2c_{\textup{cubic}}(3^p\lambda^p - y)\operatorname{Re}(Z^3) + \mc{R}_{\textup{lin}}
\end{equation}
and 
\begin{equation}\label{eqn:approx_E5_cubic_response}
\mc{Q}_{\epsilon,3}(v_Y) = 2n_\epsilon\lambda^p|Z|^2\operatorname{Re}(Z) - 2c_{\textup{cubic}}(3^p-1)\lambda^p\operatorname{Re}(Z^3) + \mc{R}_{\textup{cub}}.
\end{equation}
From here, we can see that 
\[\mc{E}_5 = \mc{R}_{\textup{lin}} + \mc{R}_{\textup{cub}}.\]
We will estimate each remainder separately. 

\begin{proposition}\label{prop:approx_E5_linear}
    Let $Y\geq8R$. For $Y\leq y\leq2Y$ and $j=0,1$, the remainder $\mc{R}_{\textup{lin}}$ in \eqref{eqn:approx_E5_linear_response} satisfies
    \begin{align}
        |\partial_y^j\mc{R}_{\textup{lin}}(y)|
        &\lesssim\epsilon\ang{y}^{-2p-1+\frac{j}{p}},
        \label{eqn:approx_E5_linear_remainder}
    \end{align}
\end{proposition}

\begin{proof}
    To estimate $\mc{R}_{\textup{lin}}$, we compute 
    \begin{align*}
        \mc{R}_{\textup{lin}} & = \mc{L}_0(f_{\textup{osc}}) - 2(\lambda^p - y)\operatorname{Re}(Z) - 2c_{\textup{cubic}}(3^p\lambda^p - y)\operatorname{Re}(Z^3) \\
        & = |D|^pf_{\textup{osc}}' + f_{\textup{osc}} - 2\lambda^p\operatorname{Re}(Z) - 2c_{\textup{cubic}}3^p\lambda^p\operatorname{Re}(Z^3) + 2y\operatorname{Re}\left(\frac{\mu}{\mu_*}Z_0\right).
    \end{align*}
    Here we used $f_{\textup{osc}}'=2\operatorname{Re}\left(Z+c_{\textup{cubic}}Z^3-\frac{\mu}{\mu_*}Z_0\right)$, since $y\geq2R$.
    We split $f_{\textup{osc}}=f_{\textup{osc},1}+f_{\textup{osc},3}$ into its fundamental and third-harmonic components, where
    \begin{equation}\label{eqn:approx_E5_oscillatory_split}
        \begin{aligned}
            f_{\textup{osc},1}(y)&=-2\chi_R(y)\operatorname{Re}\int_y^\infty\left(Z(t)-\frac{\mu}{\mu_*}Z_0(t)\right)\,dt,\\
            f_{\textup{osc},3}(y)&=-2c_{\textup{cubic}}\chi_R(y)\operatorname{Re}\int_y^\infty Z(t)^3\,dt.
        \end{aligned}
    \end{equation}
    Substituting \eqref{eqn:approx_E5_oscillatory_split} gives
    \begin{equation}\label{eqn:approx_E5_linear_split}
        \mc{R}_{\textup{lin}}=\mc{R}_{\textup{lin},1}+\mc{R}_{\textup{lin},3},
    \end{equation}
    where
    \begin{equation}\label{eqn:approx_E5_linear_components}
        \begin{aligned}
            \mc{R}_{\textup{lin},1}(y)
            &=|D|^pf_{\textup{osc},1}'(y)+f_{\textup{osc},1}(y)
            -2\lambda(y)^p\operatorname{Re}Z(y)
            +2y\operatorname{Re}\left(\frac{\mu}{\mu_*}Z_0(y)\right),\\
            \mc{R}_{\textup{lin},3}(y)
            &=|D|^pf_{\textup{osc},3}'(y)+f_{\textup{osc},3}(y)
            -2c_{\textup{cubic}}3^p\lambda(y)^p\operatorname{Re}\bigl(Z(y)^3\bigr).
        \end{aligned}
    \end{equation}
    We estimate these two remainders separately. We write
    \begin{equation}\label{eqn:approx_E5_fundamental_localization}
    |D|^pf_{\textup{osc},1}' = 2\text{Re}|D|^p \left[\chi_Y\left(Z - \frac{\mu}{\mu_*}Z_0\right)\right] + \mc{R}_{\textup{far},Y}
    \end{equation}
    where
    \begin{equation}\label{eqn:approx_E5_far_remainder}
    \mc{R}_{\textup{far},Y}(y) = c_\alpha \int f_{\textup{osc},1}(z)\partial_z \left[\frac{1 - \chi_Y(z)}{|y - z|^\alpha}\right] \, dz.
    \end{equation}
    It's not hard to see that $|\partial^j \mc{R}_{\textup{far},Y}(y)| \lesssim \epsilon \ang{y}^{-p-2-j}$ since $\|f_{\textup{osc},1}\|_{L^1} \lesssim \epsilon$ by \cref{lem:approx_phase_differences}. The remaining term can be estimated by appealing to \cref{lem:appendix_oscillatory_multiplier}. We write
    \begin{equation}\label{eqn:approx_E5_multiplier_expansions}
    \begin{aligned}
        |D|^p(\chi_Y Z)
        &=\lambda^p Z+i\frac{(\lambda^p)'}{\lambda}Z
        -ip\mu a_*'\lambda^{(p-3)/2}e^{i(\Psi_{\epsilon,\mu}-\pi/4)}+\mc{E}_{\textup{mult}},\\
        |D|^p\left(\chi_Y\frac{\mu}{\mu_*}Z_0\right)
        &=y\frac{\mu}{\mu_*}Z_0+\frac{i}{\lambda_0}\frac{\mu}{\mu_*}Z_0
        -ip\mu a_*'\lambda_0^{(p-3)/2}e^{i(\Psi_0-\pi/4)}+\mc{E}_{\textup{mult},0}.
    \end{aligned}
    \end{equation}
    Here $\mc{E}_{\textup{mult}}$ and $\mc{E}_{\textup{mult},0}$ are the corresponding remainders in \cref{lem:appendix_oscillatory_multiplier}.
    For $Y\leq y\leq2Y$, integration by parts in $f_{\textup{osc},1}$ gives
    \begin{equation}\label{eqn:approx_E5_fundamental_primitive}
        f_{\textup{osc},1}
        =-2\operatorname{Re}\left[i\left(\frac{Z}{\lambda}
        -\frac{\mu}{\mu_*}\frac{Z_0}{\lambda_0}\right)\right]+\mc{R}_{\textup{prim}},
    \end{equation}
    where
    \begin{equation}\label{eqn:approx_E5_primitive_remainder}
    \begin{aligned}
        \mc{R}_{\textup{prim}}(y)=2\operatorname{Re}\int_y^\infty\Bigg[
        &e^{i(\Psi_{\epsilon,\mu}(s)-\pi/4)}
        \partial_s\left(\frac{\mu a_*(s)}{i\lambda(s)^{(p+3)/2}}\right)\\
        &-e^{i(\Psi_0(s)-\pi/4)}
        \partial_s\left(\frac{\mu a_*(s)}{i\lambda_0(s)^{(p+3)/2}}\right)
        \Bigg]\,ds.
    \end{aligned}
    \end{equation}
    Substituting \eqref{eqn:approx_E5_multiplier_expansions} and \eqref{eqn:approx_E5_fundamental_primitive} into \eqref{eqn:approx_E5_linear_components}, the leading terms cancel and we obtain
    \begin{equation}\label{eqn:approx_E5_fundamental_error_identity}
    \begin{aligned}
        \mc{R}_{\textup{lin},1}
        ={}&2\operatorname{Re}\Bigg[
        i\bigl((\lambda^p)'-1\bigr)\frac{Z}{\lambda} -ip\mu a_*'\left(
        \lambda^{(p-3)/2}e^{i(\Psi_{\epsilon,\mu}-\pi/4)}
        -\lambda_0^{(p-3)/2}e^{i(\Psi_0-\pi/4)}
        \right)\Bigg]\\
        &+2\operatorname{Re}(\mc{E}_{\textup{mult}}-\mc{E}_{\textup{mult},0})+\mc{R}_{\textup{prim}}+\mc{R}_{\textup{far},Y}.
    \end{aligned}
    \end{equation}
    We now bound the terms in \eqref{eqn:approx_E5_fundamental_error_identity}. For $\mc{E}_{\textup{mult}}-\mc{E}_{\textup{mult},0}$, apply \cref{lem:appendix_oscillatory_multiplier} with phases $\Psi_{\epsilon,\mu}-\pi/4$, $\Psi_0-\pi/4$ and
    \[
    \begin{aligned}
        a_{1,\mu}&=\mu a_*\lambda^{-(p+1)/2},\qquad
        a_{0,\mu}=\mu a_*\lambda_0^{-(p+1)/2},\\
        A&=Y^{-\beta},\qquad \Lambda=Y^{1/p},\qquad
        \delta=\epsilon Y^{1/p-p}.
    \end{aligned}
    \]
    By \cref{lem:approx_linear_seed,lem:approx_phase_differences}, the individual symbol bounds hold, and on $(Y/4,4Y)$,
    \[
    \begin{aligned}
        |\partial_y^k(a_{1,\mu}-a_{0,\mu})|&\lesssim_k\epsilon Y^{-\beta-p-1-k}
        \lesssim_k\delta A Y^{-k},\\
        |\partial_y^k(\Psi_{\epsilon,\mu}-\Psi_0)|&\lesssim_k\delta Y^{-k}.
    \end{aligned}
    \]
    Hence, for $j=0,1$,
    \begin{equation}\label{eqn:approx_E5_multiplier_difference_bound}
        |\partial_y^j(\mc{E}_{\textup{mult}}-\mc{E}_{\textup{mult},0})|
        \lesssim\delta A Y^{-2}\Lambda^{p-2+j}
        =\epsilon Y^{-p-3\beta+j/p}.
    \end{equation}
    The same estimates give
    \[
        |\partial_y^k[(\lambda^p)'-1]|\lesssim_k\epsilon Y^{-p-1-k},
        \qquad |\partial_y^k a_*'|\lesssim_kY^{-2\beta-1-k}.
    \]
    Since $|Z/\lambda|\lesssim Y^{-\beta-1/p}$, the first term in square brackets in \eqref{eqn:approx_E5_fundamental_error_identity} is bounded by $\epsilon Y^{-p-1-\beta-1/p}=\epsilon Y^{-p-3\beta}$. For the second term, use the frequency-difference bounds and
    \[|e^{i(\Psi_{\epsilon,\mu}-\Psi_0)}-1|\lesssim\epsilon Y^{1/p-p}\]
    to obtain the bound $\epsilon Y^{-2\beta-1+(p-3)/(2p)+1/p-p}=\epsilon Y^{-p-3\beta}$. The product rule gives an additional factor at most $Y^{1/p}$ for one derivative in both estimates.

    To estimate $\mc{R}_{\textup{prim}}$ in \eqref{eqn:approx_E5_primitive_remainder}, set
    \[
        b(s)=\frac{1}{i\lambda(s)}
        \partial_s\left(\frac{\mu a_*(s)}{i\lambda(s)^{(p+3)/2}}\right),
    \]
    and let $b_0$ denote the same expression with $\lambda_0$ in place of $\lambda$. For $s\geq y$ and $k=0,1$, the amplitude and frequency bounds give
    \[
    \begin{aligned}
        |\partial_s^kb(s)|+|\partial_s^kb_0(s)|
        &\lesssim s^{-\beta-1-2/p-k},\\
        |\partial_s^k(b-b_0)(s)|
        &\lesssim\epsilon s^{-p-\beta-2-2/p-k}.
    \end{aligned}
    \]
    Integrating once more by parts in \eqref{eqn:approx_E5_primitive_remainder} produces the boundary difference of $be^{i(\Psi_{\epsilon,\mu}-\pi/4)}$ and $b_0e^{i(\Psi_0-\pi/4)}$, and the integral of the corresponding difference with $b',b_0'$. Using these bounds and $|e^{i(\Psi_{\epsilon,\mu}(s)-\Psi_0(s))}-1|\lesssim\epsilon s^{1/p-p}$, we obtain
    \[
        |\mc{R}_{\textup{prim}}(y)|
        \lesssim\epsilon y^{-p-3\beta}
        +\epsilon\int_y^\infty s^{-p-3\beta-1}\,ds
        \lesssim\epsilon y^{-p-3\beta}.
    \]
    Differentiating \eqref{eqn:approx_E5_primitive_remainder} instead gives
    \[
        |\mc{R}_{\textup{prim}}'(y)|\lesssim\epsilon y^{-p-\beta-1}
        =\epsilon y^{-p-3\beta+1/p}.
    \]
    Combining these estimates with \eqref{eqn:approx_E5_multiplier_difference_bound} and the bound for $\mc{R}_{\textup{far},Y}$, we conclude that
    \begin{equation}\label{eqn:approx_E5_fundamental_bound}
        |\partial_y^j\mc{R}_{\textup{lin},1}(y)|
        \lesssim\epsilon\bigl(Y^{-p-3\beta+j/p}+Y^{-p-2-j}\bigr)
        \lesssim\epsilon Y^{-2p-1+j/p},\qquad j=0,1,
    \end{equation}
    since $3\beta>p+1$ and $p<1$. \par
    The proof of the bound for $\mc{R}_{\textup{lin},3}$ is similar. We write
    \[
        Z^3=a_3e^{3i(\Psi_{\epsilon,\mu}-\pi/4)},\qquad
        a_3=\mu^3a_*^3\lambda^{-3(p+1)/2},
        \qquad |\partial_y^ka_3|\lesssim_kY^{-3\beta-k}.
    \]
    Applying \cref{lem:appendix_oscillatory_multiplier} with phase $3(\Psi_{\epsilon,\mu}-\pi/4)$, $A=Y^{-3\beta}$ and $\Lambda=Y^{1/p}$ gives, for $j=0,1$,
    \[
    \begin{aligned}
        \left|\partial_y^j\left[|D|^p(\chi_YZ^3)-3^p\lambda^pZ^3\right]\right|
        &\lesssim Y^{-3\beta-1}\Lambda^{p-1+j}
        +Y^{-3\beta-2}\Lambda^{p-2+j}\\
        &\lesssim Y^{1-5\beta+j/p}.
    \end{aligned}
    \]
    Here the first term bounds the two first-order corrections in the lemma, and the second bounds its remainder. Integrating once by parts in $\int_y^\infty Z^3$ gives
    \[
        \left|\int_y^\infty Z(s)^3\,ds\right|
        \lesssim y^{-3\beta-1/p}
        +\int_y^\infty s^{-3\beta-1/p-1}\,ds
        \lesssim y^{-3\beta-1/p},\qquad y\geq R.
    \]
    Since $c_{\textup{cubic}}=O(\epsilon)$ and $f_{\textup{osc},3}'=2c_{\textup{cubic}}\operatorname{Re}(Z^3)$ on $[Y,2Y]$, it follows that
    \[
        |\partial_y^jf_{\textup{osc},3}(y)|\lesssim\epsilon Y^{1-5\beta+j/p},
        \qquad j=0,1,\qquad \|f_{\textup{osc},3}\|_{L^1}\lesssim\epsilon.
    \]
    The localization argument in \eqref{eqn:approx_E5_fundamental_localization} gives
    \[
        \mc{R}_{\textup{lin},3}
        =2c_{\textup{cubic}}\operatorname{Re}\left[|D|^p(\chi_YZ^3)-3^p\lambda^pZ^3\right]
        +f_{\textup{osc},3}+\mc{R}_{\textup{far},Y,3},
    \]
    where $\mc{R}_{\textup{far},Y,3}$ is given by \eqref{eqn:approx_E5_far_remainder} with $f_{\textup{osc},3}$ in place of $f_{\textup{osc},1}$. Thus
    \[
        |\partial_y^j\mc{R}_{\textup{far},Y,3}(y)|
        \lesssim Y^{-p-2-j}\|f_{\textup{osc},3}\|_{L^1}
        \lesssim\epsilon Y^{-p-2-j},\qquad j=0,1.
    \]
    Combining these estimates, we obtain
    \begin{equation}\label{eqn:approx_E5_third_harmonic_bound}
        |\partial_y^j\mc{R}_{\textup{lin},3}(y)|
        \lesssim\epsilon\bigl(Y^{1-5\beta+j/p}+Y^{-p-2-j}\bigr)
        \lesssim\epsilon Y^{-2p-1+j/p},\qquad j=0,1,
    \end{equation}
    since $5\beta>2p+2$ and $p<1$. Together with \eqref{eqn:approx_E5_fundamental_bound}, this proves \eqref{eqn:approx_E5_linear_remainder}.

\end{proof}

\begin{proposition}\label{prop:approx_E5_cubic}
    Let $Y\geq8R$. For $Y\leq y\leq2Y$ and $j=0,1$, the remainder $\mc{R}_{\textup{cub}}$ in \eqref{eqn:approx_E5_cubic_response} satisfies
    \begin{align}
        |\partial_y^j\mc{R}_{\textup{cub}}(y)|
        &\lesssim\epsilon\ang{y}^{1-5\beta+\frac{j}{p}}.
        \label{eqn:approx_E5_cubic_remainder}
    \end{align}
    Therefore, $\mc{E}_5=\mc{R}_{\textup{lin}}+\mc{R}_{\textup{cub}}$ satisfies
    \begin{equation}\label{eqn:approx_E5_bound}
        |\partial_y^j\mc{E}_5(y)|
        \lesssim\epsilon\ang{y}^{-2p-1+\frac{j}{p}},
        \qquad Y\leq y\leq2Y,\quad j=0,1.
    \end{equation}
\end{proposition}

\begin{proof}
    Set $\lambda=\lambda(y)>0$ and consider the plane wave $P_y$ defined by
    \[P_y(z) = 2\text{Re}\left[\frac{Z(y)}{i\lambda}e^{i\lambda(z - y)}\right].\]
    If we write $Z(y) = |Z(y)|e^{i\theta}$, then 
    \[P_y(z) = \frac{2|Z(y)|}{\lambda}\sin\left(\lambda(z-y) + \theta\right).\]
    Using the cubic homogeneity and translation invariance of $\mc{Q}_{\epsilon,3}$, and changing variables by $s\mapsto\lambda s$ in its integral, the definition \eqref{eqn:approx_cubic_normalization} of $n_\epsilon$ and $c_{\textup{cubic}}$ gives
    \begin{align}
        \mc{Q}_{\epsilon,3}(P_y)(z)
        &=\left(\frac{|Z(y)|}{\lambda}\right)^3\lambda^{p+3}
        \mc{Q}_{\epsilon,3}(2\sin(\cdot))\bigl(\lambda(z-y)+\theta\bigr)\notag\\
        &=2n_\epsilon|Z(y)|^3\lambda^p\cos\bigl(\lambda(z-y)+\theta\bigr)\notag\\
        &-2(3^p-1)c_{\textup{cubic}}|Z(y)|^3\lambda^p\cos\bigl(3\lambda(z-y)+3\theta\bigr).
        \label{eqn:approx_cubic_plane_wave_response}
    \end{align}
    Evaluating \eqref{eqn:approx_cubic_plane_wave_response} at $z=y$ and using $|Z|^3\cos\theta=|Z|^2\operatorname{Re}Z$ and $|Z|^3\cos(3\theta)=\operatorname{Re}(Z^3)$ gives
    \begin{equation}\label{eqn:approx_cubic_plane_wave_frozen}
        \mc{Q}_{\epsilon,3}(P_y)(y)
        =2n_\epsilon\lambda^p|Z|^2\operatorname{Re}Z
        -2c_{\textup{cubic}}(3^p-1)\lambda^p\operatorname{Re}(Z^3).
    \end{equation}
    It follows that the remainder $\mc{R}_{\textup{cub}}$ can be expressed as
    \begin{equation}\label{eqn:approx_cubic_remainder_plane_wave}
        \begin{aligned}
            \mc{R}_{\textup{cub}}(y)
            &=\mc{Q}_{\epsilon,3}(v_Y)(y)-\mc{Q}_{\epsilon,3}(P_y)(y)\\
            &=-\frac{(p+1)\epsilon c_\alpha}{2}\int_{\R}\frac{1}{|y-z|^{p+3}}\Big[
            \bigl(v_Y'(y)-v_Y'(z)\bigr)\bigl(v_Y(y)-v_Y(z)\bigr)^2\\
            &\hspace{130pt}-\bigl(P_y'(y)-P_y'(z)\bigr)\bigl(P_y(y)-P_y(z)\bigr)^2\Big]\,dz.
        \end{aligned}
    \end{equation}
    Making the change of variables $s=y-z$ in \eqref{eqn:approx_cubic_remainder_plane_wave}, we obtain
    \begin{equation}\label{eqn:approx_cubic_remainder_translated}
        \begin{aligned}
            \mc{R}_{\textup{cub}}(y)
            &=-\frac{(p+1)\epsilon c_\alpha}{2}\int_{\R}\frac{1}{|s|^{p+3}}\Big[
            \bigl(v_Y'(y)-v_Y'(y-s)\bigr)\bigl(v_Y(y)-v_Y(y-s)\bigr)^2\\
            &\hspace{130pt}-\bigl(P_y'(y)-P_y'(y-s)\bigr)\bigl(P_y(y)-P_y(y-s)\bigr)^2\Big]\,ds.
        \end{aligned}
    \end{equation}
    To control this term, we do a smooth near-far decomposition in $s$. The near region is $|s| \leq Y/4$ and the far region is $|s| \geq Y/4$. In the near region, we use the algebraic identity
    \[a_1b_1^2 - a_2b_2^2 = (a_1-a_2)b_1^2 + a_2(b_1^2-b_2^2)\]
    to write the integrand as
    \begin{align*}
        &\frac{v_Y'(y)-v_Y'(y-s)-P_y'(y)+P_y'(y-s)}{|s|^{p+3}}
        \bigl(v_Y(y)-v_Y(y-s)\bigr)^2\\
        &\quad+\frac{P_y'(y)-P_y'(y-s)}{|s|^{p+3}}
        \bigl(v_Y(y)-v_Y(y-s)-P_y(y)+P_y(y-s)\bigr)\\
        &\qquad\qquad\times\bigl(v_Y(y)-v_Y(y-s)+P_y(y)-P_y(y-s)\bigr).
    \end{align*}
   Observe that we have the estimates
    \begin{align*}
        |v_Y(y)-v_Y(y-s)|+|P_y(y)-P_y(y-s)|
        &\lesssim Y^{-\beta}\min\{|s|,\lambda^{-1}\},\\
        |v_Y'(y)-v_Y'(y-s)|+|P_y'(y)-P_y'(y-s)|
        &\lesssim Y^{-\beta}\min\{1,\lambda|s|\}.
    \end{align*}
    Indeed, for the plane wave, we have
    \begin{align*}
        |P_y(y) - P_y(y - s)| & \leq \frac{2|Z(y)|}{\lambda}|1 - e^{-i\lambda s}| \\
        & \lesssim Y^{-\beta}\lambda^{-1}\min\{1,\lambda|s|\},\\
        & \lesssim Y^{-\beta} \min\{|s|, \lambda^{-1}\}
    \end{align*}
    For $v_Y$, the cutoff is identically one between $y$ and $y-s$, so $v_Y'=2\operatorname{Re}(Z+c_{\textup{cubic}}Z^3)$ on this interval. The mean value theorem therefore gives
    \[|v_Y(y)-v_Y(y-s)|\lesssim Y^{-\beta}|s|.\]
    To obtain the other bound we integrate by parts using $\Psi'=\lambda$:
    \begin{align*}
        \int_z^\infty Z(t)\,dt
        &=-\frac{Z(z)}{i\lambda(z)}
        -\int_z^\infty e^{i(\Psi(t)-\pi/4)}
        \partial_t\left(\frac{\mu a_*(t)}{i\lambda(t)^{(p+3)/2}}\right)\,dt.
    \end{align*}
    Using the estimates on $a_*$ from \cref{lem:approx_linear_seed}, we have
    \begin{align*}
        \left|\int_z^\infty Z(t)\,dt\right|
        &\lesssim z^{-\beta-1/p}+\int_z^\infty t^{-\beta-1/p-1}\,dt
        \lesssim z^{-\beta-1/p}.
    \end{align*}
    The same integration by parts with phase $3\Psi$ gives
    \[\left|\int_z^\infty Z(t)^3\,dt\right|\lesssim z^{-3\beta-1/p}.\]
    Since $|c_{\textup{cubic}}|\lesssim\epsilon$ and $z \approx Y$ on the support of $\chi_Y$, the definition \eqref{eqn:approx_local_oscillatory_front} yields
    \begin{equation}\label{eqn:approx_vY_amplitude_bound}
        \|v_Y\|_{L^\infty}
        \lesssim Y^{-\beta-1/p}+\epsilon Y^{-3\beta-1/p}
        \lesssim Y^{-\beta}\lambda^{-1}.
    \end{equation}
    Combining \eqref{eqn:approx_vY_amplitude_bound} with the mean value estimate proves
    \[|v_Y(y)-v_Y(y-s)|\lesssim Y^{-\beta}\min\{|s|,\lambda^{-1}\}.\]
    A similar argument gives the derivative bound, so we omit the details.
    We also have the following estimates for $|s|\leq Y/4$:
    \begin{equation}\label{eqn:approx_vY_plane_wave_comparison}
        \begin{aligned}
            &|v_Y(y)-v_Y(y-s)-P_y(y)+P_y(y-s)|\lesssim Y^{-\beta-1}s^2
            +\epsilon Y^{-3\beta}\min\{|s|,\lambda^{-1}\},\\
            &|v_Y'(y)-v_Y'(y-s)-P_y'(y)+P_y'(y-s)| \lesssim Y^{-\beta-1}\bigl(|s|+\lambda s^2\bigr)
            +\epsilon Y^{-3\beta}\min\{1,\lambda|s|\}.
        \end{aligned}
    \end{equation}
    To estimate the second term, we write
    \begin{equation}\label{eqn:approx_vY_slope_comparison_identity}
        \begin{aligned}
            &v_Y'(y)-v_Y'(y-s)-P_y'(y)+P_y'(y-s)\\
            &\qquad=2\operatorname{Re}\left[
            Z(y)e^{-i\lambda s}-Z(y-s)
            +c_{\textup{cubic}}\bigl(Z(y)^3-Z(y-s)^3\bigr)
            \right].
        \end{aligned}
    \end{equation}
    We have that
    \begin{equation}\label{eqn:approx_Z_frozen_frequency_identity}
        Z(y)e^{-i\lambda s} - Z(y - s) = \int_{y-s}^{y}e^{-i\lambda(t-y+s)}\bigl[Z'(t)-i\lambda Z(t)\bigr]\,dt.
    \end{equation}
    Here $\lambda=\lambda(y)$. For $t$ between $y-s$ and $y$, the bounds on $a_*$ in \cref{lem:approx_linear_seed} and $|\lambda'(t)|\lesssim\lambda/Y$ give
    \begin{align*}
        |Z'(t)-i\lambda(t)Z(t)|
        &=\left|\mu\partial_t\left(a_*(t)\lambda(t)^{-(p+1)/2}\right)\right|
        \lesssim Y^{-\beta-1},\\
        |\lambda(t)-\lambda|&\lesssim\frac{\lambda}{Y}|t-y|.
    \end{align*}
    Therefore,
    \begin{align*}
        |Z'(t)-i\lambda Z(t)|
        &\leq |Z'(t)-i\lambda(t)Z(t)|+|\lambda(t)-\lambda|\,|Z(t)|\\
        &\lesssim Y^{-\beta-1}\bigl(1+\lambda|t-y|\bigr).
    \end{align*}
    Therefore,
    \begin{equation}\label{eqn:approx_Z_plane_wave_comparison}
        \begin{aligned}
            |Z(y)e^{-i\lambda s}-Z(y-s)|
            &\lesssim Y^{-\beta-1}\int_0^{|s|}(1+\lambda t)\,dt\\
            &\lesssim Y^{-\beta-1}\bigl(|s|+\lambda s^2\bigr).
        \end{aligned}
    \end{equation}
    On the other hand, we can estimate the cubic term in \eqref{eqn:approx_vY_slope_comparison_identity} using the mean value theorem:
    \[|Z(y)^3-Z(y-s)^3| \lesssim Y^{-3\beta}\min\{1,\lambda|s|\}.\]
    The derivative estimate in \eqref{eqn:approx_vY_plane_wave_comparison} follows. For the other estimate, we write
    \begin{equation}\label{eqn:approx_vY_value_comparison_identity}
        \begin{aligned}
            &v_Y(y)-v_Y(y-s)-P_y(y)+P_y(y-s)\\
            &\qquad=2\operatorname{Re}\int_{y-s}^{y}
            \left[Z(t)-Z(y)e^{i\lambda(t-y)}+c_{\textup{cubic}}Z(t)^3\right]\,dt.
        \end{aligned}
    \end{equation}
    Integrating by parts and then using \eqref{eqn:approx_Z_frozen_frequency_identity}, we obtain
    \begin{align*}
        &\int_{y-s}^{y}\left[Z(t)-Z(y)e^{i\lambda(t-y)}\right]\,dt\\
        &\quad=\frac{Z(y)e^{-i\lambda s}-Z(y-s)}{i\lambda}
        -\frac{1}{i\lambda}\int_{y-s}^{y}\bigl[Z'(t)-i\lambda Z(t)\bigr]\,dt\\
        &\quad=\frac{1}{i\lambda}\int_{y-s}^{y}
        \left[e^{-i\lambda(t-y+s)}-1\right]\bigl[Z'(t)-i\lambda Z(t)\bigr]\,dt.
    \end{align*}
    For $t$ between $y-s$ and $y$, we have
    \[\frac{|e^{-i\lambda(t-y+s)}-1|}{\lambda}
    \lesssim\min\{|s|,\lambda^{-1}\}.\]
    Using the bound $|Z'(t)-i\lambda Z(t)|\lesssim Y^{-\beta-1}(1+\lambda|t-y|)$ proved above, it follows that
    \begin{align*}
        \left|\int_{y-s}^{y}\left[Z(t)-Z(y)e^{i\lambda(t-y)}\right]\,dt\right|
        &\lesssim Y^{-\beta-1}\min\{|s|,\lambda^{-1}\}
        \int_0^{|s|}(1+\lambda t)\,dt\\
        &\lesssim Y^{-\beta-1}\min\{|s|,\lambda^{-1}\}
        \bigl(|s|+\lambda s^2\bigr)\\
        &\lesssim Y^{-\beta-1}s^2.
    \end{align*}
    The last inequality follows by considering $\lambda|s|\leq1$ and $\lambda|s|\geq1$. For the cubic term, the pointwise bound on $Z$ and the primitive estimate proved above give, respectively,
    \begin{align*}
        \left|\int_{y-s}^{y}Z(t)^3\,dt\right|
        &\lesssim Y^{-3\beta}|s|,\\
        \left|\int_{y-s}^{y}Z(t)^3\,dt\right|
        &\leq\left|\int_{y-s}^{\infty}Z(t)^3\,dt\right|
        +\left|\int_y^{\infty}Z(t)^3\,dt\right| \lesssim Y^{-3\beta}\lambda^{-1}.
    \end{align*}
    Since $|c_{\textup{cubic}}|\lesssim\epsilon$, we conclude that
    \[\left|c_{\textup{cubic}}\int_{y-s}^{y}Z(t)^3\,dt\right|
    \lesssim\epsilon Y^{-3\beta}\min\{|s|,\lambda^{-1}\}.\]
    Substituting these bounds into \eqref{eqn:approx_vY_value_comparison_identity} proves the first estimate in \eqref{eqn:approx_vY_plane_wave_comparison}. We are now able to estimate the near contribution of the integral in \eqref{eqn:approx_cubic_remainder_translated}. Substituting the previous estimates, we obtain
    \begin{equation}\label{eqn:approx_cubic_near_bound}
        \begin{aligned}
            &\epsilon\int_{|s|\leq Y/4}\frac{1}{|s|^{p+3}}\Big|
            \bigl(v_Y'(y)-v_Y'(y-s)\bigr)\bigl(v_Y(y)-v_Y(y-s)\bigr)^2 -\bigl(P_y'(y)-P_y'(y-s)\bigr)\bigl(P_y(y)-P_y(y-s)\bigr)^2\Big|\,ds\\
            &\quad\lesssim\epsilon Y^{-3\beta-1}\int_{|s|\leq Y/4}
            \frac{(|s|+\lambda s^2)\min\{|s|,\lambda^{-1}\}^2}{|s|^{p+3}}\,ds +\epsilon^2Y^{-5\beta}\int_{|s|\leq Y/4}
            \frac{\min\{1,\lambda|s|\}\min\{|s|,\lambda^{-1}\}^2}{|s|^{p+3}}\,ds\\
            &\quad\lesssim\epsilon Y^{-3\beta-1}\left[
            \int_0^{\lambda^{-1}}s^{-p}\,ds
            +\lambda^{-1}\int_{\lambda^{-1}}^{Y/4}s^{-p-1}\,ds\right] +\epsilon^2Y^{-5\beta}\left[
            \lambda\int_0^{\lambda^{-1}}s^{-p}\,ds
            +\lambda^{-2}\int_{\lambda^{-1}}^{Y/4}s^{-p-3}\,ds\right]\\
            &\quad\lesssim\epsilon Y^{-3\beta-1}\lambda^{p-1}
            +\epsilon^2Y^{-5\beta}\lambda^p\\
            &\quad\lesssim\epsilon Y^{1-5\beta}.
        \end{aligned}
    \end{equation}
    For the far region where $|s| \geq Y/4$, we simply use the triangle inequality to bound the far contribution as 
    \begin{align*}
        \int_{|s| \geq Y/4}\frac{|v_Y'(y)-v_Y'(y-s)||v_Y(y)-v_Y(y-s)|^2}{|s|^{p+3}}\,ds & \lesssim Y^{-3\beta}\lambda^{-2} \int_{|s| \geq Y/4}|s|^{-p-3} \\
        & \lesssim Y^{-3\beta - p - 2 - 2/p}.
    \end{align*}
    Using that $2\beta = 1 + 1/p$, we can bound this by $Y^{1 - 5\beta}$. The same argument applies to the $P_y$ term as well, so we can conclude that 
    \[|\mc{R}_{\textup{cub}}(y)| \lesssim \epsilon \ang{y}^{1 - 5\beta}.\]
    For the derivative, we differentiate \eqref{eqn:approx_cubic_remainder_translated}. The slow amplitude bounds give
    \[
        |Z'-i\lambda Z|+\lambda^{-1}|\partial_y(Z'-i\lambda Z)|
        \lesssim Y^{-\beta-1},\qquad |\lambda'|\lesssim\lambda/Y.
    \]
    By differentiating the identities in \eqref{eqn:approx_vY_plane_wave_comparison}, we obtain for $|s|\leq Y/4$,
    \[\left|\partial_y\bigl[v_Y(y)-v_Y(y-s)-P_y(y)+P_y(y-s)\bigr]\right| \lesssim\lambda Y^{-\beta-1}s^2
        +\epsilon\lambda Y^{-3\beta}\min\{|s|,\lambda^{-1}\}\]
    \[\left|\partial_y\bigl[v_Y'(y)-v_Y'(y-s)-P_y'(y)+P_y'(y-s)\bigr]\right| \lesssim\lambda Y^{-\beta-1}(|s|+\lambda s^2)
        +\epsilon\lambda Y^{-3\beta}\min\{1,\lambda|s|\}.\]
    Thus the product rule and the same integrals as in \eqref{eqn:approx_cubic_near_bound} bound the near contribution to the derivative by
    \[
        \epsilon Y^{-3\beta-1}\lambda^p
        +\epsilon^2Y^{-5\beta}\lambda^{p+1}
        \lesssim\epsilon Y^{1-5\beta+1/p}.
    \]
    In the far region, the explicit formula for $P_y$ shows that differentiating its increments gain a factor of $\lambda(1+|s|/Y)$. Differentiating $v_Y$ gains at most $\lambda$. The far contribution to the derivative is therefore bounded by
    \[
    \begin{aligned}
        &\epsilon Y^{-3\beta}\lambda^{-1}
        \int_{|s|\geq Y/4}\frac{1+|s|/Y}{|s|^{p+3}}\,ds \lesssim\epsilon Y^{-3\beta}\lambda^{-1}Y^{-p-2}
        \lesssim\epsilon Y^{1-5\beta+1/p}.
    \end{aligned}
    \]
    Combining the two regions proves the derivative estimate in \eqref{eqn:approx_E5_cubic_remainder}.
\end{proof}

\subsection{Control of \texorpdfstring{$\mc{E}_6$}{E6}}

In this section, we will control the term
\[\mc{E}_6=D\mc{Q}_\epsilon(W)[f-W]
  -2c_\epsilon y^{-p}\operatorname{Re}\left(Z+c_{\textup{cubic}}Z^3\right).\]
\begin{proposition}\label{prop:approx_E6}
    The residual $\mc{E}_6$ satisfies the estimate
    \begin{equation}\label{eqn:approx_E6_bound}
        |\partial^j\mc{E}_6(y)|
        \lesssim\epsilon\ang{y}^{-2p-1+\frac{j}{p}},
        \qquad y\geq8R,\quad j=0,1.
    \end{equation}
\end{proposition}

\begin{proof}
    We first compute $D\mc{Q}_\epsilon(W)$. For $y>0$ and $t\geq0$, set
    \[K(y,t)=\frac{A_\epsilon(-t/(y+t))-1}{(y+t)^{1+p}}.\]
    Since $A_\epsilon(0)=1$ and $A_\epsilon'(0)=0$, the region $z>0$ contributes nothing to $D\mc{Q}_\epsilon(W)[h](y)$. Therefore we can write
    \begin{align*}
        D\mc{Q}_\epsilon(W)[h](y) & = c_\alpha\int_{-\infty}^{0}\frac{h'(y) - h'(z)}{|y - z|^\alpha}[A_\epsilon(S_W(y,z)) - 1] \, dz \\
        & + c_\alpha\int_{-\infty}^{0}\frac{W'(y) - W'(z)}{|y - z|^\alpha}A_\epsilon'(S_W(y,z))S_h(y,z) \, dz
    \end{align*}
    Setting $z=-t$ and using $S_{W}(y,z) = z / (y - z)$ gives
    \begin{align*}
        D\mc{Q}_\epsilon(W)[h](y)
        ={}&c_\alpha\int_0^\infty K(y,t)\bigl(h'(y)-h'(-t)\bigr)\,dt\\
        &+c_\alpha\int_0^\infty
        \frac{A_\epsilon'(-t/(y+t))}{(y+t)^{p+2}}
        \bigl(h(y)-h(-t)\bigr)\,dt.
    \end{align*}
    Observe that
    \[\partial_yK(y,t)-\partial_tK(y,t)
      =\frac{A_\epsilon'(-t/(y+t))}{(y+t)^{p+2}},\qquad
      \partial_t\bigl(h(y)-h(-t)\bigr)=h'(-t).\]
    We can integrate by parts to obtain
    \begin{equation}\label{eqn:approx_wedge_first_variation}
        \begin{aligned}
            D\mc{Q}_\epsilon(W)[h](y)
            &=c_\alpha\int_0^\infty
            \left[K(y,t)h'(y)+\partial_yK(y,t)\bigl(h(y)-h(-t)\bigr)\right]\,dt\\
            &=c_\alpha\partial_y\int_0^\infty
            K(y,t)\bigl(h(y)-h(-t)\bigr)\,dt.
        \end{aligned}
    \end{equation}
    The boundary term vanishes because $K(y, 0) = 0$ and $|K(y, t)| \lesssim_y t^{-1-p}$. It follows that 
    \[D\mc{Q}_\epsilon(W)[f - W] = c_\alpha\partial_y\int_0^\infty
            K(y,t)\bigl(f(y) - f(-t) + t\bigr)\,dt.\]
    Next, by the proof of \cref{lem:analysis_coefficient}, we have that 
    \[c_\alpha\int_{0}^{\infty}K(y, t) \, dt = c_\epsilon y^{-p}.\]
    Therefore we have the exact decomposition
    \[\mc{E}_6 = c_\epsilon y^{-p}[f'(y) - 2\, \text{Re}(Z + c_{\textup{cubic}}Z^3)] - pc_\epsilon y^{-p-1}f(y) - c_\alpha \int_{0}^{\infty}\partial_yK(y,t)[f(-t) - t] \, dt.\]
    By \eqref{eqn:approx_family_positive}, we can write 
    \[f' - 2\, \text{Re}(Z + c_{\textup{cubic}}Z^3) = f_{0, \textup{low}}' - pc_+ y^{-p-1}.\]
    Therefore applying \cref{lem:approx_linear_seed}, we have that
    \[|\partial^{j}(f' - 2\, \text{Re}(Z + c_{\textup{cubic}}Z^3))| \lesssim y^{-p-1-j}, \quad j = 0,1.\]
    Since $A_\epsilon-1$ and its first two derivatives are $O(\epsilon)$ on $[-1,0]$, differentiating the kernel gives
    \begin{equation}\label{eqn:approx_E6_kernel_bound}
        |\partial_y^kK(y,t)|\lesssim\epsilon(y+t)^{-1-p-k},
        \qquad k=0,1,2.
    \end{equation}
    Using $|f(-t)-t|\lesssim\ang{t}^{-p}$, we obtain for $j=0,1$,
    \begin{equation}\label{eqn:approx_E6_remote_bound}
        \begin{aligned}
            \left|\partial_y^j\int_0^\infty\partial_yK(y,t)[f(-t)-t]\,dt\right|
            &\lesssim\epsilon\int_0^\infty\frac{\ang{t}^{-p}}{(y+t)^{p+2+j}}\,dt\\
            &\lesssim\epsilon\left(y^{-p-2-j}\int_0^y\ang{t}^{-p}\,dt
                +\int_y^\infty t^{-2p-2-j}\,dt\right)\\
            &\lesssim\epsilon y^{-2p-1-j}.
        \end{aligned}
    \end{equation}
    Combining the preceding estimates with $|c_\epsilon|\lesssim\epsilon$, $|f(y)|\lesssim y^{-p}$, and $|f'(y)|\lesssim y^{-\beta}$ in the exact decomposition of $\mc{E}_6$ gives
    \[|\mc{E}_6(y)|\lesssim\epsilon y^{-2p-1},\qquad
      |\partial_y\mc{E}_6(y)|\lesssim\epsilon\left(y^{-2p-2}+y^{-p-1-\beta}\right).\]
    Since $\beta+1/p>p$, the derivative bound is at most $C\epsilon y^{-2p-1+1/p}$. This concludes the proof of \eqref{eqn:approx_E6_bound}.
\end{proof}

\subsection{Control of \texorpdfstring{$\mc{E}_7$}{E7}}

\begin{lemma}\label{lem:approx_E7}
    Let $Y\geq8R$. The residual $\mc{E}_7$ satisfies
    \begin{equation}\label{eqn:approx_E7_bound}
        |\partial_y^j\mc{E}_7(y)|
        \lesssim\epsilon\ang{y}^{-p-2\beta+\frac{j}{p}}
        \lesssim\epsilon\ang{y}^{-2p-1+\frac{j}{p}},
        \qquad Y\leq y\leq2Y,\quad j=0,1.
    \end{equation}
\end{lemma}

\begin{proof}
    Write $h_{Y} = \chi_Y f - v_Y$ and observe that
    \[h_Y = \chi_Y(f_{0,\textup{low}} + c_+ y^{-p})\]
    which is slowly varying. Recall that 
    \[\mc{Q}_{\epsilon,3}(f)(y)
    =-\frac{(p+1)\epsilon c_\alpha}{2}\int_{\R}
    \frac{f'(y)-f'(z)}{|y-z|^\alpha}
    \left(\frac{f(y)-f(z)}{y-z}\right)^2\,dz.\]
    Using the algebraic identity
    \[a_1b_1^2 - a_2b_2^2 = (a_1-a_2)b_1^2 + a_2(b_1-b_2)(b_1+b_2),\]
    we obtain
    \begin{equation}\label{eqn:approx_E7_cubic_difference}
        \begin{aligned}
            \mc{E}_7(y) & = \mc{Q}_{\epsilon,3}(v_Y+h_Y)(y)-\mc{Q}_{\epsilon,3}(v_Y)(y) \\
            & = -\frac{(p+1)\epsilon c_\alpha}{2}\int_{\R}\frac{h_Y'(y)-h_Y'(z)}{|y-z|^\alpha}\bigl(S_{v_Y}(y,z)+S_{h_Y}(y,z)\bigr)^2\,dz \\
            & \quad -\frac{(p+1)\epsilon c_\alpha}{2}\int_{\R}\frac{v_Y'(y)-v_Y'(z)}{|y-z|^\alpha}S_{h_Y}(y,z)\bigl(2S_{v_Y}(y,z)+S_{h_Y}(y,z)\bigr)\,dz.
        \end{aligned}
    \end{equation}
    By \cref{lem:approx_linear_seed}, we have the bounds 
    \[|\partial^{j} h_Y| \lesssim Y^{-p - j}, \quad 0 \leq j \leq 3\]
    and 
    \[|\partial^{j+1} v_Y| \lesssim Y^{-\beta + j/p}, \quad 0 \leq j \leq 2.\]
    Using the corresponding secant bounds, it follows that we can bound the first term in \eqref{eqn:approx_E7_cubic_difference} by 
    \[\begin{aligned}&\lesssim \epsilon Y^{-2\beta} \int_{\R} \frac{|h_Y'(y) - h_Y'(z)|}{|y-z|^{1+p}}\,dz \\ &\lesssim \epsilon Y^{-p-1-2\beta}\int_{\R}\frac{\min\{1,|y-z|/Y\}}{|y-z|^{1+p}}\,dz \\ &\lesssim \epsilon Y^{-2p-1-2\beta}.\end{aligned}\]
    Similarly, we can bound the second term by 
    \[\begin{aligned}&\lesssim \epsilon Y^{-p-1-\beta}\int_{\R}\frac{|v_Y'(y)-v_Y'(z)|}{|y-z|^{1+p}}\,dz \\ &\lesssim \epsilon Y^{-p-1-2\beta}\int_{\R}\frac{\min\{1,Y^{1/p}|y-z|\}}{|y-z|^{1+p}}\,dz \\ &\lesssim \epsilon Y^{-p-2\beta}.\end{aligned}\]
    Next, to estimate $\partial \mc{E}_7$, we make the change of variables $z=y-s$ and write
    \begin{equation}\label{eqn:approx_E7_translated}
        \begin{aligned}
            \mc{E}_7(y) & = -\frac{(p+1)\epsilon c_\alpha}{2}\int_{\R}\frac{h_Y'(y)-h_Y'(y-s)}{|s|^{1+p}}S_{v_Y+h_Y}(y,y-s)^2\,ds \\
            & \quad -\frac{(p+1)\epsilon c_\alpha}{2}\int_{\R}\frac{v_Y'(y)-v_Y'(y-s)}{|s|^{1+p}}S_{h_Y}(y,y-s)S_{2v_Y+h_Y}(y,y-s)\,ds.
        \end{aligned}
    \end{equation}
    Observe that
    \begin{equation}\label{eqn:approx_E7_secant_derivative}
        \partial_y S_w(y,y-s)=\frac{w'(y)-w'(y-s)}{s}=S_{w'}(y,y-s).
    \end{equation}
    The derivative bounds above give $|S_{v_Y+h_Y}|\lesssim Y^{-\beta}$ and $|S_{v_Y'+h_Y'}|\lesssim Y^{-\beta+1/p}$. Differentiating the first term in \eqref{eqn:approx_E7_translated} and using \eqref{eqn:approx_E7_secant_derivative} we obtain the estimate
    \begin{align*}
        &\lesssim \epsilon Y^{-2\beta}\int_{\R}\frac{|h_Y''(y)-h_Y''(y-s)|}{|s|^{1+p}}\,ds
        +\epsilon Y^{-2\beta+1/p}\int_{\R}\frac{|h_Y'(y)-h_Y'(y-s)|}{|s|^{1+p}}\,ds\\
        &\lesssim \epsilon\left(Y^{-p-2-2\beta}+Y^{-p-1-2\beta+1/p}\right)\int_{\R}\frac{\min\{1,|s|/Y\}}{|s|^{1+p}}\,ds\\
        &\lesssim \epsilon\left(Y^{-2p-2-2\beta}+Y^{-2p-1-2\beta+1/p}\right)\\
        &\lesssim \epsilon Y^{-2p-1-2\beta+1/p}.
    \end{align*}
    Similarly, we can estimate the derivative of the second term as
    \begin{align*}
        &\lesssim \epsilon Y^{-p-1-\beta}\int_{\R}\frac{|v_Y''(y)-v_Y''(y-s)|}{|s|^{1+p}}\,ds\\
        &\quad +\epsilon\left(Y^{-p-2-\beta}+Y^{-p-1-\beta+1/p}\right)\int_{\R}\frac{|v_Y'(y)-v_Y'(y-s)|}{|s|^{1+p}}\,ds\\
        &\lesssim \epsilon\left(Y^{-p-1-2\beta+1/p}+Y^{-p-2-2\beta}\right)\int_{\R}\frac{\min\{1,Y^{1/p}|s|\}}{|s|^{1+p}}\,ds\\
        &\lesssim \epsilon\left(Y^{-p-2\beta+1/p}+Y^{-p-1-2\beta}\right)\\
        &\lesssim \epsilon Y^{-p-2\beta+1/p}.
    \end{align*}
    Combining these estimates with the bounds for $\mc{E}_7$ proves \eqref{eqn:approx_E7_bound}, since $y\approx Y$ and $p+2\beta-(2p+1)=1/p-p>0$.
\end{proof}

\subsection{Control of \texorpdfstring{$\mc{E}_8$}{E8}}

In this section we control the term
\[\mc{E}_8 = \mc{Q}_\epsilon(\chi_Y f) - \mc{Q}_{\epsilon, 3}(\chi_Y f).\]
We will achieve good estimates because the difference is of quintic or higher order. 

\begin{lemma}\label{lem:approx_E8}
    The residual $\mc{E}_8$ satisfies
    \begin{equation}\label{eqn:approx_E8_bound}
        |\partial_y^j\mc{E}_8(y)|
        \lesssim\epsilon^2\ang{y}^{-2p-1+\frac{j}{p}},
        \qquad Y\leq y\leq2Y,\quad j=0,1.
    \end{equation}
\end{lemma}

\begin{proof}
    In general, we can write
    \[\mc{Q}_{\epsilon}(h)(y) - \mc{Q}_{\epsilon, 3}(h)(y) = \int \frac{h'(y) - h'(z)}{|y - z|^\alpha}\sq{A}_\epsilon\left(\frac{h(y) - h(z)}{y - z}\right) \, dz.\]
    Here, $\sq{A}_\epsilon$ is a smooth function satisfying 
    \[|\sq{A}_\epsilon(u)| \lesssim \epsilon^2 |u|^4, \quad |\sq{A}_\epsilon'(u)| \lesssim \epsilon^2 |u|^3.\]
    Therefore, using the bounds for $f \chi_Y$, we have that 
    \[|\mc{E}_8| \lesssim \epsilon^2 Y^{-4\beta} \int \frac{|(\chi_Y f)'(y) - (\chi_Y f)'(z)|}{|y - z|^\alpha} \, dz \lesssim \epsilon^2 Y^{-5\beta} \int \frac{\min\{1, Y^{1/p}|s|\}}{|s|^\alpha} \, ds \lesssim \epsilon^2 Y^{1 - 5\beta}.\]
    For the derivative, we make the change of variables $z=y-s$ and write
    \begin{equation}\label{eqn:approx_E8_translated}
        \mc{E}_8(y)=\int_{\R}\frac{(\chi_Y f)'(y)-(\chi_Y f)'(y-s)}{|s|^\alpha}\sq{A}_\epsilon\bigl(S_{\chi_Y f}(y,y-s)\bigr)\,ds.
    \end{equation}
    Using \eqref{eqn:approx_E7_secant_derivative}, we obtain
    \begin{equation}\label{eqn:approx_E8_derivative}
        \begin{aligned}
            \partial_y\mc{E}_8(y)
            & = \int_{\R}\frac{(\chi_Y f)''(y)-(\chi_Y f)''(y-s)}{|s|^\alpha}\sq{A}_\epsilon\bigl(S_{\chi_Y f}(y,y-s)\bigr)\,ds\\
            & \quad +\int_{\R}\frac{(\chi_Y f)'(y)-(\chi_Y f)'(y-s)}{|s|^\alpha}\sq{A}_\epsilon'\bigl(S_{\chi_Y f}(y,y-s)\bigr)S_{(\chi_Y f)'}(y,y-s)\,ds.
        \end{aligned}
    \end{equation}
    The derivative bounds for $\chi_Y f$ give $|S_{\chi_Y f}|\lesssim Y^{-\beta}$ and $|S_{(\chi_Y f)'}|\lesssim Y^{-\beta+1/p}$. Thus we can estimate \eqref{eqn:approx_E8_derivative} by
    \begin{align*}
        |\partial_y\mc{E}_8(y)|
        &\lesssim\epsilon^2Y^{-4\beta}\int_{\R}\frac{|(\chi_Y f)''(y)-(\chi_Y f)''(y-s)|}{|s|^\alpha}\,ds\\
        &\quad+\epsilon^2Y^{-4\beta+1/p}\int_{\R}\frac{|(\chi_Y f)'(y)-(\chi_Y f)'(y-s)|}{|s|^\alpha}\,ds\\
        &\lesssim\epsilon^2Y^{-5\beta+1/p}\int_{\R}\frac{\min\{1,Y^{1/p}|s|\}}{|s|^\alpha}\,ds\\
        &\lesssim\epsilon^2Y^{1-5\beta+1/p}
        \lesssim\epsilon^2\ang{y}^{-2p-1+1/p}.
    \end{align*}
    Together with the estimate for $\mc{E}_8$, this proves \eqref{eqn:approx_E8_bound} since $y\approx Y$ and $5\beta>2p+2$.
\end{proof}

\subsection{Control of \texorpdfstring{$\mc{E}_9$}{E9}}

In this section we control the term
\[\mc{E}_9 = \mc{Q}_\epsilon(f)-\mc{Q}_\epsilon(W)
        -D\mc{Q}_\epsilon(W)[f-W]-\mc{Q}_\epsilon(\chi_Y f).\]

\begin{lemma}\label{lem:approx_E9}
    The residual $\mc{E}_9$ satisfies
    \begin{equation}\label{eqn:approx_E9_bound}
        |\partial_y^j\mc{E}_9(y)|
        \lesssim\epsilon\ang{y}^{-p-2+\frac{j}{p}},
        \qquad Y\leq y\leq2Y,\quad j=0,1.
    \end{equation}
\end{lemma}

\begin{proof}
    We start by expressing $\mc{E}_9$ as the sum of the contributions from $z > 0$ and $z < 0$:
    \begin{equation}\label{eqn:approx_E9_integral_decomposition}
        \begin{aligned}
            \mc{E}_9(y)
            & = c_\alpha\int_0^\infty\frac{f'(y)-f'(z)}{|y-z|^\alpha}\bigl[A_\epsilon(S_f(y,z))-1\bigr]\,dz\\
            & \quad-c_\alpha\int_0^\infty\frac{f'(y)-(\chi_Yf)'(z)}{|y-z|^\alpha}\bigl[A_\epsilon(S_{\chi_Yf}(y,z))-1\bigr]\,dz\\
            & \quad+c_\alpha\int_{-\infty}^0\frac{A_\epsilon(S_f(y,z))-A_\epsilon(S_W(y,z))-A_\epsilon'(S_W(y,z))S_{f-W}(y,z)}{|y-z|^\alpha}\,dz\\
            & \quad+c_\alpha\int_{-\infty}^0\frac{f'(y)-f'(z)-1}{|y-z|^\alpha}\bigl[A_\epsilon(S_f(y,z))-A_\epsilon(S_W(y,z))\bigr]\,dz\\
            & \quad-c_\alpha\int_{-\infty}^0\frac{f'(y)}{|y-z|^\alpha}\left[A_\epsilon\left(\frac{f(y)}{y-z}\right)-1\right]\,dz \\
            & = \mc{E}_{9,+}(y) + \mc{E}_{9,-}(y).
        \end{aligned}
    \end{equation}
    For the contribution of the $z > 0$ term, observe that if $z \in [Y/2, 3Y]$ then the integrands cancel and equal zero. Therefore the $\mc{E}_{9,+}$ integral is supported where $|y - z| \gtrsim Y$. In this region, we can bound 
    \[|S_f(y, z)| + |\partial_y S_f(y, z)| \lesssim |y - z|^{-1}\]
    and therefore
    \[\left|\partial_y^j \left[\frac{f'(y) - f'(z)}{|y - z|^\alpha}[A_\epsilon(S_f(y,z)) - 1]\right]\right| \lesssim \epsilon Y^{j/p}|y - z|^{-2 - \alpha}, \quad j = 0,1.\]
    The same estimates hold for $\chi_Y f$ in place of $f$. Therefore, we can estimate
    \[|\partial_y^j \mc{E}_{9,+}| \lesssim \epsilon Y^{j/p}\int_{|y - z| \gtrsim Y}|y - z|^{-3 - p} \, dz \lesssim \epsilon Y^{-p - 2 + j/p}, \quad j = 0,1.\]
    For the contribution of the $z < 0$ term, we further decompose 
    \[\mc{E}_{9,-}(y) = \mc{I}_1 + \mc{I}_2 + \mc{I}_3\]
    where 
    \begin{equation}\label{eqn:approx_E9_negative_terms}
        \begin{aligned}
            \mc{I}_1 & = c_\alpha\int_{-\infty}^0\frac{A_\epsilon(S_f(y,z))-A_\epsilon(S_W(y,z))-A_\epsilon'(S_W(y,z))S_{f-W}(y,z)}{|y-z|^\alpha}\,dz,\\
            \mc{I}_2 & = c_\alpha\int_{-\infty}^0\frac{f'(y)-f'(z)-1}{|y-z|^\alpha}\bigl[A_\epsilon(S_f(y,z))-A_\epsilon(S_W(y,z))\bigr]\,dz,\\
            \mc{I}_3 & = -c_\alpha\int_{-\infty}^0\frac{f'(y)}{|y-z|^\alpha}\left[A_\epsilon\left(\frac{f(y)}{y-z}\right)-1\right]\,dz.
        \end{aligned}
    \end{equation}
    For the first term, we use Taylor's theorem to write
    \begin{equation}\label{eqn:approx_E9_I1_taylor}
        \mc{I}_1=c_\alpha\int_{-\infty}^0\frac{S_{f-W}(y,z)^2}{|y-z|^\alpha}\left[\int_0^1(1-t)A_\epsilon''\bigl(S_W(y,z)+tS_{f-W}(y,z)\bigr)\,dt\right]\,dz.
    \end{equation}
    Since $f - W$ is bounded, we have $|S_{f-W}(y,z)| \lesssim |y - z|^{-1}$. Using that $|A_{\epsilon}''| \lesssim \epsilon$, we have 
    \[|\mc{I}_1| \lesssim \epsilon\int_{-\infty}^{0}|y - z|^{-2 - \alpha} \, dz \lesssim \epsilon Y^{-2 - p}.\]
    For the derivative, we use that
    \[\partial_y S_{f - W}(y, z) = \frac{f'(y) - S_{f-W}(y,z)}{y - z}, \quad \partial_y S_W(y, z) = - \frac{S_{W}(y,z)}{y - z}.\]
    Therefore 
    \[|\partial_y S_{f - W}(y, z)| \lesssim \frac{Y^{-\beta}}{y - z} + |y - z|^{-2}, \quad |\partial_y S_W(y, z)| \lesssim |y - z|^{-1}.\]
    It follows that, using $|A_\epsilon''|+|A_\epsilon'''|\lesssim\epsilon$, we can bound $\partial \mc{I}_1(y)$ as
    \begin{align*}
        |\partial_y\mc{I}_1(y)|
        &\lesssim\epsilon\int_{-\infty}^0\frac{|S_{f-W}|^2}{|y-z|^{\alpha+1}}\,dz
        +\epsilon\int_{-\infty}^0\frac{|S_{f-W}|\,|\partial_yS_{f-W}|}{|y-z|^\alpha}\,dz\\
        &\quad+\epsilon\int_{-\infty}^0\frac{|S_{f-W}|^2\bigl(|\partial_yS_W|+|\partial_yS_{f-W}|\bigr)}{|y-z|^\alpha}\,dz\\
        &\lesssim\epsilon\int_{-\infty}^0\left[|y-z|^{-p-4}+Y^{-\beta}|y-z|^{-p-3}\right]\,dz\\
        &\lesssim\epsilon\left(Y^{-p-3}+Y^{-p-2-\beta}\right)
        \lesssim\epsilon Y^{-p-3}.
    \end{align*}
    Here all secants are evaluated at $(y,z)$, and the last inequality uses $\beta>1$. \par 
    For $\mc{I}_2$, we use Taylor's theorem to write
    \begin{equation}\label{eqn:approx_E9_I2_taylor}
        A_\epsilon(S_f)-A_\epsilon(S_W)
        =S_{f-W}\int_0^1A_\epsilon'\bigl(S_W+tS_{f-W}\bigr)\,dt.
    \end{equation}
    Using $|A_\epsilon'|\lesssim\epsilon$ and the secant bound above in \eqref{eqn:approx_E9_I2_taylor}, we obtain
    \[|A_\epsilon(S_f)-A_\epsilon(S_W)|\lesssim\frac{\epsilon}{y-z},\qquad
    |f'(y)-f'(z)-1|\lesssim Y^{-\beta}+\ang{z}^{-p-1}.\]
    Therefore, we can estimate
    \begin{align*}
        |\mc{I}_2|
        &\lesssim\epsilon\int_{-\infty}^0\frac{Y^{-\beta}+\ang{z}^{-p-1}}{(y-z)^{p+2}}\,dz\\
        &\lesssim\epsilon\left(Y^{-p-1-\beta}+Y^{-p-2}\int_{-\infty}^0\ang{z}^{-p-1}\,dz\right)\\
        &\lesssim\epsilon Y^{-p-2}.
    \end{align*}
    For the derivative, we use that
    \begin{equation}\label{eqn:approx_E9_I2_coefficient_derivative}
        \begin{aligned}
            \partial_y\bigl[A_\epsilon(S_f)-A_\epsilon(S_W)\bigr]
            & = \bigl[A_\epsilon'(S_f)-A_\epsilon'(S_W)\bigr]\partial_yS_W\\
            & \quad+A_\epsilon'(S_f)\partial_yS_{f-W}.
        \end{aligned}
    \end{equation}
    Applying the secant bounds above and $|A_\epsilon'|+|A_\epsilon''|\lesssim\epsilon$ to \eqref{eqn:approx_E9_I2_coefficient_derivative} gives
    \[\left|\partial_y\bigl[A_\epsilon(S_f)-A_\epsilon(S_W)\bigr]\right|
    \lesssim\epsilon\left(\frac{Y^{-\beta}}{y-z}+\frac{1}{(y-z)^2}\right).\]
    Differentiating the integral defining $\mc{I}_2$ and using $|f''(y)|\lesssim Y^{-\beta+1/p}$, we obtain
    \begin{align*}
        |\partial_y\mc{I}_2(y)|
        &\lesssim\epsilon Y^{-\beta+1/p}\int_{-\infty}^0\frac{dz}{(y-z)^{p+2}}\\
        &\quad+\epsilon\int_{-\infty}^0\frac{Y^{-\beta}+\ang{z}^{-p-1}}{(y-z)^{p+3}}\,dz\\
        &\quad+\epsilon Y^{-\beta}\int_{-\infty}^0\frac{Y^{-\beta}+\ang{z}^{-p-1}}{(y-z)^{p+2}}\,dz\\
        &\lesssim\epsilon\left(Y^{-p-1-\beta+1/p}+Y^{-p-2-\beta}+Y^{-p-3}+Y^{-p-1-2\beta}\right)\\
        &\lesssim\epsilon Y^{-p-2+1/p}.
    \end{align*}
    In both estimates, the last inequality uses $\beta>1$.

    Finally, for $\mc{I}_3$, we use the bounds
    \begin{equation}\label{eqn:approx_E9_I3_coefficient_bounds}
        |A_\epsilon(u)-1|\lesssim\epsilon|u|^2,\qquad
        |A_\epsilon'(u)|\lesssim\epsilon|u|.
    \end{equation}
    It follows that
    \begin{align*}
        |\mc{I}_3|
        &\lesssim\epsilon|f'(y)|\,|f(y)|^2\int_{-\infty}^0\frac{dz}{(y-z)^{p+3}}\\
        &\lesssim\epsilon Y^{-\beta-2p}Y^{-p-2}
        \lesssim\epsilon Y^{-p-2}.
    \end{align*}
    For the derivative, we use that
    \[\partial_y\left(\frac{f(y)}{y-z}\right)=\frac{f'(y)}{y-z}-\frac{f(y)}{(y-z)^2}.\]
    Differentiating the integral defining $\mc{I}_3$ and using \eqref{eqn:approx_E9_I3_coefficient_bounds}, we obtain
    \begin{align*}
        |\partial_y\mc{I}_3(y)|
        &\lesssim\epsilon\left(|f''(y)|\,|f(y)|^2+|f'(y)|^2|f(y)|\right)\int_{-\infty}^0\frac{dz}{(y-z)^{p+3}}\\
        &\quad+\epsilon|f'(y)|\,|f(y)|^2\int_{-\infty}^0\frac{dz}{(y-z)^{p+4}}\\
        &\lesssim\epsilon\left(Y^{-\beta-2p+1/p}+Y^{-2\beta-p}\right)Y^{-p-2}
        +\epsilon Y^{-\beta-2p}Y^{-p-3}\\
        &\lesssim\epsilon Y^{-p-2+1/p}.
    \end{align*}
    Combining the estimates for $\mc{I}_1$, $\mc{I}_2$, and $\mc{I}_3$ with the bound for $\mc{E}_{9,+}$ proves \eqref{eqn:approx_E9_bound}, since $y\approx Y$.
\end{proof}

\subsection{Residual Bounds for \texorpdfstring{$y \leq -8R$}{y <= -8R}} 

In this section, we will prove estimates on the residual $\sq{\mc{E}}$ in the negative region $(-\infty, -8R)$. The estimates are much simpler than in the positive region.

\begin{proposition}\label{prop:approx_residual_negative}
    For $y \leq -8R$, we have the estimate
    \[|\partial^j(\sq{\mc{E}}_\epsilon(f)-G_\mu)(y)| \lesssim \epsilon \ang{y}^{-\gamma - j}, \quad j = 0,1.\]
\end{proposition}

\begin{proof}
    As in the positive region, we write
    \[\sq{\mc{E}}_\epsilon(f)-G_\mu = \mc{L}_0(f_{\textup{slow}}) + \mc{L}_0(f_{\textup{osc}}) + \mc{Q}_\epsilon(f).\]
    We now split
    \begin{align*}
        \sq{\mc{E}}_\epsilon(f)-G_\mu & = \mc{L}_0(f_{\textup{slow}}) + \mc{Q}_\epsilon(W) \\
        & + |D|^p f_{\textup{osc}}' \\
        & + \mc{Q}_\epsilon(f) - \mc{Q}_\epsilon(W)
    \end{align*}
    and estimate each term separately. For $y\leq-8R$, the choice of $c_-$ gives
    \[
        \mc{Q}_\epsilon(W)(y)
        =-c_\alpha|y|^{-p}\int_0^1[A_\epsilon(t)-1]t^{p-1}\,dt
        =-(p+1)c_-|y|^{-p},
    \]
    so $\mc{L}_0(f_{\textup{slow}})+\mc{Q}_\epsilon(W)=|D|^pf_{\textup{slow}}'$. The same argument as in the proof of \cref{prop:approx_E1}, with $|y|$ in place of $y$, gives the bound \eqref{eqn:approx_E1_bound} on the negative side:
    \[
        \left|\partial_y^j\left(\mc{L}_0(f_{\textup{slow}})+\mc{Q}_\epsilon(W)\right)(y)\right|
        \lesssim\epsilon\ang{y}^{-2p-1-j}
        \lesssim\epsilon\ang{y}^{-\gamma-j},\qquad j=0,1.
    \]
    To bound $|D|^p f_{\textup{osc}}'$, by \cref{lem:approx_phase_differences}, we know that 
    \[\|f_{\textup{osc}}\|_{L^1} \lesssim \epsilon, \quad |f_{\textup{osc}}(z)| \lesssim \epsilon z^{-p - \beta}. \]
    Therefore for $y \leq -8R$, we can calculate
    \[|D|^p f_{\textup{osc}}'(y) = -c_\alpha \int_{R}^{\infty}\frac{f_{\textup{osc}}'(z)}{(z - y)^\alpha} \, dz = -c_\alpha\alpha \int_{R}^{\infty}\frac{f_{\textup{osc}}(z)}{(z - y)^{\alpha+1}} \, dz\]
    and bound
    \begin{align*}
        \left|\partial^j|D|^p f_{\textup{osc}}'(y)\right|
        &\lesssim\int_R^\infty\frac{|f_{\textup{osc}}(z)|}{(z-y)^{p+2 + j}}\,dz
        \lesssim\ang{y}^{-p-2 - j}\|f_{\textup{osc}}\|_{L^1}
        \lesssim\epsilon\ang{y}^{-\gamma - j}
    \end{align*}
    where we used $\gamma<2p+1<p+2$. \par
    Finally, to bound the difference 
    \[\mc{Q}_\epsilon(f) - \mc{Q}_\epsilon(W),\]
    we use a near-far splitting of the integral. We write
    \begin{equation}\label{eqn:approx_negative_Q_split}
    \begin{aligned}
        &\mc{Q}_\epsilon(f)(y)-\mc{Q}_\epsilon(W)(y)\\
        &\quad=
        c_\alpha\int_{\R}\chi\left(\frac{y-z}{|y|}\right)
        \frac{f'(y)-f'(z)}{|y-z|^\alpha}
        [A_\epsilon(S_f(y,z))-1]\,dz\\
        &\quad+c_\alpha\int_{\R}\left[1-\chi\left(\frac{y-z}{|y|}\right)\right]
        \frac{(f-W)'(y)-(f-W)'(z)}{|y-z|^\alpha}
        [A_\epsilon(S_f(y,z))-1]\,dz\\
        &\quad-c_\alpha\int_0^\infty
        \frac{A_\epsilon(S_f(y,z))-A_\epsilon(S_W(y,z))}
        {(z-y)^\alpha}\,dz,
    \end{aligned}
    \end{equation}
    where $\chi$ is a smooth cutoff supported in $(-1/2,1/2)$ and equal to one on $[-1/4,1/4]$. Denote the three terms in \eqref{eqn:approx_negative_Q_split} by $\mc{I}_1$, $\mc{I}_2$, and $\mc{I}_3$, respectively. By \cref{prop:approx_profile_bounds},
    \[
        |(f-W)(z)|\lesssim\ang{z}^{-p},\qquad
        |\partial^j(f-W)(y)|\lesssim |y|^{-p-j},\quad y\leq-8R.
    \]
    Moreover, $f'$ and the secants are bounded, so $|A_\epsilon-1|+|A_\epsilon'|+|A_\epsilon''|\lesssim\epsilon$ on their range. \par
    For the first term, change variables to $s=y-z$. On the support of $\chi(s/|y|)$, we have $y-s<0$ and $|y-s|\approx|y|$, so
    \[
        |\partial_y^j(f'(y)-f'(y-s))|\lesssim |y|^{-p-2-j}|s|,
        \qquad j=0,1.
    \]
    Also, $\partial_y S_f(y,y-s)=(f'(y)-f'(y-s))/s$ is bounded by $|y|^{-p-2}$, while differentiating the cutoff costs $|y|^{-1}$. Therefore
    \[
        |\partial_y^j\mc{I}_1(y)|
        \lesssim\epsilon |y|^{-p-2-j}\int_{|s|\leq|y|/2}|s|^{-p}\,ds
        \lesssim\epsilon |y|^{-2p-1-j},\qquad j=0,1.
    \]
    For the second term, integrate by parts in $z$ to write
    \begin{equation}\label{eqn:approx_negative_Q_far_ibp}
    \begin{aligned}
        \mc{I}_2(y)
        & =c_\alpha(f-W)'(y)\int_\R
        \left[1-\chi\left(\frac{y-z}{|y|}\right)\right]
        \frac{A_\epsilon(S_f(y,z))-1}{|y-z|^\alpha}\,dz\\
        &\quad+c_\alpha\int_\R(f-W)(z)\partial_z\left\{
        \left[1-\chi\left(\frac{y-z}{|y|}\right)\right]
        \frac{A_\epsilon(S_f(y,z))-1}{|y-z|^\alpha}\right\}\,dz.
    \end{aligned}
    \end{equation}
    The boundary terms at infinity vanish, and those at $z=0$ cancel because $f-W$ is continuous. The identities
    \[
        \partial_y S_f=\frac{f'(y)-S_f}{y-z},\qquad
        \partial_z S_f=\frac{S_f-f'(z)}{y-z},\qquad
        \partial_y\partial_z S_f=\frac{f'(y)+f'(z)-2S_f}{(y-z)^2}
    \]
    and the cutoff bounds imply, for $j,\ell\in\{0,1\}$,
    \[
        \left|\partial_y^j\partial_z^\ell\left\{
        \left[1-\chi\left(\frac{y-z}{|y|}\right)\right]
        \frac{A_\epsilon(S_f(y,z))-1}{|y-z|^\alpha}\right\}\right|
        \lesssim\epsilon\frac{\mathbf{1}_{\{|y-z|\geq|y|/4\}}}
        {|y-z|^{p+1+j+\ell}}.
    \]
    Applying these bounds to \eqref{eqn:approx_negative_Q_far_ibp}, we obtain, for $j=0,1$,
    \begin{align*}
        |\partial_y^j\mc{I}_2(y)|
        &\lesssim\epsilon |y|^{-2p-1-j}
        +\epsilon\int_{|y-z|\geq|y|/4}
        \frac{\ang{z}^{-p}}{|y-z|^{p+2+j}}\,dz\\
        &\lesssim\epsilon |y|^{-2p-1-j}
        +\epsilon |y|^{-p-2-j}\int_{|z|\leq2|y|}\ang{z}^{-p}\,dz
        +\epsilon\int_{|z|>2|y|}|z|^{-2p-2-j}\,dz\\
        &\lesssim\epsilon |y|^{-2p-1-j},
    \end{align*}
    where we used $p<1$. \par
    Finally, for $z>0$ we have $z-y=|y|+z$. Since $S_f-S_W=S_{f-W}$, the mean value theorem gives
    \[
        |A_\epsilon(S_f(y,z))-A_\epsilon(S_W(y,z))|
        \lesssim\epsilon\frac{|y|^{-p}+\ang{z}^{-p}}{|y|+z}.
    \]
    Differentiating the difference and using
    \[
        \partial_y(A_\epsilon(S_f)-A_\epsilon(S_W))
        =A_\epsilon'(S_f)\partial_y S_{f-W}
        +[A_\epsilon'(S_f)-A_\epsilon'(S_W)]\partial_y S_W,
    \]
    we also obtain
    \[
        |\partial_y(A_\epsilon(S_f(y,z))-A_\epsilon(S_W(y,z)))|
        \lesssim\epsilon\left[
        \frac{|y|^{-p-1}}{|y|+z}
        +\frac{|y|^{-p}+\ang{z}^{-p}}{(|y|+z)^2}\right].
    \]
    Therefore, splitting the integrals at $z=|y|$ as above,
    \begin{align*}
        |\mc{I}_3(y)|
        &\lesssim\epsilon\int_0^\infty
        \frac{|y|^{-p}+\ang{z}^{-p}}{(|y|+z)^{p+2}}\,dz
        \lesssim\epsilon |y|^{-2p-1},\\
        |\partial_y\mc{I}_3(y)|
        &\lesssim\epsilon\int_0^\infty\left[
        \frac{|y|^{-p-1}}{(|y|+z)^{p+2}}
        +\frac{|y|^{-p}+\ang{z}^{-p}}{(|y|+z)^{p+3}}\right]\,dz
        \lesssim\epsilon |y|^{-2p-2}.
    \end{align*}
    Combining the three estimates gives
    \begin{equation}\label{eqn:approx_negative_Q_difference_bound}
        |\partial_y^j(\mc{Q}_\epsilon(f)-\mc{Q}_\epsilon(W))(y)|
        \lesssim\epsilon\ang{y}^{-2p-1-j}
        \lesssim\epsilon\ang{y}^{-\gamma-j},\qquad j=0,1.
    \end{equation}
    Together with the previous bounds, this proves the proposition.
\end{proof}

\subsection{Residual Bounds for \texorpdfstring{$|y| \leq 8R$}{|y| <= 8R}} 

In this section, we will prove estimates on the residual $\sq{\mc{E}}$ in the compact region $[-8R, 8R]$.

\begin{proposition}\label{prop:approx_residual_compact}
    We have the estimate 
    \[\|\sq{\mc{E}}_\epsilon(\ov{f}_{\epsilon, \mu}) - G_\mu\|_{C^1([-8R, 8R])} \lesssim_R \epsilon.\]
\end{proposition}

\begin{proof}
    Again, we write
    \[\sq{\mc{E}}_\epsilon(\ov{f}_{\epsilon, \mu}) - G_\mu = \mc{L}_0(f_{\textup{slow}}) + \mc{L}_0(f_{\textup{osc}}) + \mc{Q}_\epsilon(\ov{f}_{\epsilon, \mu}).\]
    Since $f_{\textup{slow}}$ in \eqref{eqn:approx_slow_correction} is smooth and $\mc{O}(\epsilon)$, the term $\mc{L}_0(f_{\textup{slow}})$ is smooth and $\mc{O}(\epsilon)$. For the oscillatory term, we recall \eqref{eqn:approx_oscillatory_correction}:
    \[f_{\textup{osc}}(y) = -2 \chi_R(y)\operatorname{Re} \int_{y}^{\infty}\left(Z_{\epsilon, \mu}(t) + c_{\textup{cubic}}Z_{\epsilon, \mu}(t)^3 - \frac{\mu}{\mu_*}Z_0(t)\right) \, dt.\]
    By \cref{lem:approx_phase_differences}, we can bound
    \[
        \left|\frac{Z_{\epsilon,\mu}(y)}{\lambda_{\epsilon,\mu}(y)}
        -\frac{\mu}{\mu_*}\frac{Z_0(y)}{\lambda_0(y)}\right|
        \lesssim\epsilon y^{-p-\beta},\qquad y\geq R.
    \]
    Integrating each primitive once by parts before subtracting, and using the bounds for $a_*$ in \cref{lem:approx_linear_seed} and the phase and frequency differences in \eqref{eqn:approx_phase_difference}--\eqref{eqn:approx_lambda_difference}, gives
    \begin{equation}\label{eqn:approx_oscillatory_primitive_difference}
    \begin{aligned}
        \left|\int_y^\infty\left(Z_{\epsilon,\mu}(t)-\frac{\mu}{\mu_*}Z_0(t)\right)\,dt\right|
        &\lesssim\epsilon y^{-p-\beta}
        +\epsilon\int_y^\infty t^{-p-\beta-1}\,dt\\
        &\lesssim\epsilon y^{-p-\beta}.
    \end{aligned}
    \end{equation}
    Here the first term is the boundary contribution, and the second bounds the difference of the remaining integrands. Also,
    \[
        \left|c_{\textup{cubic}}\int_y^\infty Z_{\epsilon,\mu}(t)^3\,dt\right|
        \lesssim\epsilon y^{1-3\beta}
        \lesssim\epsilon y^{-p-\beta}.
    \]
    Together with the derivative bounds in \cref{lem:approx_phase_differences}, this proves
    \begin{equation}\label{eqn:approx_compact_oscillatory_norms}
        \|f_{\textup{osc}}\|_{L^\infty(\R)}
        +\|f_{\textup{osc}}\|_{C^3([-10R,10R])}
        \lesssim_R\epsilon.
    \end{equation}
    We now use the local estimate
    \begin{equation}\label{eqn:approx_compact_fractional_bound}
        \bigl\||D|^pu'\bigr\|_{C^1([-8R,8R])}
        \lesssim_R\|u\|_{C^3([-10R,10R])}
        +\|u\|_{L^\infty(\R)}.
    \end{equation}
    This estimate follows from a near-field/far-field decomposition. Applying \eqref{eqn:approx_compact_fractional_bound} with $u=f_{\textup{osc}}$ and using \eqref{eqn:approx_compact_oscillatory_norms} also to bound the transport term, we obtain
    \begin{equation}\label{eqn:approx_compact_oscillatory_linear}
        \|\mc{L}_0(f_{\textup{osc}})\|_{C^1([-8R,8R])}
        \lesssim_R\epsilon.
    \end{equation}
    Finally, using that $\mc{Q}_{\epsilon}(\ov{f}_{\epsilon, \mu})$ is smooth and $\mc{O}(\epsilon)$, we can bound this term locally. The proof of the proposition follows.
\end{proof}

\subsection{The Proof of \texorpdfstring{\cref{thm:fixed_point_approximation}}{}}

In this section, we will combine the estimates for the residual $\sq{\mc{E}}$ in the three regions $y \leq -8R$, $|y| \leq 8R$, and $y \geq 8R$ to prove \cref{thm:fixed_point_approximation}.

\begin{proof}[Proof of \cref{thm:fixed_point_approximation}]
    The bounds for the family $\ov{f}_{\epsilon, \mu}$ and continuity of $\mu \mapsto \ov{f}_{\epsilon, \mu}$ into $C^{2}_{\textup{loc}}$ follow from \cref{prop:approx_profile_bounds}. We combine the three estimates in \cref{prop:approx_residual_negative,prop:approx_residual_compact,prop:approx_residual_positive} to obtain the residual bound
    \[\|\sq{\mc{E}}_\epsilon(\ov{f}_{\epsilon, \mu}) - G_\mu\|_{Y_\gamma} \lesssim \epsilon.\]
    It remains to prove the continuity of $\mu \mapsto \sq{\mc{E}}_\epsilon(\ov{f}_{\epsilon, \mu})$ into $C^{1}_{\textup{loc}}$. Let $\mu_n \to \mu$ and let $f_n = \ov{f}_{\epsilon, \mu_n}$ and $f = \ov{f}_{\epsilon, \mu}$. We can write
    \[\sq{\mc{E}}_\epsilon(f_n)(y) = f_n(y) - yf_n'(y) + c_\alpha \int \frac{f_n'(y) - f_n'(y-s)}{|s|^{\alpha}}A_\epsilon \left(\frac{f_n(y) - f_n(y - s)}{s}\right) \, ds.\]
    The first two terms converge to $f(y) - yf'(y)$ in $C^1_{\textup{loc}}$ by the continuity of $\mu \mapsto \ov{f}_{\epsilon, \mu}$ into $C^{2}_{\textup{loc}}$. For the integral term, we apply the dominated convergence theorem. Choose $M$ and let $|y| \leq M$. The integrand can be bounded by
    \[ \lesssim_M \begin{cases}
        |s|^{-p}, & |s| \leq 1 \\
        |s|^{-p - 1}, & |s| > 1
    \end{cases}\]
    by the local smoothness of $f_n$ and the bounds $\|f_n'\|_{L^\infty} \lesssim 1$. Similarly, we can bound the derivative of the integrand by
    \[ \lesssim_M \begin{cases}
        |s|^{-p}, & |s| \leq 1 \\
        |s|^{-p - 1}(1 + |s|^{1/p - \beta}), & |s| > 1
    \end{cases}.\]
    In either case, the integrand is integrable. By the dominated convergence theorem, we can pass to the limit under the integral sign and conclude that $\sq{\mc{E}}_\epsilon(f_n) \to \sq{\mc{E}}_\epsilon(f)$ in $C^1([-M,M])$. This proves the continuity of $\mu \mapsto \sq{\mc{E}}_\epsilon(\ov{f}_{\epsilon, \mu})$ into $C^{1}_{\textup{loc}}$. \par
    The last step is to go from the assertions about $\sq{\mc{E}}_\epsilon$ to $\mc{E}_\epsilon$. To do this, we simply write
    \[\mc{E}_\epsilon(\ov{f}_{\epsilon, \mu}) = B_\epsilon(\ov{f}_{\epsilon, \mu}')\sq{\mc{E}_\epsilon}(\ov{f}_{\epsilon, \mu}).\]
    Continuity of $\mu \mapsto \mc{E}_\epsilon(\ov{f}_{\epsilon, \mu})$ into $C^{1}_{\textup{loc}}$ follows from the continuity of $\mu \mapsto \sq{\mc{E}}_\epsilon(\ov{f}_{\epsilon, \mu})$ into $C^{1}_{\textup{loc}}$, the smoothness of $B_\epsilon$, and the bounds for $\ov{f}_{\epsilon, \mu}$. For the residual, we write
    \[\mc{E}_\epsilon(\ov{f}_{\epsilon, \mu}) - G_\mu = B_\epsilon(\ov{f}_{\epsilon, \mu}')[\sq{\mc{E}_\epsilon}(\ov{f}_{\epsilon, \mu}) - G_\mu] + [B_\epsilon(\ov{f}_{\epsilon, \mu}') - 1]G_\mu\]
    and again, using the bounds for $\ov{f}_{\epsilon, \mu}$ and the smoothness of $B_\epsilon$, we get that
    \[\|\mc{E}_\epsilon(\ov{f}_{\epsilon, \mu}) - G_\mu\|_{Y_\gamma} \lesssim \epsilon\]
    and continuity of $\mu \mapsto \mc{E}_\epsilon(\ov{f}_{\epsilon, \mu})$ into $C_{\textup{loc}}^1$. Continuity into $Y_\gamma$ follows from \cref{prop:appendix_convergence}, using the uniform residual bound with a slightly larger weight exponent. This completes the proof of \cref{thm:fixed_point_approximation}.
\end{proof}

\section{Linear Inverse}\label{sec:linear_inverse}

The purpose of this section is to prove \cref{thm:fixed_point_linear_inverse}. We will first use a conjugation argument to reduce the perturbed operator $\mc{L}_\epsilon$ to the free operator $\mc{L}_0$. Recall that the free linear operator is defined by
\[\mc{L}_0(r) = |D|^pr' + r - yr'\]
and the perturbed operator 
\[\mc{L}_\epsilon(r) = \mc{L}_0(r) + b_\epsilon r' + \mc{R}_\epsilon(r)\]
where $b_\epsilon(y) = c_\epsilon \chi_+(y)y^{-p}$ and $\mc{R}_\epsilon$ is a remainder term that will be defined in \cref{lem:linear_conjugation}. 

\subsection{Conjugation of the Linear Operator}

In the first lemma, we will construct the isomorphisms $\mc{U}_\epsilon$ and $\mc{V}_\epsilon$ that conjugate $\mc{L}_\epsilon$ to $\mc{L}_0$.

\begin{lemma}\label{lem:linear_conjugation}
    There exist linear operators $\mc{U}_\epsilon$ and $\mc{V}_\epsilon$ which are isomorphisms of $X_\gamma$ and $Y_\gamma$ respectively so that 
    \[\mc{V}_\epsilon \mc{L}_0 \mc{U}_\epsilon  = \mc{L}_0 + b_\epsilon \partial_y + \mc{R}_\epsilon\]
    where $\mc{R}_\epsilon: X_\gamma \to Y_\gamma$ is a bounded operator with norm $\lesssim \epsilon$. Moreover, we have the uniform bounds
    \[\sup_{0 \leq \epsilon \leq \epsilon_0}\|\mc{U}_{\epsilon}\|_{X_\gamma \to X_\gamma} < \infty, \quad \sup_{0 \leq \epsilon \leq \epsilon_0}\|\mc{V}_{\epsilon}\|_{Y_\gamma \to Y_\gamma} < \infty\]
    and
    \[\sup_{0 \leq \epsilon \leq \epsilon_0}\|\mc{U}_{\epsilon}^{-1}\|_{X_\gamma \to X_\gamma} < \infty, \quad \sup_{0 \leq \epsilon \leq \epsilon_0}\|\mc{V}_{\epsilon}^{-1}\|_{Y_\gamma \to Y_\gamma} < \infty\]
    and the estimate
    \[\|(\mc{V}_\epsilon^{-1} - I)G\|_{Y_\gamma} \lesssim \epsilon(\|\ang{y}^{\gamma}G\|_{L^\infty(1, \infty)} + \|\ang{y}^{\gamma}G'\|_{L^\infty(1, \infty)} + \|\ang{y}^{\gamma}G''\|_{L^\infty(1, \infty)}).\]
\end{lemma}

\begin{proof}
    Let $\phi$ be a smooth increasing diffeomorphism, and write
    \[
        \mc{T}_\phi u(y)=u(\phi(y)).
    \]
    The idea is to choose $\phi$ so that the conjugation of $\mc{L}_0$ by $\mc{T}_\phi$ resembles $\mc{L}_\epsilon$. We claim that 
    \begin{equation}\label{eqn:linear_conjugation}
        (\phi')^{p + 1}\mc{T}_\phi \mc{L}_0 \mc{T}_{\phi^{-1}} = |D|^p \partial_y + \left(I - \phi(\phi')^p\partial_y\right) + \mc{R}
    \end{equation}
    where $\mc{R} = \mc{R}_1 + \mc{R}_2$ with 
    \begin{equation}\label{eqn:linear_R1}
        \mc{R}_1 = \left((\phi')^{p + 1} - 1\right)I
    \end{equation}
    and
    \begin{equation}\label{eqn:linear_R2}
        \mc{R}_2r(y) = c_\alpha\int_{\R}
        \left[1-\left(\frac{\phi'(y)}{S_\phi(y,z)}\right)^{p+1}\right]
        \frac{r'(z)}{|y-z|^{1+p}}\,dz.
    \end{equation}
    Indeed, first observe that
    \[
        \mc{T}_\phi\partial_y\mc{T}_\phi^{-1}
        = \frac{1}{\phi'}\partial_y,
        \qquad
        \mc{T}_\phi(y\partial_y)\mc{T}_\phi^{-1}
        =\frac{\phi}{\phi'}\partial_y.
    \]
    For the fractional term, we perform a change of variables to obtain
    \[
        (\mc{T}_\phi|D|^p\mc{T}_\phi^{-1}v)(y)
        =c_\alpha\int_{\R}\frac{v(y)-v(z)}{|y-z|^{1+p}}
        \frac{\phi'(z)}{S_\phi(y,z)^{1+p}}\,dz.
    \]
    Define the remainder by
    \begin{equation}\label{eqn:linear_E1}
        \mc{C}_{\phi,1}v(y)
        :=c_\alpha\int_{\R}\frac{v(y)-v(z)}{|y-z|^{1+p}}
        \left(\frac{\phi'(z)}{S_\phi(y,z)^{1+p}}
        -\frac{1}{\phi'(y)^p}\right)\,dz.
    \end{equation}
    Thus
    \[
        \mc{T}_\phi|D|^p\mc{T}_\phi^{-1}
        =\frac{1}{(\phi')^p}|D|^p+\mc{C}_{\phi,1}.
    \]
    Combining this identity with the formula for $\partial_y$ gives
    \[
        \mc{T}_\phi|D|^p\partial_y\mc{T}_\phi^{-1}
        =\frac{1}{(\phi')^{p+1}}|D|^p\partial_y+\mc{C}_{\phi,2},
    \]
    where
    \begin{equation}\label{eqn:linear_E2}
        \mc{C}_{\phi,2}
        :=\frac{1}{(\phi')^p}
        [|D|^p,\frac{1}{\phi'}]\partial_y
        +\mc{C}_{\phi,1}\frac{1}{\phi'}\partial_y.
    \end{equation}
    To expand $\mc{C}_{\phi,2}$ in \eqref{eqn:linear_E2}, first note that
    \[
        \phi'(y)\left[|D|^p,\frac{1}{\phi'}\right]r'(y)
        =c_\alpha\int_{\R}\left(1-\frac{\phi'(y)}{\phi'(z)}\right)
        \frac{r'(z)}{|y-z|^{1+p}}\,dz.
    \]
    On the other hand, \eqref{eqn:linear_E1} gives
    \begin{align*}
        \phi'(y)^{p+1}\mc{C}_{\phi,1}\left(\frac{r'}{\phi'}\right)(y)
        &=c_{\alpha}r'(y)\int_{\R}
        \left(\frac{\phi'(y)^p\phi'(z)}{S_\phi(y,z)^{1+p}}-1\right)
        \frac{dz}{|y-z|^{1+p}}\\
        &\quad+c_\alpha\int_{\R}
        \left[\frac{\phi'(y)}{\phi'(z)}
        -\left(\frac{\phi'(y)}{S_\phi(y,z)}\right)^{p+1}\right]
        \frac{r'(z)}{|y-z|^{1+p}}\,dz.
    \end{align*}
    Adding these identities cancels the terms containing
    $\phi'(y)r'(z)/\phi'(z)$. Using
    $S_\phi(y,z)|y-z|=|\phi(y)-\phi(z)|$ and \eqref{eqn:linear_R2}, we obtain
    \begin{align*}
        \phi'(y)^{p+1}\mc{C}_{\phi,2}r(y)
        & =\mc{R}_2r(y)\\
        &\quad+c_{\alpha}r'(y)\int_{\R}\left[
        \frac{\phi'(y)^p\phi'(z)}{|\phi(y)-\phi(z)|^{1+p}}
        -\frac{1}{|y-z|^{1+p}}\right]\,dz.
    \end{align*}
    Using a change of variables, the integral term vanishes, and thus $(\phi')^{p+1}\mc{C}_{\phi,2}=\mc{R}_2$.
    The conjugated full operator is therefore
    \[
        \mc{T}_\phi\mc{L}_0\mc{T}_\phi^{-1}
        =\frac{1}{(\phi')^{p+1}}\bigl(|D|^p\partial_y+\mc{R}_2\bigr)
        +
        \left(I-\frac{\phi}{\phi'}\partial_y\right).
    \]
    Multiplying by $(\phi')^{p+1}$ and using \eqref{eqn:linear_R1} gives
    \[
        (\phi')^{p+1}\mc{T}_\phi\mc{L}_0\mc{T}_\phi^{-1}
        =|D|^p\partial_y
        +
        \left(I-\phi(\phi')^p\partial_y\right)
        +\mc{R}_1+\mc{R}_2,
    \]
    which proves \eqref{eqn:linear_conjugation}. We now claim that we can choose $\phi$ so that
    \begin{equation}\label{eqn:linear_phi_ode}
        \phi (\phi')^p = y - b_\epsilon.
    \end{equation}
    It follows from \eqref{eqn:linear_conjugation} and \eqref{eqn:linear_phi_ode} that, if we set
    \begin{equation}\label{eqn:linear_conjugating_maps}
        \mc{V} = (\phi')^{p+1}\mc{T}_\phi, \qquad \mc{U} = \mc{T}_\phi^{-1},
    \end{equation}
    then
    \[\mc{V}\mc{L}_0\mc{U} = \mc{L}_\epsilon.\]
    It remains to show that $\phi$ is indeed a well-defined diffeomorphism, $\mc{R}$ is a bounded operator from $X_\gamma$ to $Y_\gamma$ and that $\mc{U}$ and $\mc{V}$ are isomorphisms of $X_\gamma$ and $Y_\gamma$ respectively. \par
    \textbf{Definition and properties of $\phi$:} We first need to show that $\phi$ is a diffeomorphism and that it satisfies certain bounds. For $y < 1$, we have that $b_\epsilon(y) = 0$ and therefore $\phi(y) = y$. For $y > 1$, we can solve \eqref{eqn:linear_phi_ode} to obtain 
    \[\phi(y) = \left[1 + \frac{p+1}{p} \int_{1}^{y}(t - b_\epsilon(t))^{1/p}\,dt\right]^{p / (p + 1)}.\]
    For $\epsilon$ small enough, and $t > 1$, we have 
    \[t - b_\epsilon(t) \approx t.\]
    It follows that $\phi(y) \approx y$ and so 
    \[\phi'(y) = \left(\frac{y - b_\epsilon(y)}{\phi(y)}\right)^{1/p} \approx 1.\]
    Therefore $\phi$ is a diffeomorphism. We remark that we also have the bounds
    \begin{equation}\label{eqn:linear_phi_bounds}
        |\partial^j(\phi(y) - y)| \lesssim_j \epsilon\ang{y}^{-p-j}, \quad j \geq 0.
    \end{equation}
    Indeed, write 
    \[\phi(y) - y = y[(1 + E(y))^{p/(p+1)} - 1]\] 
    where
    \[E(y) = \frac{p+1}{p}y^{-(p+1)/p}\int_{1}^{y}[(t - b_\epsilon(t))^{1/p} - t^{1/p}]\,dt.\]
    Observe that
    \[|\partial^j((t - b_\epsilon(t))^{1/p} - t^{1/p})| \lesssim_j \epsilon t^{1/p - p - 1 - j}\]
    which follows by writing
    \[(t - b_\epsilon(t))^{1/p} - t^{1/p} = t^{1/p}\left(\left(1 - \frac{b_\epsilon(t)}{t}\right)^{1/p} - 1\right)\]
    and using the product rule and $p < 1$. Therefore we have that
    \[|\partial^j E(y)| \lesssim_j \epsilon y^{-p - 1 - j}\]
    from which \eqref{eqn:linear_phi_bounds} follows. \par

    \textbf{Estimates for $\mc{U}$ and $\mc{V}$:} We now show that $\mc{U}$ and $\mc{V}$ are isomorphisms of $X_\gamma$ and $Y_\gamma$ respectively.  The proof is fairly straightforward. For $\mc{U}$, \eqref{eqn:linear_conjugating_maps} gives 
    \[\mc{U}r(y) = r(\phi^{-1}(y)).\]
    Set $\psi=\phi^{-1}$. Since $\phi$ is the identity on $y\le1$,
    so is $\psi$, and hence $\mc{U}r(y)=r(y)$ there. For $y>1$,
    the bounds in \eqref{eqn:linear_phi_bounds} give
    \begin{equation}\label{eqn:linear_inverse_phi_bounds}
        \ang{\psi(y)}\approx\ang{y},
        \qquad \psi'(y)=\frac{1}{\phi'(\psi(y))}\approx1,
        \qquad
        \psi''(y)=-\frac{\phi''(\psi(y))}{\phi'(\psi(y))^3},
    \end{equation}
    so $|\psi''(y)|\lesssim\epsilon\ang{y}^{-p-2}$.
    Therefore,
    \[
        (\mc{U}r)'=(r'\circ\psi)\psi',
        \qquad
        (\mc{U}r)''=(r''\circ\psi)(\psi')^2+(r'\circ\psi)\psi''.
    \]
    The definition of $X_\gamma$ therefore gives
    \[
        |\partial_y^j(\mc{U}r)(y)|
        \lesssim\|r\|_{X_\gamma}\ang{y}^{-\gamma+j/p},
        \qquad y>1,\quad 0\le j\le2.
    \]
    We therefore conclude that
    $\|\mc{U}r\|_{X_\gamma}\lesssim\|r\|_{X_\gamma}$.
    The same argument with $\phi$ in place of $\psi$ gives
    $\|\mc{T}_\phi r\|_{X_\gamma}\lesssim\|r\|_{X_\gamma}$.
    Since $\mc{U}^{-1}=\mc{T}_\phi$, it follows that $\mc{U}$ is an
    isomorphism of $X_\gamma$, with both operator norms bounded
    uniformly for sufficiently small $\epsilon$. \par
    The proof for $\mc{V}$ is similar. Recall from \eqref{eqn:linear_conjugating_maps} that 
    \[\mc{V}G(y) = (\phi')^{p + 1} G(\phi(y)).\]
    For $G\in Y_\gamma$, we again have $\mc{V}G(y)=G(y)$ on $y\le1$.
    For $y>1$, we compute
    \[
        (\mc{V}G)'
        =(p+1)(\phi')^p\phi''(G\circ\phi)
        +(\phi')^{p+2}(G'\circ\phi).
    \]
    Since $\ang{\phi(y)}\approx\ang{y}$, $\phi'\approx1$, and
    $|\phi''(y)|\lesssim\epsilon\ang{y}^{-p-2}$ by \eqref{eqn:linear_phi_bounds}, we obtain
    \[
        |\partial_y^j(\mc{V}G)(y)|
        \lesssim\|G\|_{Y_\gamma}\ang{y}^{-\gamma+j/p},
        \qquad y>1,\quad 0\le j\le1.
    \]
    Together with the identity on $y\le1$, this proves
    $\|\mc{V}G\|_{Y_\gamma}\lesssim\|G\|_{Y_\gamma}$.
    Moreover, with $\psi=\phi^{-1}$ as above,
    \[
        \mc{V}^{-1}G(y)
        =\frac{G(\psi(y))}{\phi'(\psi(y))^{p+1}}
        =\psi'(y)^{p+1}G(\psi(y)).
    \]
    The same argument, using \eqref{eqn:linear_inverse_phi_bounds} with $\psi$ in place of $\phi$, gives
    $\|\mc{V}^{-1}G\|_{Y_\gamma}\lesssim\|G\|_{Y_\gamma}$.
    Thus $\mc{V}$ is an isomorphism of $Y_\gamma$, with both operator
    norms bounded uniformly for sufficiently small $\epsilon$. \par
    The assertion about $(\mc{V}_\epsilon^{-1} - I)$ follows by writing 
    \[(\mc{V}_\epsilon^{-1} - I)G = [\psi'(y)^{p + 1} - 1]G(\psi(y)) + G(\psi(y)) - G(y).\]
    Write $C = \|\ang{y}^{\gamma}G\|_{L^\infty(1, \infty)} + \|\ang{y}^{\gamma}G'\|_{L^\infty(1, \infty)} + \|\ang{y}^{\gamma}G''\|_{L^\infty(1, \infty)}$.
    It follows that 
    \begin{align*}
        |(\mc{V}_\epsilon^{-1} - I)G| & \lesssim |\psi'(y) - 1||G(\psi(y))| + \sup_{t \in [0,1]}|G'(ty + (1-t)\psi(y))||\psi(y) - y| \\
        & \lesssim \epsilon C\ang{y}^{-\gamma}
    \end{align*}
    and 
    \begin{align*}
        |\partial_y(\mc{V}_\epsilon^{-1} - I)G| & \lesssim |\psi'(y)^{p} \psi''(y)G(\psi(y))| + |\psi'(y)^{p+2}-1||G'(\psi(y))| + |G'(\psi(y)) - G'(y)| \\
        & \lesssim \epsilon C\ang{y}^{-\gamma-p - 2} + \epsilon C\ang{y}^{-\gamma-p - 1} + \sup_{t \in [0,1]}|G''(ty + (1-t)\psi(y))||\psi(y) - y| \\
        & \lesssim \epsilon C \ang{y}^{-\gamma + \frac{1}{p}}
    \end{align*}
    as wanted. \par
    \textbf{Estimates for $\mc{R}_1$:} Recall from \eqref{eqn:linear_R1} that 
    \[\mc{R}_1r(y) = \left((\phi')^{p + 1} - 1\right)r(y).\]
    Using \eqref{eqn:linear_phi_bounds}, we have
    \[
        |(\phi')^{p+1}-1|
        +\ang{y}\bigl|\partial_y((\phi')^{p+1}-1)\bigr|
        \lesssim\epsilon\ang{y}^{-p-1}.
    \]
    The product rule and the definition of $X_\gamma$ therefore give
    \[
        |\partial_y^j(\mc{R}_1r)(y)|
        \lesssim\epsilon\|r\|_{X_\gamma}
        \ang{y}^{-\gamma-p-1+j/p},
        \qquad y>1,\quad 0\le j\le1.
    \]
    Since $\mc{R}_1r=0$ on $y\le1$, it follows that
    \begin{equation}\label{eqn:linear_R1_bound}
        \|\mc{R}_1r\|_{Y_\gamma}
        \lesssim\epsilon\|r\|_{X_\gamma}.
    \end{equation}
\textbf{Estimates for $\mc{R}_2$:} We now show that $\mc{R}_2$ is a bounded operator from $X_\gamma$ to $Y_\gamma$ with norm $\lesssim \epsilon$. Recall from \eqref{eqn:linear_R2} that
\[\mc{R}_2r(y) = c_\alpha\int_{\R}
        \left[1-\left(\frac{\phi'(y)}{S_\phi(y,z)}\right)^{p+1}\right]
        \frac{r'(z)}{|y-z|^{1+p}}\,dz = \int K_\phi(y, z)r'(z) \, dz\]
where
\[K_\phi(y, z) = c_\alpha\left[1-\left(\frac{\phi'(y)}{S_\phi(y,z)}\right)^{p+1}\right]\frac{1}{|y-z|^{1+p}}.\]
To compute the derivative of $\mc{R}_2$, we change variables $u = y - z$, then differentiate, and then change variables back to $z$ to obtain
\[\partial_y(\mc{R}_2r)(y) = \int (\partial_y + \partial_z) K_\phi(y, z) r'(z) \, dz + \int K_\phi(y, z) r''(z) \, dz.\]
To bound $\mc{R}_2$ and its derivative, we will simply bound the kernel $K_\phi$ and its derivatives.
We compute 
\begin{align*}
    (\partial_y + \partial_z)K_\phi(y, z) & = -c_\alpha \frac{1}{|y - z|^{p + 1}}(\partial_y + \partial_z)\left(\frac{\phi'(y)}{S_\phi(y,z)}\right)^{p+1} \\
    & = -c_\alpha \frac{p + 1}{|y - z|^{p + 1}}\left(\frac{\phi'(y)}{S_\phi(y,z)}\right)^{p+1}\left(\frac{\phi''(y)}{\phi'(y)} - \frac{S_{\phi'}(y, z)}{S_\phi(y,z)}\right).
\end{align*}
In the near region when $|y - z| \leq \frac{1}{4}\ang{y}$ (and in particular, $\ang{y} \approx \ang{z}$), we have that
\[|K_\phi(y, z)| \lesssim |\phi'(y) - S_\phi(y, z)||y - z|^{-p-1} \lesssim \epsilon \ang{y}^{-p - 2}|y - z|^{-p}\]
and that
\begin{align*}
    |(\partial_y + \partial_z)K_\phi(y, z)| & = c_\alpha \frac{1}{|y - z|^{p + 1}}\left|(\partial_y + \partial_z)\left(\frac{\phi'(y)}{S_\phi(y,z)}\right)^{p+1}\right| \\
    & \lesssim |y - z|^{-p-1}(|\phi''(y)||S_\phi(y,z) - \phi'(y)| + |\phi'(y)||S_{\phi'}(y,z) - \phi''(y)|) \\
    & \lesssim \epsilon \ang{y}^{-p - 3}|y - z|^{-p}
\end{align*}
Since $\ang{z}\approx\ang{y}$ in this region, the definition of $X_\gamma$ gives
\begin{align*}
    \left|\int_{|y - z| \leq \frac{1}{4}\ang{y}} K_\phi(y, z) r'(z) \, dz\right| & \lesssim \epsilon \ang{y}^{-p-2}\int_{|y - z| \leq \frac{1}{4}\ang{y}} |y - z|^{-p}|r'(z)| \, dz \\
    & \lesssim \epsilon \|r\|_{X_\gamma}\ang{y}^{-p-2}\int_{|y - z| \leq \frac{1}{4}\ang{y}} |y - z|^{-p} \, dz \begin{cases}
        \ang{y}^{-\gamma+1/p}, & y>0,\\
        \ang{y}^{-\gamma-1}, & y<0,
    \end{cases}\\
    & \lesssim \epsilon \|r\|_{X_\gamma}\begin{cases}
        \ang{y}^{-\gamma+1/p-2p-1}, & y>0,\\
        \ang{y}^{-\gamma-2p-2}, & y<0,
    \end{cases}\\
    & \lesssim \epsilon \|r\|_{X_\gamma}\ang{y}^{-\gamma}.
\end{align*}
Here we used $\int_{|y - z| \leq \frac{1}{4}\ang{y}}|y-z|^{-p}\,dz\lesssim\ang{y}^{1-p}$, since $p<1$, and $1/p-2p-1<0$. For the near contribution to the derivative it will be more convenient to introduce a bump function to do the splitting. Let $\chi$ be a bump function such that $\chi(z) = 1$ for $|z| \leq \frac{1}{8}$ and $\chi(z) = 0$ for $|z| \geq \frac{1}{4}$. Then we similarly have that
\begin{align*}
    &\left|\int_{|y - z| \leq \frac{1}{4}\ang{y}} \chi\left(\frac{y - z}{\ang{y}}\right) \bigl[(\partial_y+\partial_z)K_\phi(y,z)r'(z)+K_\phi(y,z)r''(z)\bigr]\,dz\right|\\
    &\qquad\lesssim\epsilon\ang{y}^{-p-2}\int_{|y - z| \leq \frac{1}{4}\ang{y}}|y-z|^{-p}\bigl(\ang{y}^{-1}|r'(z)|+|r''(z)|\bigr)\,dz\\
    &\qquad\lesssim\epsilon\|r\|_{X_\gamma}\ang{y}^{-p-2}\int_{|y - z| \leq \frac{1}{4}\ang{y}}|y-z|^{-p}\,dz\begin{cases}
        \ang{y}^{-\gamma+2/p}, & y>0,\\
        \ang{y}^{-\gamma-2}, & y<0,
    \end{cases}\\
    &\qquad\lesssim\epsilon\|r\|_{X_\gamma}\begin{cases}
        \ang{y}^{-\gamma+2/p-2p-1}, & y>0,\\
        \ang{y}^{-\gamma-2p-3}, & y<0,
    \end{cases}\\
    &\qquad\lesssim\epsilon\|r\|_{X_\gamma}\begin{cases}
        \ang{y}^{-\gamma+1/p}, & y>0,\\
        \ang{y}^{-\gamma-1}, & y<0.
    \end{cases}
\end{align*}
In the far region when $|y - z| > \frac{1}{8}\ang{y}$, we have that
\begin{align*}
    |\phi'(y) - S_\phi(y, z)| &= \left|\phi'(y) - 1 - \frac{(\phi(y) - y) - (\phi(z) - z)}{y - z}\right| \\
    & \lesssim \epsilon\left(\ang{y}^{-p-1} + \frac{\ang{y}^{-p} + \ang{z}^{-p}}{|y - z|}\right) \\
    & \lesssim \epsilon\left(\ang{y}^{-p-1} + \frac{\ang{z}^{-p}}{|y - z|}\right)
\end{align*}
and therefore
\[|K_\phi(y, z)| \lesssim \epsilon\left(\frac{\ang{y}^{-p-1}}{|y - z|^{p+1}} + \frac{\ang{z}^{-p}}{|y - z|^{p+2}}\right).\]
It follows that we can bound the far contribution to $\mc{R}_2$ as
\begin{align*}
    \left|\int_{|y - z| > \frac{1}{4}\ang{y}} K_\phi(y, z) r'(z) \, dz\right| & \lesssim \epsilon\int_{|y - z| > \frac{1}{4}\ang{y}}\left(\frac{\ang{y}^{-p-1}}{|y - z|^{p+1}} + \frac{\ang{z}^{-p}}{|y - z|^{p+2}}\right)|r'(z)| \, dz\\
    & \lesssim \epsilon\|r\|_{X_\gamma}\int_{|y - z| > \frac{1}{4}\ang{y}}\left(\frac{\ang{y}^{-p-1}\ang{z}^{-\gamma+1/p}}{|y - z|^{p+1}} + \frac{\ang{z}^{-\gamma-p+1/p}}{|y - z|^{p+2}}\right) \, dz\\
    & \lesssim \epsilon\|r\|_{X_\gamma}\ang{y}^{-p-2} \\
    & \lesssim \epsilon\|r\|_{X_\gamma}\ang{y}^{-\gamma}
\end{align*}
where we used $\gamma-1/p>1$, so both integrals in $z$ are finite. In particular, $\|r'\|_{L^1}\lesssim\|r\|_{X_\gamma}$. For the far contribution to the derivative, first note that
\[
    \partial_yS_\phi(y,z)=\frac{\phi'(y)-S_\phi(y,z)}{y-z}.
\]
Thus, when $|y-z|\geq\ang{y}/8$, differentiating the kernel at fixed $z$ gives
\begin{equation}\label{eqn:linear_R2_far_kernel_derivative}
\begin{aligned}
    |\partial_yK_\phi(y,z)|
    &\lesssim\frac{|\phi''(y)|}{|y-z|^{p+1}}
    +\frac{|\phi'(y)-S_\phi(y,z)|}{|y-z|^{p+2}}\\
    &\lesssim\epsilon\left(
    \frac{\ang{y}^{-p-2}}{|y-z|^{p+1}}
    +\frac{\ang{z}^{-p}}{|y-z|^{p+3}}\right).
\end{aligned}
\end{equation}
Using $1-\chi$ for the far contribution and integrating by parts in $z$, we obtain
\begin{equation}\label{eqn:linear_R2_far_derivative_ibp}
\begin{aligned}
    &\int_{\R}\left[1-\chi\left(\frac{y-z}{\ang{y}}\right)\right]
    \bigl[(\partial_y+\partial_z)K_\phi(y,z)r'(z)+K_\phi(y,z)r''(z)\bigr]\,dz\\
    &\quad=\int_{\R}\left[1-\chi\left(\frac{y-z}{\ang{y}}\right)\right]
    \partial_yK_\phi(y,z)r'(z)\,dz\\
    &\qquad\quad-\frac{1}{\ang{y}}\int_{\R}
    \chi'\left(\frac{y-z}{\ang{y}}\right)K_\phi(y,z)r'(z)\,dz.
\end{aligned}
\end{equation}
Here the boundary terms vanish, and $\partial_z[1-\chi((y-z)/\ang{y})]=\chi'((y-z)/\ang{y})/\ang{y}$. Since $1-\chi((y-z)/\ang{y})$ vanishes when $|y-z|\leq\ang{y}/8$, \eqref{eqn:linear_R2_far_kernel_derivative} gives
\begin{align*}
    \left|\int_{\R}\left[1-\chi\left(\frac{y-z}{\ang{y}}\right)\right]\partial_yK_\phi(y,z)r'(z)\,dz\right| & \lesssim \epsilon\left(\ang{y}^{-2p-3}+\ang{y}^{-p-3}\right)\int_{\R}|r'(z)|\,dz\\
    & \lesssim \epsilon\|r\|_{X_\gamma}\ang{y}^{-p-3}.
\end{align*}
For the term containing $\chi'$, its support lies in $\ang{y}/8\leq|y-z|\leq\ang{y}/4$. The near kernel bound therefore gives
\begin{align*}
    \left|\frac{1}{\ang{y}}\int_{\R}\chi'\left(\frac{y-z}{\ang{y}}\right)K_\phi(y,z)r'(z)\,dz\right| & \lesssim \epsilon\ang{y}^{-p-3}\int_{\ang{y}/8\leq|y-z|\leq\ang{y}/4}|y-z|^{-p}|r'(z)|\,dz\\
    & \lesssim \epsilon\ang{y}^{-2p-3}\int_{\R}|r'(z)|\,dz\\
    & \lesssim \epsilon\|r\|_{X_\gamma}\ang{y}^{-2p-3}.
\end{align*}
Combining these estimates in \eqref{eqn:linear_R2_far_derivative_ibp} and using $\gamma<2p+1<p+2$, we obtain
\begin{align*}
    \left|\int_{\R}\left[1-\chi\left(\frac{y-z}{\ang{y}}\right)\right]\bigl[(\partial_y+\partial_z)K_\phi(y,z)r'(z)+K_\phi(y,z)r''(z)\bigr]\,dz\right| & \lesssim \epsilon\|r\|_{X_\gamma}\ang{y}^{-p-3}\\
    & \lesssim \epsilon\|r\|_{X_\gamma}\begin{cases}
        \ang{y}^{-\gamma+1/p}, & y>0,\\
        \ang{y}^{-\gamma-1}, & y<0.
    \end{cases}
\end{align*}
Together with the near estimates, this proves $\|\mc{R}_2r\|_{Y_\gamma}\lesssim\epsilon\|r\|_{X_\gamma}$.
\end{proof}

\subsection{Inverse of the Free Operator}

Now we focus on constructing the inverse of the free operator $\mc{L}_0$ given by 
\[\mc{L}_0(r) = |D|^pr' + r - yr'.\]
By taking the Fourier transform, we can explicitly define the formal inverse of this operator by the following formula:
\begin{equation}\label{eqn:linear_free_inverse}
\begin{aligned}
    \widehat{\mc{L}_0^{-1}G}(\xi)
    &=\frac{e^{-i\Phi_0(\xi)}}{\xi^2}
    \int_0^\xi \eta e^{i\Phi_0(\eta)}\widehat G(\eta)\,d\eta\\
    &=\int_0^1 t e^{-i(1-t^{p+1})\Phi_0(\xi)}
    \widehat G(t\xi)\,dt,
\end{aligned}
\end{equation}
where $\Phi_0(\xi)= \frac{1}{\alpha} \xi|\xi|^p$, as above. Recall that the linear functional $\ell$ is defined by
\[\ell(G) = \int_{0}^{\infty}\eta e^{i\Phi_0(\eta)}\widehat{G}(\eta)\,d\eta.\]
Observe that if $\ell(G) = 0$ then we can write for $\xi > 0$, 
\begin{equation}\label{eqn:linear_projected_tail}\widehat{\mc{L}_0^{-1}G}(\xi) = -\frac{e^{-i\Phi_0(\xi)}}{\xi^2} \int_\xi^\infty \eta e^{i\Phi_0(\eta)}\widehat G(\eta)\,d\eta\end{equation}
and we have a similar formula for $\xi < 0$. This is the observation that allows us to define the inverse of $\mc{L}_0$ on the subspace of $Y_\gamma$ where $\ell(G) = 0$.

\begin{definition} Define the subspace of $Y_\gamma$ by 
    \[\sq{Y}_\gamma = \{G \in Y_\gamma : \ell(G) = 0\}.\]
    For $G \in \sq{Y}_\gamma$, we define the linear operator $\mc{T}_0$ by
    \[\widehat{\mc{T}_0 G}(\xi) = \frac{e^{-i\Phi_0(\xi)}}{\xi^2}
    \int_0^\xi \eta e^{i\Phi_0(\eta)}\widehat G(\eta)\,d\eta.\]
\end{definition}

Before showing that the inverse is bounded, we will show that $\ell$ is actually a well-defined bounded linear functional on $Y_\gamma$.

\begin{proposition}\label{prop:linear_ell_continuity}
    Let $\ell$ be the functional defined by 
    \[\ell(G) = \int_{0}^{\infty}\xi e^{i\Phi_0(\xi)}\widehat{G}(\xi)\,d\xi.\]
    Then for $\gamma > \beta + \frac{1}{p}$, $\ell$ is a well-defined bounded linear functional on $Y_\gamma$.
\end{proposition}

\begin{proof}
    Using the fact that $G,G' \in L^1$, we can write $\xi \widehat{G}(\xi) = -i\widehat{G'}(\xi)$. Therefore
    \[\ell(G) = \int G'(y)K(y) \, dy, \quad K(y) = -i\int_{0}^{\infty} e^{i(\Phi_0(\xi) - y\xi)}\,d\xi.\]
    By \cref{lem:appendix_stationary}, we have that
    \[|K(y)| \lesssim \begin{cases}
        1, & y \leq 0 \\
        \ang{y}^{\frac{1 - p}{2p}}, & y \geq 0.
    \end{cases}\]
    Since we also have that
    \[|G'(y)| \lesssim \|G\|_{Y_\gamma}\begin{cases}
        \ang{y}^{-\gamma - 1}, & y \leq 0 \\
        \ang{y}^{- \gamma + \frac{1}{p}}, & y \geq 0
    \end{cases}\]
    we can conclude that $\ell$ is bounded. Here we used that $\gamma > \beta + \frac{1}{p}$ implies $G'K \in L^1$.
\end{proof}

To bound the free inverse, we decompose it into a low-frequency and high-frequency part. Choose some even smooth bump function $\chi$ supported on $[-2,2]$ and equal to one on $[-1,1]$. We write 
\begin{equation}\label{eqn:linear_frequency_split}
\begin{aligned}
    \widehat{\mc{T}_{\textup{l}}G}(\xi) &= \chi(\xi) \widehat{\mc{T}_{0}G}(\xi),\\
    \widehat{\mc{T}_{\textup{h}}G}(\xi) &= (1 - \chi(\xi)) \widehat{\mc{T}_{0}G}(\xi).
\end{aligned}
\end{equation}

\subsection{Low-Frequency Estimate}\label{sec:linear_low_frequency}

The only real complication in bounding the low-frequency part is the fact that the phase $\Phi_0$ is not smooth at the origin. 

\begin{proposition}\label{prop:linear_low_frequency}(Low-frequency estimate)
    The operator $\mc{T}_{\textup{l}}$ is bounded from $Y_\gamma$ to $X_\gamma$ for any $1 < \gamma < p + 2$.
\end{proposition}

\begin{remark}
    For the low-frequency part, we don't require the condition $\ell(G) = 0$. The proposition also holds for a larger range of $\gamma$ than we require.
\end{remark}

\begin{proof}
    We write 
    \[\widehat{\mc{T}_0 G}(\xi) = \int_0^1 t e^{-i(1-t^{p+1})\Phi_0(\xi)}
    \widehat G(t\xi)\,dt\]
    and we expand the Taylor series
    \[e^{-i(1 - t^{p + 1})\Phi_0(\xi)} = \sum_{k = 0}^{\infty} \frac{(-i)^k}{k!} (1 - t^{p + 1})^k \Phi_0(\xi)^k.\]
    If we define the averaging operator 
    \begin{equation}\label{eqn:linear_averaging_operator}\mc{A}_kG(y) = \int_{0}^{1}(1 - t^\alpha)^kG(y / t) \, dt\end{equation}
    and let 
    \begin{equation}\label{eqn:linear_low_frequency_multipliers}m_k(\xi) = \chi(\xi)\Phi_0(\xi)^k, \quad \widehat{\mc{M}_k G}(\xi) = m_k(\xi)\widehat{G}(\xi),\end{equation} 
    then we can write 
    \begin{equation}\label{eqn:linear_low_frequency_series}\mc{T}_{\textup{l}}(G) = \sum_{k = 0}^{\infty} \frac{(-i)^k}{k!} (\mc{M}_k \circ \mc{A}_k)(G).\end{equation}
    Define the intermediate space
    \[Z_\gamma := \{G : \|G\|_{Z_\gamma} < \infty\}\]
    where
    \[\|G\|_{Z_\gamma} = \sup_{y \in \R}\ang{y}^{\gamma}|G(y)| + \sup_{y \leq -1}\ang{y}^{\gamma}(\ang{y}|G'(y)| + \ang{y}^{2}|G''(y)|).\]
    We claim that
    \begin{equation}\label{eqn:linear_low_frequency_factor_bounds}\|\mc{A}_k\|_{Y_\gamma \to Z_\gamma} \lesssim 1 + k, \quad \text{and} \quad \|\mc{M}_k\|_{Z_\gamma \to X_\gamma} \lesssim C^k.\end{equation}
    Here, $C$ is a universal constant. It would follow immediately that $\mc{T}_{\textup{l}}$ is bounded from $Y_\gamma$ to $X_\gamma$. \par
    We start by estimating the norm of $\mc{A}_k$. It follows immediately from \eqref{eqn:linear_averaging_operator} that
    \[|\mc{A}_kG(y)| \lesssim \|G\|_{Y_\gamma} \ang{y}^{-\gamma}.\]
    For $y \leq - 1$, we need to take derivatives. We compute
    \begin{align*}
        (\mc{A}_kG)'(y) & = \int_{0}^{1}(1 - t^\alpha)^k t^{-1}G'(y / t) \, dt \\
        & = \frac{1}{y}\int_{0}^{1}\left[(1-t^\alpha)^k-\alpha k t^\alpha(1-t^\alpha)^{k-1}\right]G(y/t)\,dt - \begin{cases}
            \frac{G(y)}{y}, & k = 0,\\
            0, & k \geq 1,
        \end{cases}
    \end{align*}
    and so 
    \begin{align*}
        (\mc{A}_kG)''(y) & = \begin{cases}
            -\dfrac{G'(y)}{y}, & k = 0,\\[6pt]
            -\dfrac{\alpha k}{y}\displaystyle\int_0^1 t^{\alpha-1}(1-t^\alpha)^{k-1}G'(y/t)\,dt, & k \geq 1.
        \end{cases}
    \end{align*}
    From these formulas, we can immediately conclude that
    \[|(\mc{A}_kG)'(y)| \lesssim (1 + k) \|G\|_{Y_\gamma} \ang{y}^{-\gamma - 1}, \quad |(\mc{A}_kG)''(y)| \lesssim (1 + k) \|G\|_{Y_\gamma} \ang{y}^{-\gamma - 2}\]
    so that $\|\mc{A}_k\|_{Y_\gamma \to Z_\gamma} \lesssim 1 + k$. \par
    Now we estimate the norm of $\mc{M}_k$, defined in \eqref{eqn:linear_low_frequency_multipliers}.  By \cref{prop:appendix_nonsmooth_fourier_inverse}, we have that 
    \begin{equation}\label{eqn:linear_Hk_kernel}
        |\partial^{j}\widecheck{m_k}(y)| \lesssim C^k \ang{y}^{-p - 2 - j}, \quad j = 0, 1, 2.
    \end{equation}
    Therefore, for every $y \in \R$, we can estimate
    \begin{align*}
        |\partial^j\mc{M}_kG(y)| & \lesssim C^k \|G\|_{Z_\gamma} \int \ang{y - z}^{-p - 2 - j}\ang{z}^{-\gamma}\,dz \\
        & \lesssim C^k \|G\|_{Z_\gamma} \ang{y}^{-\gamma}.
    \end{align*}
    This gives the required estimates when $y \geq -2$ and the zeroth-order estimate for all $y$. When $y \leq -2$, we need a decay gain in the derivatives. For $j = 1, 2$, we split the convolution
    \[\int \partial^j \widecheck{m_k}(y-z)G(z)\,dz\]
   using the cutoff $\chi(8(y-z)/\ang{y})$. On its support, $|y-z| \leq \ang{y}/4$, so $z \leq -1$ and $\ang{z} \approx \ang{y}$. Integrating by parts $j$ times and using \eqref{eqn:linear_Hk_kernel}, we obtain
    \begin{align*}
        \left|\int \chi\left(\frac{8(y-z)}{\ang{y}}\right)\partial^j\widecheck{m_k}(y-z)G(z)\,dz\right| & = \left|\int \widecheck{m_k}(y-z)\partial_z^j\left[\chi\left(\frac{8(y-z)}{\ang{y}}\right)G(z)\right]\,dz\right| \\
        & \lesssim \|G\|_{Z_\gamma}\ang{y}^{-\gamma-j}\int |\widecheck{m_k}(y-z)|\,dz \\
        & \lesssim C^k\|G\|_{Z_\gamma}\ang{y}^{-\gamma-j}.
    \end{align*}
    Here each derivative falling on the cutoff contributes a factor $\ang{y}^{-1}$, while the derivatives falling on $G$ are controlled by its $Z_\gamma$ norm. For the far contribution, $|y-z| \geq \ang{y}/8$, and \eqref{eqn:linear_Hk_kernel} gives
    \begin{align*}
        \left|\int \left(1-\chi\left(\frac{8(y-z)}{\ang{y}}\right)\right)\partial^j\widecheck{m_k}(y-z)G(z)\,dz\right| & \lesssim C^k\|G\|_{Z_\gamma}\ang{y}^{-p-2-j}\int \ang{z}^{-\gamma}\,dz \\
        & \lesssim C^k\|G\|_{Z_\gamma}\ang{y}^{-\gamma-j},
    \end{align*}
    since $1 < \gamma < p+2$. Combining the estimates, we conclude that
    \begin{equation}\label{eqn:linear_Hk_bound}
        \|\mc{M}_kG\|_{X_\gamma} \lesssim C^k\|G\|_{Z_\gamma}.
    \end{equation}
    This completes \eqref{eqn:linear_low_frequency_factor_bounds}; substituting these bounds into \eqref{eqn:linear_low_frequency_series} proves the proposition.
\end{proof}

\subsection{High-Frequency Estimate}\label{sec:linear_high_frequency}

In this section, we control the high-frequency part of the inverse defined in \eqref{eqn:linear_frequency_split}: $\mc{T}_{\textup{h}}(G) = \mc{T}_0(G) - \mc{T}_{\textup{l}}(G)$.

\begin{definition}
    We define the positive high-frequency part of $\mc{T}_{\textup{h}}$ as:
    \[\widehat{\mc{T}_{+}G}(\xi) := \mathbbm{1}_{\{\xi > 0\}}(\xi)(1 - \chi(\xi)) \widehat{\mc{T}_{0}G}(\xi)\]
\end{definition}

\begin{remark}
    We remark that if $\ell(G) = 0$, for $\xi > 0$, \eqref{eqn:linear_projected_tail} gives 
    \begin{align*}
        \widehat{\mc{T}_{+}G}(\xi) & = (1 - \chi(\xi))\frac{e^{-i\Phi_0(\xi)}}{\xi^2}
        \int_0^\xi \eta e^{i\Phi_0(\eta)}\widehat G(\eta)\,d\eta \\
        & = -(1 - \chi(\xi))\frac{e^{-i\Phi_0(\xi)}}{\xi^2}
        \int_\xi^{\infty} \eta e^{i\Phi_0(\eta)}\widehat G(\eta)\,d\eta.
    \end{align*}
\end{remark}
This remark motivates the next definition of the kernel of the operator $\mc{T}_+$.

\begin{definition}
    For $0 \leq j \leq 2$, we define the kernel $K_{+,j}$ by
    \begin{equation}\label{eqn:linear_positive_kernel}K_{+,j}(y, z) = -\frac{i^{j+1}}{2\pi}\int_{1 < \xi < \eta}(\chi(\xi) - 1)\xi^{j - 2} e^{i(\Phi_0(\eta) - \Phi_0(\xi) - z\eta + y\xi)} \, d\xi d\eta.\end{equation}
\end{definition}

It follows from \eqref{eqn:linear_positive_kernel} that we can write 
\[\mc{T}_{+}G(y) = -\int \partial_zK_{+,0}(y, z) G(z)\,dz = \int K_{+,0}(y, z) G'(z)\,dz,\]
and, for $j=0,1,2$,
\begin{equation}\label{eqn:linear_positive_kernel_representation}\partial_y^j\mc{T}_{+}G(y) = \int K_{+,j}(y,z)G'(z)\,dz.\end{equation}

Our bounds on the kernel will be based on stationary phase estimates. To do this, observe that the phase in $\xi$ is given by $y\xi - \Phi_0(\xi)$ which is stationary when $y > 0$ and $\xi = y^{1/p}$. We split the kernel into its stationary and nonstationary components. 

\begin{definition}
    We define the decomposition of the kernel of the operator $\mc{T}_+$ into stationary and nonstationary parts as follows. 
    \begin{enumerate}
        \item Define the smooth cutoff function that localizes to the stationary region
        \begin{equation}\label{eqn:linear_stationary_cutoffs}\chi_{\textup{s}}(y,\xi) = \begin{cases}
            \chi\left(\dfrac{100(\xi - y^{1/p})}{y^{1/p}}\right), & y>0,\\
            0, & y\leq0,
        \end{cases}
        \qquad \chi_{\textup{ns}}(y,\xi)=1-\chi_{\textup{s}}(y,\xi).\end{equation}
        \item Define the symbols
        \[a_j(y, \xi) = -\frac{i^{j+1}}{2\pi}(\chi(\xi) - 1)\xi^{j-2}\]
        and
        \[a_j^{\textup{s}}(y, \xi) = \chi_{\textup{s}}(y,\xi)a_j(y,\xi), \quad a_j^{\textup{ns}}(y, \xi) = \chi_{\textup{ns}}(y,\xi)a_j(y,\xi)\]
        \item Define the stationary and nonstationary parts of the kernel by
        \begin{equation}\label{eqn:linear_stationary_kernel_split}K_{+,j} = K_{+,j}^{\textup{s}} + K_{+,j}^{\textup{ns}},\end{equation}
        where
        \[K_{+,j}^{\textup{s}}(y, z) = \int_{1 < \xi < \eta} a_j^{\textup{s}}(y, \xi) e^{i(\Phi_0(\eta) - \Phi_0(\xi) - z\eta + y\xi)} \, d\xi d\eta\]
        and similarly for $K_{+,j}^{\textup{ns}}$. We define the operators $\mc{T}_{+,j}^{\textup{s}}$ and $\mc{T}_{+,j}^{\textup{ns}}$ by
        \begin{equation}\label{eqn:linear_stationary_operators}\mc{T}_{+,j}^{\textup{s}}G(y)=\int K_{+,j}^{\textup{s}}(y,z)G'(z)\,dz,\qquad
        \mc{T}_{+,j}^{\textup{ns}}G(y)=\int K_{+,j}^{\textup{ns}}(y,z)G'(z)\,dz\end{equation}
        so that $\partial^j \mc{T}_+ = \mc{T}_{+,j}^{\textup{s}} + \mc{T}_{+,j}^{\textup{ns}}$.
    \end{enumerate}
\end{definition}

The following proposition gives the estimate for the kernel $K_{+,j}^{\textup{s}}$.

\begin{proposition}\label{prop:linear_stationary}(Stationary estimate)
    For a small enough constant $c > 0$, the kernel $K_{+,j}^{\textup{s}}$ satisfies the bound 
    \[|K_{+,j}^{\textup{s}}(y, z)| \lesssim \begin{cases}
        \ang{y}^{\frac{j}{p} - \frac{p + 3}{2p}} \ang{z}^{\frac{1 - p}{2p}}, & z \geq cy \\
        C_M\ang{y}^{-M}(\ang{y} + |z|)^{-M}, & z < cy
    \end{cases}.\]
\end{proposition}

\begin{proof}
    By \eqref{eqn:linear_stationary_cutoffs}, on the support of $\chi_{\textup{s}}(y,\xi)$, we need $|\xi - y^{1/p}| \lesssim y^{1/p}$ and thus $y \gtrsim 1$ since $\xi > 1$. We first consider when $z \geq cy$. This in particular implies that $z \gtrsim 1$ and this includes the case when $y = z$. We write
    \[K_{+,j}^{\textup{s}}(y,z) = \int a_j^{\textup{s}}(y, \xi) e^{i(y\xi - \Phi_0(\xi))} \int_{\xi}^{\infty}e^{i(\Phi_0(\eta) - z\eta)} \, d\eta  \, d\xi.\]
    We apply \cref{lem:appendix_triangle} to this integral with
    \[I = (y^{1/p} - c'y^{1/p}, y^{1/p} + c'y^{1/p}), \quad a(\xi) = a_j^{\textup{s}}(y,\xi), \quad \Phi(\xi) = y\xi - \Phi_0(\xi), \quad \Psi(\eta) = \Phi_0(\eta) - z\eta.\]
    To check the hypotheses of \cref{lem:appendix_triangle}, we have 
    \[\Phi'(y^{1/p}) = 0, \quad |\Phi''(\xi)| \approx \ang{y}^{(p-1)/p}, \quad  |\Phi'''(\xi)| \approx \ang{y}^{(p - 2)/p}.\]
    Also, for $\ell = 0, 1, 2$, we have that
    \[|\partial^{\ell}a(\xi)| \lesssim \ang{y}^{\frac{j - 2 - \ell}{p}}.\]
    Also, by \cref{lem:appendix_stationary}, we have that
    \[\sup_{\xi > 0}\left|\int_{\xi}^{\infty} e^{i(\Phi_0(\eta) - z\eta)} \, d\eta\right| \lesssim \begin{cases}
        1, & z \leq 0 \\
        \ang{z}^{\frac{1-p}{2p}}, & z > 0.
    \end{cases}\]
    Therefore by \cref{lem:appendix_triangle}, we have that
     \[|K_{+,j}^{\textup{s}}(y,z)| \lesssim \ang{y}^{\frac{j - 2}{p}} \left(\ang{y}^{\frac{1-p}{p}} + \ang{y}^{\frac{1-p}{2p}}\ang{z}^{\frac{1 - p}{2p}} \right) \lesssim \ang{y}^{\frac{j}{p} - \frac{p + 3}{2p}} \ang{z}^{\frac{1 - p}{2p}}. \]
    In the case that $z < cy$, the phase in $\eta$ is nonstationary. Indeed,
    \[\partial_\eta (\Phi_0(\eta) - z\eta) = \eta^{p} - z \gtrsim \eta^{p} + |z|\]
    where we used that $\eta \geq \xi \approx y^{1/p}$ and $y \gtrsim z$. Write
    \[E_z(\xi) = \int_{\xi}^\infty e^{i(\Phi_0(\eta) - z\eta)} \, d\eta.\]
    Let $D_{\Phi_0}$ denote the operator
    \[D_{\Phi_0}h(\xi) = \frac{\partial_\xi h(\xi)}{i(\xi^p - z)}.\]
    We have that 
    \[\begin{aligned}E_z(\xi) & = e^{i(\Phi_0(\xi) - z\xi)}\sum_{k = 0}^{M}(-1)^{k+1}D_{\Phi_0}^k \left(\frac{1}{i(\xi^p - z)}\right) \\ & \quad + (-1)^{M+1}\int_{\xi}^{\infty}e^{i(\Phi_0(\eta) - z\eta)}\partial_\eta\left[D_{\Phi_0}^{M}\left(\frac{1}{i(\eta^p - z)}\right)\right]\,d\eta.\end{aligned}\]
    Therefore we can write
    \begin{align*}
        K_{+,j}^{\textup{s}}(y, z) & = \sum_{k = 0}^{M}(-1)^{k + 1}\int a_j^{\textup{s}}(y, \xi)e^{i\xi(y - z)}D_{\Phi_0}^k\left(\frac{1}{i(\xi^p - z)}\right)\,d\xi \\ & \quad + (-1)^{M+1}\int_{1 < \xi < \eta}a_j^{\textup{s}}(y, \xi)e^{i(\Phi_0(\eta) - \Phi_0(\xi) - z\eta + y\xi)}\partial_\eta\left[D_{\Phi_0}^{M}\left(\frac{1}{i(\eta^p - z)}\right)\right]\,d\xi d\eta.
    \end{align*}
    Since $y - z  \gtrsim y + |z|$, the phase in the first integral is nonstationary so we can integrate by parts to obtain 
    \begin{equation}\label{eqn:linear_stationary_boundary_ibp}
        \begin{aligned}\left|\int a_j^{\textup{s}}(y, \xi)e^{i\xi(y-z)}D_{\Phi_0}^k\left(\frac{1}{i(\xi^p-z)}\right)\,d\xi\right| 
            & \leq \frac{1}{|y-z|^M}\int\left|\partial_\xi^M\left[a_j^{\textup{s}}(y, \xi)D_{\Phi_0}^k\left(\frac{1}{i(\xi^p-z)}\right)\right]\right|\,d\xi \\ 
            & \lesssim \frac{\ang{y}^{\frac{j-1-k-M}{p}}}{(\ang{y}+|z|)^{k+M+1}}.\end{aligned}
    \end{equation} 
    In particular, the sum over $k$ is bounded by $\ang{y}^{\frac{j - 1 - M}{p}}(\ang{y} + |z|)^{-M-1}$. The second term can be estimated directly:
    \begin{equation}\label{eqn:linear_stationary_remainder_bound}
        \begin{aligned}&\left|\int_{1 < \xi < \eta}a_j^{\textup{s}}(y, \xi)e^{i(\Phi_0(\eta) - \Phi_0(\xi) - z\eta + y\xi)}\partial_\eta\left[D_{\Phi_0}^{M}\left(\frac{1}{i(\eta^p-z)}\right)\right]\,d\xi d\eta\right| \\ &\quad \lesssim_M \int_1^\infty |a_j^{\textup{s}}(y, \xi)|\int_\xi^\infty \frac{\eta^{p-M-1}}{(\eta^p+|z|)^{M+2}}\,d\eta\,d\xi \\ &\quad \lesssim_M \int_1^\infty |a_j^{\textup{s}}(y, \xi)|\frac{\xi^{-M}}{(\xi^p+|z|)^{M+1}}\,d\xi \\ &\quad \lesssim_M \frac{\ang{y}^{\frac{j-1-M}{p}}}{(\ang{y}+|z|)^{M+1}}.
        \end{aligned}
    \end{equation} 
    In the second inequality, we used $\eta^{-M}\leq\xi^{-M}$ and $\int_\xi^\infty\eta^{p-1}(\eta^p+|z|)^{-M-2}\,d\eta=[p(M+1)]^{-1}(\xi^p+|z|)^{-M-1}$. Integrating by parts sufficiently many times and relabeling $M$, we obtain the bound 
    \[|K_{+,j}^{\textup{s}}(y, z)| \lesssim_M \ang{y}^{-M}(\ang{y} + |z|)^{-M}, \quad M \geq 0.\]
\end{proof}

To estimate the nonstationary part of the kernel, the plan is to integrate by parts in $\xi$ to gain decay in $y$. Since the integration domain has the boundary $\xi = \eta$, the integration by parts will gain decay in the remainder term, but will also produce boundary terms that reduce to Fourier multipliers. The following lemma makes this precise.

\begin{lemma}\label{lem:linear_nonstationary_symbol_expansion}
    For $k \geq 0$ inductively define the symbols 
    \begin{equation}\label{eqn:linear_nonstationary_symbols}m_{j,0}(y, \xi) = \frac{a_j^{\textup{ns}}(y, \xi)}{i(y - \xi^p)}, \quad m_{j,k+1}(y, \xi) = -\frac{\partial_\xi m_{j,k}(y, \xi)}{i(y - \xi^p)}.\end{equation}
    Then, for $M \geq 1$, we can write 
    \[ \int_{1}^{\eta}a_j^{\textup{ns}}(y,\xi)e^{i(y\xi - \Phi_0(\xi))} \, d\xi
         = e^{i(y\eta - \Phi_0(\eta))}\sum_{k = 0}^{M - 1}m_{j,k}(y, \eta)  - \int_{1}^{\eta}e^{i(y\xi - \Phi_0(\xi))}\partial_\xi m_{j,M-1}(y, \xi) \, d\xi.\]
    Consequently, the nonstationary contribution to $\partial_y^j\mc{T}_{+}G$ is
    \begin{equation}\label{eqn:linear_nonstationary_expansion}\mc{T}_{+,j}^{\textup{ns}}G(y) = 2\pi\sum_{k = 0}^{M-1} [\widecheck{m_{j,k}(y, \cdot)} * G'](y) + \mc{R}_{j,M}(G)(y)\end{equation}
    where
    \begin{equation}\label{eqn:linear_nonstationary_remainder_definition}\mc{R}_{j,M}(G)(y) = -\int_{\R} \left[\int_{1<\xi<\eta}
    \partial_\xi m_{j,M-1}(y,\xi)
    e^{i(\Phi_0(\eta)-\Phi_0(\xi)-z\eta+y\xi)}\,d\xi\,d\eta \right] G'(z) \,dz.\end{equation}
\end{lemma}

\begin{proof}
    The proof of the identity involving $a_j^{\textup{ns}}$ essentially follows from integration by parts and 
    \[\partial_\xi(y\xi - \Phi_0(\xi)) = y - \xi^p.\]
    Therefore
    \begin{align*}
        K_{+,j}^{\textup{ns}}(y, z)
        & = \sum_{k=0}^{M-1}\int_1^\infty m_{j,k}(y,\eta)e^{i\eta(y-z)}\,d\eta\\
        &\quad-\int_{1<\xi<\eta}\partial_\xi m_{j,M-1}(y,\xi)
        e^{i(\Phi_0(\eta)-\Phi_0(\xi)-z\eta+y\xi)}\,d\xi\,d\eta.
    \end{align*}
    The final identity then follows.
\end{proof}

\begin{lemma}[Symbol bounds]\label{lem:linear_nonstationary_leading_kernels}
    The symbols $m_{j,k}$ satisfy the following estimates for every $L \geq 0$.
    \begin{enumerate}
        \item For $k=0$, 
    \begin{equation}\label{eqn:linear_nonstationary_leading_kernel_j0}
        \left|\operatorname{Re}\partial_z\widecheck{m_{0,0}(y,\cdot)}(z)\right|
        \lesssim_L \ang{y}^{-1}\ang{z}^{-L}
    \end{equation}
    and
    \begin{equation}\label{eqn:linear_nonstationary_leading_kernel_j12}
    \begin{aligned}
        \left|\operatorname{Re}\partial_z\widecheck{m_{1,0}(y,\cdot)}(z)\right| + \left|\operatorname{Re}\widecheck{m_{2,0}(y,\cdot)}(z)\right| 
        &\lesssim_L
        \begin{cases}
            |z|^{p-1}, & 0<|z|\leq \ang{y}^{-1/p},\\
            \ang{y}^{-2}|z|^{-p-1}+\ang{y}^{-1}, & \ang{y}^{-1/p}<|z|\leq1,\\
            \ang{y}^{-1}|z|^{-L}, & |z|>1.
        \end{cases}
    \end{aligned}
    \end{equation}
    \item For $k \geq 1$, we have that 
    \begin{equation}\label{eqn:linear_nonstationary_higher_kernels}
        \sum_{j=0}^1\left|\operatorname{Re}\partial_z\widecheck{m_{j,k}(y,\cdot)}(z)\right|
        +\left|\operatorname{Re}\widecheck{m_{2,k}(y,\cdot)}(z)\right|
        \lesssim_{k,L}\ang{y}^{-k-1}\ang{z}^{-L}.
    \end{equation}
    \end{enumerate}
\end{lemma}

\begin{proof}
    We start by proving the result for $k = 0$. Recall from \eqref{eqn:linear_nonstationary_symbols} that
    \[
        m_{j,0}(y,\xi)
        =\frac{i^j}{2\pi}
        \frac{(1-\chi(\xi))(1-\chi_{\textup{s}}(y,\xi))
        \xi^{j-2}}{y-\xi^p}.
    \]
    We suppress the input $y$ in $m_{j,0}$ for convenience. For $j = 2$, we have that
    \[m_{2,0}(\xi) = -\frac{1}{2\pi}\frac{(1-\chi(\xi))(1-\chi_{\textup{s}}(y,\xi))}{y-\xi^p} = (1 - \chi(\xi))q(\xi)\]
    where 
    \[q(\xi) = -\frac{1}{2\pi}\frac{(1-\chi_{\textup{s}}(y,\xi))}{y-\xi^p}.\]
    Note that $q$ satisfies the bounds 
    \[
    |q(\xi)|\lesssim\frac1{|y|+\xi^p},\qquad
    |q'(\xi)|\lesssim\frac{\xi^{p-1}}{(|y|+\xi^p)^2},\qquad
    |q''(\xi)|\lesssim\frac{\xi^{p-2}}{(|y|+\xi^p)^2}.
    \] 
    We will estimate $\widecheck{m_{2,0}}(z)$ in three different ways. First write
    \[2\pi\text{Re} \, \widecheck{m_{2,0}}(z) = \int (1 - \chi(\xi)) q(\xi) \cos(z\xi) \, d\xi = \int q(\xi) \cos(z\xi) \, d\xi - \int \chi(\xi) q(\xi) \cos(z\xi) \, d\xi.\]
    The second term is bounded by $C \ang{y}^{-1}$ from the bounds for $q$:
    \[\left|\int \chi(\xi) q(\xi) \cos(z\xi) \, d\xi\right| \lesssim \int_{0}^{2}\frac{1}{|y| + \xi^p} \, d\xi \lesssim \ang{y}^{-1}.\]
    For the first term, we can estimate this in two different ways. First,
    \begin{align*}
        \left|\int_{0}^{\infty} q(\xi) \cos(z\xi) \, d\xi\right| & \leq \left|\int_{0}^{|z|^{-1}} q(\xi) \cos(z\xi) \, d\xi\right| + \left|\int_{|z|^{-1}}^{\infty} q(\xi) \cos(z\xi) \, d\xi\right| \\
        & \lesssim |z|^{p - 1} + |z|^{-1}\left(q(|z|^{-1}) + \int_{|z|^{-1}}^{\infty} |q'(\xi)| \, d\xi\right) \\
        & \lesssim |z|^{p-1} 
    \end{align*}
    Next, we can also first integrate by parts to obtain
    \begin{align*}
         \left|\int_{0}^{\infty} q(\xi) \cos(z\xi) \, d\xi\right| & =  |z|^{-1} \left|\int q'(\xi) \sin(z\xi) \, d\xi\right| \\
         & \lesssim \int_{0}^{|z|^{-1}}|\xi||q'(\xi)| \, d\xi + |z|^{-2}\left( |q'(|z|^{-1})| + \int_{|z|^{-1}}^{\infty}|q''(\xi)| \, d\xi \right) \\
         & \lesssim |y|^{-2}|z|^{-p-1} + |y|^{-2}|z|^{-2}\int_{|z|^{-1}}^{\infty}\xi^{p-2} \, d\xi \\
         & \lesssim |y|^{-2}|z|^{-p-1}.
    \end{align*}
    Therefore, for $0<|z|\leq1$, combining these bounds for $|y|\geq1$ and using the first bound for $|y|\leq1$, we obtain
    \[\left|\int_{0}^{\infty} q(\xi) \cos(z\xi) \, d\xi\right| \lesssim |z|^{p-1}\min\{1, \ang{y}^{-2}|z|^{-2p}\}.\]
    Finally, we can also estimate $\text{Re} \, \widecheck{m_{2,0}}(z)$ by integrating by parts $\max\{1, L\}$ times to get
    \[|\text{Re} \, \widecheck{m_{2,0}}(z)| \lesssim_L |z|^{-L} \int_{1}^\infty |\partial^L_\xi((1 - \chi(\xi))q(\xi))| \, d\xi \lesssim \ang{y}^{-1}|z|^{-L}.\]
    This concludes the bounds for $(j, k) = (2, 0)$. For the other bounds when $k = 0$, observe that
    \[\left|\operatorname{Re}\partial_z\widecheck{m_{1,0}(y,\cdot)}(z)\right| = \left|\operatorname{Re}\widecheck{m_{2,0}(y,\cdot)}(z)\right|\]
    and 
    \[\left|\operatorname{Re}\partial_z\widecheck{m_{0,0}(y,\cdot)}(z)\right| = \left|\operatorname{Re}\widecheck{m_{1,0}(y,\cdot)}(z)\right|
        \]
    by the definition \eqref{eqn:linear_nonstationary_symbols} of the symbols. The estimates for $m_{1,0}$ follow. For $m_{0,0}$, we write
    \[\left|\operatorname{Re}\partial_z\widecheck{m_{0,0}(y,\cdot)}(z)\right| = \left|\int_{0}^{z}\partial_z\operatorname{Re}\widecheck{m_{1,0}(y,\cdot)}(z') \, dz'\right|.\]
    The boundary term vanishes because $m_{1,0}$ is imaginary. We can therefore bound this by integrating the bounds for $m_{1,0}$ to obtain 
    \[\left|\operatorname{Re}\partial_z\widecheck{m_{0,0}(y,\cdot)}(z)\right| \lesssim \ang{y}^{-1}\]
    for $|z| \leq 1$. For $|z| > 1$, we simply write
    \[\partial_z \widecheck{m_{0,0}(y,\cdot)}(z) = -\frac{i}{2\pi}\int_{1}^{\infty}(1 - \chi(\xi))q(\xi)\xi^{-1} e^{iz\xi} \, d\xi,\]
    and integrate by parts $L$ times to bound this by $\ang{y}^{-1}|z|^{-L}$. This concludes the proof for $k = 0$. For $k \geq 1$, observe that we have the estimates
    \[
    |\partial_\xi^\ell m_{j,k}(y,\xi)|
    \lesssim_{k,\ell}
    \xi^{j-2-k-\ell}(|y|+\xi^p)^{-k-1},
    \qquad \xi\ge1
    \]
    for all $j = 0,1,2$ and $k,\ell \geq 0$. Indeed, this follows from the estimates
    \[|\partial_\xi^{\ell} \chi_{\textup{ns}}(y, \xi)| \lesssim \xi^{-\ell}, \quad  |\partial_\xi^{\ell} a_j(y, \xi)| \lesssim \xi^{j - 2 - \ell}, \quad \left|\partial_\xi^\ell \left(\frac{1}{y - \xi^p}\right)\right| \lesssim \xi^{-\ell}(|y| + \xi^p)^{-1}\]
    on the support of our symbols, and induction in $k$ using \eqref{eqn:linear_nonstationary_symbols}. It follows that
    \[
        \sum_{j=0}^{1}
        |\partial_\xi^\ell(i\xi m_{j,k})|
        +
        |\partial_\xi^\ell m_{2,k}|
        \lesssim_{k,\ell}
        \xi^{-k-\ell}(|y|+\xi^p)^{-k-1}.
    \]
    Therefore, since $|y| + \xi^p \gtrsim \ang{y}$, we have that, for $\ell+k \geq 2$,
    \[
        \begin{aligned}
        \sum_{j=0}^{1}
        \|\partial_\xi^\ell(i\xi m_{j,k})\|_{L^1_\xi}
        +\|\partial_\xi^\ell m_{2,k}\|_{L^1_\xi}
        &\lesssim_{k,\ell}
        \langle y\rangle^{-k-1}
        \int_1^\infty \xi^{-k-\ell}\,d\xi\\
        &\lesssim_{k,\ell}\langle y\rangle^{-k-1}.
        \end{aligned}
    \]
    Now assume $k = 1$. For $j = 0$, we have 
    \[\|i \xi m_{0,1}(y, \cdot)\|_{L^1_\xi} \lesssim \int_{1}^{\infty}\xi^{-2}(|y| + \xi^p)^{-2} \, d\xi \lesssim \ang{y}^{-2}\]
    and for $j = 2$, we have 
    \[m_{2,1}(y, \xi) = -\frac{i}{2\pi}\left[\frac{\partial_\xi[(1-\chi(\xi))\chi_{\textup{ns}}(y,\xi)]}{(y-\xi^p)^2}+\frac{p(1-\chi(\xi))\chi_{\textup{ns}}(y,\xi)\xi^{p-1}}{(y-\xi^p)^3}\right].\]
    Therefore we can bound 
    \begin{align*}
        \| m_{2,1}(y, \cdot)\|_{L^1_\xi} & \lesssim \ang{y}^{-2}\int_{1}^{\infty}|\partial_\xi[(1-\chi(\xi))\chi_{\textup{ns}}(y,\xi)]| \, d\xi \\
        & + \int_{1}^{\infty}\xi^{p-1}(|y| + \xi^p)^{-3} \, d\xi \\
        & \lesssim \ang{y}^{-2}.
    \end{align*}
    For $j = 1$, we write
    \[\xi m_{1,1}(y,\xi) = -i m_{2,1}(y,\xi)+\frac{(1-\chi(\xi))\chi_{\textup{ns}}(y,\xi)}{2\pi\xi(y-\xi^p)^2}.\]
    Now for $m = i\xi m_{0, 1}, i\xi m_{1,1}$, or $m_{2,1}$ we write
    \[\widecheck{m}(z) = \frac1{2\pi}\int_{1}^{\infty}m(\xi)e^{iz\xi} \, d\xi.\]
    For $|z| \leq 1$, we can bound this by the $L^1$ norm when $m = i\xi m_{0, 1}$ or $m_{2,1}$ but when $m = i\xi m_{1,1}$, we instead write 
    \[\left|\text{Re} \, \partial_z\widecheck{m_{1,1}(y, \cdot)}(z)\right| = \frac1{2\pi}\left|\int_{1}^{\infty} \xi m_{1,1}(y,\xi)\sin(z\xi) \, d\xi\right|.\]
    We can bound the $m_{2,1}$ term as above. For the second term, the integral we get is
    \begin{align*}
        \left|\int_1^\infty\frac{(1-\chi(\xi))\chi_{\textup{ns}}(y,\xi)}{2\pi\xi(y-\xi^p)^2}\sin(z\xi)\,d\xi\right| & =\frac1{2\pi}\left|\int_1^\infty
        \partial_\xi\left[\frac{(1-\chi(\xi))\chi_{\textup{ns}}(y,\xi)}{(y-\xi^p)^2}\right]
        \left(\int_1^\xi\frac{\sin(z\eta)}{\eta}\,d\eta\right)\,d\xi\right|\\
        &\quad\lesssim\int_1^\infty\left|\partial_\xi\left[\frac{(1-\chi(\xi))\chi_{\textup{ns}}(y,\xi)}{(y-\xi^p)^2}\right]\right|\,d\xi\\
        &\quad\lesssim\ang{y}^{-2}\int_1^\infty\left|\partial_\xi[(1-\chi(\xi))\chi_{\textup{ns}}(y,\xi)]\right|\,d\xi
        +\int_1^\infty\frac{\xi^{p-1}}{(|y|+\xi^p)^3}\,d\xi\\
        &\quad\lesssim\ang{y}^{-2}.
    \end{align*}
    Therefore in any case, we have that for $|z| \leq 1$,
    \[|\text{Re} \, \widecheck{m}(z)| \lesssim \ang{y}^{-2}.\]
    For $|z| \geq 1$, choose an integer $N\geq\max\{1,L\}$. We integrate by parts $N$ times in the integral for $\widecheck{m}(z)$ and use the derivative bounds for $m$ to obtain 
    \[
        \sum_{j=0}^1
        \left|\operatorname{Re}\partial_z\widecheck{m_{j,1}(y,\cdot)}(z)\right|
        +
        \left|\operatorname{Re}\widecheck{m_{2,1}(y,\cdot)}(z)\right|
        \lesssim_L
        \langle y\rangle^{-2}\langle z\rangle^{-L}.
    \]
    Finally, for $k \geq 2$, we write 
    \[\widecheck{m_{2,k}(y, \cdot)}(z) = \frac1{2\pi}\int_{1}^{\infty} m_{2,k}(y, \xi) e^{iz\xi} \, d\xi.\]
    For $|z| \leq 1$, we use the $L^1$ bound
    \[|\widecheck{m_{2,k}(y, \cdot)}(z)| \lesssim \ang{y}^{-k-1}\]
    and for $|z| > 1$, we integrate by parts $N$ times, with an integer $N\geq\max\{1,L\}$, to obtain
    \begin{align*}
        |\widecheck{m_{2,k}(y, \cdot)}(z)| & = \frac{|z|^{-N}}{2\pi}\left|\int_{1}^{\infty} \partial^N_\xi m_{2,k}(y, \xi) e^{iz\xi} \, d\xi\right| \\ 
        & \lesssim_N |z|^{-N}\int_{1}^{\infty}\xi^{-k - N}(|y| + \xi^p)^{-k-1} \, d\xi \\
        & \lesssim_L \ang{y}^{-k-1}|z|^{-L}.
    \end{align*}
    The estimates for $m_{0,k}$ and $m_{1,k}$ follow similarly. This proves the result.

\end{proof}

\begin{lemma}\label{lem:linear_nonstationary_remainder}
    Consider the kernel of the remainder term $\mc{R}_{j,M}(G)$ in \eqref{eqn:linear_nonstationary_remainder_definition}:
    \[K_{\mc{R}_{j,M}}(y, z) = -\int_{1<\xi<\eta}\partial_\xi m_{j,M-1}(y,\xi)
        e^{i(\Phi_0(\eta)-\Phi_0(\xi)-z\eta+y\xi)}\,d\xi\,d\eta.\]
    Then for $M \geq 2$, we have the estimate
    \[|K_{\mc{R}_{j,M}}(y, z)| \lesssim_M \ang{y}^{-M} \begin{cases}
        \ang{z}^{-1}, & z \leq 0 \\
        \ang{z}^{\frac{1-p}{2p}}, & z > 0
    \end{cases}.\]
\end{lemma}

\begin{proof}
    By definition, we have that $|a_{j}^{\textup{ns}}(y, \xi)| \lesssim \xi^{j - 2}$ and so by induction using \eqref{eqn:linear_nonstationary_symbols}, we have that
    \[
    |\partial_\xi m_{j,M-1}(y, \xi)| \lesssim_M \xi^{j - 2 - M}(|y| + \xi^p)^{-M}, \quad M \geq 1.
    \]
    We write the kernel as
    \[K_{\mc{R}_{j,M}}(y,z) = -\int_{1}^{\infty}\left[\int_{\xi}^{\infty}e^{i(\Phi_0(\eta) - z\eta)} \, d\eta\right]\partial_\xi m_{j,M-1}(y, \xi) e^{i(y\xi - \Phi_0(\xi))}  \,d\xi.\]
    By \cref{lem:appendix_stationary} and the estimates for $\partial_\xi m_{j,M-1}$, we have that
    \begin{align*}
        |K_{\mc{R}_{j,M}}(y,z)| \lesssim_M \left(\int_{1}^{\infty}\xi^{j - 2 - M}(|y| + \xi^p)^{-M} \, d\xi\right) \begin{cases}
        \ang{z}^{-1}, & z \leq 0 \\
        \ang{z}^{\frac{1-p}{2p}}, & z > 0
    \end{cases} \\
        \lesssim_M \ang{y}^{-M} \begin{cases}
        \ang{z}^{-1}, & z \leq 0 \\
        \ang{z}^{\frac{1-p}{2p}}, & z > 0
    \end{cases}
    \end{align*}
    as wanted.
\end{proof}

\begin{proposition}\label{prop:linear_high_frequency}
For $\beta+\frac1p<\gamma<2p+1$, we have the following estimate for the high-frequency part of the inverse:
\[\|\mc{T}_{\textup{h}}\|_{\sq{Y}_\gamma \to X_\gamma} < \infty.\]
\end{proposition}

\begin{proof}
    By the frequency decomposition \eqref{eqn:linear_frequency_split}, we can express the operator as $\mc{T}_{\textup{h}} = 2 \text{Re} \, \mc{T}_{+}$. Using \eqref{eqn:linear_positive_kernel_representation}, \eqref{eqn:linear_stationary_kernel_split}, and \eqref{eqn:linear_stationary_operators}, we now write the operator and its derivatives as
    \[\partial^j \mc{T}_{+} = \mc{T}_{+,j}^{\textup{s}} + \mc{T}_{+,j}^{\textup{ns}}.\]
    To bound the stationary contribution, we can appeal to \cref{prop:linear_stationary} which gives the bound
    \[|K_{+,j}^{\textup{s}}(y, z)| \lesssim \begin{cases}
        \ang{y}^{\frac{j}{p} - \frac{p + 3}{2p}} \ang{z}^{\frac{1 - p}{2p}}, & z \geq cy,\\
        C_M\ang{y}^{-M}(\ang{y} + |z|)^{-M}, & z < cy.
    \end{cases}\]
    Therefore we have that 
    \begin{equation}\label{eqn:linear_positive_stationary_contribution}
        \begin{aligned}
            \left|\int K_{+,j}^{\textup{s}}(y, z)G'(z)\,dz\right| & \leq \int_{cy}^{\infty}|K_{+,j}^{\textup{s}}(y, z)||G'(z)|\,dz + \int_{-\infty}^{cy}|K_{+,j}^{\textup{s}}(y, z)||G'(z)|\,dz \\
            & \lesssim_M \|G\|_{Y_\gamma}\ang{y}^{\frac{j}{p}-\frac{p+3}{2p}}\int_{cy}^{\infty}\ang{z}^{\frac{3-p}{2p}-\gamma}\,dz + \ang{y}^{-2M}\|G'\|_{L^1} \\
            & \lesssim_M \|G\|_{Y_\gamma}\left(\ang{y}^{-\gamma+\frac{j}{p}} + \ang{y}^{-2M}\right) \\
            & \lesssim \|G\|_{Y_\gamma}\ang{y}^{-\gamma+\frac{j}{p}}
        \end{aligned}
    \end{equation}
    Since $\mc{T}_{+,j}^{\textup{s}}G(y) = 0$ for $y \leq 0$, this establishes the estimate for the stationary contribution.
    \par
    For the nonstationary contribution, we use the decomposition \eqref{eqn:linear_nonstationary_expansion} from \cref{lem:linear_nonstationary_symbol_expansion}. Indeed, we write
    \[\mc{T}_{+,j}^{\textup{ns}}G(y) = 2\pi\sum_{k = 0}^{M-1} [\widecheck{m_{j,k}(y, \cdot)} * G'](y) + \mc{R}_{j,M}(G)(y).\]
    By \cref{lem:linear_nonstationary_remainder}, we can bound the remainder term by
    \begin{align*}
        |\mc{R}_{j,M}(G)(y)| & \lesssim \ang{y}^{-M} \left(\int_{-\infty}^{0}|G'(z)|\ang{z}^{-1} \, dz + \int_{0}^{\infty}|G'(z)|\ang{z}^{\frac{1-p}{2p}} \, dz \right)\\
        & \lesssim_M \|G\|_{Y_\gamma}\ang{y}^{-M}\left(\int_{-\infty}^{0}\ang{z}^{-\gamma-2}\,dz+\int_{0}^{\infty}\ang{z}^{\frac{3-p}{2p}-\gamma}\,dz\right)\\
        & \lesssim_M \|G\|_{Y_\gamma}\ang{y}^{-M}\lesssim \|G\|_{Y_\gamma}\ang{y}^{-\gamma-j}.
    \end{align*}
    Here we used $\gamma>\beta+\frac1p$, and in the last inequality we take $M=5>\gamma+j$ for $j=0,1,2$. For the convolution terms, we use \cref{lem:linear_nonstationary_leading_kernels}. Taking $L>\gamma+2$, we obtain, for $0\leq k\leq4$,
    \begin{equation}\label{eqn:linear_nonstationary_weighted_kernels}
        \begin{aligned}
            &\int_{\R}\ang{z}^{\gamma+1}\left(\sum_{j=0}^1\left|\operatorname{Re}\partial_z\widecheck{m_{j,k}(y,\cdot)}(z)\right|
            +\left|\operatorname{Re}\widecheck{m_{2,k}(y,\cdot)}(z)\right|\right)\,dz\\
            &\qquad\lesssim_k\ang{y}^{-k-1}.
        \end{aligned}
    \end{equation}
    For $k=0$, this follows by integrating the three regions in the lemma and using
    \[
        \int_0^{\ang{y}^{-1/p}}z^{p-1}\,dz
        +\ang{y}^{-2}\int_{\ang{y}^{-1/p}}^1z^{-p-1}\,dz+\ang{y}^{-1}
        \lesssim\ang{y}^{-1}.
    \]
    For $k\geq1$, it follows directly from \eqref{eqn:linear_nonstationary_higher_kernels}. \par
    First, for $j=0,1$, we integrate by parts. Since $G$ is real-valued and $\ang{y}\lesssim\ang{y-z}\ang{z}$, \eqref{eqn:linear_nonstationary_weighted_kernels} gives
    \begin{equation}\label{eqn:linear_nonstationary_convolution_j01}
        \begin{aligned}
            \left|\operatorname{Re}[\widecheck{m_{j,k}(y,\cdot)}*G'](y)\right|
            &=\left|\int_{\R}\operatorname{Re}\partial_z\widecheck{m_{j,k}(y,\cdot)}(z)G(y-z)\,dz\right|\\
            &\lesssim\|G\|_{Y_\gamma}\ang{y}^{-\gamma}
            \int_{\R}\ang{z}^{\gamma}\left|\operatorname{Re}\partial_z\widecheck{m_{j,k}(y,\cdot)}(z)\right|\,dz\\
            &\lesssim_k\|G\|_{Y_\gamma}\ang{y}^{-\gamma-k-1}.
        \end{aligned}
    \end{equation}
    For $j=2$, we keep the derivative on $G$. By the definition of $Y_\gamma$, we have
    \begin{equation}\label{eqn:linear_nonstationary_shifted_derivative}
        |G'(y-z)|\lesssim\|G\|_{Y_\gamma}\ang{z}^{\gamma+1}
        \begin{cases}
            \ang{y}^{-\gamma+1/p},&y\geq0,\\
            \ang{y}^{-\gamma-1},&y\leq0.
        \end{cases}
    \end{equation}
    Indeed, for $y\geq0$, use the global bound $|G'(y-z)|\lesssim\|G\|_{Y_\gamma}\ang{y-z}^{-\gamma+1/p}$ and $\ang{y}\lesssim\ang{y-z}\ang{z}$. For $y\leq0$, the same argument applies when $y-z\leq0$ using the negative-side bound for $G'$. If $y-z>0$, then $|z|\geq|y|$, so the uniform bound $|G'(y-z)|\lesssim\|G\|_{Y_\gamma}$ proves \eqref{eqn:linear_nonstationary_shifted_derivative}. Therefore, by \eqref{eqn:linear_nonstationary_weighted_kernels},
    \begin{equation}\label{eqn:linear_nonstationary_convolution_j2}
        \begin{aligned}
            \left|\operatorname{Re}[\widecheck{m_{2,k}(y,\cdot)}*G'](y)\right|
            &\leq\int_{\R}\left|\operatorname{Re}\widecheck{m_{2,k}(y,\cdot)}(z)\right||G'(y-z)|\,dz\\
            &\lesssim_k\|G\|_{Y_\gamma}
            \begin{cases}
                \ang{y}^{-\gamma+1/p-k-1},&y\geq0,\\
                \ang{y}^{-\gamma-k-2},&y\leq0.
            \end{cases}
        \end{aligned}
    \end{equation}
    Summing \eqref{eqn:linear_nonstationary_convolution_j01} and \eqref{eqn:linear_nonstationary_convolution_j2} over $0\leq k\leq4$, and adding the remainder bound, gives the required estimates for the nonstationary contribution. Combining these with \eqref{eqn:linear_positive_stationary_contribution} and $\mc{T}_{\textup{h}}=2\operatorname{Re}\mc{T}_+$, we conclude that
    \[\|\mc{T}_{\textup{h}}G\|_{X_\gamma}\lesssim\|G\|_{Y_\gamma}.\]
\end{proof}

Finally, we are able to prove \cref{thm:fixed_point_linear_inverse}.

\begin{proof}[Proof of \cref{thm:fixed_point_linear_inverse}]
Continuity of $\ell$ follows from \cref{prop:linear_ell_continuity}.
By the definition of $\mc{T}_0$, we know that $\mc{L}_0\mc{T}_0 G = G$ for any $G \in \sq{Y}_\gamma$ and by \eqref{eqn:linear_frequency_split}, \cref{prop:linear_low_frequency} and \cref{prop:linear_high_frequency}, we conclude that $\mc{T}_0$ is bounded from $\sq{Y}_{\gamma}$ to $X_\gamma$. The first assertion of \cref{thm:fixed_point_linear_inverse} follows. \par
Next, the existence of the isomorphisms $\mc{U}_\epsilon, \mc{V}_\epsilon$ and their bounds are supplied by \cref{lem:linear_conjugation}. Finally, the bound for $G_\mu$ follows because $G_\mu$ is a Schwartz function with the corresponding bounds on the norms in \cref{lem:linear_conjugation} being independent of $\mu$. It follows that
\[\|(\mc{V}_\epsilon^{-1} - I)G_\mu\|_{Y_\gamma} \lesssim \epsilon.\]
This finishes the proof of \cref{thm:fixed_point_linear_inverse}.
\end{proof}

\section{Nonlinear Estimates}\label{sec:nonlinear_estimates}

In this section, we will prove the nonlinear estimates in \cref{thm:fixed_point_nonlinear_estimates}. Recall that the nonlinear operator is given by 
\[\mc{N}_{\epsilon, \ov{f}}(r) = \mc{N}_1(r) + \mc{N}_2(r) + \mc{N}_3(r) - \mc{R}_\epsilon(r)\]
where 
\begin{equation}\label{eq:nonlinear_decomposition}
\begin{aligned}
    \mc{N}_1(r) & = (B_\epsilon(\ov{f}' + r') - B_\epsilon(\ov{f}'))\left(\ov{f} - y \ov{f}'\right) \\
    \mc{N}_2(r) & = (B_\epsilon(\ov{f}' + r') - 1)\left(r - y r'\right) \\
    \mc{N}_3(r) & = \mc{I}_\epsilon(\ov{f} + r) - \mc{I}_\epsilon(\ov{f}) - b_\epsilon r'.
\end{aligned}
\end{equation}

We will need to bound these operators with estimates that are uniform in $\epsilon$ and in the family of approximate solutions $\{\ov{f}_{\epsilon, \mu}\}$. This motivates the following definition.

\begin{definition}\label{def:nonlinear_admissible_profile}
    Fix a large constant $M$. We say that $\ov{f} \in C^2$ is an \emph{admissible profile} if it satisfies the following bounds. For $y \geq 0$,
    \begin{equation}\label{eq:nonlinear_admissible_profile_plus}|\ov{f}(y)| \leq M \ang{y}^{-p}, \quad |\partial^{j+1} \ov{f}(y)| \leq M \ang{y}^{-\beta + \frac{j}{p}}, \quad j = 0,1\end{equation}
    and for $y \leq 0$, 
    \begin{equation}\label{eq:nonlinear_admissible_profile_minus}| \partial^j (\ov{f}(y) + y)| \leq M \ang{y}^{-p - j} \quad j = 0,1, 2.\end{equation}
\end{definition}

It's clear from the definition that the approximate solutions $\{\ov{f}_{\epsilon, \mu}\}$ are admissible profiles with a uniform constant $M$ independent of $\epsilon$ and $\mu$.

\begin{lemma}\label{lem:N1}
Let $0\le\epsilon\le1$ and $\beta + \frac{1}{p} < \gamma < 2p + 1$, and let $\ov{f}$ be an admissible profile and let $r_1, r_2 \in X_\gamma$ with $\|r_1\|_{X_\gamma}, \|r_2\|_{X_\gamma} \leq N$. Then 
\[\|\mc{N}_1(r_1) - \mc{N}_1(r_2)\|_{Y_\gamma} \lesssim_{M,N} \epsilon \|r_1 - r_2\|_{X_\gamma}.\]
\end{lemma}

\begin{proof}
    For the first term in \eqref{eq:nonlinear_decomposition}, write
    \begin{equation}\label{eq:N1_delta_B}
        \Delta B_\epsilon(y)
        =B_\epsilon(\ov{f}'(y)+r_1'(y))-B_\epsilon(\ov{f}'(y)+r_2'(y)).
    \end{equation}
    All implicit constants may depend on $M,N,p,\gamma$.
    We use the estimates
    \begin{equation}\label{eq:N1_profile_factor}\left|\ov{f} - y\ov{f}'\right| \lesssim \begin{cases}
        \ang{y}^{1-\beta}, & y > 0,\\
        \ang{y}^{-p}, & y < 0,
    \end{cases}\end{equation}
    and
    \begin{equation}\label{eq:N1_profile_factor_derivative}|\partial_y(\ov{f} - y\ov{f}')| = |y\ov{f}''| \lesssim \begin{cases}
        \ang{y}^{1-\beta+1/p}, & y > 0,\\
        \ang{y}^{-p-1}, & y < 0.
    \end{cases}\end{equation}
    On the bounded range of $\ov{f}'+r_j'$, the function
    $B_\epsilon(u)=(1+\epsilon u^2)^{\alpha/2}$ satisfies
    \[
        |B_\epsilon'(u)|\lesssim\epsilon|u|,\qquad
        |B_\epsilon''(u)|\lesssim\epsilon.
    \]
    Thus, using $\gamma>\beta+1/p$, we obtain
    \begin{equation}\label{eq:N1_delta_B_bound}
\begin{aligned}
        |\Delta B_\epsilon(y)|
        &\lesssim\epsilon\bigl(|\ov{f}'|+|r_1'|+|r_2'|\bigr)|r_1'-r_2'|\\
        &\lesssim\epsilon\|r_1-r_2\|_{X_\gamma}
        \begin{cases}
            \ang{y}^{-\beta-\gamma+1/p}, & y>0,\\
            \ang{y}^{-\gamma-1}, & y<0.
        \end{cases}
    \end{aligned}
\end{equation}
    For the derivative, we write
    \begin{align*}
        \partial_y\Delta B_\epsilon
        &=B_\epsilon'(\ov{f}'+r_1')(r_1'' - r_2'')\\
        &\quad+\bigl[B_\epsilon'(\ov{f}'+r_1')
        -B_\epsilon'(\ov{f}'+r_2')\bigr](\ov{f}''+r_2'').
    \end{align*}
    It follows that
    \begin{equation}\label{eq:N1_delta_B_derivative_bound}
\begin{aligned}
        |\partial_y\Delta B_\epsilon(y)|
        &\lesssim\epsilon\bigl(|\ov{f}'|+|r_1'|\bigr)|r_1''-r_2''|
        +\epsilon|r_1'-r_2'|\bigl(|\ov{f}''|+|r_2''|\bigr)\\
        &\lesssim\epsilon\|r_1-r_2\|_{X_\gamma}
        \begin{cases}
            \ang{y}^{-\beta-\gamma+2/p}, & y>0,\\
            \ang{y}^{-\gamma-2}, & y<0.
        \end{cases}
    \end{aligned}
\end{equation}
    Since
    \[
        \mc{N}_1(r_1)-\mc{N}_1(r_2)
        =\Delta B_\epsilon\left(\ov{f}-y\ov{f}'\right),
    \]
    the bounds in \eqref{eq:N1_profile_factor}--\eqref{eq:N1_profile_factor_derivative} and \eqref{eq:N1_delta_B_bound}--\eqref{eq:N1_delta_B_derivative_bound}, the product rule, and $2\beta=1+1/p$ give
    \[
        |\mc{N}_1(r_1)(y)-\mc{N}_1(r_2)(y)|
        \lesssim\epsilon\|r_1-r_2\|_{X_\gamma}
        \begin{cases}
            \ang{y}^{-\gamma}, & y>0,\\
            \ang{y}^{-\gamma-p-1}, & y<0,
        \end{cases}
    \]
    and
    \[
        |\partial_y(\mc{N}_1(r_1)-\mc{N}_1(r_2))(y)|
        \lesssim\epsilon\|r_1-r_2\|_{X_\gamma}
        \begin{cases}
            \ang{y}^{-\gamma+1/p}, & y>0,\\
            \ang{y}^{-\gamma-p-2}, & y<0.
        \end{cases}
    \]
    We conclude that
    \[
        \|\mc{N}_1(r_1)-\mc{N}_1(r_2)\|_{Y_\gamma}
        \lesssim_{M,N}\epsilon\|r_1-r_2\|_{X_\gamma}.
    \]
\end{proof}

\begin{lemma}\label{lem:N2}
Let $0\le\epsilon\le1$ and $\beta + \frac{1}{p} < \gamma < 2p + 1$, and let $\ov{f}$ be an admissible profile and let $r_1, r_2 \in X_\gamma$ with $\|r_1\|_{X_\gamma}, \|r_2\|_{X_\gamma} \leq N$. Then 
\[\|\mc{N}_2(r_1) - \mc{N}_2(r_2)\|_{Y_\gamma} \lesssim_{M,N} \epsilon \|r_1 - r_2\|_{X_\gamma}.\]
\end{lemma}

\begin{proof}
    As in \eqref{eq:N1_delta_B}, write
    \[
        \Delta B_\epsilon(y)
        =B_\epsilon(\ov{f}'(y)+r_1'(y))-B_\epsilon(\ov{f}'(y)+r_2'(y)).
    \]
    All implicit constants may depend on $M,N,p,\gamma$.
    We use the estimates, valid for every $r\in X_\gamma$,
    \begin{equation}\label{eq:N2_transport_bound}|r-yr'|\lesssim\|r\|_{X_\gamma}\begin{cases}
        \ang{y}^{1-\gamma+1/p}, & y>0,\\
        \ang{y}^{-\gamma}, & y<0,
    \end{cases}\end{equation}
    and
    \begin{equation}\label{eq:N2_transport_derivative_bound}|\partial_y(r-yr')|=|yr''|\lesssim\|r\|_{X_\gamma}\begin{cases}
        \ang{y}^{1-\gamma+2/p}, & y>0,\\
        \ang{y}^{-\gamma-1}, & y<0.
    \end{cases}\end{equation}
    On the bounded range of $\ov{f}'+r_j'$, the function
    $B_\epsilon(u)=(1+\epsilon u^2)^{\alpha/2}$ satisfies
    \[
        |B_\epsilon(u)-1|\lesssim\epsilon|u|^2,\qquad
        |B_\epsilon'(u)|\lesssim\epsilon|u|.
    \]
    Thus, using $\gamma>\beta+1/p$, we obtain
    \begin{equation}\label{eq:N2_coefficient_bound}
\begin{aligned}
        |B_\epsilon(\ov{f}'+r_j')-1|
        &\lesssim\epsilon\bigl(|\ov{f}'|+|r_j'|\bigr)^2\\
        &\lesssim\epsilon\begin{cases}
            \ang{y}^{-2\beta}, & y>0,\\
            1, & y<0.
        \end{cases}
    \end{aligned}
\end{equation}
    For the derivative, we write
    \[
        \partial_y\bigl(B_\epsilon(\ov{f}'+r_j')-1\bigr)
        =B_\epsilon'(\ov{f}'+r_j')(\ov{f}''+r_j'').
    \]
    It follows that
    \begin{equation}\label{eq:N2_coefficient_derivative_bound}
\begin{aligned}
        |\partial_y(B_\epsilon(\ov{f}'+r_j')-1)|
        &\lesssim\epsilon\bigl(|\ov{f}'|+|r_j'|\bigr)
        \bigl(|\ov{f}''|+|r_j''|\bigr)\\
        &\lesssim\epsilon\begin{cases}
            \ang{y}^{-2\beta+1/p}, & y>0,\\
            \ang{y}^{-p-2}, & y<0.
        \end{cases}
    \end{aligned}
\end{equation}
    Recalling \eqref{eq:N1_delta_B_bound} and \eqref{eq:N1_delta_B_derivative_bound}, we also have
    \[
        |\Delta B_\epsilon(y)|
        \lesssim\epsilon\|r_1-r_2\|_{X_\gamma}
        \begin{cases}
            \ang{y}^{-\beta-\gamma+1/p}, & y>0,\\
            \ang{y}^{-\gamma-1}, & y<0,
        \end{cases}
    \]
    and
    \[
        |\partial_y\Delta B_\epsilon(y)|
        \lesssim\epsilon\|r_1-r_2\|_{X_\gamma}
        \begin{cases}
            \ang{y}^{-\beta-\gamma+2/p}, & y>0,\\
            \ang{y}^{-\gamma-2}, & y<0.
        \end{cases}
    \]
    Since
    \begin{align*}
        \mc{N}_2(r_1)-\mc{N}_2(r_2)
        &=\bigl(B_\epsilon(\ov{f}'+r_1')-1\bigr)
        \left((r_1-r_2)-y(r_1'-r_2')\right)\\
        &\quad+\Delta B_\epsilon\left(r_2-yr_2'\right),
    \end{align*}
    the transport bounds \eqref{eq:N2_transport_bound}--\eqref{eq:N2_transport_derivative_bound}, the coefficient bounds \eqref{eq:N2_coefficient_bound}--\eqref{eq:N2_coefficient_derivative_bound}, and the difference bounds \eqref{eq:N1_delta_B_bound}--\eqref{eq:N1_delta_B_derivative_bound}, together with the product rule, $2\beta=1+1/p$, and
    $\gamma>\beta+1/p$ give
    \[
        |\mc{N}_2(r_1)(y)-\mc{N}_2(r_2)(y)|
        \lesssim\epsilon\|r_1-r_2\|_{X_\gamma}\ang{y}^{-\gamma},
    \]
    and
    \[
        |\partial_y(\mc{N}_2(r_1)-\mc{N}_2(r_2))(y)|
        \lesssim\epsilon\|r_1-r_2\|_{X_\gamma}
        \begin{cases}
            \ang{y}^{-\gamma+1/p}, & y>0,\\
            \ang{y}^{-\gamma-1}, & y<0.
        \end{cases}
    \]
    We conclude that
    \[
        \|\mc{N}_2(r_1)-\mc{N}_2(r_2)\|_{Y_\gamma}
        \lesssim_{M,N}\epsilon\|r_1-r_2\|_{X_\gamma}.
    \]
\end{proof}

\begin{lemma}\label{lem:N3}
Let $0\le\epsilon\le1$ and $\beta + \frac{1}{p} < \gamma < 2p + 1$, and let $\ov{f}$ be an admissible profile and let $r_1, r_2 \in X_\gamma$ with $\|r_1\|_{X_\gamma}, \|r_2\|_{X_\gamma} \leq N$. Then 
\[\|\mc{N}_3(r_1) - \mc{N}_3(r_2)\|_{Y_\gamma} \lesssim_{M,N} \epsilon \|r_1 - r_2\|_{X_\gamma}.\]
\end{lemma}

\begin{proof}
It suffices to prove that 
\[\|D\mc{N}_3(r)[h]\|_{Y_\gamma} \lesssim_{M,N} \epsilon \|h\|_{X_\gamma}\]
for all $r \in X_\gamma$ with $\|r\|_{X_\gamma} \leq N$ and all $h \in X_\gamma$. The case $\epsilon=0$ is immediate, so assume $0<\epsilon\le1$. Write $f = \ov{f} + r$.
All implicit constants below may depend on $p,\gamma,M,N$ and the fixed
cutoffs. The assumptions imply
\begin{gather}
    |f(y)-W(y)|\lesssim\ang{y}^{-p},\qquad
    \|f'\|_\infty+\|f\|_{W^{2,\infty}([-3,3])}\lesssim1,
    \qquad \|h\|_{L^1}\lesssim\|h\|_{X_\gamma},\label{eq:N3_global}\\
    |f'(y)|\lesssim y^{-\beta},\quad
    |f''(y)|\lesssim y^{-\beta+1/p}\quad(y\ge2),\label{eq:N3_plus_profile}\\
    |f'(y)+1|\lesssim|y|^{-p-1},\quad
    |f''(y)|\lesssim|y|^{-p-2}\quad(y\le-2).\label{eq:N3_minus_profile}
\end{gather}
We compute the linearization of $\mc{N}_3$ at $r$ in the direction $h$:
\begin{equation}\label{eq:N3_variation}
    D\mc{N}_3(r)[h](y)
    =\int_{\R}\frac{\partial_1 \Gamma_\epsilon(f'(y),S_f(y,z))h'(y)
          +\partial_2 \Gamma_\epsilon(f'(y),S_f(y,z))S_h(y,z)}{|y-z|^{\alpha}}\,dz
       -b_\epsilon(y)h'(y).
\end{equation}
Let's write 
\begin{equation}\label{eq:N3_integrand}K_{f,h}(y, z) = \partial_1 \Gamma_\epsilon(f'(y), S_f(y, z))h'(y) + \partial_2 \Gamma_\epsilon(f'(y), S_f(y, z))S_h(y, z)\end{equation}
so that
\[D\mc{N}_3(r)[h](y)
    =\int_{\R}\frac{K_{f,h}(y, z)}{|y-z|^{\alpha}}\,dz
       -b_\epsilon(y)h'(y)\]
and
\begin{equation}\label{eq:N3_variation_derivative}\partial_y D\mc{N}_3(r)[h](y) = \int_{\R}\frac{\partial_y K_{f,h}(y, z) - \alpha K_{f,h}(y, z) / (y - z)}{|y-z|^{\alpha}}\,dz - \partial_y (b_\epsilon(y)h'(y)).\end{equation}
We estimate \eqref{eq:N3_variation} and \eqref{eq:N3_variation_derivative} by splitting the integrals into a near-field and a far-field part given by $|y - z| \leq \frac{1}{2}|y|$ and $|y - z| > \frac{1}{2}|y|$. The coefficient term $b_\epsilon(y)$ will only be relevant in the control of the far-field term. \par 
\textbf{Near-field estimates:} In this region, we want to preserve the quadratic cancellation of $\Gamma_\epsilon$ across the diagonal. Therefore we will write 
\begin{equation}\label{eq:N3_near_decomposition}K_{f,h}(y, z) = (\partial_1 + \partial_2) \Gamma_\epsilon(f'(y), S_f(y, z))h'(y) + \partial_2 \Gamma_\epsilon(f'(y), S_f(y, z))(S_h(y, z) - h'(y))\end{equation}
We will use the following bounds in the near region $|y - z| \leq \frac{1}{2} |y|$
\begin{equation}\label{eq:N3_profile_increment}|f'(y) - S_f(y, z)| \lesssim 
\begin{cases}
    \ang{y}^{-\beta}\min(1, \ang{y}^{1/p}|y - z|) & y > 0 \\
    \ang{y}^{-p - 2}|y-z|, & y < 0
\end{cases}  \end{equation}
and 
\begin{equation}\label{eq:N3_direction_increment}|h'(y) - S_h(y, z)| \lesssim \|h\|_{X_\gamma} 
\begin{cases}
    \ang{y}^{-\gamma+1/p}\min(1, \ang{y}^{1/p}|y - z|) & y > 0 \\
    \ang{y}^{-\gamma - 2}|y-z|, & y < 0
\end{cases}.\end{equation}
Applying \eqref{eqn:analysis_kernel_derivative_bounds} and \eqref{eqn:analysis_kernel_diagonal_derivative_bounds} to \eqref{eq:N3_near_decomposition} and using \eqref{eq:N3_profile_increment}--\eqref{eq:N3_direction_increment}, we obtain
\begin{equation}\label{eq:N3_near_integrand_bound}
\begin{aligned}
    |K_{f,h}(y, z)| & \lesssim \epsilon|f'(y) - S_f(y, z)|^2 |h'(y)| \\ &\quad + \epsilon|f'(y) - S_f(y, z)|(|f'(y)| + |S_f(y, z)|) |S_h(y, z) - h'(y)| \\
    & \lesssim \epsilon\|h\|_{X_\gamma}
    \begin{cases}
        \ang{y}^{-2\beta-\gamma+1/p}\min\{1,\ang{y}^{1/p}|y-z|\}^2, & y>0,\\
        \ang{y}^{-\gamma-p-4}|y-z|^2, & y<0.
    \end{cases}
\end{aligned}
\end{equation}
Since $\alpha=1+p$ and $0<p<1$,
\[
    \int_{\R}\frac{\min\{1,\ang{y}^{1/p}|\rho|\}^2}{|\rho|^{1+p}}\,d\rho
    \lesssim \ang{y},\qquad
    \int_{|\rho|\le |y|/2}|\rho|^{1-p}\,d\rho
    \lesssim\ang{y}^{2-p}.
\]
Therefore, integrating \eqref{eq:N3_near_integrand_bound} and using $2\beta=1+1/p$, the near integral satisfies
\begin{equation}\label{eq:N3_near_value}
    \int_{|y-z|\le |y|/2}\frac{|K_{f,h}(y,z)|}{|y-z|^\alpha}\,dz
    \lesssim\epsilon\|h\|_{X_\gamma}
    \begin{cases}
        \ang{y}^{-2\beta-\gamma+1/p+1}=\ang{y}^{-\gamma}, & y>0,\\
        \ang{y}^{-\gamma-2p-2}, & y<0.
    \end{cases}
\end{equation}
We now estimate the derivative, keeping $z$ fixed. We have
\begin{equation}\label{eq:N3_secant_derivatives}
    \partial_yS_f(y,z)=\frac{f'(y)-S_f(y,z)}{y-z},\qquad
    \partial_yS_h(y,z)=\frac{h'(y)-S_h(y,z)}{y-z}.
\end{equation}
Differentiating \eqref{eq:N3_near_decomposition}, we obtain
\begin{equation}\label{eq:N3_near_integrand_derivative}
\begin{aligned}
    \partial_yK_{f,h}(y,z)
    &=(\partial_{11}\Gamma_\epsilon+\partial_{12}\Gamma_\epsilon)f''(y)h'(y)\\
    &\quad +(\partial_{12}\Gamma_\epsilon+\partial_{22}\Gamma_\epsilon)\partial_yS_f(y,z)h'(y)\\
    &\quad +(\partial_1\Gamma_\epsilon+\partial_2\Gamma_\epsilon)h''(y)\\
    &\quad +\bigl(\partial_{12}\Gamma_\epsilon f''(y)+\partial_{22}\Gamma_\epsilon\partial_yS_f(y,z)\bigr)
        (S_h(y,z)-h'(y))\\
    &\quad +\partial_2\Gamma_\epsilon\bigl(\partial_yS_h(y,z)-h''(y)\bigr),
\end{aligned}
\end{equation}
where all derivatives of $\Gamma_\epsilon$ are evaluated at $(f'(y),S_f(y,z))$.
The secant identities \eqref{eq:N3_secant_derivatives} and the bounds on $f''$ and $h''$ give
\[
    |f''(y)|+|\partial_yS_f(y,z)|\lesssim
    \begin{cases}
        \ang{y}^{-\beta+1/p}, & y>0,\\
        \ang{y}^{-p-2}, & y<0,
    \end{cases}
\]
and
\[
    |h''(y)|+|\partial_yS_h(y,z)|\lesssim\|h\|_{X_\gamma}
    \begin{cases}
        \ang{y}^{-\gamma+2/p}, & y>0,\\
        \ang{y}^{-\gamma-2}, & y<0.
    \end{cases}
\]
We also recall the bounds from \eqref{eq:N3_direction_increment},
\[
    |S_h(y,z)-h'(y)|\lesssim\|h\|_{X_\gamma}
    \begin{cases}
        \ang{y}^{-\gamma+1/p}\min\{1,\ang{y}^{1/p}|y-z|\}, & y>0,\\
        \ang{y}^{-\gamma-2}|y-z|, & y<0,
    \end{cases}
\]
and $|f'(y)-S_f(y,z)|\lesssim\ang{y}^{-p-2}|y-z|$ for $y<0$.
Applying \eqref{eqn:analysis_kernel_derivative_bounds} and \eqref{eqn:analysis_kernel_diagonal_derivative_bounds} to \eqref{eq:N3_near_integrand_derivative}, we obtain
\begin{equation}\label{eq:N3_near_integrand_derivative_bound}
    |\partial_yK_{f,h}(y,z)|\lesssim\epsilon\|h\|_{X_\gamma}
    \begin{cases}
        \ang{y}^{-2\beta-\gamma+2/p}\min\{1,\ang{y}^{1/p}|y-z|\}, & y>0,\\
        \ang{y}^{-\gamma-p-4}|y-z|, & y<0.
    \end{cases}
\end{equation}
The same bound holds for $|K_{f,h}(y,z)|/|y-z|$. For $y>0$, this follows
from \eqref{eq:N3_near_integrand_bound} and
\[
    \frac{\min\{1,\ang{y}^{1/p}|y-z|\}}{|y-z|}\le\ang{y}^{1/p},
\]
and for $y<0$ it follows directly by dividing \eqref{eq:N3_near_integrand_bound} by
$|y-z|$. Since $0<p<1$,
\[
    \int_{\R}\frac{\min\{1,\ang{y}^{1/p}|\rho|\}}{|\rho|^{1+p}}\,d\rho
    \lesssim\ang{y},\qquad
    \int_{|\rho|\le |y|/2}|\rho|^{-p}\,d\rho
    \lesssim\ang{y}^{1-p}.
\]
Therefore, using \eqref{eq:N3_near_integrand_derivative_bound} and $2\beta=1+1/p$, the near contribution to the derivative satisfies
\begin{equation}\label{eq:N3_near_derivative}
    \int_{|y-z|\le |y|/2}
    \frac{|\partial_yK_{f,h}(y,z)-\alpha K_{f,h}(y,z)/(y-z)|}{|y-z|^\alpha}\,dz
    \lesssim\epsilon\|h\|_{X_\gamma}
    \begin{cases}
        \ang{y}^{-\gamma+1/p}, & y>0,\\
        \ang{y}^{-\gamma-2p-3}, & y<0.
    \end{cases}
\end{equation}
Multiplying by $\ang{y}^{-1/p}$ for $y>0$ and by $\ang{y}$ for $y<0$
gives the required derivative bounds in $Y_\gamma$.

\textbf{Far-field Estimates:} For the far-field estimates in the region $|y - z| \geq \frac{1}{2}|y|$, it is now important to incorporate the subtraction of $b_\epsilon(y)h'(y)$ in \eqref{eq:N3_variation}. For $y > 2$, we will write
\begin{equation}\label{eq:N3_model_coefficient}b_\epsilon(y) = \int_{z < 0}\frac{\partial_1 \Gamma_\epsilon(0, S_W(y, z))}{|y - z|^\alpha} \, dz = \int_{|y - z| \geq \frac{1}{2}|y|} \frac{\partial_1 \Gamma_\epsilon(0, S_W(y, z))}{|y - z|^\alpha} \, dz.\end{equation}
Using \eqref{eq:N3_model_coefficient}, we can express the far-field contribution of $D\mc{N}_3$ as 
\begin{equation}\label{eq:N3_far_subtraction}
\begin{aligned}
    \int_{|y - z| \geq \frac{1}{2}|y|} \frac{K_{f,h}(y, z)}{|y - z|^\alpha} \, dz - b_\epsilon(y)h'(y) & = \int_{|y - z| \geq \frac{1}{2}|y|} \frac{K_{f,h}(y, z) - \partial_1\Gamma_\epsilon(0, S_W(y, z))h'(y)}{|y - z|^\alpha} \, dz 
\end{aligned}
\end{equation}
and the far-field contribution of $\partial_y D\mc{N}_3$ as 
\begin{equation}\label{eq:N3_far_subtraction_derivative}
\begin{aligned}
    &\int_{|y-z|\geq \frac{1}{2}|y|}
    \frac{\partial_y K_{f,h}(y,z)-\alpha K_{f,h}(y,z)/(y-z)}{|y-z|^\alpha}\,dz
    -\partial_y(b_\epsilon(y)h'(y))\\
    &\qquad=\int_{|y-z|\geq \frac{1}{2}|y|}\frac{1}{|y-z|^\alpha}
    \left(\partial_y-\frac{\alpha}{y-z}\right)
    \left(K_{f,h}(y,z)-\partial_1\Gamma_\epsilon(0,S_W(y,z))h'(y)\right)\,dz.
\end{aligned}
\end{equation}
Indeed, for $y>2$ and $z<0$, we have
\[
    S_W(y,z)=\frac{z}{y-z},\qquad
    \partial_yS_W(y,z)=-\frac{S_W(y,z)}{y-z},\qquad
    \partial_y|y-z|^{-\alpha}=-\frac{\alpha}{y-z}|y-z|^{-\alpha}.
\]
We therefore obtain
\begin{equation}\label{eq:N3_model_coefficient_derivative}
\begin{aligned}
    \partial_y(b_\epsilon(y)h'(y))
    &=\partial_y\left(\int_{z<0}
    \frac{\partial_1\Gamma_\epsilon(0,S_W(y,z))h'(y)}{|y-z|^\alpha}\,dz\right)\\
    &=\int_{z<0}\frac{\partial_1\Gamma_\epsilon(0,S_W(y,z))h''(y)}{|y-z|^\alpha}\,dz\\
    &\quad-\int_{z<0}
    \frac{\partial_{12}\Gamma_\epsilon(0,S_W(y,z))S_W(y,z)h'(y)}{(y-z)|y-z|^\alpha}\,dz\\
    &\quad-\alpha\int_{z<0}
    \frac{\partial_1\Gamma_\epsilon(0,S_W(y,z))h'(y)}{(y-z)|y-z|^\alpha}\,dz\\
    &=\int_{z<0}\frac{1}{|y-z|^\alpha}
    \left(\partial_y-\frac{\alpha}{y-z}\right)
    \left(\partial_1\Gamma_\epsilon(0,S_W(y,z))h'(y)\right)\,dz\\
    &=\int_{|y-z|\geq y/2}\frac{1}{|y-z|^\alpha}
    \left(\partial_y-\frac{\alpha}{y-z}\right)
    \left(\partial_1\Gamma_\epsilon(0,S_W(y,z))h'(y)\right)\,dz.
\end{aligned}
\end{equation}
Substituting \eqref{eq:N3_model_coefficient_derivative} into the far contribution of \eqref{eq:N3_variation_derivative} gives \eqref{eq:N3_far_subtraction_derivative}.
To bound the undifferentiated far-field integral for $y > 2$, we write
\begin{equation}\label{eq:N3_far_decomposition}
    K_{f,h}(y,z)-\partial_1\Gamma_\epsilon(0,S_W(y,z))h'(y)
    =[\partial_1\Gamma_\epsilon-\partial_1\Gamma_\epsilon(0,S_W(y,z))]h'(y)
    +(\partial_2\Gamma_\epsilon)S_h(y,z),
\end{equation}
where the unmarked derivatives of $\Gamma_\epsilon$ are evaluated at $(f'(y),S_f(y,z))$.
Using \eqref{eq:N3_global}--\eqref{eq:N3_plus_profile}, $\beta>1$, and $|y-z|\geq y/2$, we have
\begin{equation}\label{eq:N3_far_coefficient_bound}
\begin{aligned}
    |\partial_1\Gamma_\epsilon(f'(y),S_f(y,z))-\partial_1\Gamma_\epsilon(0,S_W(y,z))|
    &\lesssim\epsilon\bigl(|f'(y)|+|S_f(y,z)-S_W(y,z)|\bigr)\\
    &\lesssim\epsilon\left(\ang{y}^{-\beta}+\frac{1}{|y-z|}\right)
    \lesssim\epsilon\ang{y}^{-1}.
\end{aligned}
\end{equation}
Also, $|\partial_2\Gamma_\epsilon(f'(y),S_f(y,z))|\lesssim\epsilon$ by
\eqref{eqn:analysis_kernel_derivative_bounds}, since $f'$ and $S_f$ are uniformly bounded.
It follows that
\begin{equation}\label{eq:N3_far_integrand_bound}
\begin{aligned}
    |K_{f,h}(y,z)-\partial_1\Gamma_\epsilon(0,S_W(y,z))h'(y)|
    &\lesssim\epsilon\ang{y}^{-1}|h'(y)|
    +\epsilon\frac{|h(y)|+|h(z)|}{|y-z|}\\
    &\lesssim\epsilon\|h\|_{X_\gamma}
    \left(\ang{y}^{-\gamma+1/p-1}
    +\frac{\ang{y}^{-\gamma}+\ang{z}^{-\gamma}}{|y-z|}\right).
\end{aligned}
\end{equation}
Since $\gamma>1$, we have
\[
    \int_{|y-z|\geq y/2}\frac{\ang{z}^{-\gamma}}{|y-z|^{\alpha+1}}\,dz
    \lesssim\ang{y}^{-p-2}.
\]
Integrating \eqref{eq:N3_far_integrand_bound}, the far integral satisfies
\begin{equation}\label{eq:N3_far_plus_value}
\begin{aligned}
    &\int_{|y-z|\geq y/2}
    \frac{|K_{f,h}(y,z)-\partial_1\Gamma_\epsilon(0,S_W(y,z))h'(y)|}{|y-z|^\alpha}\,dz\\
    &\qquad\lesssim\epsilon\|h\|_{X_\gamma}
    \left(\ang{y}^{-\gamma+1/p-p-1}
    +\ang{y}^{-\gamma-p-1}+\ang{y}^{-p-2}\right)\\
    &\qquad\lesssim\epsilon\|h\|_{X_\gamma}\ang{y}^{-\gamma},
\end{aligned}
\end{equation}
where we used $p+1>1/p$ and $\gamma<2p+1<p+2$. To bound \eqref{eq:N3_far_subtraction_derivative} when $y > 2$, we again use \eqref{eq:N3_far_decomposition}: 
\[
    K_{f,h}(y,z)-\partial_1\Gamma_\epsilon(0,S_W(y,z))h'(y)
    =[\partial_1\Gamma_\epsilon-\partial_1\Gamma_\epsilon(0,S_W(y,z))]h'(y)
    +(\partial_2\Gamma_\epsilon)S_h(y,z).
\]
It follows that
\begin{align*}
    &\partial_y\left(K_{f,h}(y,z)-\partial_1\Gamma_\epsilon(0,S_W(y,z))h'(y)\right)\\
    &\qquad=[\partial_1\Gamma_\epsilon-\partial_1\Gamma_\epsilon(0,S_W(y,z))]h''(y)\\
    &\qquad\quad+\bigl[\partial_{11}\Gamma_\epsilon f''(y)
    +\partial_{12}\Gamma_\epsilon\partial_yS_f(y,z)
    -\partial_{12}\Gamma_\epsilon(0,S_W(y,z))\partial_yS_W(y,z)\bigr]h'(y)\\
    &\qquad\quad+\bigl[\partial_{12}\Gamma_\epsilon f''(y)
    +\partial_{22}\Gamma_\epsilon\partial_yS_f(y,z)\bigr]S_h(y,z)
    +\partial_2\Gamma_\epsilon\partial_yS_h(y,z).
\end{align*}
In the far region, the secant identities \eqref{eq:N3_secant_derivatives} give
\[
    |\partial_yS_f(y,z)|+|\partial_yS_W(y,z)|\lesssim\frac{1}{|y-z|},
\]
and
\[
    |\partial_yS_h(y,z)|
    \leq\frac{|h'(y)|}{|y-z|}
    +\frac{|h(y)|+|h(z)|}{|y-z|^2}.
\]
Since $\beta>1$ and $|y-z|\geq y/2$, we also have
\[
    |f''(y)|+\frac{1}{|y-z|}
    \lesssim\ang{y}^{-\beta+1/p}+\ang{y}^{-1}
    \lesssim\ang{y}^{1/p-1}.
\]
Using \eqref{eq:N3_far_coefficient_bound}, the preceding secant bounds, and $|D^2\Gamma_\epsilon|\lesssim\epsilon$
from \eqref{eqn:analysis_kernel_derivative_bounds}, we obtain
\begin{equation}\label{eq:N3_far_integrand_derivative_bound}
\begin{aligned}
    &\left|\partial_y\left(K_{f,h}(y,z)-\partial_1\Gamma_\epsilon(0,S_W(y,z))h'(y)\right)\right|\\
    &\qquad\lesssim\epsilon\left(
    \ang{y}^{-1}|h''(y)|+\ang{y}^{1/p-1}|h'(y)|
    +\ang{y}^{1/p-1}\frac{|h(y)|+|h(z)|}{|y-z|}\right)\\
    &\qquad\lesssim\epsilon\|h\|_{X_\gamma}\left(
    \ang{y}^{-\gamma+2/p-1}
    +\ang{y}^{1/p-1}\frac{\ang{y}^{-\gamma}+\ang{z}^{-\gamma}}{|y-z|}\right).
\end{aligned}
\end{equation}
The same bound holds for
$|K_{f,h}(y,z)-\partial_1\Gamma_\epsilon(0,S_W(y,z))h'(y)|/|y-z|$, by the
bound \eqref{eq:N3_far_integrand_bound} and $|y-z|\geq y/2$.
Integrating \eqref{eq:N3_far_integrand_derivative_bound}, the differentiated far integral satisfies
\begin{equation}\label{eq:N3_far_plus_derivative}
\begin{aligned}
    &\int_{|y-z|\geq y/2}\frac{1}{|y-z|^\alpha}
    \left|\left(\partial_y-\frac{\alpha}{y-z}\right)
    \left(K_{f,h}(y,z)-\partial_1\Gamma_\epsilon(0,S_W(y,z))h'(y)\right)\right|\,dz\\
    &\qquad\lesssim\epsilon\|h\|_{X_\gamma}\left(
    \ang{y}^{-\gamma+2/p-p-1}
    +\ang{y}^{-\gamma+1/p-p-2}
    +\ang{y}^{1/p-p-3}\right)\\
    &\qquad\lesssim\epsilon\|h\|_{X_\gamma}\ang{y}^{-\gamma+1/p},
\end{aligned}
\end{equation}
where we used $p+1>1/p$ and $\gamma<2p+1<p+3$.
Multiplying by $\ang{y}^{-1/p}$ gives the required derivative bound
in $Y_\gamma$ for $y>2$. \par
For $y < 2$, the estimates are simpler. In this region, we treat the integral and $b_\epsilon(y)h'(y)$ separately. Since $b_\epsilon(y)$ is smooth and supported on $y > 1$, the estimates for $b_\epsilon(y)h'(y)$ when $y < 2$ are immediate. It remains to estimate the integral terms 
\begin{align*}
    \int_{|y - z| \geq \frac{1}{2}|y|} \frac{K_{f,h}(y, z)}{|y - z|^\alpha} \, dz
\end{align*}
and 
\begin{align*}
    &\int_{|y-z|\geq \frac{1}{2}|y|}
    \frac{\partial_y K_{f,h}(y,z)-\alpha K_{f,h}(y,z)/(y-z)}{|y-z|^\alpha}\,dz
\end{align*}
when $y < 2$.
For the undifferentiated integral, first suppose that $y\leq-2$.
Since $f'$ and $S_f$ are uniformly bounded, \eqref{eqn:analysis_kernel_derivative_bounds} gives
$|\partial_1\Gamma_\epsilon|+|\partial_2\Gamma_\epsilon|\lesssim\epsilon$. Hence
\begin{equation}\label{eq:N3_far_minus_integrand_bound}
\begin{aligned}
    |K_{f,h}(y,z)|
    &\lesssim\epsilon|h'(y)|
    +\epsilon\frac{|h(y)|+|h(z)|}{|y-z|}\\
    &\lesssim\epsilon\|h\|_{X_\gamma}\left(
    \ang{y}^{-\gamma-1}
    +\frac{\ang{y}^{-\gamma}+\ang{z}^{-\gamma}}{|y-z|}\right).
\end{aligned}
\end{equation}
Since $\gamma>1$, integration over the far region yields
\begin{equation}\label{eq:N3_far_minus_value}
\begin{aligned}
    \int_{|y-z|\geq |y|/2}\frac{|K_{f,h}(y,z)|}{|y-z|^\alpha}\,dz
    &\lesssim\epsilon\|h\|_{X_\gamma}\left(
    \ang{y}^{-\gamma-p-1}+\ang{y}^{-p-2}\right)\\
    &\lesssim\epsilon\|h\|_{X_\gamma}\ang{y}^{-\gamma},
\end{aligned}
\end{equation}
where we used $\gamma<2p+1<p+2$.
For $|y|\leq2$ and $|y-z|\leq1$, the local $W^{2,\infty}$ bounds give
\[
    |f'(y)-S_f(y,z)|\lesssim|y-z|,\qquad
    |h'(y)-S_h(y,z)|\lesssim\|h\|_{X_\gamma}|y-z|.
\]
Using the decomposition \eqref{eq:N3_near_decomposition}, we obtain
$|K_{f,h}(y,z)|\lesssim\epsilon\|h\|_{X_\gamma}|y-z|^2$.
For $|y-z|\geq1$, the uniform bounds on the first derivatives of $\Gamma_\epsilon$
give $|K_{f,h}(y,z)|\lesssim\epsilon\|h\|_{X_\gamma}$.
Therefore
\begin{equation}\label{eq:N3_far_compact_value}
\begin{aligned}
    \int_{|y-z|\geq |y|/2}\frac{|K_{f,h}(y,z)|}{|y-z|^\alpha}\,dz
    &\leq\int_{\R}\frac{|K_{f,h}(y,z)|}{|y-z|^\alpha}\,dz\\
    &\lesssim\epsilon\|h\|_{X_\gamma}\left(
    \int_{|\rho|\leq1}|\rho|^{1-p}\,d\rho
    +\int_{|\rho|\geq1}|\rho|^{-1-p}\,d\rho\right)\\
    &\lesssim\epsilon\|h\|_{X_\gamma}
    \lesssim\epsilon\|h\|_{X_\gamma}\ang{y}^{-\gamma}.
\end{aligned}
\end{equation}
The bounds \eqref{eq:N3_far_minus_value} and \eqref{eq:N3_far_compact_value} give the required bound for the undifferentiated far integral for all $y<2$.

For the differentiated integral, first suppose that $y\leq-2$.
In the far region, we have
\[
    |f''(y)|+|\partial_yS_f(y,z)|
    \lesssim\ang{y}^{-p-2}+\frac{1}{|y-z|}
    \lesssim\ang{y}^{-1},
\]
and
\[
    |\partial_yS_h(y,z)|
    \leq\frac{|h'(y)|}{|y-z|}
    +\frac{|h(y)|+|h(z)|}{|y-z|^2}.
\]
Using $|\nabla \Gamma_\epsilon|+|D^2\Gamma_\epsilon|\lesssim\epsilon$ and the product rule, we obtain
\begin{align*}
    |\partial_yK_{f,h}(y,z)|
    &\lesssim\epsilon\left(|h''(y)|+\ang{y}^{-1}|h'(y)|
    +\ang{y}^{-1}\frac{|h(y)|+|h(z)|}{|y-z|}\right)\\
    &\lesssim\epsilon\|h\|_{X_\gamma}\left(
    \ang{y}^{-\gamma-2}
    +\ang{y}^{-1}\frac{\ang{y}^{-\gamma}+\ang{z}^{-\gamma}}{|y-z|}\right).
\end{align*}
The same bound holds for $|K_{f,h}(y,z)|/|y-z|$, by \eqref{eq:N3_far_minus_integrand_bound} and $|y-z|\geq |y|/2$.
Since $\gamma>1$, integration over the far region yields
\begin{equation}\label{eq:N3_far_minus_derivative}
\begin{aligned}
    &\int_{|y-z|\geq |y|/2}
    \frac{|\partial_yK_{f,h}(y,z)-\alpha K_{f,h}(y,z)/(y-z)|}{|y-z|^\alpha}\,dz\\
    &\qquad\lesssim\epsilon\|h\|_{X_\gamma}\left(
    \ang{y}^{-\gamma-p-2}+\ang{y}^{-p-3}\right)\\
    &\qquad\lesssim\epsilon\|h\|_{X_\gamma}\ang{y}^{-\gamma-1},
\end{aligned}
\end{equation}
where we used $\gamma<2p+1<p+2$.
For $|y|\leq2$ and $|y-z|\leq1$, the local $W^{2,\infty}$ bounds and the
differentiated decomposition \eqref{eq:N3_near_integrand_derivative} give
\[
    |\partial_yK_{f,h}(y,z)|+\frac{|K_{f,h}(y,z)|}{|y-z|}
    \lesssim\epsilon\|h\|_{X_\gamma}|y-z|.
\]
For $|y-z|\geq1$, the product rule and the uniform bounds on the first
and second derivatives of $\Gamma_\epsilon$ give
$|\partial_yK_{f,h}(y,z)|+|K_{f,h}(y,z)|/|y-z|\lesssim\epsilon\|h\|_{X_\gamma}$.
Since $0<p<1$, we obtain
\begin{equation}\label{eq:N3_far_compact_derivative}
\begin{aligned}
    &\int_{|y-z|\geq |y|/2}
    \frac{|\partial_yK_{f,h}(y,z)-\alpha K_{f,h}(y,z)/(y-z)|}{|y-z|^\alpha}\,dz\\
    &\qquad\leq\int_{\R}
    \frac{|\partial_yK_{f,h}(y,z)-\alpha K_{f,h}(y,z)/(y-z)|}{|y-z|^\alpha}\,dz\\
    &\qquad\lesssim\epsilon\|h\|_{X_\gamma}\left(
    \int_{|\rho|\leq1}|\rho|^{-p}\,d\rho
    +\int_{|\rho|\geq1}|\rho|^{-1-p}\,d\rho\right)\\
    &\qquad\lesssim\epsilon\|h\|_{X_\gamma}
    \lesssim\epsilon\|h\|_{X_\gamma}\ang{y}^{-\gamma-1}.
\end{aligned}
\end{equation}
The bounds \eqref{eq:N3_far_minus_derivative} and \eqref{eq:N3_far_compact_derivative} give the required bound for the differentiated far integral for all $y<2$.

Combining the near estimates \eqref{eq:N3_near_value} and \eqref{eq:N3_near_derivative}, the positive far estimates \eqref{eq:N3_far_plus_value} and \eqref{eq:N3_far_plus_derivative}, and the far estimates for $y<2$ with the bounds for
$b_\epsilon(y)h'(y)$ on $y<2$, we conclude that
\[
    \|D\mc{N}_3(r)[h]\|_{Y_\gamma}
    \lesssim_{M,N}\epsilon\|h\|_{X_\gamma}.
\]
\end{proof}

Finally, we study the continuity of the operator $\mc{N}_{\epsilon, \ov{f}}$ in the parameter $\ov{f}$. 

\begin{proposition}\label{prop:nonlinear_continuity}
    Fix $\epsilon > 0$ small and $\beta + \frac{1}{p} < \gamma' < \gamma < 2p + 1$ and let $r$ be a fixed function in $X_{\gamma}$. Let $\{\ov{f}_k\}$ be a sequence of admissible functions in the sense of \cref{def:nonlinear_admissible_profile} so that $\ov{f}_k \to \ov{f}$ in $C^2_{\textup{loc}}$ as $k \to \infty$. Then 
    \[\mc{N}_{\epsilon, \ov{f}_k}(r) \to \mc{N}_{\epsilon, \ov{f}}(r)\]
    in $Y_{\gamma'}$ as $k \to \infty$.
\end{proposition}

\begin{proof}
    We use the admissibility bounds \eqref{eq:nonlinear_admissible_profile_plus} and \eqref{eq:nonlinear_admissible_profile_minus} with a common constant $M$ independent of $k$; the limit $\ov{f}$ satisfies the same bounds. Since $\mc{R}_\epsilon(r)$ is independent of the profile, \cref{lem:N1,lem:N2,lem:N3} imply that
    \begin{equation}\label{eq:nonlinear_continuity_uniform_bound}\sup_k\|\mc{N}_{\epsilon, \ov{f}_k}(r)-\mc{N}_{\epsilon, \ov{f}}(r)\|_{Y_\gamma}<\infty.\end{equation}
    By \cref{prop:appendix_convergence}, it suffices to show that
    \[\mc{N}_{\epsilon, \ov{f}_k}(r) \to \mc{N}_{\epsilon, \ov{f}}(r)\]
    in $W^{1, \infty}_{\textup{loc}}$ as $k \to \infty$. Fix a radius $R > 0$. We will show that $\mc{N}_{\epsilon, \ov{f}_k}(r) \to \mc{N}_{\epsilon, \ov{f}}(r)$ in $W^{1, \infty}([-R, R])$. \par
    \textbf{Convergence of $\mc{N}_1$ and $\mc{N}_2$:} The convergence of these terms is immediate. Indeed, $B_{\epsilon}$ is a smooth function so $W^{1,\infty}$ convergence of the terms $B_\epsilon(\ov{f}_k' + r')$ and $B_\epsilon(\ov{f}_k')$ follows from the $C^2_{\textup{loc}}$ convergence of $\ov{f}_k$ to $\ov{f}$. The remaining factors converge in $W^{1,\infty}([-R,R])$, since
    \[\partial_y(\ov{f}_k-y\ov{f}_k')=-y\ov{f}_k'',\qquad \partial_y(r-yr')=-yr''.\]
    The product rule gives the claimed convergence. \par
    \textbf{Convergence of $\mc{N}_3$:} It suffices to show that 
    \begin{equation}\label{eq:nonlinear_continuity_local}\mc{I}_\epsilon(g_k) \to \mc{I}_{\epsilon}(g), \quad \text{in} \quad W^{1, \infty}([-R, R])\end{equation}
    where either $(g_k, g) = (\ov{f}_k + r, \ov{f} + r)$ or $(g_k, g) = (\ov{f}_k, \ov{f})$. We will only be using the properties
    \begin{equation}\label{eq:nonlinear_continuity_profiles}\sup_{k}\|g_k'\|_{L^\infty} + \|g'\|_{L^\infty} < \infty, \quad g_k \to g \quad \text{in} \quad W^{2,\infty}_{\textup{loc}}\end{equation}
    which hold for both choices of $(g_k, g)$. All estimates involving second derivatives are understood almost everywhere. \par
    Recall that
    \[\mc{I}_\epsilon(g)(y) = \int \frac{\Gamma_\epsilon(g'(y), S_g(y, z))}{|y-z|^\alpha}\,dz\]
    and so 
    \begin{equation}\label{eq:nonlinear_continuity_derivative}
        \partial_y \mc{I}_\epsilon(g)(y) = \int \frac{K_g^{(1)}(y, z)}{|y-z|^\alpha}\,dz
    \end{equation}
    where 
    \[\begin{aligned}
        K_g^{(1)}(y, z) & = \partial_1\Gamma_\epsilon(g'(y), S_g(y, z))g''(y)\\
        &\quad + \partial_2\Gamma_\epsilon(g'(y), S_g(y, z)) \frac{g'(y) - S_g(y, z)}{y - z}
        - (1 + p)\frac{\Gamma_\epsilon(g'(y), S_g(y, z))}{y - z}.
    \end{aligned}\]
    Fix $0<\delta<\min\{1,R\}$. For each $|y|\le R$, we decompose the integration region into three regions: 
    \[|y - z| < \delta, \quad \delta < |y - z| < 4R, \quad |y - z| > 4R.\]
    In the first region, we have that 
    \[|\Gamma_\epsilon(g'(y), S_g(y, z))| \lesssim \epsilon (g'(y) - S_g(y, z))^2(|g'(y)| + |S_g(y, z)|) \lesssim_R \epsilon |y - z|^2\]
    and 
    \[|K_g^{(1)}(y, z)| \lesssim_R \epsilon |y - z|\]
    where we used \eqref{eqn:analysis_kernel_quadratic_factor}, \eqref{eqn:analysis_kernel_derivative_bounds}, and the bounds on $g''$ on $[-2R,2R]$. The same estimates hold uniformly for $g_k$. Hence
    \begin{equation}\label{eq:nonlinear_continuity_near}
    \begin{aligned}
        &\sup_{|y|\le R}\int_{|y-z|<\delta}
        \frac{|\Gamma_\epsilon(g_k'(y),S_{g_k}(y,z))|+|\Gamma_\epsilon(g'(y),S_g(y,z))|}{|y-z|^\alpha}\,dz
        \lesssim_R\epsilon\delta^{3-\alpha},\\
        &\mathop{\mathrm{ess\,sup}}_{|y|\le R}\int_{|y-z|<\delta}
        \frac{|K_{g_k}^{(1)}(y,z)|+|K_g^{(1)}(y,z)|}{|y-z|^\alpha}\,dz
        \lesssim_R\epsilon\delta^{2-\alpha}.
    \end{aligned}
    \end{equation}
    These bounds also justify \eqref{eq:nonlinear_continuity_derivative}. Indeed, differentiating the integral truncated to $|y-z|>\delta$ gives the corresponding truncated integral of $K_g^{(1)}$ and the boundary term
    \[\delta^{-\alpha}\bigl[\Gamma_\epsilon(g'(y),S_g(y,y-\delta))-\Gamma_\epsilon(g'(y),S_g(y,y+\delta))\bigr].\]
    This term is $O_R(\epsilon\delta^{2-\alpha})$, and the omitted derivative integral satisfies the same bound by \eqref{eq:nonlinear_continuity_near}. Since $\alpha<2$, passing to the limit proves the identity in the weak sense. \par
    In the second region, \eqref{eq:nonlinear_continuity_profiles} gives 
    \[|S_{g_k}(y, z) - S_g(y, z)| \le \|g_k' - g'\|_{L^\infty([-5R,5R])}\]
    and
    \[|g_k'(y) - g'(y)| + |g_k''(y) - g''(y)| \lesssim \|g_k - g\|_{W^{2, \infty}([-5R,5R])}.\]
    By the smoothness of $\Gamma_\epsilon$ and the bound on $\nabla \Gamma_\epsilon$ in \eqref{eqn:analysis_kernel_derivative_bounds}, we have that
    \[|\Gamma_\epsilon(g_k'(y), S_{g_k}(y, z)) - \Gamma_\epsilon(g'(y), S_g(y, z))| \lesssim_R \epsilon\|g_k - g\|_{W^{2, \infty}([-5R,5R])}|y-z|\]
    and 
    \[|K_{g_k}^{(1)}(y,z)-K_g^{(1)}(y,z)| \lesssim_{R,\delta}\epsilon\|g_k-g\|_{W^{2,\infty}([-5R,5R])}.\]
    Integrating over $\delta<|y-z|<4R$ gives
    \begin{equation}\label{eq:nonlinear_continuity_middle}
    \begin{aligned}
        &\mathop{\sup}_{|y|\le R}\int_{\delta<|y-z|<4R}
        \frac{|\Gamma_\epsilon(g_k'(y),S_{g_k}(y,z))-\Gamma_\epsilon(g'(y),S_g(y,z))|}{|y-z|^\alpha}\,dz\\
        &\quad+\mathop{\sup}_{|y|\le R}\int_{\delta<|y-z|<4R}
        \frac{|K_{g_k}^{(1)}(y,z)-K_g^{(1)}(y,z)|}{|y-z|^\alpha}\,dz\\
        &\qquad\lesssim_{R,\delta}\epsilon\|g_k-g\|_{W^{2,\infty}([-5R,5R])}\to0.
    \end{aligned}
    \end{equation}
    In the third region, the uniform bounds in \eqref{eq:nonlinear_continuity_profiles} imply uniform bounds on $S_{g_k}$, and
    \[|\partial_yS_{g_k}(y,z)|\le\frac{2\|g_k'\|_{L^\infty}}{|y-z|}.\]
    Using the local bounds on $g_k''$ and \eqref{eqn:analysis_kernel_derivative_bounds}, we have that
    \[|\Gamma_\epsilon(g_k'(y),S_{g_k}(y,z))|+|K_{g_k}^{(1)}(y,z)|\lesssim_R\epsilon.\]
    The same estimate holds for $g$. Since $|y-z|>4R$ and $|y|\le R$ imply $|z|>3R$ and $|y-z|\ge 2|z|/3$, both integrands in 
    \[\int_{|y - z| > 4R}\frac{\Gamma_\epsilon(g_k'(y),S_{g_k}(y,z))}{|y-z|^\alpha}\,dz, \quad \int_{|y - z| > 4R}\frac{K_{g_k}^{(1)}(y,z)}{|y-z|^\alpha}\,dz\]
    are bounded by $C_R\epsilon|z|^{-\alpha}$. By applying the dominated convergence theorem, we get that
    \begin{equation}\label{eq:nonlinear_continuity_far}
    \begin{aligned}
        &\mathop{\sup}_{|y|\le R}\int_{|y-z|>4R}
        \frac{|\Gamma_\epsilon(g_k'(y),S_{g_k}(y,z))-\Gamma_\epsilon(g'(y),S_g(y,z))|}{|y-z|^\alpha}\,dz\\
        &\quad+\mathop{\sup}_{|y|\le R}\int_{|y-z|>4R}
        \frac{|K_{g_k}^{(1)}(y,z)-K_g^{(1)}(y,z)|}{|y-z|^\alpha}\,dz\to0.
    \end{aligned}
    \end{equation}
    Combining \eqref{eq:nonlinear_continuity_near}, \eqref{eq:nonlinear_continuity_middle}, and \eqref{eq:nonlinear_continuity_far}, we conclude that
    \[\limsup_{k\to\infty}\|\mc{I}_\epsilon(g_k)-\mc{I}_\epsilon(g)\|_{W^{1,\infty}([-R,R])}
    \lesssim_R\epsilon\delta^{2-\alpha}.\]
    Letting $\delta\to0$ proves \eqref{eq:nonlinear_continuity_local}. Applying this to both choices of $(g_k,g)$ proves the local $W^{1,\infty}$ convergence of $\mc{N}_3$, since $b_\epsilon r'$ is independent of the profile. The uniform $Y_\gamma$ bound \eqref{eq:nonlinear_continuity_uniform_bound} and \cref{prop:appendix_convergence} then give convergence in $Y_{\gamma'}$.
\end{proof}

Now, we are ready to prove \cref{thm:fixed_point_nonlinear_estimates}.

\begin{proof}[Proof of \cref{thm:fixed_point_nonlinear_estimates}]
    By \cref{prop:approx_profile_bounds}, the approximate profiles $\ov{f}_{\epsilon,\mu}$ are admissible in the sense of \cref{def:nonlinear_admissible_profile} with a uniform constant $M$ independent of $\epsilon$ and $\mu$. Therefore \cref{lem:N1,lem:N2,lem:N3} and \cref{lem:linear_conjugation} give the uniform bounds on $\mc{N}_{\epsilon, \ov{f}_{\epsilon, \mu}}$ stated in the theorem. The last assertion about continuity follows from \cref{prop:nonlinear_continuity}. This concludes the proof of the theorem.
\end{proof}

\appendix

\section{Elementary Estimates}

\begin{proposition}\label{prop:appendix_nonsmooth_fourier_inverse}
    Let $\chi$ be a smooth bump function supported on $[-2, 2]$ with $\chi = 1$ on $[-1, 1]$ and define
    \[m_k(\xi) = \chi(\xi)(\Phi_0(\xi))^k, \quad \Phi_0(\xi) = \frac{\xi|\xi|^p}{p+1}.\]
    Then there exists a constant $C$ so that for $j = 0, 1, 2$ and $k \geq 0$, we have the estimate
    \[|\partial^j \widecheck{m_k}(y)| \lesssim C^k \ang{y}^{-p - 2 - j}.\]
\end{proposition}

\begin{proof}
    When $k = 0$, the result is trivial since $m_0$ is smooth, so assume $k \geq 1$. It also suffices to assume that $|y| \geq 1$. Write
    \[a_{k,j}(\xi) = (i\xi)^j \chi(\xi)\Phi_0(\xi)^k\]
    so that $\partial^j \widecheck{m_k}(y) = \widecheck{a_{k,j}}(y)$. Note that
    \[|\partial^l a_{k,j}(\xi)| \lesssim C^k |\xi|^{p + 1 + j - l}, \quad 0 \leq l \leq j + 3.\]
    We integrate by parts $j + 2$ times to obtain
    \[|\widecheck{a_{k,j}}(y)| \lesssim |y|^{-j - 2}\left|\int e^{iy\xi}\partial^{j+2}a_{k,j}(\xi)\, d\xi\right|.\]
    We split the integral at $|y|^{-1}$. For the $|\xi| \leq |y|^{-1}$ contribution, we have
    \[\left|\int_{|\xi| \leq |y|^{-1}} e^{iy\xi}\partial^{j+2}a_{k,j}(\xi)\, d\xi\right| \lesssim C^k \int_{|\xi| \leq |y|^{-1}}|\xi|^{p - 1} \, d\xi \lesssim C^k |y|^{-p}.\]
    For the $|\xi| \geq |y|^{-1}$ contribution, we integrate by parts one more time to obtain
    \begin{align*}
        &\left|\int_{|\xi| \geq |y|^{-1}} e^{iy\xi}\partial^{j+2}a_{k,j}(\xi)\, d\xi\right| \\
        &\quad \lesssim |y|^{-1}\left(\int_{|\xi| \geq |y|^{-1}}|\partial^{j+3}a_{k,j}(\xi)|\, d\xi + \left|\partial^{j+2}a_{k,j}(|y|^{-1})\right| + \left|\partial^{j+2}a_{k,j}(-|y|^{-1})\right|\right) \\
        &\quad \lesssim C^k |y|^{-1}\int_{|\xi| \geq |y|^{-1}}|\xi|^{p - 2}\, d\xi \\
        &\quad \lesssim C^k |y|^{-p}.
    \end{align*}
    By combining the estimates, we obtain the result.
\end{proof}

\begin{proposition}\label{prop:appendix_convergence}
    Let $0 < \gamma' < \gamma < \infty$ and let $\{g_k\}$ be a uniformly bounded sequence of functions in $X_{\gamma}$ so that $g_k \to g$ in $W^{2, \infty}_{\textup{loc}}$. Then $g_k \to g$ in $X_{\gamma'}$. Similarly, if $\{G_k\}$ is a uniformly bounded sequence of functions in $Y_{\gamma}$ so that $G_k \to G$ in $W^{1, \infty}_{\textup{loc}}$, then $G_k \to G$ in $Y_{\gamma'}$.
\end{proposition}

\begin{proof}
    We will prove the first statement; the proof of the second statement is the same. For $R > 0$, we estimate
    \begin{align*}
        \|(g_k - g)\ang{y}^{\gamma'}\|_{L^\infty} & \leq \|(g_k - g)\ang{y}^{\gamma'}\|_{L^\infty([-R, R])} + \|(g_k - g)\ang{y}^{\gamma'}\|_{L^\infty([-R, R]^c)} \\
        & \leq C_R\|g_k - g\|_{L^\infty([-R, R])} + R^{\gamma' - \gamma}\|(g_k - g)\ang{y}^{\gamma}\|_{L^\infty} \\
        & \leq C_R\|g_k - g\|_{L^\infty([-R, R])} + 2R^{\gamma' - \gamma}\sup_k\|g_k\|_{X_\gamma}.
    \end{align*}
    Letting $k \to \infty$ and then $R \to \infty$ gives the desired convergence of $g_k - g$. The convergence of the derivatives follows similarly.
\end{proof}

\section{Stationary and Nonstationary Phase Estimates}

\begin{lemma}\label{lem:van_der_Corput}
    Let $\chi, \Phi$ be smooth functions. If $\Phi'(\xi)$ is monotone and $|\Phi'(\xi)| \geq \lambda$ on the support of $\chi$, then
    \[\left|\int \chi(\xi)e^{i\Phi(\xi)}\,d\xi\right| \lesssim \lambda^{-1}(\|\chi\|_{L^\infty} + \|\chi'\|_{L^1}).\]
    If $|\Phi''(\xi)| \geq \lambda$ on the support of $\chi$, then
    \[\left|\int \chi(\xi)e^{i\Phi(\xi)}\,d\xi\right| \lesssim \lambda^{-1/2}(\|\chi\|_{L^\infty} + \|\chi'\|_{L^1}).\] 
\end{lemma}

\begin{proof}
    The proof can be found in \cite[Chapter~VIII]{Stein:harmonic-analysis}.
\end{proof}

\begin{lemma}\label{lem:appendix_stationary}
    Let $0 < p < 1$ and let $\Phi_0(\xi) = \frac{\xi|\xi|^{p}}{p+1}$. Then we have that
    \[\sup_{\eta > 0}\left|\int_{\eta}^{\infty}e^{i(\Phi_0(\xi) - y\xi)}\,d\xi\right| \lesssim \begin{cases}
        \ang{y}^{-1}, & y \leq 0, \\
        \ang{y}^{\frac{1 - p}{2p}}, & y \geq 0.
    \end{cases}\]
\end{lemma}

\begin{proof}
    We will use \cref{lem:van_der_Corput}. First assume that $|y| \geq 1$. If $y \leq 0$, then for $\xi \geq 0$, the derivative 
    \[\partial_\xi(\Phi_0(\xi) - y\xi) = \xi^{p} - y\]
    is bounded below by $|y|$ and is increasing. Therefore by \cref{lem:van_der_Corput}, the desired estimate follows for $y \leq 0$. For $y \geq 0$, the phase is stationary at $\xi = y^{1/p}$. We decompose the integral into the regions
    \[|\xi - y^{1/p}| < \frac{1}{2} y^{1/p}, \quad |\xi - y^{1/p}| \geq \frac{1}{2} y^{1/p}.\]
    In the first region, which contains the stationary point, we have that 
    \[\partial^2_\xi(\Phi_0(\xi) - y\xi) = p\xi^{p-1} \approx y^{(p-1)/p}.\]
    Therefore by applying \cref{lem:van_der_Corput} with $k = 2$, we have 
    \[\left|\int_{|\xi - y^{1/p}| < \frac{1}{2} y^{1/p}} e^{i(\Phi_0(\xi) - y\xi)}\,d\xi\right| \lesssim \ang{y}^{-(p-1)/2p}.\]
    Next, the region $|\xi - y^{1/p}| \geq \frac{1}{2} y^{1/p}$ can be split into the two regions $(0, \frac{1}{2} y^{1/p})$ and $(\frac{3}{2} y^{1/p}, \infty)$. In both regions, $\xi^p - y$ is monotone and $|\xi^p - y| \gtrsim y$. Therefore by applying \cref{lem:van_der_Corput} with $k = 1$, we have
    \[\left|\int_{|\xi - y^{1/p}| \geq \frac{1}{2} y^{1/p}} e^{i(\Phi_0(\xi) - y\xi)}\,d\xi\right| \lesssim \ang{y}^{-1} \lesssim \ang{y}^{-(p-1)/2p}.\]
    Combining the estimates for the two regions, we have that the supremum is bounded by $C\ang{y}^{-(p-1)/2p}$ for $y \geq 0$, as desired. \par Finally, for $|y| \leq 1$, it suffices to bound the supremum over $\eta \geq 2$. For every $\xi \geq \eta$, we then have that 
    \[|\xi^{p} - y| \gtrsim \xi^p \gtrsim 1\]
    and $\xi^p - y$ is increasing. Applying \cref{lem:van_der_Corput} with $k = 1$ gives that the supremum is bounded by a constant independent of $y \leq 1$. The result follows.
\end{proof}

\begin{lemma}\label{lem:appendix_homogeneous_stationary_phase}
    Let $0<p<1$, $m \leq 0$, and suppose $a\in C^\infty((0,\infty))$ vanishes for $\xi\leq1$ and satisfies
    \[
        |\partial_\xi^k a(\xi)|\leq C_k\xi^{m-k},
        \qquad \xi\geq1,\quad k\geq0.
    \]
    Set
    \[
        \lambda_0(y)=y^{1/p},\qquad
        \Psi_0(y)=\frac{p}{p+1}y^{1+1/p}.
    \]
    Then, for sufficiently large $y$, the oscillatory integral satisfies
    \begin{equation}\label{eqn:appendix_homogeneous_stationary_expansion}
        \int_0^\infty a(\xi)e^{i(y\xi-\xi^{p+1}/(p+1))}\,d\xi
        =
        e^{i(\Psi_0(y)-\pi/4)}
        \left(\frac{2\pi}{p\lambda_0(y)^{p-1}}\right)^{1/2}
        \bigl[a(\lambda_0(y))+r(y)\bigr],
    \end{equation}
    where $r$ is smooth and, for every $j\geq0$,
    \begin{equation}\label{eqn:appendix_homogeneous_stationary_remainder}
        |\partial_y^j r(y)|
        \lesssim_j \lambda_0(y)^{m-(p+1)}y^{-j}.
    \end{equation}
    The implicit constants depend on $p,m,j$ and finitely many of the amplitude bounds $C_k$.
\end{lemma}

\begin{proof}
This follows from the stationary phase expansion in
\cite[Chapter~VIII]{Stein:harmonic-analysis}, after rescaling and localization
near the stationary point.
\end{proof}

\begin{lemma}\label{lem:appendix_triangle}
    Let $I = (x_0 - L,x_0 + L )$ and suppose $\Phi$ is a function on $I$ satisfying
    \[\Phi'(x_0) = 0, \quad |\Phi''(t)| \approx \kappa, \quad |\Phi'''(t)| \lesssim \kappa L^{-1}.\]
    Suppose $a \in C^2_c(I)$ satisfies $|\partial^j a(t)| \lesssim A L^{-j}$ for $j = 0, 1, 2$. Write
    \[E(t) = \int_{t}^{\infty}e^{i\Psi(s)} \, ds\]
    and suppose that $|E(t)| \leq M$ for all $t \in I$. Finally, suppose that $\Phi(t) + \Psi(t) = \omega t$. Then we have the bound 
    \[\left|\int_{I}a(t)e^{i\Phi(t)}E(t) \, dt\right| \lesssim A(M\kappa^{-1/2} + \kappa^{-1}).\]
\end{lemma}

\begin{proof}
    By a translation, we may assume that $x_0 = 0$. If $L \leq \kappa^{-1/2}$, we can simply bring the absolute value inside the integral and bound
    \[\left|\int_{I}a(t)e^{i\Phi(t)}E(t) \, dt\right| \lesssim LMA \leq AM \kappa^{-1/2}.\]
    Now assume $L \geq \kappa^{-1/2}$. Since $\Phi''$ has constant sign, we have $|\Phi'(t)|\approx\kappa|t|$, and hence
    \[
        \left|\frac{a(t)}{\Phi'(t)}\right|\lesssim\frac{A}{\kappa|t|},
        \qquad
        \left|\left(\frac{a(t)}{\Phi'(t)}\right)'\right|
        \lesssim\frac{A}{\kappa L|t|}+\frac{A}{\kappa t^2}
        \lesssim\frac{A}{\kappa t^2}.
    \]
    The contribution from $|t|\leq\kappa^{-1/2}$ is bounded by $AM\kappa^{-1/2}$:
    \[\left|\int_{|t| \leq \kappa^{-1/2}}a(t)e^{i\Phi(t)}E(t) \, dt\right| \lesssim AM\kappa^{-1/2}.\]
    On the remaining two intervals $(-L, -\kappa^{-1/2})$ and $(\kappa^{-1/2}, L)$ we integrate by parts and use $E'=-e^{i\Psi}$ and $\Phi+\Psi=\omega t$ to get
    \begin{align*}
        \int_{\kappa^{-1/2}<|t|<L}a(t)e^{i\Phi(t)}E(t)\,dt 
        &=\int_{\kappa^{-1/2}<|t|<L}
        \frac{a(t)E(t)}{i\Phi'(t)}\,\partial_t\bigl(e^{i\Phi(t)}\bigr)\,dt \\
        &=\left[\frac{a(t)E(t)e^{i\Phi(t)}}{i\Phi'(t)}\right]_{-L}^{-\kappa^{-1/2}}
        +\left[\frac{a(t)E(t)e^{i\Phi(t)}}{i\Phi'(t)}\right]_{\kappa^{-1/2}}^L \\
        &-\frac1i\int_{\kappa^{-1/2}<|t|<L}
        \left(\left(\frac{a(t)}{\Phi'(t)}\right)'E(t)
        +\frac{a(t)E'(t)}{\Phi'(t)}\right)e^{i\Phi(t)}\,dt \\
        &=\frac{a(-\kappa^{-1/2})E(-\kappa^{-1/2})e^{i\Phi(-\kappa^{-1/2})}}{i\Phi'(-\kappa^{-1/2})}
        -\frac{a(\kappa^{-1/2})E(\kappa^{-1/2})e^{i\Phi(\kappa^{-1/2})}}{i\Phi'(\kappa^{-1/2})} \\
        &-\frac1i\int_{\kappa^{-1/2}<|t|<L}
        \left(\frac{a(t)}{\Phi'(t)}\right)'E(t)e^{i\Phi(t)}\,dt \\
        &+\frac1i\int_{\kappa^{-1/2}<|t|<L}
        \frac{a(t)}{\Phi'(t)}e^{i\omega t}\,dt.
    \end{align*}
    The boundary terms at $\pm L$ vanish because $a$ is compactly supported in $I$. By the bounds above and $|E|\leq M$, the boundary terms at $\pm\kappa^{-1/2}$ and the integral containing $(a(t)/\Phi'(t))'$ are $O(AM\kappa^{-1/2})$.
    To estimate the remaining integral, observe that
    \[
        \frac{\Phi'(t)}t=\int_0^1\Phi''(st)\,ds,
        \qquad
        \partial_t\!\left(\frac{\Phi'(t)}t\right)
        =\int_0^1s\Phi'''(st)\,ds.
    \]
    By the assumptions on $\Phi''$ and $\Phi'''$, we see that $h(t) := ta(t)/\Phi'(t)$ extends continuously to $t=0$ with value $a(0)/\Phi''(0)$, and its derivative is bounded by $A/(\kappa L)$. It follows that
    \[
        \left|\frac{a(t)}{\Phi'(t)}-\frac{a(0)}{\Phi''(0)t}\right| = \left|\frac{h(t) - h(0)}{t}\right|
        \lesssim\frac{A}{\kappa L}.
    \]
    It follows that we can estimate
    \begin{align*}
        \left|\int_{\kappa^{-1/2}<|t|<L}\frac{a(t)}{\Phi'(t)}e^{i\omega t}\,dt\right|
        &\lesssim\frac{A}{\kappa L}\int_{\kappa^{-1/2}<|t|<L}dt
        +\left|\frac{a(0)}{\Phi''(0)}\int_{\kappa^{-1/2}<|t|<L}\frac{e^{i\omega t}}t\,dt\right|\\
        &\lesssim\frac{A}{\kappa}+\frac{A}{\kappa}\left|\int_{\kappa^{-1/2}<t<L}\frac{\sin(\omega t)}{ t}\,dt\right| \\
        & \lesssim \frac{A}{\kappa}
    \end{align*}
    as wanted.
\end{proof}

\begin{lemma}\label{lem:appendix_oscillatory_multiplier}
    Let $0<p<1$, $A>0$, and $Y, \Lambda\geq1$. Let $a\in C_c^\infty((Y/4,4Y))$ be complex-valued and $\Psi\in C^\infty((Y/4,4Y))$ real-valued. Write $\lambda=\Psi'$ and assume that, on this interval,
    \[
        \lambda(y)\approx\Lambda,\qquad
        |\partial_y^k\lambda(y)|\lesssim_k\Lambda Y^{-k},\qquad
        |\partial_y^ka(y)|\lesssim_k A Y^{-k},
        \qquad k\geq0.
    \]
    Set $m(\xi)=\xi^p$ for $\xi>0$ and, for $z$ near zero,
    \[
        \Theta_y(z)=\Psi(y+z)-\Psi(y)-\lambda(y)z.
    \]
    Then, for every integer $N\geq1$ and $Y\leq y\leq2Y$,
    \begin{equation}\label{eqn:appendix_oscillatory_multiplier_expansion}
        |D|^p(ae^{i\Psi})(y)
        =e^{i\Psi(y)}\biggl[\sum_{k=0}^{2N-2}B_k(y)
        +\mc{R}_{\textup{mult},N}(y)\biggr].
    \end{equation}
    Here the coefficients are
    \begin{equation}\label{eqn:appendix_oscillatory_multiplier_coefficients}
        B_k(y)=\frac{(-i)^k}{k!}m^{(k)}(\lambda(y))
        \left.\partial_z^k
        \bigl[a(y+z)e^{i\Theta_y(z)}\bigr]\right|_{z=0},
        \qquad k\geq0,
    \end{equation}
    and the amplitude remainder satisfies, for $j=0,1$,
    \begin{equation}\label{eqn:appendix_oscillatory_multiplier_amplitude_remainder}
        |\partial_y^j\mc{R}_{\textup{mult},N}(y)|\lesssim_{N,j} A Y^{-N-j}\Lambda^{p-N},
        \qquad Y\leq y\leq2Y.
    \end{equation}
    In particular, the first two asymptotic orders are
    \begin{equation}\label{eqn:appendix_oscillatory_multiplier_two_term_expansion}
        |D|^p(ae^{i\Psi})(y)
        =e^{i\Psi(y)}\biggl[\lambda^pa-ip\lambda^{p-1}a'
        -\frac{ip(p-1)}2\lambda^{p-2}\lambda'a
        +\widetilde{\mc{R}}_{\textup{mult},2}(y)\biggr],
    \end{equation}
    where $\widetilde{\mc{R}}_{\textup{mult},2}
    =\mc{R}_{\textup{mult},2}-\frac{p(p-1)}2\lambda^{p-2}a''$ satisfies
    \[
        |\partial_y^j\widetilde{\mc{R}}_{\textup{mult},2}(y)|
        \lesssim_j AY^{-2-j}\Lambda^{p-2},\qquad j=0,1.
    \]
    Moreover, let $(a_\nu,\Psi_\nu)$, $\nu=0,1$, satisfy these hypotheses with the same $Y,\Lambda,A$ and uniform constants. Suppose that, for some $\delta\geq0$,
    \[
        |\partial_y^k(a_1-a_0)|\lesssim_k\delta A Y^{-k},
        \qquad
        |\partial_y^k(\Psi_1-\Psi_0)|\lesssim_k\delta Y^{-k},
        \qquad k\geq0,
    \]
    on $(Y/4,4Y)$. If $\mc{R}_{\textup{mult},N,\nu}$ denotes the amplitude remainder in \eqref{eqn:appendix_oscillatory_multiplier_expansion} for $(a_\nu,\Psi_\nu)$, then, for every integer $j\geq0$,
    \begin{equation}\label{eqn:appendix_oscillatory_multiplier_amplitude_difference}
        \begin{aligned}
            |\partial_y^j(\mc{R}_{\textup{mult},N,1}-\mc{R}_{\textup{mult},N,0})(y)|
            &\lesssim_{N,j}\min\{1,\delta\}\,A Y^{-N-j}\Lambda^{p-N},\\
            &\hspace{2em}Y\leq y\leq2Y.
        \end{aligned}
    \end{equation}
    The implicit constants are independent of $\delta$ and depend only on $p,N,j$.
\end{lemma}

\begin{proof}
    The singular-kernel formula gives
    \begin{equation}\label{eqn:appendix_oscillatory_multiplier_kernel}
        \begin{aligned}
            e^{-i\Psi(y)}|D|^p(ae^{i\Psi})(y)-\lambda(y)^pa(y) = -c_p\int_{\mathbb R}e^{i\lambda(y)z}
            \frac{a(y+z)e^{i\Theta_y(z)}-a(y)}{|z|^{1+p}}\,dz,
        \end{aligned}
    \end{equation}
    where $c_p$ is normalized so that
    \[
        c_p\int_{\mathbb R}\frac{1-e^{i\lambda z}}{|z|^{1+p}}\,dz
        =\lambda^p,\qquad \lambda>0.
    \]
    In particular, differentiating this identity implies that
    \begin{equation}\label{eqn:appendix_oscillatory_multiplier_moments}
        -c_p\!\int_{\mathbb R}
        e^{i\lambda z}\frac{z^k}{|z|^{1+p}}\,dz
        =(-i)^km^{(k)}(\lambda),\qquad k\geq1.
    \end{equation}
    Define the function 
    \[F_y(z) = a(y + z)e^{i\Theta_y(z)}.\]
    We Taylor expand this function at $z=0$ through degree $2N-2$ to get
    \[F_y(z) = \sum_{k=0}^{2N-2} \frac{F_y^{(k)}(0)}{k!} z^k + R_{2N-1}(z).\]
    Polynomially growing integrals below are understood in the oscillatory sense, by inserting $e^{-\eta z^2}$ and letting $\eta\downarrow0$.
    It follows from \eqref{eqn:appendix_oscillatory_multiplier_kernel} that
    \begin{align*}
        e^{-i\Psi(y)}|D|^p(ae^{i\Psi})(y)-\lambda(y)^pa(y) = \sum_{k = 1}^{2N-2}(-i)^k m^{(k)}(\lambda(y)) \frac{F_y^{(k)}(0)}{k!}
        - c_p \int_{\mathbb R} \frac{e^{i\lambda(y)z}}{|z|^{p+1}} R_{2N-1}(z) \, dz.
    \end{align*}
    Taylor's integral remainder, followed by Leibniz and the chain rule, gives
    \begin{equation}\label{eqn:appendix_oscillatory_multiplier_taylor_remainder}
        \begin{aligned}
            R_{2N-1}(z)
            &=\frac{z^{2N-1}}{(2N-2)!}\int_0^1(1-t)^{2N-2}F_y^{(2N-1)}(tz)\,dt\\
            &=(2N-1)z^{2N-1}
            \sum_{\substack{r,q_1,\ldots,q_{2N-1}\geq0\\
                r+\sum_{j=1}^{2N-1}jq_j=2N-1}}
            \frac{i^{\sum_{j=1}^{2N-1}q_j}}
            {r!\prod_{j=1}^{2N-1}q_j!(j!)^{q_j}}\\
            &\qquad\times\int_0^1(1-t)^{2N-2}e^{i\Theta_y(tz)}
            a^{(r)}(y+tz)\prod_{j=1}^{2N-1}
            \bigl(\Theta_y^{(j)}(tz)\bigr)^{q_j}\,dt.
        \end{aligned}
    \end{equation}
    We estimate each summand separately. Fix a summand and set $\ell=\sum_{j=2}^{2N-1}q_j$. For $0<t\leq1$, write
    \[
        \phi_t(z)=\lambda(y)z+\Theta_y(tz),\qquad
        b_t(z)=z^{2N-1}a^{(r)}(y+tz)
        \prod_{j=1}^{2N-1}\bigl(\Theta_y^{(j)}(tz)\bigr)^{q_j}.
    \]
    Therefore the integral we are considering corresponding to this summand is 
    \[\int e^{i\phi_t(z)}\frac{b_t(z)}{|z|^{1+p}} \, dz.\]
    Since
    \[
        |\Theta_y'(tz)|\lesssim\Lambda Y^{-1}t|z|,\qquad
        |\Theta_y^{(j)}(tz)|\lesssim_j\Lambda Y^{1-j}\quad(j\geq2),
    \]
    and $|tz|\lesssim Y$ on the support, for every $s\geq0$,
    \begin{equation}\label{eqn:appendix_oscillatory_multiplier_remainder_symbols}
        |\partial_z^s b_t(z)|
        \lesssim_{N,s}A\Lambda^{q_1+\ell}Y^{-2N+1+\ell}
        t^{q_1}|z|^{2N-1+q_1-s}.
    \end{equation}
    Moreover,
    \begin{equation}\label{eqn:appendix_oscillatory_multiplier_phase_derivative_bounds}
        \phi_t'(z)=(1-t)\lambda(y)+t\lambda(y+tz)\approx\Lambda,
        \qquad
        \left|\partial_z^s\frac1{\phi_t'(z)}\right|
        \lesssim_s\Lambda^{-1}(t/Y)^s
        \lesssim_s\Lambda^{-1}|z|^{-s}.
    \end{equation}
    Set 
    \[K=2N-1+q_1, \quad B=A\Lambda^{q_1+\ell}Y^{-2N+1+\ell}t^{q_1}.\]
    We decompose the integral
    \[\int e^{i\phi_t(z)}\frac{b_t(z)}{|z|^{1+p}} \, dz = \int \chi_\Lambda e^{i\phi_t(z)}\frac{b_t(z)}{|z|^{1+p}} \, dz + \int (1 - \chi_\Lambda) e^{i\phi_t(z)}\frac{b_t(z)}{|z|^{1+p}} \, dz\]
    where $\chi_\Lambda$ is a smooth bump function supported on $[-2/\Lambda, 2/\Lambda]$ and equal to $1$ on $[-1/\Lambda, 1/\Lambda]$. To bound the integral on the support of $\chi_\Lambda$,  we can bound the integral absolutely. To bound the second integral, we integrate by parts $L>K-p$ times. By \eqref{eqn:appendix_oscillatory_multiplier_remainder_symbols} and \eqref{eqn:appendix_oscillatory_multiplier_phase_derivative_bounds},
    \begin{equation}\label{eqn:appendix_oscillatory_multiplier_remainder_summand}
        \begin{aligned}
            \left|\int_{\mathbb R}e^{i\phi_t(z)}\frac{b_t(z)}{|z|^{1+p}}\,dz\right|
            &\lesssim_N B\int_0^{2/\Lambda}z^{K-1-p}\,dz
            +B\Lambda^{-L}\int_{1/\Lambda}^{\infty}z^{K-1-p-L}\,dz\\
            &\lesssim_N B\Lambda^{p-K}
            =At^{q_1}Y^{-2N+1+\ell}\Lambda^{p-2N+1+\ell}\\
            &\leq AY^{-N}\Lambda^{p-N}.
        \end{aligned}
    \end{equation}
    The last inequality follows from $\ell\leq N-1$, $t\leq1$, and $Y\Lambda\geq1$ and the estimates are uniform in $t$. Therefore,
    \[
        |\mc{R}_{\textup{mult},N}(y)|
        =\left|-c_p\int_{\mathbb R}e^{i\lambda(y)z}
        \frac{R_{2N-1}(z)}{|z|^{1+p}}\,dz\right|
        \lesssim_N AY^{-N}\Lambda^{p-N}.
    \]
    The estimates for the derivatives of $\mc{R}_{\textup{mult},N}$ follow similarly.

    To prove \eqref{eqn:appendix_oscillatory_multiplier_amplitude_difference}, first assume $\delta\leq1$.
    Given $\nu$, for the corresponding summands, we denote the phases and amplitudes by $\phi_{\nu,t}$ and $b_{\nu,t}$, and write
    \[
        I_\nu(y,t)=\int_{\mathbb R}e^{i\phi_{\nu,t}(z)}
        \frac{b_{\nu,t}(z)}{|z|^{1+p}}\,dz,\qquad \nu=0,1.
    \]
    The assumptions on the amplitude and phase differences give
    \[
        |\Theta_{1,y}'(tz)-\Theta_{0,y}'(tz)|\lesssim\delta Y^{-2}t|z|,
        \qquad
        |\Theta_{1,y}^{(k)}(tz)-\Theta_{0,y}^{(k)}(tz)|
        \lesssim_k\delta Y^{-k}\quad(k\geq2).
    \]
    It follows that the difference $b_{1,t}-b_{0,t}$ gains a factor $\delta$ in the sense that 
    \[|\partial_z^{s}(b_{1,t}-b_{0,t})|\lesssim\delta B|z|^{K-s}.\]
    Moreover, writing $\Delta\Psi=\Psi_1-\Psi_0$,
    \[
        \phi_{1,t}(z)-\phi_{0,t}(z)
        =(1-t)\Delta\Psi'(y)z+\Delta\Psi(y+tz)-\Delta\Psi(y).
    \]
    Therefore, we have that
    \begin{equation}\label{eqn:appendix_oscillatory_multiplier_difference_symbols}
        \begin{aligned}
            |\partial_y^j\partial_z^s(b_{1,t}-b_{0,t})|
            &\lesssim_{N,j,s}\delta BY^{-j}|z|^{K-s},\\
            |\partial_y^j\partial_z^s(\phi_{1,t}-\phi_{0,t})|
            &\lesssim_{j,s}\delta Y^{-j-1}|z|^{1-s},
            \qquad j,s\geq0.
        \end{aligned}
    \end{equation}
    The identity $e^{iu}-e^{iv}=i(u-v)\int_0^1e^{i((1-s)v+su)}\,ds$ yields
    \begin{equation}\label{eqn:appendix_oscillatory_multiplier_difference_integrals}
        \begin{aligned}
            I_1-I_0
            &=\int_{\mathbb R}e^{i\phi_{1,t}}
            \frac{b_{1,t}-b_{0,t}}{|z|^{1+p}}\,dz\\
            &\quad+i\int_0^1\int_{\mathbb R}
            e^{i((1-s)\phi_{0,t}+s\phi_{1,t})}
            \frac{(\phi_{1,t}-\phi_{0,t})b_{0,t}}{|z|^{1+p}}\,dz\,ds.
        \end{aligned}
    \end{equation}
    Every phase in this identity has derivative comparable to $\Lambda$, with the same inverse-derivative bounds as in \eqref{eqn:appendix_oscillatory_multiplier_phase_derivative_bounds}. By the integration by parts argument used to obtain \eqref{eqn:appendix_oscillatory_multiplier_remainder_summand}, we have that
    \begin{equation}\label{eqn:appendix_oscillatory_multiplier_difference_summand}
        \begin{aligned}
            |\partial_y^j(I_1-I_0)|
            &\lesssim_{N,j}\delta BY^{-j}\Lambda^{p-K}
            +\delta BY^{-j-1}\Lambda^{p-K-1}\\
            &\lesssim_{N,j}\delta AY^{-N-j}\Lambda^{p-N}.
        \end{aligned}
    \end{equation}
    This concludes the proof for $\delta \leq 1$. For $\delta \geq 1$, we can use the triangle inequality instead.
\end{proof}

\section{Lean formalization}\label{app:lean}

The main result of this paper, Theorem~\ref{thm:main}, has been verified in Lean~4 using its mathematical library Mathlib. The objective of this appendix is to explain the correspondence between the mathematical statement and its Lean formulation without assuming broad familiarity with Lean. We repeat the relevant mathematical definitions alongside their formal counterparts. The Lean code is available at the accompanying repository:
\begin{center}
  \url{https://github.com/javigomez2001/gSQGFormalization}.
\end{center}
The repository follows the structure of the paper: the folders \lean{Sec1_Introduction}, \lean{Sec3_FixedPoint}, \ldots, \lean{Sec7_NonlinearEstimates} and \lean{AppendixAB} contain the formal counterparts of the corresponding sections, and the file \texttt{docs/PAPER\_MAP.md} lists every numbered statement of the paper, from Theorem~\ref{thm:main} to Lemma~\ref{lem:appendix_oscillatory_multiplier}, together with its Lean counterpart.

\subsection{Reading Lean statements}\label{app:lean-reading}

Lean records the type of each object. Writing \lean{p : ℝ} means that $p\in\mathbb R$, and \lean{f : ℝ → ℝ} denotes a function $f:\mathbb R\to\mathbb R$. A hypothesis is a named argument whose type is a proposition: \lean{(hp : 3 / 4 < p)} reads ``assume $3/4<p$, and call this assumption \lean{hp}''. A definition is introduced by \lean{def}. The symbols \lean{∀}, \lean{∃}, \lean{∧} and \lean{→} are the usual quantifiers and connectives, and the order in which the quantifiers appear is the order in which the corresponding choices are made. In particular, a constant introduced by \lean{∃} before a variable is independent of that variable.

We use the following Mathlib notation. \lean{deriv g y} is the derivative $g'(y)$ and \lean{iteratedDeriv j g y} is $g^{(j)}(y)$. \lean{ContDiff ℝ (⊤ : ℕ∞) g} means that $g\in C^\infty$, and \lean{ContDiffOn ℝ k G (univ ×ˢ I)} means that $G$ is of class $C^k$ on $\mathbb R\times I$. \lean{Iio T} is the interval $(-\infty,T)$ and \lean{Ioo a b} is $(a,b)$. The integral \lean{∫ z, h z} is the Lebesgue integral over $\mathbb R$. By convention it takes the value $0$ when $h$ is not integrable. Finally, \lean{ENNReal.ofReal} maps $[0,\infty)$ into $[0,\infty]$, and \lean{⨆ x ∈ A, g x} is the supremum of $g$ over $A$, taken in $[0,\infty]$.

\subsection{The equation}\label{app:lean-equation}

Throughout the formalization we write $\alpha=p+1$, so that $\alpha\in(\frac74,2)$ corresponds to $p\in(\frac34,1)$. The constant $c_\alpha$ of \eqref{eqn:front_equation} is
\[
 c_\alpha=\frac{\Gamma(1+p)\sin(\pi p/2)}{\pi},
\]
the normalizing constant of $|D|^p$ used in Section~\ref{sec:main_ideas}. The front equation is written using the secant $S_F(x,x')=\frac{F(x)-F(x')}{x-x'}$ and the factor $A(u)=(1+u^2)^{-\alpha/2}$, so that
\[
 \frac{\partial_x F(x)-\partial_x F(x')}{\big((x-x')^2+(F(x)-F(x'))^2\big)^{\alpha/2}}
 =\frac{\partial_x F(x)-\partial_x F(x')}{|x-x'|^{\alpha}}\,A\big(S_F(x,x')\big).
\]
In Lean, these objects read as follows (we display the definitions with $\epsilon=1$; the same definitions with general $\epsilon$ are used for the rescaled equation of Section~\ref{sec:fixed_point}):

\par
\noindent\begin{minipage}{\linewidth}
\begin{leancode}
def calpha (p : ℝ) : ℝ := Real.Gamma (1 + p) * Real.sin (π * p / 2) / π

def A_eps (p ε u : ℝ) : ℝ := (1 + ε * u ^ 2) ^ (-(p + 1) / 2)

def c01_secant (f : ℝ → ℝ) (y z : ℝ) : ℝ :=
  if y = z then deriv f y else (f y - f z) / (y - z)

def c01_Jfull_eps (p ε : ℝ) (f : ℝ → ℝ) (y : ℝ) : ℝ :=
  calpha p * ∫ z, (deriv f y - deriv f z) / |y - z| ^ (p + 1) * A_eps p ε (c01_secant f y z)
\end{leancode}
\end{minipage}
\par

A classical solution of \eqref{eqn:front_equation} on $\mathbb R\times I$, for a time interval $I$, is encoded by the following predicate. It requires $F\in C^1(\mathbb R\times I)$, $F(\cdot,t)\in C^2(\mathbb R)$ with $\partial_x^2F$ continuous on $\mathbb R\times I$, $\partial_xF(\cdot,t)$ bounded, and the equation at every point:

\par
\noindent\begin{minipage}{\linewidth}
\begin{leancode}
def gsqg_front_equation (p : ℝ) (I : Set ℝ) (F : ℝ → ℝ → ℝ) : Prop :=
  ContDiffOn ℝ 1 (fun q : ℝ × ℝ => F q.1 q.2) (univ ×ˢ I) ∧
  (∀ t ∈ I, ContDiff ℝ 2 (fun x => F x t)) ∧
  ContinuousOn (fun q : ℝ × ℝ => iteratedDeriv 2 (fun x => F x q.2) q.1) (univ ×ˢ I) ∧
  (∀ t ∈ I, ∃ M : ℝ, ∀ x, |deriv (fun x => F x t) x| ≤ M) ∧
  (∀ t ∈ I, ∀ x, deriv (fun s => F x s) t = c01_Jfull_eps p 1 (fun x' => F x' t) x)
\end{leancode}
\end{minipage}
\par

To rule out any role of the convention on non-integrable integrals, the formalization proves that every solution in this sense satisfies \eqref{eqn:front_equation} literally, with an absolutely convergent integral:

\par
\noindent\begin{minipage}{\linewidth}
\begin{leancode}
theorem c01_front_equation_paper (p : ℝ) (hp : 0 < p) (hp1 : p < 1)
    (I : Set ℝ) (F : ℝ → ℝ → ℝ) (hF : gsqg_front_equation p I F) :
    ∀ t ∈ I, ∀ x : ℝ,
      Integrable (fun z => (deriv (fun y => F y t) x - deriv (fun y => F y t) z) /
        ((x - z) ^ 2 + (F x t - F z t) ^ 2) ^ ((p + 1) / 2)) ∧
      deriv (fun s => F x s) t = calpha p * ∫ z,
        (deriv (fun y => F y t) x - deriv (fun y => F y t) z) /
          ((x - z) ^ 2 + (F x t - F z t) ^ 2) ^ ((p + 1) / 2)
\end{leancode}
\end{minipage}
\par

\subsection{The self-similar ansatz}\label{app:lean-ansatz}

The ansatz \eqref{eqn:intro_self_similar_ansatz}, the corner $x_-=\max\{-x,0\}$, the Japanese bracket $\ang{y}=\sqrt{1+y^2}$ and the exponent $\beta=\frac{\alpha}{2(\alpha-1)}=\frac{p+1}{2p}$ are defined by

\par
\noindent\begin{minipage}{\linewidth}
\begin{leancode}
def corner (y : ℝ) : ℝ := max (-y) 0
def jb (y : ℝ) : ℝ := Real.sqrt (1 + y ^ 2)
def beta (p : ℝ) : ℝ := (p + 1) / (2 * p)
def self_similar_ansatz_L (p T t : ℝ) : ℝ := ((p + 1) * (T - t)) ^ (1 / (p + 1))
def self_similar_ansatz (p ε T : ℝ) (f : ℝ → ℝ) (x t : ℝ) : ℝ :=
  Real.sqrt ε * self_similar_ansatz_L p T t * f (x / self_similar_ansatz_L p T t)
\end{leancode}
\end{minipage}
\par

so that \lean{self_similar_ansatz p ε T f x t} is $\sqrt\epsilon\,L(t)\,f(x/L(t))$ with $L(t)=(\alpha(T-t))^{1/\alpha}$. In the profile bounds of Theorem~\ref{thm:main} the exponents $-\alpha+1-j$ and $-\frac{\alpha}{2(\alpha-1)}+\frac{j}{\alpha-1}$ become \lean{-p - j} and \lean{-beta p + j / p}.

\subsection{The formal theorem}\label{app:lean-main}

With these definitions, Theorem~\ref{thm:main} has the following Lean statement.

\par
\noindent\begin{minipage}{\linewidth}
\begin{leancode}
theorem main (p : ℝ) (hp : 3 / 4 < p) (hp1 : p < 1) :
    ∃ ε₀ : ℝ, 0 < ε₀ ∧ ∃ C : ℝ, 0 < C ∧
      ∀ ε : ℝ, 0 < ε → ε < ε₀ →
        ∃ f : ℝ → ℝ, ContDiff ℝ (⊤ : ℕ∞) f ∧
          -- (1) profile bounds
          (∀ y : ℝ, y ≤ 0 → ∀ j : ℕ, j ≤ 2 →
              |iteratedDeriv j (fun y' => f y' + y') y| ≤ C * jb y ^ (-p - (j : ℝ))) ∧
          (∀ y : ℝ, 0 ≤ y → |f y| ≤ C * jb y ^ (-p)) ∧
          (∀ y : ℝ, 0 ≤ y → ∀ j : ℕ, j ≤ 1 →
              |iteratedDeriv (j + 1) f y| ≤ C * jb y ^ (-beta p + (j : ℝ) / p)) ∧
          -- (2) classical, smooth solution of the front equation
          (∀ T : ℝ, 0 < T →
              gsqg_front_equation p (Iio T) (self_similar_ansatz p ε T f) ∧
              ContDiffOn ℝ (⊤ : ℕ∞) (fun q : ℝ × ℝ => self_similar_ansatz p ε T f q.1 q.2)
                (univ ×ˢ Iio T)) ∧
          -- (3) corner convergence
          (∀ T : ℝ, 0 < T → ∀ t : ℝ, t < T → ∀ x : ℝ,
              |self_similar_ansatz p ε T f x t - Real.sqrt ε * corner x| ≤
                C * Real.sqrt ε * (T - t) ^ (1 / (p + 1))) ∧
          -- (4) curvature blow-up
          (∃ c : ℝ, 0 < c ∧ ∀ T : ℝ, 0 < T → ∀ δ : ℝ, 0 < δ →
              ∃ tδ : ℝ, tδ < T ∧ ∀ t : ℝ, tδ < t → t < T →
                ENNReal.ofReal (c * (T - t) ^ (-1 / (p + 1))) ≤
                  ⨆ x ∈ Ioo (-δ) δ,
                    ENNReal.ofReal |iteratedDeriv 2 (fun x' => self_similar_ansatz p ε T f x' t) x|)
\end{leancode}
\end{minipage}
\par

The arguments \lean{(p : ℝ)}, \lean{(hp : 3 / 4 < p)} and \lean{(hp1 : p < 1)} fix $\alpha=p+1\in(\frac74,2)$. All other data are conclusions: there are $\epsilon_0>0$ and a constant $C>0$ (the implicit constant of the symbols $\lesssim$ in Theorem~\ref{thm:main}) such that, for every $0<\epsilon<\epsilon_0$, there is a profile $f=f_\epsilon\in C^\infty(\mathbb R)$ with the following properties.
\begin{enumerate}
\item[(1)] The three lines after \lean{-- (1)} are the profile bounds of Theorem~\ref{thm:main}, for $y\le 0$ and for $y\ge 0$.
\item[(2)] For every $T>0$, the function $F(x,t)=\sqrt\epsilon\,L(t)f_\epsilon(x/L(t))$ is a classical solution of \eqref{eqn:front_equation} on $\mathbb R\times(-\infty,T)$ in the sense of Subsection~\ref{app:lean-equation}, and it is jointly $C^\infty$ in $(x,t)$ on $\mathbb R\times(-\infty,T)$.
\item[(3)] The corner convergence \eqref{eqn:intro_corner_convergence}, written as a pointwise bound for all $x$; this is equivalent to the $L^\infty$ bound since $F(\cdot,t)$ is continuous.
\item[(4)] The curvature blow-up \eqref{eqn:intro_curvature_blowup}. The supremum of $|\partial_x^2F(\cdot,t)|$ over $(-\delta,\delta)$ is taken in $[0,\infty]$, and it is the $L^\infty(-\delta,\delta)$ norm since $\partial_x^2F(\cdot,t)$ is continuous.
\end{enumerate}
The order of the quantifiers makes explicit several uniformity properties that are implicit in Theorem~\ref{thm:main}: the constant $C$ is independent of $\epsilon$ and $T$; the profile $f_\epsilon$ is chosen before $T$, so that the solutions for different blow-up times are rescalings and time translations of each other; and the constant $c$ in (4) depends only on $\epsilon$, and not on $T$ or $\delta$. The remark following Theorem~\ref{thm:main}, on the desingularization of a corner, is formalized as a separate theorem, \lean{c01_desingularization}.

\subsection{Organization of the formal proof}\label{app:lean-organization}

The formal proof follows the paper. The theorem \lean{main} is derived, as in Section~\ref{sec:fixed_point}, from the formal counterparts of Theorems~\ref{thm:fixed_point_approximation}, \ref{thm:fixed_point_linear_inverse} and \ref{thm:fixed_point_nonlinear_estimates}, the fixed point argument of Theorem~\ref{thm:fixed_point_correction}, and the smoothness result of Theorem~\ref{thm:fixed_point_smoothness}. The proof was organized with a blueprint, in which every definition and result of the paper, together with the intermediate lemmas needed to formalize it, is a node linked to its Lean declaration; the final blueprint has 553 nodes, all of which are formally proved.

\subsection{Beyond the paper}\label{app:lean-extensions}

The formalization contains two results that are used, but not proved, in this paper.

First, the interior Schauder estimate for $|D|^p$ quoted from \cite{RosOton-Serra:dirichlet-fractional-laplacian} in the proof of Theorem~\ref{thm:fixed_point_smoothness} is proved in dimension one, for smooth functions and for exponents $\beta,\beta+p\notin\mathbb N$, by a direct argument based on annular Littlewood--Paley blocks, the vanishing moments of their kernels and a cutoff commutator (folder \lean{SchauderEstimate}). The approximation argument that removes the smoothness assumption is formalized as well.

Second, the formalization derives the front equation \eqref{eqn:front_equation} from the patch formulation of \eqref{eqn:gsqg} (folder \lean{PatchEquation}). It defines the two-dimensional operator $u=-\nabla^\perp(-\Delta)^{-1+\alpha/2}\theta$ through its singular kernel, identifies its Fourier multiplier, and proves that, for every graph $F$ that is smooth on $\mathbb R\times I$ and has bounded slope on compact time intervals, the indicator function of $\{x_2<F(x_1,t)\}$ is a weak solution of \eqref{eqn:gsqg} if and only if $F$ satisfies \eqref{eqn:front_equation}, with the constant $c_\alpha$ given by the normalization of the kernel. As a consequence, the solutions of Theorem~\ref{thm:main} give, after the explicit rescaling of time that relates the two normalizations, weak patch solutions of \eqref{eqn:gsqg} whose boundary is smooth for $t<T$ and forms a corner at $t=T$ (theorems \lean{c13_patch_front_equivalence} and \lean{c13_patch_main}).

\bibliographystyle{amsalpha}
\bibliography{ref}

\end{document}